\documentclass[a4paper, 12pt]{book}

\usepackage{ucs}
\usepackage[utf8x]{inputenc}
\usepackage[T1]{fontenc}
\usepackage{lmodern}

\usepackage[english,francais]{babel}

\usepackage{amsmath, amsfonts, amssymb, amsthm}
\usepackage{shuffle}
\usepackage{stmaryrd}
\usepackage[all]{xy} 

\usepackage{vmargin}
\usepackage[Lenny]{fncychap}
\usepackage{fancyhdr}

\usepackage{graphicx}

\usepackage{color}
\usepackage[dvipsnames]{xcolor}

\usepackage[pagebackref = true, hyperindex = true, frenchlinks = true,
colorlinks = true, citecolor = NavyBlue, linkcolor = Mulberry,
urlcolor = LimeGreen, linktocpage]{hyperref}

\usepackage{makeidx}
\usepackage[nottoc,notlof,notlot]{tocbibind} 

\usepackage{tikz}
\usetikzlibrary{positioning,shapes,shadows,arrows}
\usepackage{dpfloat}

\usepackage{enumitem}
\usepackage{array}
\usepackage{youngtab}
\usepackage{rotating}
\usepackage{hhline}
\usepackage{listings}
\newcommand{\Auteur}{Viviane Pons}
\newcommand{\Titre}{Combinatoire algébrique liée aux ordres sur les permutations}
\newcommand{\TitreEN}{Algebraic combinatorics on orders of permutations}
\newcommand{\Date}{7 octobre 2013}
\author{\Auteur}
\title{\Titre}
\date{\today}

\newcommand{\ConcDec}{~\vec{.}~} 
\newcommand{\Dec}[1]{\vec{#1}}  
\newcommand{\Ret}[1]{\widetilde{#1}} 
\DeclareMathOperator{\std}{std} 

\newcommand{\Sym}[1]{\mathfrak{S}_{#1}} 
\newcommand{\kSym}[2]{\mathfrak{S}_{#1}^{#2}} 
\newcommand{\bSym}[2]{{}^{#2}\!\mathfrak{S}_{#1}}
\newcommand{\kbSym}[3]{{}^{#3}\!\mathfrak{S}_{#1}^{#2}}
\newcommand{\brho}[1]{{}^{#1}\!\rho} 
\newcommand{\krho}[1]{\rho^{#1}}
\newcommand{\kbrho}[2]{{}^{#2}\!\rho^{#1}}
\newcommand{\mrho}[1]{\widetilde{\rho}^{#1}}
\DeclareMathOperator{\coinv}{coinv} 
\DeclareMathOperator{\inv}{inv} 
\DeclareMathOperator{\cle}{cle} 
\newcommand{\pmax}{\vee} 
\newcommand{\pmin}{\wedge} 
\newcommand{\infd}{\leq_R} 
\newcommand{\infg}{\leq_L}
\newcommand{\infb}{\leq_B}

\newcommand{\syng}[1]{~\tiny{\yng(#1)}} 
\newcommand{\hpi}{\hat{\pi}} 
\newcommand{\hK}{\hat{K}} 

\DeclareMathOperator{\Succs}{Succs} 
\newcommand{\W}{W} 
\newcommand{\tprec}{\prec} 

\newcommand{\E}[1]{\mathfrak{E}_{#1}} 
\newcommand{\ttprec}{\vartriangleleft} 
 
\newcommand{\vleq}{\begin{turn}{-90}$\leq$\end{turn}} 
\newcommand{\vgeq}{\begin{turn}{-90}$\geq$\end{turn}}
\newcommand{\veq}{\begin{turn}{-90}$=$\end{turn}}
\newcommand{\bblock}{\beta} 
\newcommand{\Ww}{\mathcal{W}} 

\newcommand{\noeud}{n\oe{}ud } 
\newcommand{\noeuds}{n\oe{}uds }
\newcommand{\trprec}{\vartriangleleft} 
\newcommand{\trsucc}{\vartriangleright} 
\newcommand{\ntrprec}{\ntriangleleft}

\newcommand{\cp}{\hat{+}} 
\DeclareMathOperator{\ABR}{ABR} 
\DeclareMathOperator{\ABD}{ABD} 
\DeclareMathOperator{\shape}{shape} 
\DeclareMathOperator{\exec}{exec} 
\DeclareMathOperator{\BF}{\mathbf{F}} 
\DeclareMathOperator{\BG}{\mathbf{G}} 
\DeclareMathOperator{\BE}{\mathbf{E}} 
\DeclareMathOperator{\BH}{\mathbf{H}} 
\DeclareMathOperator{\BP}{\mathbf{P}} 

\newcommand{\Aa}{\mathcal{A}} 

\DeclareMathOperator{\FQSym}{\mathbf{FQSym}}
\DeclareMathOperator{\PBT}{\mathbf{PBT}}
\newcommand{\FQSymm}{\FQSym^{(m)}} 

\newcommand{\PBTm}{\PBT^{(m)}} 

\DeclareMathOperator{\OB}{B} 
\DeclareMathOperator{\OBT}{\mathcal{B}} 
\newcommand\dec{F_{\ge}} 
\newcommand\inc{F_{\le}} 
\DeclareMathOperator{\BI}{\mathbf{I}} 
\DeclareMathOperator{\Itrees}{trees} 
\DeclareMathOperator{\Isize}{size} 
\DeclareMathOperator{\PI}{\mathcal{P}} 
\newcommand{\BB}{\mathbb{B}} 
\newcommand{\SI}{\mathbb{S}} 
\newcommand{\pleft}{\vec{\bullet}} 
\newcommand{\pright}{\overleftarrow{\delta}} 
\newcommand{\polleft}{\succ}
\newcommand{\polright}{\prec_\delta}

\newcommand{\mpaths}{\emph{$m$-ballot paths}} 
\newcommand{\mpath}{\emph{$m$-ballot path}}
\newcommand{\kpaths}[1]{\emph{#1-ballot paths}}
\newcommand{\kpath}[1]{\emph{#1-ballot path}}
\newcommand{\Tamnm}{\mathcal{T}_{n}^{(m)}} 
\newcommand{\Tam}[2]{\mathcal{T}_{#1}^{(#2)}}
\DeclareMathOperator{\OBm}{B^{(m)}} 
\newcommand{\OBk}[1]{\OB^{(#1)}} 
\newcommand{\Phim}{\Phi^{(m)}} 

\newcommand{\OBTm}[1]{\OBT_{#1}^{(m)}} 
\newcommand{\OBTk}[2]{\OBT_{#1}^{(#2)}}
\newcommand{\polrightx}{\prec_{\frac{\delta}{x}}} 
\newcommand{\prightx}{\tfrac{\pright}{x}}
\newcommand{\BBm}{\BB^{(m)}} 
\newcommand{\BBk}[1]{\BB^{(#1)}}
\DeclareMathOperator{\BR}{R} 
\newcommand{\PIm}{\PI^{(m)}} 

\newcommand{\lret}{\widetilde{\leq}}
\newcommand{\stdm}{\std^{(m)}} 

\newcommand{\BGm}[1]{\BG_{#1}^{(m)}} 
\newcommand{\BFm}[1]{\BF_{#1}^{(m)}} 
\newcommand{\BPm}[1]{\BP_{#1}^{(m)}} 
\newcommand{\Ti}{\widetilde{T}} 

\newcommand{\NN}{\mathbb{N}}

\newcommand{\CC}{\mathbb{C}}
\newcommand{\ZZ}{\mathbb{Z}}

\newcommand{\PP}{\mathbb{P}}
\newcommand{\KK}{\mathbb{K}}

\definecolor{Noir}{RGB}{0,0,0}
\definecolor{Rouge}{RGB}{205,35,38}
\definecolor{Bleu}{RGB}{2,60,195}
\definecolor{Vert}{RGB}{23,103,1}
\definecolor{Orange}{RGB}{255,113,15}
\definecolor{Blanc}{RGB}{255,255,255}

\newcommand{\red}[1]{\textbf{\textcolor{red}{#1}}}
\newcommand{\sred}[1]{\textcolor{red}{#1}}
\newcommand{\blue}[1]{\textcolor{blue}{#1}}
\newcommand{\green}[1]{\textcolor{Vert}{#1}}

\tikzstyle{Coupled} = [opacity = .5, color=green, line width = 2]

\tikzstyle{Red} = [color = red]
\tikzstyle{Blue} = [color = blue]
\tikzstyle{Green} = [color = Vert]
\tikzstyle{Gray} = [color = gray]

\tikzstyle{Path} = [line width = 1.2]
\tikzstyle{StrongPath} =  [line width=2.5]
\tikzstyle{DPoint} = [fill, radius=0.1]

\tikzstyle{Point} = [fill, radius=0.08]
\tikzstyle{BigPoint} = [fill, radius=0.13]
\tikzstyle{Leaf} = [color = gray]
\tikzstyle{Line1} = [dashed]
\tikzstyle{Line2} = [dotted, ultra thick]

\tikzstyle{SageClass}=[rectangle, draw=black, rounded corners, fill=blue!30, drop shadow, anchor=north, rectangle split, rectangle split parts=2]

\tikzstyle{SimpleSageClass}=[rectangle, draw=black, rounded corners, fill=blue!30, drop shadow, anchor=north]

\tikzstyle{coercionArrow}=[thick, color=red]

\tikzstyle{Chap}=[rectangle,draw=blue!100,fill=blue!20,thick,text width=3.8cm,text centered,line width=1.5pt, font=\scriptsize]
\tikzstyle{ArrowChap} = [Rouge!80, thick, draw, line width = 2pt]

\newtheorem{Theoreme}{Théorème}[section]
\newtheorem{Corollaire}[Theoreme]{Corollaire}
\newtheorem{Proposition}[Theoreme]{Proposition}
\newtheorem{Lemme}[Theoreme]{Lemme}
\newtheorem{Definition}[Theoreme]{Définition}

\theoremstyle{remark}
\newtheorem{Remarque}[Theoreme]{Remarque}
\newtheorem{Exemple}[Theoreme]{Exemple}

\numberwithin{equation}{chapter}

\setmarginsrb{3.2cm}{2cm}{3.2cm}{2.5cm}{1cm}{1cm}{1cm}{1cm}

\makeindex

\begin{document}

\frontmatter
\lhead[\oldstylenums \thepage]{\rightmark}
\rhead[\leftmark]{\oldstylenums \thepage}


\thispagestyle{empty}

\begin{center}
    \huge
    \textsc{Thèse de doctorat}
\end{center}

\begin{center}
\includegraphics[width=130px,height=130px]{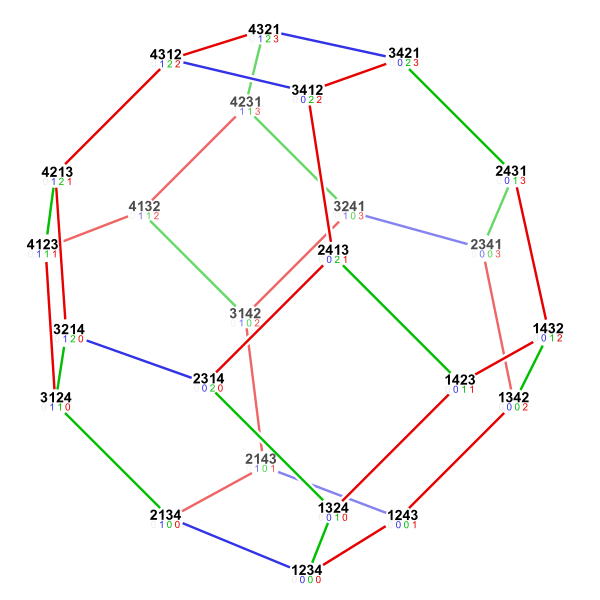}
\end{center}

\begin{center}
    \huge
    \textsc{\Titre}
\end{center}

\begin{center}
    \large
    pour l'obtention du grade de
\end{center}

\begin{center}
    \Large
    Docteur de l'université Paris-Est
\end{center}

\begin{center}
    \large
    \textbf{Spécialité Informatique} \\
\footnotesize
\'Ecole Doctorale de Mathématiques et des Sciences et Techniques de l'Information et de la Communication
\end{center}


\vspace{1em}

\begin{center}
    \large
    Présentée et soutenue publiquement par
\end{center}

\begin{center}
    \Large
    \Auteur
\end{center}

\begin{center}
    \large
    le \Date
\end{center}

\begin{center}
    \large
    \textbf{Devant le jury composé de}
\end{center}


\begin{center}
    \large
    \begin{tabular}{ll}
        Jean-Christophe Aval    \qquad & \qquad \'Examinateur \\
		François Bergeron       \qquad & \qquad Rapporteur \\
        Frédéric Chapoton       \qquad & \qquad Rapporteur \\
        Sylvie Corteel          \qquad & \qquad \'Examinateur \\
        Jean-Christophe Novelli \qquad & \qquad Directeur de thèse \\
        Frédéric Patras         \qquad & \qquad \'Examinateur \\
        Jean-Yves Thibon        \qquad & \qquad Directeur de thèse \\
    \end{tabular}
\end{center}


\vspace{0.5em}

\begin{center}
    \footnotesize
    Laboratoire d'informatique Gaspard-Monge \\
    UMR 8049 LIGM \\
    5, bd Descartes, Champs-sur-Marne, 77454 Marne-la-Vallée Cedex 2, France
\end{center}

\vspace*{\fill}

\cleardoublepage

\lhead[\oldstylenums \thepage]{}
\rhead[Remerciements]{\oldstylenums \thepage}

\begin{small}
\textit{Pour Elyah Ferrari Bullet, ma filleule qui aura un an le 8 octobre.}
\section*{Remerciements}

C'est avec beaucoup d'émotion que j'arrive aujourd'hui au bout de mes trois années de thèse, conclusion de ma vie d'étudiante et début de ma vie de chercheur. Ces trois années et les années d'études qui ont précédé ont été pour moi source de beaucoup de bonheur. Le plaisir d'étudier et d'apprendre ne s'est jamais tari et, je l'espère, m'accompagnera encore longtemps. Ce plaisir, je le dois aussi aux personnes que j'ai côtoyées, qui m'ont donné envie de continuer et que je voudrais remercier aujourd'hui.

Ma première pensée va à Alain Lascoux. Il est celui qui manque à cette soutenance, celui qui aurait dû me poser des questions de géométrie algébrique (et y répondre). Je sais qu'aujourd'hui il est malade et je souhaite de tout mon c\oe{}ur qu'il nous revienne. Cette dernière année, il a été souvent absent et il nous a manqué à tous. Sans lui, le séminaire du vendredi n'est jamais au complet. C'est à ce séminaire que je l'ai vu pour la première fois alors que je finissais ma première année de master. J'étais tout à fait incapable de comprendre ce dont il parlait. Je crois qu'une de mes plus grande fierté est d'avoir aujourd'hui l'impression, parfois, un peu, de temps en temps, de suivre le fil de ses explications. J'ai été particulièrement impressionnée lorsque j'ai commencé à travailler avec lui au début de ma thèse. Je me sentais minuscule face à son monceau de connaissances (et c'est toujours le cas), face à son esprit vif de combinatoriste, à sa compréhension profonde des objets, à sa vision globale des problèmes au delà de la simple question posée. Je dois dire qu'honnêtement, pour ces mêmes raisons, ça n'a pas toujours été facile de travailler avec lui. Mais il a fait preuve avec moi, comme je pense avec tous ses étudiants, d'une très grande patience car il a le désir très fort de transmettre sa vision des mathématiques. J'ai appris énormément avec lui et je suis fière que l'on puisse remarquer sa présence et son influence dans mon travail. Je le remercie de tout mon c\oe{}ur pour ce qu'il m'a apporté et pour tout ce qu'il doit encore m'apprendre.

Mes prochains remerciements vont évidemment à mes deux directeurs de thèse Jean-Yves Thibon et Jean-Christophe Novelli. J'ai d'abord rencontré Jean-Yves pour mon premier projet de recherche en M1, puis, très vite, Jean-Christophe à travers le séminaire du vendredi matin. Très tôt, ils ont su m'accueillir dans l'équipe, me donner confiance en moi, confirmer mon désir de continuer en thèse. Ils savent orienter sans diriger pour laisser à chacun de leurs étudiants la possibilité de trouver sa propre voie dans la recherche. Ils ont été tous les deux présents tout au long de ma thèse pour me conseiller scientifiquement et humainement, relire scrupuleusement mes articles et mes exposés. Plus récemment, j'ai eu le plaisir de collaborer avec eux et j'espère avoir encore de nombreuses occasions de le faire. Je pense qu'on leur doit en grande partie l'ambiance si agréable de l'équipe de Marne-la-Vallée, en particulier, grâce au séminaire du vendredi matin. Je les remercie chaleureusement de leur soutien et de leur présence ces trois dernières années.

Je remercie aussi François Bergeron et Frédéric Chapoton d'avoir accepté d'\^etre les rapporteurs de mon manuscrit. J'ai beaucoup apprécié l'accueil que j'ai reçu dans l'équipe de François Bergeron, au LACIM, en 2011 et j'espère que nous aurons d'autres occasions de travailler ensemble. Frédéric Chapoton m'a invitée récemment dans son équipe à l'Institut Camille Jordan et j'ai l'espoir que cette thèse soit le point de départ de nombreuses et fructueuses collaborations. Je remercie aussi Jean-Christophe Aval, Sylvie Corteel et Frédéric Patras d'avoir accepté de participer à mon jury. Je connais personnellement Jean-Christophe Aval et Sylvie Corteel avec qui j'ai de très bons souvenirs, que ce soit à l'assaut des volcans avec Jean-Christophe ou plus simplement autour de la machine à café du LIAFA avec Sylvie. Je tiens d'ailleurs à remercier l'ensemble de l'équipe du LIAFA pour l'accueil chaleureux qui m'a été fait lors du groupe de travail du mardi matin. Je souhaite aussi remercier Florent Hivert et Nicolas Thierry du LRI. Gr\^ace à eux, j'ai pu m'initier à Sage et rejoindre la grande famille de Sage-Combinat. \`A travers la communauté Sage et gr\^ace à leurs invitations, j'ai pu me rendre à de nombreuses conférences et rencontrer ainsi une large communauté de chercheurs. Je pense que cela a été un atout inestimable pour ma thèse et, plus tard, pour ma carrière en tant que chercheur.

Pendant ces trois années, j'ai partagé mon bureau avec les autres doctorants, ATER ou postdocs de l'équipe combinatoire. Ce fut toujours un plaisir, une occasion de discuter de recherche et d'autres choses, une bonne ambiance qui donne envie d'aller travailler. Au grès des soutenances, des postes et des départs, ils ont été nombreux, je remercie donc (par ordre chronologique) Valentin Féray, Adrien Boussicault, Pierre-Loïc Méliot, Hayat Cheballah, Jean-Paul Bultel, Samuele Giraudo, Marc Sage, Rémi Maurice, Vincent Vong, Grégory Chatel, Nicolas Borie, Olivier Bouillot ainsi que tous les autres membres de l'équipe. Par ailleurs, le labo ne se limite pas à la combinatoire. Je voudrais déjà remercier tout le personnel administratif sans qui rien ne fonctionnerait et en particulier Line Fonfrede pour le champagne, mon contrat (!) et tout le reste. Je cite aussi toutes les autres : Sylvie Cach, Gabrielle Brossard,  Séverine Giboz, Marie-Hélène Duprat, Pascale Souliez, Corinne Palescandolo et Angélique Crombez. Que ce soit "quelque part au troisième étage" ou près de la machine à café, j'ai toujours trouvé une ambiance amicale et je remercie tous ceux que j'ai croisés qui ont su l'entretenir. Certains, avant d'être mes collègues, ont été mes enseignants. Je pense en particulier à Marc Zipstein que j'ai connu dès ma première année d'université. J'espère qu'il ne se découragera pas devant les flots d'étudiants indifférents qu'il combat chaque année et qu'il donnera encore à beaucoup d'autres l'envie de continuer. Et comme chacun sait qu'une thèse en combinatoire est aussi un examen de cuisine dont il est le principal examinateur, j'espère ne pas le décevoir ! Je le remercie pour tout ce qu'il m'a appris en cuisine ou en programmation ainsi que pour le foie gras (et vous le remercierez aussi quand vous l'aurez goûté).

En terminant ma thèse, je quitte aussi l'université dans laquelle j'ai obtenu mes trois diplômes. Ce fut pour moi plus qu'une école ou qu'un lieu de passage. Je m'y suis engagée profondément et personnellement à travers les divers conseils et instances où j'ai tenu mon rôle d'élue étudiante. Je remercie tous ceux, étudiants, enseignants, personnels, qui ont rendu ces années inoubliables. Et plus globalement, je remercie tous les enseignants qui m'ont donné le goût de la science, des mathématiques et de la connaissance au cours de ma scolarité. Je pense en particulier à M. et Mme Vangioni lorsque j'étais au lycée. Enfin, je remercie mes amis et ma famille qui n'ont jamais cessé de me soutenir. En particulier, je remercie ma mère pour avoir relu patiemment mon manuscrit, Rébecca car, bien que loin, elle est toujours là, et mon compagnon Sébastien toujours à mes côtés. 
\end{small}
\cleardoublepage

\lhead[\oldstylenums \thepage]{}
\rhead[Résumé]{\oldstylenums \thepage}
\begin{center}
    {\bf \Titre \quad --- \quad \TitreEN}
\end{center}
\begin{small}

\section*{\normalsize Résumé}

Cette thèse se situe dans le domaine de la combinatoire algébrique et porte sur l'étude et les applications de trois ordres sur les permutations : les deux ordres faibles  (gauche et droit) et l'ordre fort ou de Bruhat. 

Dans un premier temps, nous étudions l'action du groupe symétrique sur les polynômes multivariés. En particulier, les opérateurs de \emph{différences divisées} permettent de définir des bases de l'anneau des polynômes qui généralisent les fonctions de Schur aussi bien du point de vue de leur construction que de leur interprétation géométrique. Nous étudions plus particulièrement la base des polynômes de Grothendieck introduite par Lascoux et Schützenberger. Lascoux a montré qu'un certain produit de polynômes peut s'interpréter comme un produit d'opérateurs de différences divisées. En développant ce produit, nous ré-obtenons un résultat de Lenart et Postnikov et prouvons de plus que le produit s'interprète comme une somme sur un intervalle de l'ordre de Bruhat.

Nous présentons aussi l'implantation que nous avons réalisée sur Sage des polynômes multivariés. Cette implantation permet de travailler formellement dans différentes bases et d'effectuer des changements de bases. Elle utilise l'action des différences divisées sur les vecteurs d'exposants des polynômes multivariés. Les bases implantées contiennent en particulier les polynômes de Schubert, les polynômes de Grothendieck et les polynômes clés (ou caractères de Demazure).

Dans un second temps, nous étudions le \emph{treillis de Tamari} sur les arbres binaires. Celui-ci s'obtient comme un quotient de l'ordre faible sur les permutations : à chaque arbre est associé un intervalle de l'ordre faible formé par ses extensions linéaires. Nous montrons qu'un objet plus général, les intervalles-posets, permet de représenter l'ensemble des intervalles du treillis de Tamari. Grâce à ces objets, nous obtenons une formule récursive donnant pour chaque arbre binaire le nombre d'arbres plus petits ou égaux dans le treillis de Tamari. Nous donnons aussi une nouvelle preuve que la fonction génératrice des intervalles de Tamari vérifie une certaine équation fonctionnelle décrite par Chapoton.

Enfin, nous généralisons ces résultats aux treillis de $m$-Tamari. Cette famille de treillis introduite par Bergeron et Préville-Ratelle était décrite uniquement sur les chemins. Nous en donnons une interprétation sur une famille d'arbres binaires en bijection avec les arbres $m+1$-aires. Nous utilisons cette description pour généraliser les résultats obtenus dans le cas du treillis de Tamari classique. Ainsi, nous obtenons une formule comptant le nombre d'éléments plus petits ou égaux qu'un élément donné ainsi qu'une nouvelle preuve de l'équation fonctionnelle des intervalles de $m$-Tamari. Pour finir, nous décrivons des structures algébriques $m$ qui généralisent les algèbres de Hopf $\FQSym$ et $\PBT$ sur les permutations et les arbres binaires.

\subsection*{\small Mots-clés~:} 

combinatoire algébrique; ordres du groupe symétrique; polynômes de Schubert; polynômes de Grothendieck; polynômes clés; différence divisée; algèbre 0-Hecke; treillis de Tamari; algèbre de Hopf; treillis de $m$-Tamari.
\end{small}

\begin{small}
\begin{otherlanguage}{english}
\section*{\normalsize Abstract}

This thesis comes within the scope of algebraic combinatorics and studies problems related to three orders on permutations: the two said weak orders (right and left) and the strong order or Bruhat order. 

We first look at the action of the symmetric group on multivariate polynomials. By using the \emph{divided differences} operators, one can obtain some generalisations of the Schur function and form bases of non symmetric multivariate polynomials. This construction is similar to the one of Schur functions and also allows for geometric interpretations. We study more specifically the Grothendieck polynomials which were introduced by Lascoux and Schützenberger. Lascoux proved that a product of these polynomials can be interpreted in terms of a product of divided differences. By developing this product, we reobtain a result of Lenart and Postnikov and also prove that it can be interpreted as a sum over an interval of the Bruhat order.

We also present our implementation of multivariate polynomials in Sage. This program allows for formal computation on different bases and also implements many changes of bases. It is based on the action of the divided differences operators. The bases include Schubert polynomials, Grothendieck polynomials and Key polynomials.

In a second part, we study the \emph{Tamari lattice} on binary trees. This lattice can be obtained as a quotient of the weak order. Each tree is associated with the interval of its linear extensions. We introduce a new object called, \emph{interval-posets} of Tamari and show that they are in bijection with the intervals of the Tamari lattice. Using these objects, we give the recursive formula counting the number of elements smaller than or equal to a given tree. We also give a new proof that the generating function of the intervals of the Tamari lattice satisfies some functional equation given by Chapoton.

Our final contributions deals with the $m$-Tamari lattices. This family of lattices is a generalization of the classical Tamari lattice. It was introduced by Bergeron and Préville-Ratelle and was only known in terms of paths. We give the description of this order in terms of some family of binary trees, in bijection with $m+1$-ary trees. Thus, we generalize our previous results and obtain a recursive formula counting the number of elements smaller than or equal to a given one and a new proof of the functional equation. We finish with the description of some new $"m"$ Hopf algebras which are generalizations of the known $\FQSym$ on permutations and $\PBT$ on binary trees.

\subsection*{\small Keywords:}

algebraic combinatorics; orders of the symmetric group; Schubert polynomials; Grothendieck polynomials; key polynomials; divided difference; 0-Hecke algebra; Tamari lattice; Hopf algebra; $m$-Tamari lattices.

\end{otherlanguage}
\end{small}
\cleardoublepage

\lhead[\oldstylenums \thepage]{Table des matières}
\rhead[Table des matières]{\oldstylenums \thepage}
\tableofcontents
\cleardoublepage

\lhead[\oldstylenums \thepage]{Table des figures}
\rhead[Table des figures]{\oldstylenums \thepage}
\listoffigures
\cleardoublepage

\mainmatter

\lhead[\oldstylenums \thepage]{Introduction}
\rhead[Introduction]{\oldstylenums \thepage}
\addcontentsline{toc}{chapter}{Introduction}

\chapter*{Introduction}

\section*{Avant-propos}

\subsection*{Combinatoire et algèbre}

On peut définir la combinatoire comme l'étude des ensembles finis ou dénombrables d'objets. Les problèmes les plus couramment cités sont sans doute ceux de \emph{combinatoire énumérative} où l'on cherche à dénombrer ces ensembles. Ils apparaissent en particulier en probabilités discrètes et certains ont été étudiés dès l'antiquité : nombre de tirages de $k$ numéros parmi $n$, nombre de façon d'asseoir $n$ personnes autour d'une table, etc. En Europe, on cite souvent Fibonacci comme l'un des précurseurs du domaine. On en retrouve plus tard chez Newton, Pascal, Leibnitz ou Euler. Au XX\ieme~siècle, la combinatoire prend un nouvel essor avec l'arrivée de l'informatique et l'importance prise par les algorithmes. Non seulement on utilise la combinatoire comme outil dans l'analyse d'algorithmes, mais les avancées technologiques justifient l'optimisation des calculs proposée par l'approche combinatoire. L'utilisation de plus en plus courante de \emph{l'exploration par ordinateur} permet d'expérimenter \emph{a priori} sur les objets  pour découvrir de nouvelles propriétés.

En \emph{combinatoire algébrique}, on cherche plus particulièrement à \emph{structurer} les ensembles d'objets à l'aide des outils fournis par l'algèbre. On étudiera alors l'agencement interne propre à la nature de l'objet pour définir, par exemple, un produit, un ordre partiel ou d'autres relations. Cette approche permet une compréhension plus profonde des objets étudiés et l'on peut s'en servir pour résoudre des questions énumératives. Inversement, la vision combinatoire apporte à l'algèbre un éclairage nouveau : les caractères d'un groupe s'identifient à des partitions, un produit de polynômes devient une somme sur un intervalle... Le sens combinatoire donné à des problèmes algébriques permet souvent d'appréhender les questions d'une façon plus globale et moins technique, révélant les aspects structurels de certaines opérations. Concrètement, la combinatoire permet de prouver des nouveaux résultats et propose des méthodes algorithmiques effectives pour réaliser des calculs. Elle utilise pour cela les structures  et méthodes de l'informatique classique comme les algorithmes de tris ou les arbres binaires de recherche \cite{AVL}. Elle apparaît aujourd'hui fondamentale dans de très nombreux domaines : théorie des nombres, théorie des groupes, théorie des représentations, géométrie algébrique, etc.

\subsection*{Structures d'ordre}

La notion d'ordre sur les éléments est tout à fait naturelle en mathématiques, ne faisant que formaliser la succession des nombres entiers. L'ordre sur les entiers est dit \emph{total} car deux nombres sont toujours comparables. Les \emph{ordres partiels}, bien que moins immédiats, se retrouvent aussi dans le monde réel : relations des individus dans un arbre généalogique, ensemble d'actions dépendant les unes des autres. En mathématique, on peut donner comme exemple l'inclusion des ensembles qui est souvent utilisée comme base de cas plus généraux. Au delà d'une simple relation, les ordres partiels sont étudiés eux-mêmes en tant qu'objets mathématiques et sont à la base d'une vaste théorie \cite{MacNeille, lattices, Birkhoff}. Dans ce mémoire, nous n'y apportons pas de nouvelle contribution mais utilisons ces résultats dans un but algébrique.

En effet, un ordre partiel est un outil d'interprétation dans un calcul algébrique. Lorsque l'on développe une expression, on obtient le résultat sous forme de somme. L'un des objectifs est alors de comprendre cette somme pour, par exemple, améliorer l'algorithme de calcul mais aussi pour en découvrir de nouvelles propriétés. Si on arrive à déterminer une structure d'ordre sur les objets permettant d'exprimer la somme comme un intervalle, on a atteint une partie de cet objectif. On peut par exemple ainsi optimiser l'espace de stockage car les éléments extrémaux contiennent l'ensemble de l'information. On peut aussi se contenter de calculer ces éléments extrémaux plutôt que l'ensemble de la somme, obtenant par là des calculs bien plus rapides. De plus, en découvrant la structure d'ordre sous-jacente à un calcul on comprend dans quel contexte les objets peuvent être étudiés, ou sous quelle forme. On relie ainsi le problème à une structure plus globale parfois en lien avec une théorie qui \emph{a priori} semblait différente. En effet, entre eux, les ordres possèdent des relations : quotient, isomorphisme, ordre plus fin ou plus grossier, sous-ordre, etc. Et ces relations existent parfois aussi sur le plan algébrique. C'est par exemple le cas pour l'ordre de Tamari sur les arbres binaires que nous décrivons dans la partie \ref{part:tamari} : en termes d'ordres, il est un quotient de l'ordre faible sur les permutations et en termes d'algèbres, l'algèbre $\PBT$ sur les arbres binaires est une sous-algèbre de celle sur les permutations $\FQSym$ \cite{Tamari2}. Cela nous amène à l'objet de base de notre mémoire : les permutations.

\subsection*{Permutations}

En combinatoire, une \emph{permutation} est à la fois un objet simple et fondamental. Elle correspond au résultat d'un réarrangement d'un ensemble donné d'objets. Par exemple, un jeu de cartes mélangé est une permutation des cartes. De façon générale, on s'intéressera aux permutations des entiers $1, \dots, n$. Si le concept apparaît depuis longtemps dans divers domaines des mathématiques, il prend tout son sens au XIX\ieme~siècle quand la théorie des groupes est introduite par Galois à travers les permutations des racines d'un polynôme. Les permutations sont des applications de $\lbrace 1, \dots, n \rbrace$ dans $\lbrace 1, \dots, n \rbrace$ et en ce sens, l'ensemble des permutations possède une structure de groupe. C'est ce qu'on appelle le \emph{groupe symétrique} $\Sym{n}$, il sera l'objet principal de ce mémoire. Il est fondamental en mathématiques et plus particulièrement en algèbre. Le théorème de Cayley prouve ainsi que tout groupe fini est un sous-groupe de $\Sym{n}$. Par ailleurs, le groupe symétrique peut souvent être compris comme une version spécialisée et simple d'objets plus complexes. Ainsi, au sein de la théorie des \emph{groupes de Coxeter} apparue au début du XX\ieme~siècle, $\Sym{n}$ est le groupe de Coxeter de type $A_{n-1}$. De nombreux résultats sont prouvés au départ sur $\Sym{n}$ avant d'être généralisés aux autres types. C'est par exemple le cas des propriétés des ordres sur les permutations que nous décrivons dans le chapitre \ref{chap:prelim_groupe_sym}. Un autre exemple est celui de l'algèbre de Hecke dont les générateurs vérifient, comme les générateurs du groupe symétrique, les \emph{relations de tresses} mais où des paramètres sont ajoutés aux relations quadratiques. En ce sens, c'est une généralisation de l'algèbre du groupe symétrique et en particulier le calcul des représentations est similaire \cite{HoefHecke,LascouxHecke}. 

\subsection*{Ordres du groupe symétrique}

Il existe de nombreuses façons d'ordonner les permutations. Dans cette thèse, on décrit trois ordres liés à la structure de groupe : les ordres faibles, droit et gauche, et l'ordre fort ou \emph{de Bruhat}. La longueur d'une permutation $\sigma$ est donnée par son nombre \emph{d'inversions}, c'est-à-dire par le nombre de couples $(i,j)$ tels que $\sigma_i > \sigma_j$. Aussi bien pour les ordres faibles que pour l'ordre fort, la relation d'ordre est basée sur les multiplications par des éléments particuliers (les transpositions) qui augmentent la longueur de la permutation. Ces ordres apparaissent dans de très nombreux contextes. Ils sont en particulier définis de façon plus générale sur les groupes de Coxeter et ont été étudiés entre autres par Björner \cite{BjornerCoxeter} et Lascoux et Sch\"utzenberger \cite{LascouxCoxeter}. 

L'ordre fort a été introduit par Ehresmann dans le contexte de la géométrie algébrique \cite{Ehresmann} : il correspond à l'ordre d'inclusion des variétés de Schubert dans la variété de drapeau. Sur les permutations, c'est l'ordre d'inclusion des facteurs gauches réordonnés. Il apparaît donc dans de nombreux problèmes liés à cette thématique, en particulier dans les calculs sur les polynômes de Grothendieck \cite{LenartPostnikov}. Il permet aussi de décrire les changements de base dans l'algèbre 0-Hecke \cite{LascouxCoxeter} comme nous le verrons dans le chapitre \ref{chap:polynomes_action}. Les ordres faibles se décrivent aussi par une inclusion : celle des ensembles d'inversions des permutations. On peut prouver qu'ils possèdent une structure de treillis \cite{PermTreillis, BjornerCoxeter} et que l'objet géométrique sous-jacent est un polytope connu : le \emph{permutoèdre}. Ils apparaissent de façon récurrente en combinatoire. Par exemple, les classes plaxiques décrites par Lascoux et Sch\"utzenberger \cite{LascouxPlaxique} forment des ensembles connexes de l'ordre faible droit. On les trouve aussi dans la théorie des fonctions non commutatives où ils décrivent en particulier le produit des éléments de $\FQSym$ \cite{NCSF6}.  

\section*{Contexte}

Dans cette thèse, nous présentons en particulier deux domaines où les ordres du groupe symétrique jouent un rôle fondamental : les bases des polynômes multivariés d'une part et les algèbres de Hopf liées aux arbres binaires et permutations d'autre part. Nous donnons à présent le contexte de notre travail dans ces deux domaines. 

\subsection*{Polynômes et action du groupe symétrique}

Le groupe symétrique joue un rôle important dans l'étude des polynômes en plusieurs variables. L'action d'une permutation sur un monôme est simplement la permutation des variables. Les polynômes invariants par cette action sont appelés \emph{polynômes symétriques}. Lorsqu'on travaille avec un nombre infini de variables dans l'espace des séries formelles, on parle alors de \emph{fonctions symétriques}. \'Etudiées au XIX\ieme~siècle par des mathématiciens tels que MacMahon \cite{MacMahon}, Cauchy ou Jacobi, elles sont au cœur de nombreuses questions liées à la combinatoire algébrique.

L'anneau des fonctions symétriques se décrit dans différentes bases indexées par les \emph{partitions}, c'est-à-dire les vecteurs $\lambda = (\lambda_1, \dots, \lambda_n)$ avec $\lambda_1 \geq \lambda_2 \geq \dots \geq \lambda_n$. On trouvera les définitions des différentes bases dans le premier chapitre du livre de Macdonald \cite{MAC}. Une de ces bases revêt une importance particulière : la base des fonctions de Schur. Bien qu'étudiées en premier lieu par Jacobi et Cauchy, elles portent le nom du mathématicien Isaai Schur connu pour son travail en théorie des représentations. En effet, les fonctions de Schur correspondent aux caractères irréductibles du groupe symétrique. Elles forment le lien principal entre la théorie des représentations et la combinatoire, justifiant à elles seules de très nombreux articles et généralisations. En particulier, elles motivent les recherches liées à la combinatoire des tableaux à travers la célèbre règle de Littlewood-Richardson.

Un autre domaine fortement lié à la théorie des représentation est la géométrie algébrique et le rôle des fonctions de Schur y est aussi fondamental. Le \emph{calcul de Schubert}, introduit par le mathématicien du même nom au XIX\ieme~siècle, pose la question du nombre de solutions à des problèmes d'intersections. On décompose les éléments d'une variété géométrique en \emph{cellules de Schubert} en fonctions de leurs intersections avec un \emph{drapeau}, c'est-à-dire un ensemble de sous-espaces vectoriels donnés $V_1 \subset V_2 \subset \dots \subset V_n$. Le nombre d'intersections de deux sous-variétés données peut alors s'obtenir par un calcul algébrique dans \emph{l'anneau de cohomologie}. Dans le cas où la variété de départ est la \emph{Grassmannienne}, l'anneau de cohomologie est un quotient de l'anneau des fonctions symétriques. L'image des cellules de Schubert correspond alors aux fonctions de Schur \cite{Fulton}. Le produit des fonctions de Schur permet ainsi d'encoder les intersections des sous-variétés dans la Grassmannienne. C'est dans ce cadre qu'a d'abord été prouvée la \emph{formule de Pieri} \cite{MAC} qui en donne l'interprétation combinatoire en termes de tableau.

Si la variété de départ n'est plus la Grassmannienne mais une \emph{variété de drapeaux}, l'anneau de cohomologie devient isomorphe à un quotient non plus des fonctions symétriques mais des \emph{polynômes non symétriques}. On peut alors décrire des bases des polynômes qui généralisent les fonctions de Schur et répondent aux mêmes types de questions géométriques. Ces bases sont toutes définies à partir d'un opérateur clé : la différence divisée. L'opérateur de différence divisée $\partial_i$ peut être compris comme une dérivation discrète qu'on attribue généralement à Newton. On la définit par
\begin{equation}
f.\partial_i = \frac{f(\dots, x_i, x_{i+1}, \dots) - f(\dots, x_{i+1}, x_i, \dots)}{x_i - x_{i+1}}.
\end{equation}
C'est une opération qui \emph{symétrise} le polynôme en $x_i$ et $x_{i+1}$. Comme on peut le lire dans un article historique récapitulatif de Lascoux \cite{LascouxSchubHist}, c'est en 1973 qu'un intérêt pour les différences divisées naît en combinatoire à travers les articles de Bernstein-Gelfand-Gelfand \cite{BGG} et Demazure \cite{Dem2}. Les auteurs prouvent que tout comme les générateurs du groupe symétrique, les différences divisées $\partial_i$ vérifient les relations de tresses (en outre, elles sont aussi de carré nul). Cette propriété fondamentale permet d'indexer les différences divisées par des permutations en interprétant une décomposition réduite comme un produit d'opérateurs $\partial_i$. C'est ce qui fait le lien entre ces opérateurs et les ordres faibles sur les permutations qui se définissent eux aussi à partir des décompositions réduites \cite{LascouxCoxeter}. Les différences divisées sont étudiées de façon très poussée par Lascoux et Sch\"utzenberger \cite{LascouxDiffDiv} qui travaillent à l'époque sur les polynômes de Kazhdan-Lusztig \cite{LascouxKaz}. En interprétant en termes d'opérateurs les relations de couverture de l'ordre faible, ils définissent une nouvelle famille de polynômes indexés par les permutations. Plus précisément, on note $\partial_\omega$ la différence divisée correspondant à la permutation maximale pour l'ordre faible $\omega = n, n-1, \dots, 1$. En appliquant $\partial_\omega$ à un monôme dit \emph{dominant} $x_1^{\lambda_1} x_2^{\lambda_2} \dots x_n^{\lambda_n}$ où $\lambda$ est une partition, on obtient un polynôme symétrique et plus précisément, une \emph{fonction de Schur}. On interprète ce calcul comme une symétrisation progressive du polynôme : le polynôme le "moins" symétrique est le monôme dominant de départ et le "plus" symétrique est la fonction de Schur. \`A eux tous, les polynômes intermédiaires forment une base des polynômes non symétriques, ce sont les polynômes de Schubert. En d'autres termes, c'est l'ensemble des images par différences divisées des monômes dominants  \cite{LascouxSchub}. Cette base généralise au cas non symétrique celle des fonctions de Schur. D'un point de vue géométrique, les polynômes de Schubert sont l'image des variétés de Schubert dans l'anneau de cohomologie de la variété de drapeaux \cite{Chern}.

Lascoux et Sch\"utzenberger introduisent aussi une autre famille de polynômes appelés \emph{polynômes de Grothendieck} \cite{LascouxGroth}. La définition est très similaire à celle des polynômes de Schubert mais utilise une différence divisée dite \emph{isobare} qui conserve le degré du polynôme. D'un point de vue géométrique, les polynômes de Grothendieck permettent de travailler dans l'anneau de Grothendieck de la variété de drapeau plutôt que dans l'anneau de cohomologie. On définit deux types de différences divisées isobares qui s'interprètent aussi comme deux bases de l'algèbre dite 0-Hecke. Les changements de bases, et, par là, toute la combinatoire des polynômes de Grothendieck, sont alors liés à l'ordre de Bruhat sur les permutations \cite{LascouxCoxeter}. Dans le cas des polynômes de Schubert comme dans celui des polynômes de Grothendieck, l'enjeu principal est de comprendre la structure multiplicative de l'anneau par une formule généralisant la formule de Pieri sur les fonctions de Schur. Dans les deux cas, la question d'une interprétation combinatoire du produit de deux éléments quelconques est un problème ouvert. Cependant, la formule de Monk \cite{Monk} décrit le cas particulier important du produit d'un polynôme de Schubert par un élément indexé par un générateur. Pour les polynômes de Grothendieck, Lascoux propose un résultat similaire en interprétant un produit de polynômes comme un produit d'opérateurs de différences divisées \cite{LascouxGroth}. Nous traitons la question du développement de ce produit dans le chapitre \ref{chap:polynomes_grothendieck}. Encore récemment, les questions relatives aux polynômes de Grothendieck et de Schubert ont fait l'objet de très nombreux articles \cite{LenartPostnikov, SottileLenart, LenartSimple, BergeronKbruhat, LenartGreedyAlgo} tous liés à la combinatoire de l'ordre de Bruhat.

\subsection*{Treillis de Tamari et algèbres de Hopf}

Dans la partie \ref{part:tamari} de cette thèse, nous étudions un ordre sur les arbres binaires appelé \emph{treillis de Tamari} qui est un quotient de l'ordre faible sur les permutations. Le lien entre les deux structures fait intervenir une nouvelle structure algèbrique : les \emph{algèbres de Hopf}.

Lors des dernières décennies, la structure \emph{d'algèbre de Hopf} a pris une importance particulière en combinatoire. Définir une structure \emph{d'algèbre} sur un ensemble d'objets combinatoires revient à définir un produit, c'est-à-dire un algorithme pour \emph{composer} deux objets. L'opération duale est le \emph{coproduit} qui décompose un objet en deux objets distincts. Apparues dans le cadre de la topologie algébrique, les \emph{algèbres de Hopf} sont des structures munies à la fois d'un produit et d'un coproduit vérifiant certaines propriétés de compatibilité \cite{Hopf1,Hopf2}. Les fonctions symétriques en particulier sont munies d'une telle structure. L'opération de coproduit passe par le doublement d'alphabet : on ne travaille plus sur un seul ensemble de variables $x_1, x_2, x_3, \dots$ mais sur deux ensembles $x_1, x_2, x_3, \dots$ et $y_1, y_2, y_3, \dots$. Chaque monôme d'une fonction symétrique se découpe alors en un monôme sur les $x$ et un autre sur les $y$. Cette structure est déjà implicitement présente dans le travail de MacMahon \cite{MacMahon}. Plus tard, Rota insiste sur l'importance des algèbres de Hopf en combinatoire \cite{Rota1,Rota2}.

En 1995, Malvenuto et Reutenauer décrivent une algèbre de Hopf sur les permutations \cite{MalReut}. Le produit s'exprime comme un intervalle de l'ordre faible. \`A la même époque, Gelfand, Krob, Lascoux, Leclerc, Retakh et Thibon donnent un équivalent en variables non commutatives de l'algèbre de Hopf des fonctions symétriques \cite{NCSF1}. Cet article sera suivi de nombreux autres explorant les différentes ramifications de cette nouvelle théorie \cite{NCSF2,NCSF3,NCSF4,NCSF5,NCSF6,NCSF7}. Ce travail met à jour un nouveau principe de calcul : un objet combinatoire peut "s'exprimer" comme une somme de mots sur un alphabet (commutatif ou non). Plutôt que de décrire le produit sur l'objet en tant que tel, on le décrit sur le polynôme associé. De même, le coproduit s'obtient par un doublement d'alphabet correctement défini. Cette méthode est appelée la \emph{réalisation polynomiale}. Elle simplifie grandement les calculs car parmi les nombreux axiomes que doit vérifier une algèbre de Hopf, beaucoup sont naturellement présents sur les développements en mots. C'est en utilisant ce principe que Duchamp, Hivert, Thibon et Novelli obtiennent une nouvelle description de l'algèbre de Malvenuto-Reutenauer \cite{NCSF6,NCSF7} qu'ils nomment \emph{algèbre des fonctions quasi-symétriques libres} ou $\FQSym$. Dans \cite{CHA}, Hivert donne l'exemple de $\FQSym$ comme illustration du principe de réalisation polynomiale, c'est aussi ce que nous ferons dans le chapitre \ref{chap:tamari_prelim}. Les éléments de $\FQSym$ sont indexés par des permutations et sont développés comme sommes de mots par l'opération de \emph{standardisation} qui n'est rien d'autre que l'application du \emph{tri par bulles}. Les auteurs définissent plusieurs bases et retrouvent très simplement les formules de produit et coproduit données par Malvenuto et Retenauer. Ils obtiennent en particulier que l'algèbre $\FQSym$ est autoduale. Sa dualité reflète celle qui existe entre l'ordre faible droit et l'ordre faible gauche. En effet, le produit sur une des bases est donné par les intervalles de l'ordre faible droit tandis que le produit sur la base duale est donnée par l'ordre faible gauche.

Le lien avec les arbres binaires apparaît par le biais de l'algèbre de Hopf que Loday et Ronco décrivent sur ces objets en tant que sous-algèbre de celle de Malvenuto et Reutenauer \cite{PBT1}. Comme conséquence du travail sur $\FQSym$, Hivert, Novelli et Thibon en donnent une réalisation polynomiale nommée $\PBT$ en tant que sous-algèbre de $\FQSym$ \cite{PBT2}. La structure d'algèbre de $\PBT$ est liée à un treillis sur les arbres binaires appelé \emph{treillis de Tamari}. Ce treillis a été initialement décrit par Tamari sur les parenthésages \cite{Tamari1, Tamari2}. Vu comme un objet géométrique, c'est \emph{l'associaèdre} ou polytope de Stasheff. Les relations entre l'associaèdre et le permutoèdre sont souvent étudiées d'un point de vue géométrique : l'associaèdre est un permutoèdre auquel on a supprimé des faces. En termes d'ordres partiels, le treillis de Tamari est un quotient de l'ordre faible sur les permutations. Cette propriété est en particulier prouvée dans \cite{PBT2} et nous revenons en détail dessus dans ce mémoire. L'idée principale est que certains intervalles de l'ordre faible peuvent être décrits comme des \emph{extensions linéaires} de posets, et plus précisément d'arbres \cite{BW}. Les extensions linéaires des arbres binaires découpent l'ordre faible en intervalles distincts et l'ordre entre ces intervalles correspond à l'ordre de Tamari. Cette correspondance existe aussi en termes d'algèbres : un élément $\BP_T$ de la base de $\PBT$ se développe dans $\FQSym$ comme la somme des extensions linéaires de son arbre binaire $T$. Ainsi le produit dans $\PBT$ est à la fois un intervalle du treillis de Tamari et de l'ordre faible sur les permutations. 

Le treillis de Tamari en tant que tel suscite de nombreuses publications. On peut citer en particulier le résultat de Chapoton \cite{Chap} dénombrant les intervalles. Très récemment, de nouvelles pistes de recherche ont été ouvertes par l'objet plus général des treillis de $m$-Tamari décrits par Bergeron et Préville-Ratelle \cite{BergmTamari}. En particulier, une formule générale dénombrant les intervalles dans $m$-Tamari a été prouvée \cite{mTamari}.

\section*{Contributions et plan du mémoire}

Notre travail est divisé en deux parties distinctes qui reposent sur un même principe : utiliser les propriétés des ordres du groupe symétrique pour obtenir des résultats algébriques et énumératifs nouveaux. Après une partie \ref{part:prelim} introductive, la partie \ref{part:polynomes} porte sur les polynômes multivariés. On trouvera les principaux résultats dans les chapitres \ref{chap:polynomes_grothendieck} et \ref{chap:polynomes_sage}. La partie \ref{part:tamari} est dédiée aux treillis de Tamari et $m$-Tamari, nos contributions sont rassemblées dans les chapitres \ref{chap:tamari_intervalles} et \ref{chap:mtamari}. Le schéma suivant représente l'organisation globale des chapitres entre eux. 

\begin{figure}[ht]
\centering
\input{includes/figures/chapitres}
\end{figure}

\subsection*{Préliminaires}

La partie \ref{part:prelim} est composée de deux courts chapitres préliminaires introduisant les notions dont nous aurons besoin. Dans le chapitre \ref{chap:prelim_posets}, nous rappelons la définition des \emph{ensembles partiellement ordonnés} et des \emph{treillis} et en donnons les principales propriétés. Le chapitre \ref{chap:prelim_groupe_sym} est consacré au groupe symétrique. Nous donnons la définition des ordres faibles et forts en termes de décompositions réduites. Pour l'ordre  fort en particulier, nous décrivons précisément les propriétés de comparaison des éléments obtenus par les projections sur les permutations bigrassmanniennes. En outre, dans le paragraphe \ref{sub-sec:prelim_groupe_sym:bruhat:coset} nous prouvons un lemme annexe sur des intersections d'intervalles dans l'ordre de Bruhat dont nous aurons besoin par la suite.

\subsection*{Polynômes multivariés}

La partie \ref{part:polynomes} est dédiée à l'étude des polynômes multivariés dans le cadre que nous avons déjà évoqué. Le chapitre \ref{chap:polynomes_action} rappelle l'action du groupe symétrique sur les polynômes et les définitions des différences divisées. Nous effectuons un bref survol de la théorie des fonctions symétriques en décrivant en particulier la construction des fonctions de Schur à partir des différences divisées. Nous rappelons la formule de Pieri et l'interprétation géométrique qui en découle. Nous présentons ensuite certaines bases des polynômes non symétriques telles qu'elles ont été définies par Lascoux et Schützenberger \cite{LascouxSchub, LascouxGroth}, c'est-à-dire construites à partir des différences divisées. Nous terminons par les propriétés de l'algèbre $0$-Hecke dont l'action sur les polynômes est celle des différences divisées isobares. En particulier, il existe deux sortes de différences divisées isobares qui correspondent à deux bases de l'algèbre. La combinatoire qui en découle est celle de l'ordre de Bruhat.

Le chapitre \ref{chap:polynomes_grothendieck} présente un équivalent de la formule de Pieri pour les polynômes de Grothendieck. Nous avons vu qu'une question importante est celle du produit des polynômes. Dans le cas des fonctions de Schur, un résultat fondamental est donné par la formule de Pieri qui décrit la multiplication d'une fonction de Schur par une fonction complète. Le produit équivalent pour les polynômes de Grothendieck est $G_\sigma G_{s_k}$ où $G_\sigma$ et $G_{s_k}$ sont des polynômes de Grothendieck indexés respectivement par une permutation quelconque $\sigma$ et une transposition simple $s_k$. Dans \cite{LascouxGroth}, Lascoux interprète ce produit comme un produit d'opérateurs de deux sortes de différences divisées isobares. Par ailleurs, on trouve dans \cite{LenartPostnikov} un développement de ce même produit utilisant une énumération de chaînes dans l'ordre de Bruhat. \`A partir du résultat de Lascoux, nous commençons par donner une nouvelle preuve en type $A$ du théorème de Lenart et Postnikov  \cite{LenartPostnikov}. Notre approche permet une meilleure compréhension du résultat et nous prouvons alors que les permutations qui apparaissent dans le développement du produit forment un intervalle de l'ordre de Bruhat. Ce résultat a donné lieu à un article publié \cite{Me_Grothendieck}.

Dans le chapitre \ref{chap:polynomes_sage}, nous présentons l'implantation réalisée en Sage des bases des polynômes multivariés. Sage \cite{SAGE_WEBSITE} est un logiciel libre de calcul formel. Pour les besoins de notre recherche et dans l'objectif d'offrir un outil de calcul à la communauté, nous avons développé au sein du groupe Sage-combinat \cite{SAGE_COMBINAT} un cadre complet de travail sur les bases des polynômes. Notre logiciel permet de définir les polynômes multivariés comme des sommes formelles de vecteurs (les exposants) et d'y appliquer les opérateurs de différences divisées. Nous avons implanté les différentes bases décrites dans le chapitre \ref{chap:polynomes_action} : polynômes de Schubert, polynômes de Grothendieck, polynômes clés. Il est possible de travailler formellement dans ces bases, de les développer en monômes ou d'effectuer des changements de bases. Enfin, notre implantation permet de définir de nouveaux opérateurs et de nouvelles bases pour les besoins de l'expérimentation. Dans ce chapitre, nous présentons les fonctionnalités de base ainsi que les choix que nous avons fait en termes d'architecture logicielle. Enfin, nous donnons plusieurs exemples d'applications avancées possibles avec notre implantation. En particulier, nous expliquons comment les résultats du chapitre \ref{chap:polynomes_grothendieck} peuvent être obtenus. Cette implantation est disponible en tant que patch additionnel à Sage \cite{SAGE_Polynomials}. Elle a fait l'objet de nombreuses présentations et d'un article \cite{Me_Polynomes}.

\subsection*{Treillis de ($m$)-Tamari et algèbres}

Dans la partie \ref{part:tamari}, nous utilisons le lien entre l'ordre faible sur les permutations et l'ordre de Tamari sur les arbres binaires pour démontrer des résultats sur les treillis de Tamari et $m$-Tamari. Le chapitre \ref{chap:tamari_prelim} est un chapitre préliminaire où nous rappelons les définitions et différentes descriptions de l'ordre de Tamari. Nous expliquons en quoi l'ordre de Tamari est un quotient et un sous-treillis de l'ordre faible par la description des \emph{classes sylvestres}. C'est aussi le premier chapitre où nous abordons la notion d'algèbre de Hopf. Nous en donnons donc la définition et présentons les deux exemples qui nous intéressent : $\FQSym$ et $\PBT$.

Comme nous le rappelons dans le chapitre \ref{chap:tamari_prelim}, les classes sylvestres sont des intervalles de l'ordre faible correspondant à des extensions linéaires d'arbres binaires. Par ailleurs, à partir d'un arbre binaire $T$, il est possible d'obtenir deux arbres planaires $\dec$ et $\inc$ dont les extensions linéaires correspondent à des intervalles respectivement finaux et initiaux de l'ordre faible. De façon plus générale, dans le chapitre \ref{chap:tamari_intervalles}, nous définissons un nouvel objet appelé \emph{intervalle-poset} de Tamari dont les extensions linéaires correspondent à des intervalles de l'ordre faible entre deux classes sylvestres. Ces intervalles-posets sont en bijection avec les intervalles de Tamari. Nous utilisons leurs propriétés pour donner une nouvelle preuve de la formule de Chapoton \cite{Chap} dénombrant le nombre d'intervalles dans le treillis de Tamari. Nous prouvons par ailleurs un nouveau résultat : pour chaque arbre, on calcule de façon récursive un polynôme que l'on nomme \emph{polynôme de Tamari} de l'arbre. Nous prouvons que pour un arbre $T$, ce polynôme compte le nombre d'arbres $T'\leq T$ dans l'ordre de Tamari en fonction d'une statistique particulière. Ce résultat a fait l'objet d'une récente publication sous forme d'un \emph{extended abstract} \cite{Me_Tamari}.

Dans \cite{BergmTamari}, Bergeron et Préville-Ratelle décrivent une généralisation du treillis de Tamari qu'ils nomment treillis de $m$-Tamari. Leur description est donnée en termes de chemins qui généralisent les mots de Dyck. Le chapitre \ref{chap:mtamari} est dédié à ces treillis généralisés de Tamari. Tout d'abord, en utilisant le plongement des treillis de $m$-Tamari dans les treillis de Tamari classiques, nous expliquons quelle description de l'ordre peut être faite en termes d'arbres. Cette description n'était pas connue jusqu'alors et et faisait l'objet d'une question à la fin de l'article \cite{mTamari}. Nous l'utilisons pour généraliser les résultats du chapitre \ref{chap:tamari_intervalles} aux treillis de $m$-Tamari. En particulier, nous obtenons une nouvelle preuve que les intervalles de $m$-Tamari vérifient l'équation fonctionnelle résolue dans \cite{mTamari} pour les dénombrer. Tout comme dans le cas de treillis de Tamari classique, nous obtenons un résultat plus précis et donnons pour chaque arbre la formule calculant le nombre d'éléments inférieurs ou égaux dans le treillis de $m$-Tamari. Enfin, nous donnons des constructions analogues des algèbres $\FQSym$ et $\PBT$ que nous appelons $\FQSymm$ et $\PBTm$ et qui possèdent des propriétés voisines.

Les implantations réalisées dans le cadre de cette dernière partie ne nous ont pas semblé nécessiter un chapitre à part entière. Cependant, comme dans la partie précédente, notre recherche a été basée sur de l'exploration préliminaire effectuée grâce au logiciel Sage. Une partie des implantations réalisées ont pu être intégrées au logiciel et nous avons participé aux différents projets sur l'implantation des arbres en Sage \cite{SAGE_BinaryTrees, SAGE_BinaryMaps, SAGE_MaryTrees}.

\cleardoublepage

\lhead[\oldstylenums \thepage]{\S~\thesection \; --- \; \rightmark}
\rhead[Chapitre~\thechapter \; --- \; \leftmark]{\oldstylenums \thepage}

\part{Préliminaires}
\label{part:prelim}

\chapter{Structures algébriques et ordres sur les objets combinatoires}
\label{chap:prelim_posets}

Nous commençons ce mémoire par un court chapitre d'introduction à quelques notions de combinatoire. Le paragraphe \ref{sec:prelim_posets:struct_elementaires} rappelle le principe de la construction d'espaces vectoriels et d'algèbres sur les objets combinatoires à travers deux exemples : les permutations et les mots. Les structures de posets et de treillis sont définies dans les paragraphes \ref{sec:prelim_posets:posets} et \ref{sec:prelim_posets:treillis}. Ces objets sont à la base de notre travail et nous donnons un aperçu de leurs propriétés fondamentales. Pour une approche plus complète, nous invitons le lecteur à se reporter aux ouvrages suivants \cite{Stanley}, \cite{Birkhoff}, \cite{lattices} et \cite{LascouxCoxeter}.

\section{Objets et structures élémentaires}
\label{sec:prelim_posets:struct_elementaires}

\subsection{Espaces vectoriels d'objets combinatoires}
\label{sub-sec:prelim_posets:struct_elementaires:ev}
\index{classe combinatoire}

Une \emph{classe combinatoire} est un ensemble d'objets munis d'une notion de taille. L'ensemble des objets ayant une taille donnée est toujours fini. Par exemple, si $A$ est l'alphabet $\lbrace a, b \rbrace$, on note $A^*$ l'ensemble des mots sur $A$.  On a $A^* = \lbrace \epsilon, a, b, aa, ab, ba, bb, aaa, aab, \dots \rbrace$, où $\epsilon$ est le mot vide. La taille d'un mot $u$, notée $|u|$, est donnée par son nombre de lettres. L'ensemble des mots de taille $n$ est fini et de taille $2^n$. 

Les \emph{permutations} sont les objets de base dans notre travail. Nous en donnons une première définition en termes de mots.

\begin{Definition}
\label{def:prelim_posets:permutations}
Une permutation de taille $n$ est un mot sur l'alphabet $\lbrace 1, 2, \dots, n \rbrace$ où chaque lettre apparaît exactement une fois.
\end{Definition}

Par exemple, les permutations de taille 3 sont $123$, $213$, $132$, $231$, $312$, $321$. L'ensemble des permutations de taille $n$ est fini et de taille $n!$.

On sera rapidement amené à former des sommes formelles d'objets combinatoires. Cela revient à se placer dans un espace vectoriel dont la base est l'ensemble des objets. On considérera que le corps de base de l'espace vectoriel est un corps quelconque $\mathbb{K}$ de caractéristique nulle. L'espace vectoriel $E$ dont la base est une classe combinatoire $C$ est \emph{gradué}, c'est-à-dire

\begin{equation}
\label{eq:prelim_posets:ev}
E = \bigoplus_{n \in \mathbb{N}} E_n
\end{equation}
où $E_n$ est de base $C_n$, les objets combinatoires de taille $n$. 

\subsection{Produits et algèbres}
\label{sub-sec:prelim_posets:struct_elementaires:alg}

On enrichit souvent les espaces vectoriels d'objets combinatoires d'une notion de produit, on obtient alors une algèbre.

\begin{Exemple}
\label{ex:prelim_posets:concat}
Le produit basé sur la concaténation des mots\index{produit!concaténation} : $A^* \times A^* \rightarrow A^*$,
\begin{equation}
u.v \rightarrow uv.
\end{equation}
Par exemple $aab.abab = aababab$. De façon similaire, on définit la \emph{concaténation décalée}\index{produit!concaténation décalée} sur les permutations
\begin{equation}
\sigma \ConcDec \mu = \sigma \Dec{\mu}
\end{equation}
où $\Dec{\mu}$ est le mot $\mu$ où les lettres ont été décalées de $|\sigma|$. Par exemple $132\ConcDec3421 = 1326754$.
\end{Exemple}

\begin{Exemple}
\label{ex:prelim_poset:shuffle}
Le \emph{produit de mélange}\index{produit!de mélange} sur sur les mots se définit récursivement par
\begin{equation} 
\label{eq:prelim_posets:shuffle}
    u \shuffle v :=
    \begin{cases}
        u & \mbox{si $v = \epsilon$}, \\
        v & \mbox{si $u = \epsilon$}, \\
        u_1(u' \shuffle v) \enspace + \enspace  v_1(u \shuffle v')
                & \mbox{sinon, où $u_1, v_1 \in A$ et $u = u_1 u'$ et $v = v_1 v'$}.
    \end{cases}
\end{equation}
C'est la somme de tous les "mélanges" des lettres de $u$ et $v$ tels que l'ordre des lettres dans $u$ et $v$ respectivement ne soit pas modifié. Par exemple, 
\begin{align}
\label{eq:prelim_posets:shuffle_ex}
\red{ab} \shuffle ba &= \red{ab}ba + \red{a}b\red{b}a + \red{a}ba\red{b} + b\red{ab}a
+ b \red{a} a \red{b} + ba \red{ab} \\
&= 2~abba + 2~baab + abab + baba
\end{align}

Comme pour la concaténation, on peut définir le \emph{produit de mélange décalé}\index{produit!de mélange décalé} sur les permutations.
\begin{equation}
\label{eq:prelim_posets:shuffle_dec}
\sigma \cshuffle \mu = \sigma \shuffle \Dec{\mu}
\end{equation} 
par exemple :
\begin{equation}
\label{eq:prelim_posets:shuffle_dec_ex}
\red{12} \cshuffle 21 = \red{12}43 + \red{1}4\red{2}3 + \red{1}43\red{2} + 4\red{12}3
+ 4\red{1}3\red{2} + 43\red{12}
\end{equation}
\end{Exemple}

\begin{Definition}
\label{def:prelim_posets:alg_graduee}
L'algèbre $A$ d'une classe combinatoire est dite graduée\index{algèbre!graduée} si son produit vérifie la relation suivante :
\begin{equation}
\label{eq:prelim_posets:prod_gradue}
|x \times y | = |x| + |y|
\end{equation}
pour tout $x,y \in A$. Dit autrement, pour tout $n,m \in \mathbb{N}$, on a que le produit $\times$ est une application de $A_n \times A_m$ vers  $A_{n+m}$.

\end{Definition}

Les deux exemples précédents, la concaténation et le produit de mélange (resp. la concaténation décalée et le produit de mélange décalé) sont des produits gradués sur les mots (resp. sur les permutations).

\subsection{Opérations sur les mots et les permutations}
\label{sub-sec:prelim_posets:struct_elementaires:operations}
\index{facteur}
\index{préfixe}
\index{suffixe}
\index{sous-mot}

Les mots et les permutations sont des objets de base en combinatoire. De nombreuses opérations sur des objets plus complexes peuvent en fait s'exprimer à l'aide des mots et de la concaténation. Nous aurons besoin de certaines notions et opérations classiques que nous définissons dès maintenant.

Soit $u=u_1 \dots u_n$ un mot de taille $n$. Un \emph{facteur} de $u$ est un mot $v=u_i u_{i+1} \dots u_{i+m}$ avec $1 \leq i \leq n$ et $i+m \leq n$. Le mot $u$ s'écrit alors $u = u'vu''$ où $u'$ et $u''$ sont deux autres facteurs de $u$. Si $v$ est placé au début de $u$, c'est-à-dire si $v = u_1 \dots u_m$, on dit que $v$ est un \emph{facteur gauche} ou \emph{préfixe} de $u$. De même si $v$ est placé à la fin de $u$, c'est-à-dire si $v = u_{n-m} \dots u_n$, on dit que $v$ est un \emph{facteur droit} ou \emph{suffixe} de $u$. Un \emph{sous-mot} de $u$ est un mot $v = u_{j_1} \dots u_{j_m}$ où $1 \leq j_1 < j_2 < \dots < j_m \leq n$. C'est-à-dire que le sous-mot $v$ est composé d'un sous-ensemble des lettres de $u$ dont on a conservé l'ordre. Par exemple si $u= aabcba$ alors $aab$ est un préfixe, $ba$ un suffixe, $bc$ un facteur et $abb$ un sous-mot de $u$. Les préfixes et suffixes sont en particulier des facteurs, et les facteurs, des sous-mots. Le mot vide $\epsilon$ est à la fois préfixe, suffixe, facteur et sous-mot de n'importe quel mot $u$. On parle de facteur (resp. préfixe, suffixe ou sous-mot) propre quand on n'inclut pas le mot vide.

Toutes ces notions s'appliquent aussi aux permutations qui sont des mots particuliers. Par exemple la permutation $52431$ admet entre autres le préfixe $52$, le suffixe $431$, le facteur $24$ et le sous-mot $231$. On remarque qu'en général ces mots ne sont pas eux-mêmes des permutations. Une solution pour rester dans la même classe combinatoire est de \emph{standardiser} les mots. 

\begin{Definition}
\label{def:prelim_posets:std}
\index{standardisation}
Soit $u = u_1 \dots u_n$ un mot de taille $n$ sur un alphabet ordonné $A = \lbrace a_1, a_2, \dots \rbrace$ (par exemple, les entiers positifs). On dit que $u$ admet une inversion  $(i,j)$ avec $i<j$ si $u_i > u_j$. Le \emph{standardisé} de $u$, noté $\std(u)$, est l'unique permutation $\sigma$ de taille $n$ telle que les inversions de $\sigma$ soient exactement les inversions de $u$. 
\end{Definition}

Par exemple, si $u=baa$ sur l'alphabet ordonné $A = \lbrace a < b \rbrace$, alors $u$ admet deux inversions $(1,2)$ et $(1,3)$. La seule permutation de taille 3 admettant uniquement ces deux inversions est $\std(u) = 312$. De façon générale, la standardisation revient à numéroter les lettres de $u$ des plus petites lettres vers les plus grandes lettres et de la gauche vers la droite, cf. figure \ref{fig:prelim_posets:std}.

\begin{figure}[ht]
\centering

\begin{tabular}{ccccccccc}
$d$     & $d$     & $a$     & $b$     & $c$     & $b$     & $b$     & $a$     & $a$ \\
        &         & \red{1} &         &         &         &         & \red{2} & \red{3} \\
\\
$d$     & $d$     & $a$     & $b$     & $c$     & $b$     & $b$     & $a$     & $a$ \\
        &         & 1       & \red{4} &         & \red{5} & \red{6} & 2       & 3       \\
\\
$d$     & $d$     & $a$     & $b$     & $c$     & $b$     & $b$     & $a$     & $a$ \\
        &         & 1       & 4       & \red{7} & 5       & 6       & 2       & 3       \\
\\
$d$     & $d$     & $a$     & $b$     & $c$     & $b$     & $b$     & $a$     & $a$ \\
\red{8} & \red{9} & 1       & 4       & 7       & 5       & 6       & 2       & 3 \\
\\
\end{tabular}

$\std(ddabcbbaa) = 891475623$
\caption[Standardisé du mot $ddabcbbaa$.]{Standardisé du mot $ddabcbbaa$, on numérote d'abord les $a$, puis les $b$, et ainsi de suite.}
\label{fig:prelim_posets:std}
\end{figure}

Par définition, le standardisé d'une permutation est la permutation elle-même. On peut utiliser la standardisation sur les facteurs et sous-mots d'une permutation pour obtenir à nouveau des permutations. Par exemple, les préfixes standardisés de $52431$ sont $\lbrace \epsilon, 1,21, 312, 4132, 52431 \rbrace$. 

\begin{Definition}
\label{def:prelim_posets:motif}
\index{motif d'une permutation}
On dit qu'une permutation $\sigma$ admet comme \emph{motif} la permutation $\mu$ s'il existe un sous-mot de $\sigma$ dont le standardisé est la permutation $\mu$. Si $\sigma$ n'admet pas le motif $\mu$, on dit qu'elle \emph{évite} le motif.
\end{Definition}

Par exemple, la permutation $\sigma = 4213$ admet le motif $312$ car $413$ est sous-mot de $\sigma$ et $\std(413) = 312$. 

\section{Posets}
\label{sec:prelim_posets:posets}

Il est possible de définir des relations d'ordre sur les objets combinatoires.
Les ensembles partiellement ordonnés sont couramment appelés \emph{posets}, de l'anglais \emph{Partially Ordered Set}. Cet objet est fondamental dans notre travail. Les posets peuvent s'étudier à la fois comme des objets combinatoires en tant que tels ou comme des structures appliquées à d'autres objets combinatoires. C'est surtout dans ce second contexte que nous les rencontrerons. Nous présentons ici les notions fondamentales dont nous aurons besoin.

\subsection{Définition}
\label{sub-sec:prelim_posets:posets:def}

\begin{Definition}
\label{def:prelim_posets:posets}
\index{poset}
\index{relation!d'ordre}
Un  \emph{poset} est un ensemble $P$ muni d'une relation d'ordre $\leq$ vérifiant les conditions
\begin{enumerate}
\item de réflexivité : $\forall x \in P, x \leq x$;
\item de transitivité : $\forall x,y,z \in P$, si $x \leq y$ et $y \leq z$ alors $x \leq z$;
\item d'antisymétrie : $\forall x,y \in P$, si $x \leq y$ et $y \leq x$ alors $x=y$.
\end{enumerate}

Si la relation d'ordre est telle que pour tout $x,y \in P$, alors soit $x \leq y$ , soit $y \leq x$, on dit que l'ordre est total ou linéaire.
\end{Definition}

Nous ne traiterons dans ce mémoire que des cas où l'ensemble $P$ est fini.

\begin{Definition}
\label{def:prelim_posets:couv}
On dit que $y$ \emph{couvre} $x$ et on écrit $x \lessdot y$ si on a $x < y$ et qu'il n'existe pas de $z \in P$ tel que $x < z < y$. On appelle l'ensemble de ces relations les \emph{relations de couverture}.\index{relation!de couverture}
\end{Definition}

\begin{Definition}
\label{def:prelim_posets:minimal}
On dit qu'un élément $x\in P$ est \emph{minimal} (resp. \emph{maximal}) s'il n'existe pas d'élément $y \in P$, $y \neq x$, tel que $y \leq x$ (resp. $y \geq x$).
\end{Definition}

Les relations de couverture suffisent à définir le poset (les autres relations étant obtenues par transitivité). Pour représenter un poset, il suffit donc d'indiquer les éléments et leurs relations de couverture, c'est ce qu'on appelle un \emph{diagramme de Hasse}\index{diagramme de Hasse}. De façon générale, les posets sont représentés par leur diagramme de Hasse de la façon suivante : si $x$ est relié à $y$ par une arête et que $x$ est \emph{en dessous} de $y$ alors $x \lessdot y$, les plus petits éléments se trouvent donc en bas du diagramme, cf. figure \ref{fig:prelim_posets:exemple_poset}.

\begin{figure}[ht]
\centering
\input{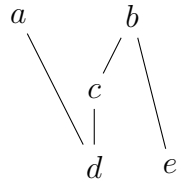}
\caption[Exemple de diagramme de Hasse.]{Exemple de diagramme de Hasse. Les relations de couverture sont $d\lessdot c \lessdot b$, $d\lessdot a$ et $e \lessdot b$. Les éléments $d$ et $e$ sont minimaux. L'ordre est partiel, par exemple $a$ et $b$ ne sont pas comparables.}
\label{fig:prelim_posets:exemple_poset}
\end{figure}

Lorsque les éléments du poset sont les entiers ou un alphabet quelconque, on peut considérer le poset comme un objet combinatoire en tant que tel. La taille est alors donnée par le nombre d'éléments et l'on peut par exemple chercher à dénombrer les posets d'une taille donnée sur l'alphabet $\lbrace 1, \dots, n \rbrace$ \cite{OEISPosets}. Dans le chapitre \ref{chap:tamari_intervalles}, nous définirons ainsi une classe spécifique de posets, les \emph{intervalles-posets} de Tamari liés à l'ordre de Tamari que nous dénombrerons. Cependant, la plupart du temps, nous considérons les posets non pas comme des objets combinatoires mais comme des structures sur les objets combinatoires. Ce sera le cas par exemple quand nous étudierons les ordres sur les permutations dans le chapitre \ref{chap:prelim_groupe_sym}. Pour des raisons de cohérence avec la littérature actuelle et nos propres articles, nous dessinons alors les posets "à l'envers". Les éléments minimaux sont en haut du diagramme et la relation de couverture se lit $y$ est \emph{en dessous} de $x$ si $x \lessdot y$ ($y$ couvre $x$). C'est le cas par exemple des ordres sur les permutations représentés dans les figures \ref{fig:perm_droit}, \ref{fig:perm_gauche} et \ref{fig:bruhat}.

\subsection{Vocabulaire de base}
\label{sub-sec:prelim_posets:posets:voc}

\begin{Definition}
\label{def:prelim_posets:chaine}
Une chaîne\index{chaîne} d'un poset $P$ est un ensemble d'éléments $\lbrace x_1, \dots, x_m \rbrace$ tel que
\begin{equation}
x_1 \leq x_2 \leq \dots \leq x_m.
\end{equation}
Si quelque soit $i \in \llbracket 1,m-1 \rrbracket$, on a $x_i \lessdot x_{i+1}$, alors la chaîne est dite \emph{saturée}\index{chaîne! saturée}.
\end{Definition}

Par exemple, $d\lessdot c \lessdot b$ est une chaîne saturée du poset donné figure \ref{fig:prelim_posets:exemple_poset}. 

\begin{Definition}
\label{def:prelim_posets:gradue}
\index{poset!gradué}
On dit qu'un poset est \emph{gradué} s'il existe une application $\Phi$ bien définie telle que
\begin{enumerate}
\item $\Phi(x) = 0$ si $x$ est minimal,
\item $\Phi(y) = \Phi(x) + 1$ si $x \lessdot y$.
\end{enumerate}
\end{Definition}

De façon équivalente, un poset est gradué si la longueur d'une chaîne saturée entre un élément $y \in P$ et un élément minimal $x$ du poset ne dépend ni de $x$, ni de la chaîne choisie. Ainsi, le poset donné figure \ref{fig:prelim_posets:exemple_poset} n'est pas gradué car $d \lessdot c \lessdot b$ et $e \lessdot b$ sont deux chaînes saturées de longueurs différentes d'éléments minimaux vers $b$.

Les définitions qui suivent donnent des constructions de posets à partir d'un poset donné. Des exemples de toutes les constructions (sous-poset, intervalle, poset quotient, etc.) sont illustrés figure \ref{fig:prelim_posets:sous-posets}.

\begin{figure}[ht]
\input{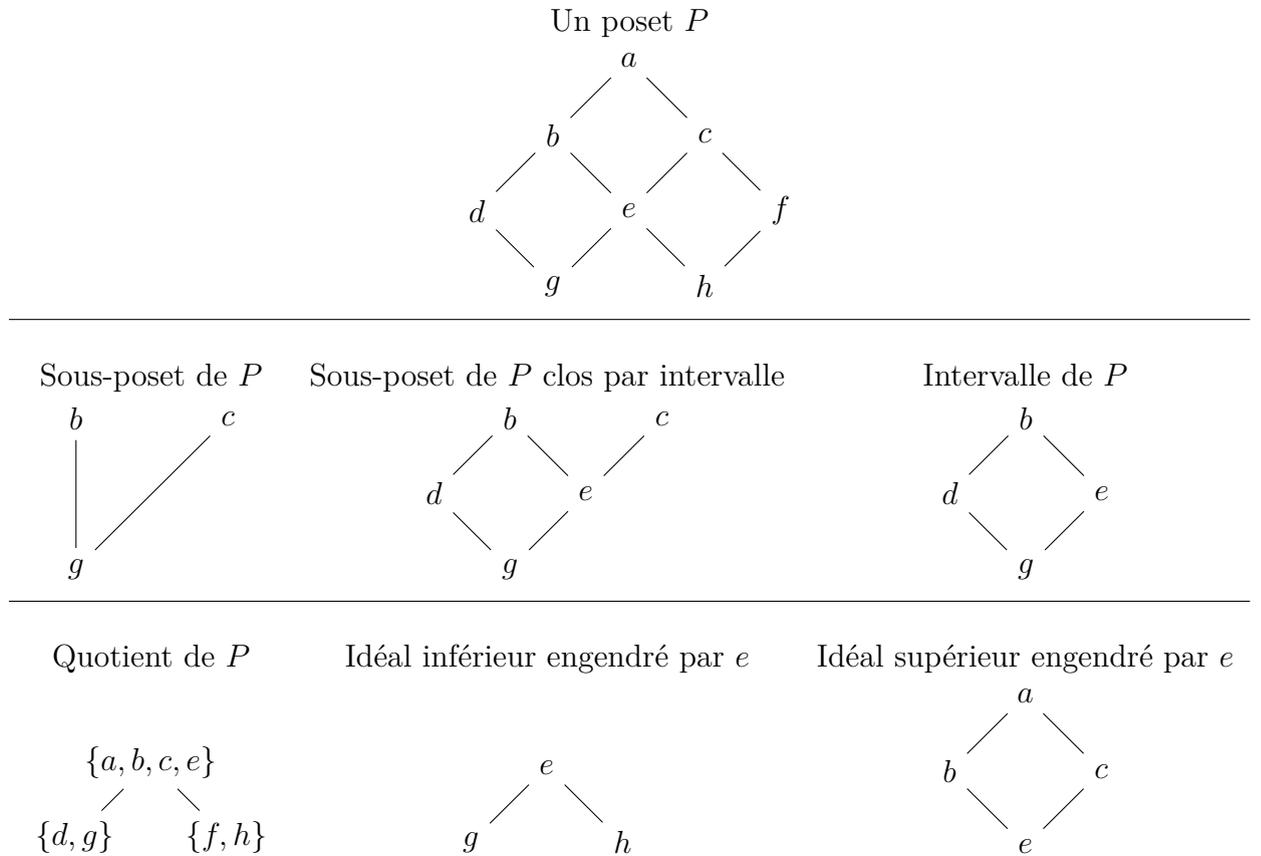}
\caption{Exemples de sous-posets et posets quotients}
\label{fig:prelim_posets:sous-posets}
\end{figure}

\begin{Definition}
\label{def:prelim_posets:sous-poset}
Un poset $P'$ est un sous-poset\index{poset!sous-poset} de $P$ si en tant qu'ensemble, $P' \subset P$ et si la relation d'ordre de $P'$ est la même que celle de $P$ restreinte aux éléments de $P'$.

Si pour tout $x,y \in P'$, on a que $x \leq y$ implique que $\forall z \in P$ avec $x \leq z \leq y$ alors $z \in P'$, on dit que $P'$ est \emph{clos par intervalle}.\index{clos par intervalle} Si de plus, $P'$ a un unique élément minimal et un unique élément maximal, on dit que $P'$ est un \emph{intervalle} de $P$.\index{intervalle}
\end{Definition}

\begin{Definition}
\label{def:prelim_posets:posets_iso}
Soit $\varphi : P_1 \rightarrow P_2$ une application d'un poset $P_1$ vers un poset $P_2$. On dit que $\varphi$ est un morphisme de posets si $\varphi$ préserve l'ordre des éléments. C'est-à-dire, que pour tout $x,y \in P_1$ on a
\begin{equation}
x \leq_{P_1} y \Rightarrow \varphi(x) \leq_{P_2} \varphi(y). 
\end{equation}
Si $\varphi$ est bijective et si $\varphi^{-1}$ est aussi un morphisme de posets, on dit que $\varphi$ est un isomorphisme de posets.\index{morphisme!de poset}\index{isomorphisme!de poset}
\end{Definition}

\begin{Definition}
\label{def:prelim_posets:poset_quotient}
Soit $\mathcal{P}$ une partition du poset $P$, on dit que $\mathcal{P}$ est un \emph{poset quotient}\index{poset!quotient} de $P$ si la relation définie sur $\mathcal{P}$ par 
\begin{equation}
\dot{x} \leq \dot{y} \Leftrightarrow \exists x \in \dot{x}, y \in \dot{y} \text{ tels que } x \leq y.
\end{equation}
pour tout $\dot{x},\dot{y} \in \mathcal{P}$ est une relation d'ordre.
\end{Definition}

On trouve dans \cite{PosetCongru} les conditions nécessaires et suffisantes que doit vérifier une telle partition pour que le quotient forme bien un poset.

\begin{Definition}
Soit $P$ un poset et $x \in P$, l'idéal inférieur (resp. supérieur) de $P$ engendré par $x$ est l'ensemble des éléments $y \leq x$ (resp. $y \geq x$).\index{poset!idéal}
\end{Definition}

\subsection{Extensions et extensions linéaires de poset}
\label{sub-sec:prelim_posets:posets:ext}
\index{extension!d'un poset}
\index{extension!linéaire}

Soient $(P,\leq_P)$ et $(Q, \leq_Q)$ deux posets sur un même ensemble d'éléments $E$. Si on a
\begin{equation}
\label{eq:prelim_posets:ext}
x \leq_P y \Rightarrow x \leq_Q y
\end{equation}
pour tout $x,y \in E$, alors $Q$ est une \emph{extension} de $P$. Construire une extension de $P$ revient donc à rajouter des relations au poset $P$. Si l'ordre du poset $Q$ est total, on dit que $Q$ est une extension linéaire de $P$. 

Une extension linéaire de $P$ peut se représenter comme un mot $u$ dont les lettres sont les éléments de $P$ selon la règle : si $a\leq_P b$ alors $a$ doit \^etre placé avant $b$ dans $u$. Par exemple, les mots $dceab$ et $edcba$ sont des extensions linéaires du poset représenté figure \ref{fig:prelim_posets:exemple_poset}. Si les éléments du poset sont des entiers de 1 à $n$, alors les extensions linéaires du poset peuvent \^etre vues comme des permutations, cf. figure \ref{fig:prelim_posets:ext-lin}.

\begin{figure}[ht]
\centering
\input{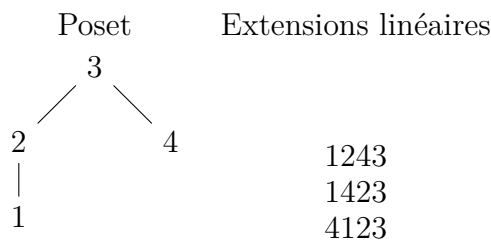}
\caption{Exemple d'extensions linéaires d'un poset}
\label{fig:prelim_posets:ext-lin}
\end{figure}

\subsection{Base d'un poset}
\label{sub-sec:prelim_posets:posets:base}
\index{base!d'un poset}

\begin{Definition}
\label{def:prelim_posets:base}
La base d'un poset $P$ est le plus petit ensemble $B \subset P$ tel que pour tout $x,y \in P$, on ait
\begin{equation}
x \leq_P y \Leftrightarrow \lbrace z \in B, z \leq x \rbrace \subset \lbrace z \in B, z \leq y \rbrace.
\end{equation}
\end{Definition}

C'est-à-dire qu'à chaque élément correspond une projection sur la base et que la comparaison des éléments peut se faire par comparaison ensembliste des projections. Dans \cite{LascouxCoxeter}, Lascoux et Schützenberger donnent une caractérisation des éléments de la base.

\begin{Lemme}
Un élément $x$ d'un poset $P$ appartient à la base de $P$ si et seulement si il existe $y\in P$ tel que $x$ soit un élément minimal de $P \backslash \lbrace z \leq y \rbrace$.
\end{Lemme}

Un exemple de calcul de la base d'un poset est donné figure \ref{fig:prelim_posets:base_poset}. De façon équivalente, il est possible de définir la \emph{cobase} d'un poset en utilisant les éléments maximaux et des projections d'idéaux supérieurs au lieu d'inférieurs.

\begin{figure}
\centering
\input{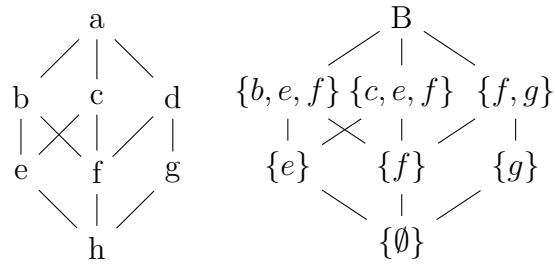}
\caption[Exemple d'un poset et de sa base]{Exemple d'un poset et de sa base. La base du poset est formée par les éléments $b,c,e,f,g$. En effet, $e,f$ et $g$ sont minimaux pour $P \backslash \lbrace z \leq h \rbrace$ et $b$ et $c$ sont minimaux pour respectivement $P \backslash \lbrace z \leq c \rbrace$ et $P \backslash \lbrace z \leq b \rbrace$. A droite, on a représenté le m\^eme poset où les éléments ont été remplacés par leur projection sur la base.} 
\label{fig:prelim_posets:base_poset}
\end{figure}

\section{Treillis}
\label{sec:prelim_posets:treillis}
\index{treillis}

\subsection{Définition et exemples}
\label{sub-sec:prelim_posets:treillis:def}

Soit $Z$ une partie d'un poset $P$. La borne inférieure de $Z$, notée $\pmin Z$, est l'unique élément $z$ tel que

\begin{equation}
y \leq z \Leftrightarrow \forall x \in Z, y \leq x 
\end{equation}
s'il existe ou $\emptyset$ sinon. De façon symétrique, la borne supérieure de $Z$, notée $\pmax Z$ est l'unique élément $z$ tel que
\begin{equation}
y \geq z \Leftrightarrow \forall x \in Z, y \geq x 
\end{equation}
s'il existe ou $\emptyset$ sinon. Quand une partie $Z$ ne comporte que deux éléments $x$ et $y$ on écrit aussi $x \pmin y = \pmin Z$ et $x \pmax y = \pmax Z$.

\begin{Definition}
\label{def:prelim_posets:treillis}
Un treillis est un poset $P$ tel que, pour toute partie $X$ de $P$, $\pmin X$ et $\pmax X$ sont différents de $\emptyset$.
\end{Definition}

La figure \ref{fig:prelim_posets:treillis} présente un exemple de poset qui n'est pas un treillis. En effet, on a que $b \pmin c = \emptyset$ et symétriquement $e \pmax f = \emptyset$. 

\begin{figure}[ht]
\centering
\input{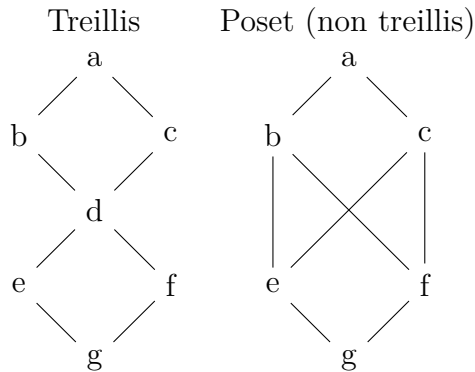}
\caption{Exemple et contre-exemple de treillis}
\label{fig:prelim_posets:treillis}
\end{figure}

\index{base!d'un treillis}
Dans le cas d'un treillis, le calcul de la base est trivial. On a que $x$ appartient à la base du treillis si et seulement si $x$ couvre un unique élément. Pour le treillis donné en exemple dans la figure \ref{fig:prelim_posets:treillis}, la base est donc formée de quatre éléments : $b,c,e$ et $f$.

\subsection{Treillis enveloppant}
\label{sub-sec:prelim_posets:treillis:env}
\index{treillis!enveloppant}

Un résultat de MacNeille \cite{MacNeille} est que tout poset $P$ possède un \emph{treillis enveloppant}, c'est-à-dire qu'il existe un treillis minimal $T$ tel que $P$ soit un sous-poset de $T$ à isomorphisme d'ordre près. La construction de ce treillis est donnée en particulier dans \cite{lattices} et \cite{LascouxCoxeter}. 

Le treillis enveloppant $T$ d'un poset $P$ est le plus petit ensemble de parties de $P$ clos par intersection, contenant $P$ ainsi que tous les idéaux supérieurs de $P$. La relation d'ordre est l'inclusion ensembliste et le morphisme plongeant $P$ dans $T$ est celui qui envoie chaque élément sur son idéal. On pourra vérifier dans la figure \ref{fig:prelim_posets:treillis} que le treillis de gauche est le treillis enveloppant du poset de droite. En effet, l'élément supplémentaire $d$ est en fait l'intersection des idéaux de $b$, $\lbrace b,e,f,g \rbrace$ et de $c$, $\lbrace c,e,f,g \rbrace$. 

La propriété suivante du treillis enveloppant est donnée dans \cite{LascouxCoxeter}.

\begin{Proposition}
La base d'un poset $P$ est la même que celle de son treillis enveloppant.
\end{Proposition}

Pour l'exemple de la figure \ref{fig:prelim_posets:treillis}, le poset et son treillis enveloppant ont tous les deux comme base $\lbrace b,c,e,f\rbrace$.

\chapter{Ordres du groupe symétrique}
\label{chap:prelim_groupe_sym}

Dans ce chapitre, nous rappelons des propriétés connues du groupe symétrique et en particulier les structures d'ordre qui lui sont associées. Le paragraphe \ref{sec:prelim_groupe_sym:group_sym} rappelle la structure de groupe sur les permutations et énonce les propriétés de base dont nous aurons besoin. Les ordres dits \emph{faibles} sont définis dans le paragraphe~\ref{sec:prelim_groupe_sym:perm}, ils seront utilisés dans la suite de cette thèse principalement dans la partie \ref{part:tamari}. Le paragraphe \ref{sec:prelim_groupe_sym:bruhat} est dédiée à l'étude de l'ordre dit \emph{fort} ou \emph{de Bruhat}. Nous rappelons en particulier comment comparer les éléments et calculer des bornes supérieures, propriétés utilisées dans le chapitre \ref{chap:polynomes_grothendieck}. 

Les résultats que nous énonçons ici sont pour la plupart bien connus. Notre principale référence est l'article \cite{LascouxCoxeter} de Lascoux et Schützenberger qui décrit en détail les ordres sur les groupes de Coxeter. 
 Par ailleurs, dans le paragraphe \ref{sub-sec:prelim_groupe_sym:bruhat:coset}, nous prouvons aussi quelques propriétés particulières sur les intervalles de type coset dont nous nous servirons dans le chapitre~\ref{chap:polynomes_grothendieck}.

\section{Le groupe symétrique}
\label{sec:prelim_groupe_sym:group_sym}
\index{groupe symétrique}
\index{permutations}

\subsection{Structure de groupe sur les permutations}
\label{sub-sec:prelim_groupe_sym:groupe_sym:def}

La définition \ref{def:prelim_posets:permutations} donne une description des permutations en terme de mots sur l'alphabet $1, \dots, n$. On peut aussi les interpréter comme des bijections de $\lbrace 1, \dots, n \rbrace$ vers $\lbrace 1, \dots, n \rbrace$. Dans ce cas, le mot correspondant à la permutation est simplement la lecture des images. Par exemple, la permutation $\mu = 4213$ est l'application qui envoie 1 sur 4, 2 sur 2, 3 sur 1 et 4 sur 3. 

Vu de cette façon, l'ensemble des permutations de taille $n$ forme un groupe non commutatif dont la loi est la composition des applications. L'élément neutre est la permutation identité, c'est-à-dire le mot $123\dots n$. Par exemple, si $\sigma = 2314$ et $\mu = 4213$, on a

\begin{align}
\sigma \mu &:= \sigma \circ \mu = 4321, \\
\mu \sigma &:= \mu \circ \sigma = 2143, \\
\sigma^{-1} &= 3124.
\end{align}

\index{point fixe}
\index{transposition}
\index{transposition!simple}
Un \emph{point fixe} de $\sigma$ est un entier $i$ tel que $\sigma(i) = i$. La permutation identité ne contient que des points fixes. Le seul point fixe de la permutation $\mu = 4213$ est 2. Une \emph{transposition} est une permutation où seuls deux points ne sont pas fixes, on dit qu'elle "échange" deux valeurs. Par exemple, $153426$ est la transposition qui échange 2 et 5 et on la note $(2,5)$. Une \emph{transposition simple} est une transposition où les valeurs échangées sont consécutives, on l'écrit $s_i := (i,i+1)$. Lorsqu'on mutliplie par la droite une permutation $\sigma$ par une transposition $(i,j)$, cela revient à échanger les lettres en positions $i$ et $j$ dans $\sigma$. Si l'on multiplie par la gauche, alors on échange les valeurs $i$ et $j$. Par exemple :

\begin{align}
342615.(2,5) &= 312645 \\
(2,5).342615 &= 345612.
\end{align}

\index{cycles d'une permutation}
En tant qu'élément de groupe, une permutation peut se décomposer en un produit de cycles disjoints. Un \emph{cycle} est une permutation particulière qui s'écrit $(a_1, a_2, \dots, a_r)$ et qui signifie que l'image de $a_1$ est $a_2$, l'image de $a_2$ est $a_3$, etc, et l'image de $a_r$ est $a_1$. Les éléments qui n'apparaissent pas dans $a_1, \dots, a_r$ sont des points fixes. Deux cycles $c_1$ et $c_2$ sont dits disjoints s'ils n'agissent pas sur les mêmes valeurs, c'est-à-dire si les éléments qui apparaissent dans $c_1$ sont des points fixes de $c_2$. Une permutation peut toujours s'écrire comme un produit de cycles disjoints. Par exemple, si $\sigma = 524613$, on a $\sigma = (1,5)(3,4,6)$. Le produit entre deux cycles disjoints est commutatif, on a aussi $\sigma = (3,4,6)(1,5)$. Cependant, la décomposition en cycle est unique à commutation des cycles près. Les cycles de taille 1 sont des points fixes et les cycles de taille 2 des transpositions ce qui justifie l'écriture $(a,b)$ pour la transposition qui échange $a$ et $b$.

Le groupe sur les permutations est aussi appelé \emph{groupe symétrique} et on le note $\Sym{n}$. Le groupe symétrique de taille $n$ est engendré par les transpositions simples $s_i$ pour $1\leq i < n$ et admet la présentation suivante :

\begin{align}
\label{eq:prelim_groupe_sym:rel_carree}
s_i ^2 &= 1, \\
\label{eq:prelim_groupe_sym:tresse1}
s_i s_j &= s_j s_i \text{ pour } |i-j| > 1, \\
\label{eq:prelim_groupe_sym:tresse2}
s_i s_{i+1} s_i &= s_{i+1} s_i s_{i+1}.
\end{align}

\index{relation! de tresses}
Les relations \eqref{eq:prelim_groupe_sym:tresse1} et \eqref{eq:prelim_groupe_sym:tresse2} sont appelées \emph{relations de tresses}. Elles ne sont pas propres au groupe symétrique. En particulier, on peut définir des déformations du groupe symétrique où seule la relation \eqref{eq:prelim_groupe_sym:rel_carree} est modifiée. C'est le cas par exemple de l'algèbre de Hecke que nous verrons dans le chapitre \ref{chap:polynomes_action}. 

\subsection{Décompositions réduites et code}
\label{sub-sec:prelim_groupe_sym:groupe_sym:decomp}

\index{décomposition réduite}
\index{longueur d'une permutation}
Toute permutation peut donc \^etre écrite comme un produit de transpositions simples. On appelle cette écriture une \emph{décomposition} de la permutation. Au vu des relations entre générateurs, elle n'est pas unique. La taille d'une décomposition est donnée par le nombre de générateurs qu'elle contient. Si la taille d'une décomposition est minimale, c'est-à-dire s'il n'existe pas d'autres décompositions de la permutation contenant moins de générateurs, on l'appelle une \emph{décomposition réduite} de la permutation et sa taille définit la \emph{longueur} de la permutation, notée $\ell(\sigma)$. Une permutation admet en général plusieurs décompositions réduites.

\begin{Exemple}
La permutation $\sigma = 523461$ admet entre autres comme décompositions réduites :
\begin{align}
\sigma &= s_4s_3s_2s_1s_2s_3s_4s_5, \\
\sigma &= s_1s_2s_3s_4s_5s_3s_2s_1,
\end{align}
elle est donc de longueur 8.
\end{Exemple}

Notons que si $\sigma$ admet une décomposition réduite $v$, une décomposition réduite de $\sigma^{-1}$ est simplement le retourné de $v$ car $s_i^2 = 1$. Une permutation et son inverse ont donc m\^eme longueur. La longueur d'une permutation est aussi donnée par son nombre \emph{d'inversions}. 

\index{inversions}
\index{coinversions}
\begin{Definition}
Une inversion d'une permutation $\sigma$ est un couple $(i,j)$ avec $i<j$ et $\sigma(i)>\sigma(j)$. 

Une coinversion d'une permutation $\sigma$ est un couple $(a,b)$ tel que $a<b$ et $\sigma^{-1}(a) > \sigma^{-1}(b)$. 
\end{Definition}

Par exemple, les inversions de $\sigma = 523461$ sont $(1,2)$, $(1,3)$, $(1,4)$, $(1,6)$, $(2,6)$, $(3,6)$, $(4,6)$ et $(5,6)$ ses coinversions sont $(1,2)$, $(1,3)$, $(1,4)$, $(1,5)$, $(1,6)$, $(2,5)$, $(3,5)$, et $(4,5)$. Une permutation est entièrement codée par ses inversions. On peut aussi utiliser son \emph{code de Lehmer}.

\index{code de Lehmer}
\begin{Definition}
Le \emph{code de Lehmer} (ou tout simplement code) d'une permutation $\sigma$ est le vecteur $v$ défini par  $v(i) = \# \lbrace \sigma(j) < \sigma(i), j>i \rbrace$. 
\end{Definition}

Par exemple, le code de la permutation $\sigma = 523461$ est $v=\left[4,1,1,1,1,0\right]$. Le code de Lehmer suffit pour retrouver la permutation de départ. En effet, si $v(1) = i$, alors $\sigma(1)=i\!+\!1$. Puis si $v(2) = j$, alors $\sigma(2)$ est le $j\!+\!1$ème nombre restant par ordre croissant et ainsi de suite. On peut donc travailler indifféremment avec une permutation ou avec son code.

\index{permutation!maximale}
De façon évidente, pour $n$ donné, il n'existe qu'une seule permutation de taille $n$ qui n'ait pas d'inversions. C'est la permutation identité $12\dots n$ dont la longueur est $0$, c'est-à-dire dont la décomposition réduite est le mot vide. Il existe aussi une seule permutation de longueur maximale, donnée par $n~n-1 \dots 1$. Son code est $\left[ n-1, n-2, \dots, 1, 0 \right]$ et sa longueur est donc $\frac{n(n-1)}{2}$. On l'appelle la \emph{permutation maximale} et on la note $\omega$.

Nous aurons besoin de deux autres notions sur les permutations.

\index{descente d'une permutation}
\index{recul d'une permutation}
\begin{Definition}
Une \emph{descente} d'une permutation $\sigma$ est une position $i$ telle que $\sigma(i) > \sigma(i+1)$. Un recul de $\sigma$ est une descente de $\sigma^{-1}$, c'est-à-dire une valeur $a$ telle que $\sigma^{-1}(a)>\sigma^{-1}(a+1)$.
\end{Definition}

Les descentes de la permutation $523461$ sont 1 et 5, ses reculs sont 1 et 4.

\subsection{Groupes de Coxeter}
\label{sub-sec:prelim_groupe_sym:groupe_sym:coxeter}

Le groupe symétrique appartient à une catégorie plus large de groupes appelés les \emph{groupes de Coxeter}. Bien que l'objet principal de cette thèse soit le groupe symétrique, de nombreuses notions que nous abordons admettent une définition plus générale. Notre travail s'inscrit donc au sein de cette théorie et nous avons jugé utile d'en donner les définitions de base. Pour une approche plus complète, nous invitons le lecteur à lire \cite{Coxeter}. 

\begin{Definition}
\label{def:prelim_groupe_sym:coxeter}
Un groupe de Coxeter est un groupe engendré par une famille de générateurs $r_1, \dots, r_n$ admettant une présentation sous forme $(r_i r_j)^{m_{ij}} = 1$ où $m_{ij} \in \mathbb{N} \cup +\infty$ et vérifie :
\begin{enumerate}
\item $m_{ii} = 1$
\item $m_{ij} = m_{ji}$
\item $m_{ij} \geq 2$ si $i\neq j$.
\end{enumerate}
\end{Definition}

Pour $\Sym{n}$, les générateurs sont $s_1, \dots, s_{n-1}$ et on a $m_{ij} = 2$ si $|i-j| \neq 1$ et $m_{i~i+1} = 3$. Les groupes de Coxeter finis ont été classifiés dès 1935, ils correspondent à des groupes de réflexions dans l'espace euclidien. Le groupe $\Sym{n}$ est appelé \emph{groupe de Coxeter de type $A$} ou $A_{n-1}$. Les groupes de type $BC$ et $D$ seront abordés dans le chapitre \ref{chap:polynomes_sage} et nous en donnons donc une définition ici. 

\begin{Definition}
\label{def:prelim_groupe_sym:bn}
Le groupe de Coxeter de type $BC$ de taille $n$ est le groupe engendré par l'ensemble des générateurs et relations de $A_n$ ($s_1, \dots, s_{n-1}$) et un générateur supplémentaire $s_n^{B}$ (aussi noté $s_n^C$) vérifiant les relations:
\begin{align}
s_i s_n^B &= s_n^B s_i \text{ pour } i \neq n-1 \\
s_{n-1} s_n^B s_{n-1} s_n^B &= s_n^B s_{n-1} s_n^B s_{n-1}.
\end{align}
C'est-à-dire $m_{i n} = 2$ pour $i \neq n-1$ et $m_{n-1~n} = 4$.
\end{Definition}

\begin{Definition}
\label{def:prelim_groupe_sym:dn}
Le groupe de Coxeter de type $D$ de taille $n$ est le groupe engendré par l'ensemble des générateurs et relations de $A_n$ ($s_1, \dots, s_{n-1}$) et un générateur supplémentaire $s_n^{D}$ vérifiant les relations:
\begin{align}
s_i s_n^B &= s_n^B s_i \text{ pour } i \neq n-2 \\
s_{n-2} s_n^D s_{n-2} &= s_{n-2} s_n^D s_{n-2}.
\end{align}
C'est-à-dire $m_{i n} = 2$ pour $i \neq n-2$ et $m_{n-2~n} = 3$.
\end{Definition}

Nous verrons dans le chapitre \ref{chap:polynomes_action} que les groupes $BC$ et $D$ peuvent être identifiés respectivement aux permutations signées et permutations signées avec un nombre pair de négatifs. 

Les notions de \emph{décomposition réduite} et \emph{longueur} des éléments décrites dans le paragraphe \ref{sub-sec:prelim_groupe_sym:groupe_sym:decomp} s'étendent naturellement à l'ensemble des groupes de Coxeter. De m\^eme, les descentes peuvent \^etre définies comme suit: un générateur $r_i$ est une descente de $\sigma$ si $\ell(\sigma r_i) = \ell(\sigma) - 1$, c'est un recul de $\sigma$ si $\ell(r_i \sigma) = \ell(\sigma) - 1$. Dans le cas du groupe symétrique, le générateur $s_i$ est identifié à la position $i$ et on dit que la permutation a une descente en $i$ si $s_i$ vérifie la propriété ci-dessus.

\section{Ordres faibles : treillis sur les permutations}
\label{sec:prelim_groupe_sym:perm}

\subsection{Définition}
\label{sub-sec:prelim_groupe_sym:perm:def}
\index{ordre!faible!droit}
\index{ordre!faible!gauche}

\begin{Definition}
\label{def:permu-facteur}
Deux éléments $\sigma$ et $\mu$ de $\Sym{n}$ sont comparables pour \emph{l'ordre faible droit} (ou simplement ordre droit) s'il existe une décomposition réduite de $\sigma$ qui soit un facteur gauche d'une décomposition réduite de $\mu$. On écrit $\sigma \infd \mu$ .

De même, deux éléments sont comparables pour \emph{l'ordre faible gauche} s'il existe une décomposition réduite de $\sigma$ qui soit un facteur droit d'une décomposition réduite de $\mu$, et on écrit $\sigma \infg \mu$.
\end{Definition}

Par exemple, la permutation $3142= s_2 s_1 s_3$ est plus petite pour l'ordre droit que $4312 = s_2 s_1 s_3 s_2 s_1$ et plus petite pour l'ordre gauche que $4231 = s_3 s_1 s_2 s_1 s_3$. Cette définition n'est pas spécifique au type $A$ mais nous étudierons ici uniquement le cas de $\Sym{n}$.

Les relations de couverture se déduisent facilement de la définition. Une permutation $\sigma$ est couverte par une permutation $\mu$ dans l'ordre droit (c'est-à-dire que $\mu$ est son successeur direct) si $\mu = \sigma s_i$ avec $\ell(\mu) = \ell(\sigma) +1$. On a multiplié $\sigma$ \emph{par la droite} avec une transposition simple, d'où la terminologie \emph{ordre droit}. De la m\^eme façon, dans le cas de l'ordre gauche, on multiplie \emph{par la gauche} par une transposition simple.

On a $\ell(\sigma s_i) = \ell(\sigma) +1$ si et seulement si $\sigma s_i$ a une descente en $i$, la transposition $s_i$ échange alors $\sigma(i)$ et $\sigma(i+1)$. Le nombre d'éléments couverts par une permutation $\mu$ dans l'ordre droit est donc exactement son nombre de descentes.

De m\^eme, dans l'ordre gauche, $\ell(s_i \sigma) = \ell(\sigma) +1$ si et seulement si $s_i \sigma$ a un recul en $i$ et la transposition $s_i$ échange les valeurs $i$ et $i+1$ dans $\sigma$. Le nombre d'éléments couverts par une permutation $\mu$ dans l'ordre gauche est donc son nombre de reculs.

\index{permutoèdre}
De par leur définition, les ordres droits et gauches sont gradués par le nombre d'inversions des permutations. Ils sont en fait isomorphes par passage à l'inverse. Plongé dans $\mathbb{R}^{n-1}$, leur diagramme de Hasse est un polytope connu, le \emph{permutoèdre}, et on appelle parfois ces ordres les \emph{ordres du permutoèdre}. Dans toute cette thèse, on représentera toujours l'ordre faible tel que le plus petit élément (la permutation identité) soit en haut du diagramme. Les diagrammes de Hasse des ordres faibles droits et gauches pour les tailles 3 et 4 sont donnés figures \ref{fig:perm_droit} et \ref{fig:perm_gauche}. Les relations de couverture correspondent aux transpositions simples : à chaque passage par une arête du graphe, on échange deux valeurs à des positions consécutives dans le cas de l'ordre droit et deux valeurs consécutives dans le cas de l'ordre gauche. Une décomposition réduite d'une permutation est un chemin dans l'ordre droit entre l'identité et la permutation où chaque arête correspond à une transposition simple. 

\begin{figure}[ht]
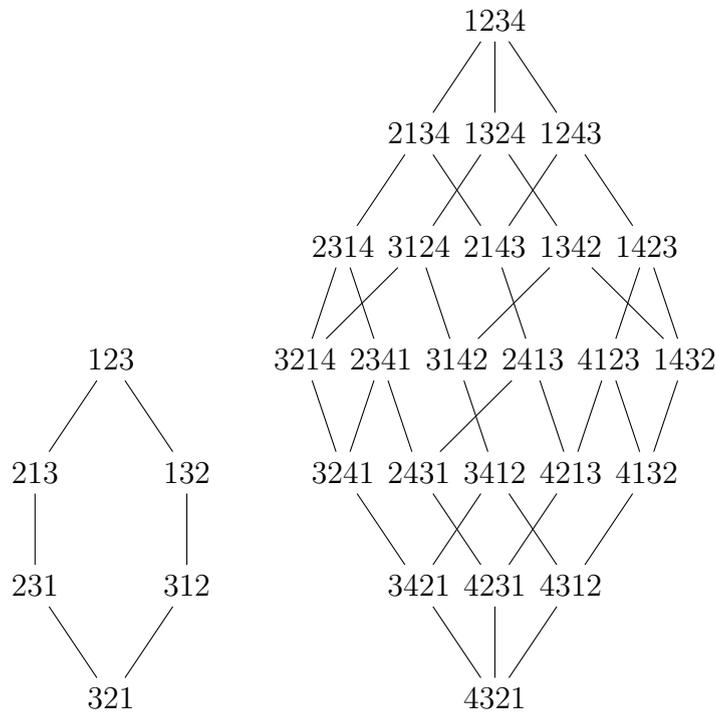

\centering
\begin{tabular}{cc}
\input{includes/figures/perm_droit3.tex} &
\input{includes/figures/perm_droit4.tex}
\end{tabular}
\caption{Ordre faible droit pour les tailles 3 et 4.}
\label{fig:perm_droit}
\end{figure}

\begin{figure}[ht]
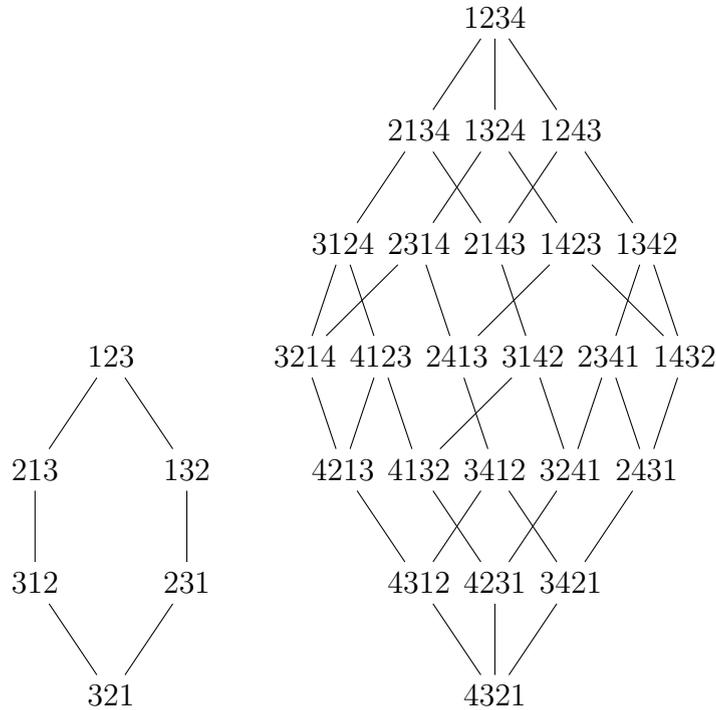

\centering
\begin{tabular}{cc}
\input{includes/figures/perm_gauche3.tex} &
\input{includes/figures/perm_gauche4.tex}
\end{tabular}
\caption{Ordre faible gauche pour les tailles 3 et 4.}
\label{fig:perm_gauche}
\end{figure}

\subsection{Structure de treillis}
\label{sub-sec:prelim_groupe_sym:perm:treillis}
\index{treillis}

\index{permutations!grassmanniennes}
Les ordres, droit et gauche, possèdent chacun une structure de treillis \cite{PermTreillis}. Comme on l'a vu paragraphe \ref{sec:prelim_posets:treillis}, la base d'un treillis est formée des éléments couvrant un unique élément. De ce fait, la base de l'ordre droit est formée des permutations n'ayant qu'une seule descente, appelées permutations \emph{grassmanniennes}. De même, la base de l'ordre gauche est formée par l'inverse des permutations grassmanniennes, c'est-à-dire les permutations n'ayant qu'un seul recul. La figure \ref{fig:prelim_groupe_sym:perm:bases} représente les deux bases en taille 4.

\begin{figure}[ht]
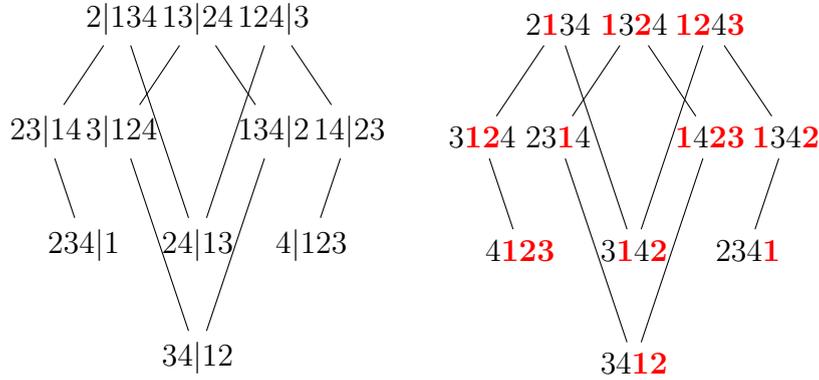

\centering
\begin{tabular}{cc}
\input{includes/figures/grassmanniennes_4} &
\input{includes/figures/grassmanniennes_inv_4}
\end{tabular}
\caption[Permutations grassmanniennes et inverses.]{Permutations grassmanniennes et inverses, bases de l'ordre faible.}
\label{fig:prelim_groupe_sym:perm:bases}
\end{figure}

Pour comparer deux permutations, il suffit donc de comparer les ensembles de permutations grassmanniennes qui leur sont inférieures. On a en fait une propriété plus forte,

\begin{align}
\sigma \infd \mu &\Leftrightarrow \coinv(\sigma) \subset \coinv(\mu), \\
\sigma \infg \mu &\Leftrightarrow \inv(\sigma) \subset \inv(\mu), 
\end{align}
où $\coinv(\sigma)$ et $\inv(\sigma)$ correspondent respectivement à l'ensemble des coinversions et des inversions de $\sigma$. Par exemple, les coinversions de la permutation $3124$ sont $(1,3)$ et $(2,3)$. Elle est plus petite pour l'ordre droit que $4312$ dont les coinversions sont $(1,3)$, $(1,4)$, $(2,3)$, $(2,4)$ et $(3,4)$.

On appelle $\kSym{n}{k}$ l'ensemble des éléments de $\Sym{n}$ ayant une unique descente en $k$ auquel on adjoint l'identité. Cet ensemble est un treillis pour l'ordre droit \cite{BW2}. Ses éléments sont les permutations minimales des classes du quotient de $\Sym{n}$ par $\Sym{k} \times \Sym{n-k}$, le groupe engendré par les $s_i$ avec $i \neq k$. On aura besoin de l'application suivante $\krho{k} : \Sym{n} \rightarrow \kSym{n}{k}$ qui à chaque permutation $\sigma$ associe l'élément minimal de son coset $\sigma (\Sym{k} \times \Sym{n-k})$. Cela revient à réordonner les valeurs de la permutations à gauche et à droite de $k$. Par exemple :

\begin{equation}
\krho{3}(361425) = 136 \vert 245.
\end{equation}

De même, $\bSym{n}{r}$, l'ensemble des éléments de $\Sym{n}$ ayant un unique recul en $r$ auquel on adjoint l'identité, est aussi un treillis. On définit $\brho{r}$ qui à chaque permutation $\sigma$ associe l'élément minimal de son coset $(\Sym{r} \times \Sym{n-r})\sigma$. Cela revient à réordonner séparément les valeurs inférieures et supérieures à $r$. Par exemple :

\begin{equation}
\brho{4}(\red{3}6\red{142}5) = \red{1}5\red{234}6.
\end{equation} 

Par ailleurs, nous aurons aussi besoin de l'application $\mrho{k}$ qui à une permutation $\sigma$ associe cette fois l'élément maximal de son coset $\sigma(\Sym{k} \times \Sym{n-k})$. Dans l'exemple précédent, cela donne :

\begin{equation}
\mrho{3}(361425) = 631 \vert 542.
\end{equation}

\section{Ordre de Bruhat}
\label{sec:prelim_groupe_sym:bruhat}
\index{ordre!de Bruhat}
\index{ordre!fort}

\subsection{Définition}
\label{sub-sec:prelim_groupe_sym:bruhat:def}

\begin{Definition}
\label{def:bruhat-sous-mot}
Deux éléments $\sigma$ et $\mu$ de $\Sym{n}$ sont comparables pour l'\emph{ordre de Bruhat} (ou \emph{ordre fort}) si une décomposition réduite de $\sigma$ est un sous-mot d'une décomposition réduite de $\mu$. On écrit $\sigma \infb \mu$.
\end{Definition}

Par exemple, la permutation $2143 = s_1 s_3$ est plus petite que $3421 = s_2 s_1 s_2 s_3 s_2$. Pour alléger l'écriture, la notations $\leq$ pour les permutations désignera dans toute cette section $\infb$. De façon générale, il suffit de considérer les sous-mots d'une décomposition réduite arbitraire d'une permutation pour obtenir l'ensemble des permutations qui lui sont inférieures \cite{bourbaki}. Un facteur étant un sous-mot particulier, l'ordre de Bruhat est une extension à la fois de l'ordre faible gauche et de l'ordre faible droit, cependant il n'est pas restreint à l'union de ces deux ordres. Par exemple, $4123 = s_3 s_2 s_1$ est plus grande que $2143 = s_3 s_1$ mais n'est son successeur ni pour l'ordre droit, ni pour l'ordre gauche.

\index{lemme d'échange}
Historiquement, l'ordre de Bruhat a d'abord été décrit comme la clôture transitive de la relation de couverture suivante.

\begin{Proposition}
\label{prop:bruhat-couv}
Une permutation $\sigma$ est couverte par une permutation $\mu$ si et seulement si $\mu = \sigma \tau$ avec $\tau$ une transposition et $\ell(\mu) = \ell(\sigma) + 1$. 

On dira alors que la transposition $\tau$ est une \emph{transposition de Bruhat} pour la permutation $\sigma$.
\end{Proposition}

Cette relation diffère de celle de l'ordre droit car la transposition $\tau$ n'est plus nécessairement une transposition simple. Par ailleurs, on peut donner une définition similaire avec un produit à gauche qui donnera le même ordre. En effet, si $\mu = \sigma \tau$ alors il existe une transposition $\tau'$ telle que $\mu = \tau' \sigma$. La proposition suivante donne un critère simple pour savoir si une transposition $\tau$ est une transposition de Bruhat pour une permutation $\sigma$.

\begin{Proposition}
\label{prop:prelim_groupe_sym:bruhat-critere}
Soit $\tau = (a,b)$ et $\sigma \in \Sym{n}$. Alors $\tau$ est une transposition de Bruhat pour $\sigma$ si et seulement si :
\begin{enumerate}
\item $\sigma(a) < \sigma(b)$
\item Il n'existe pas de $i$ tel que $a < i < b$ et $\sigma(a) < \sigma(i) < \sigma(b)$.
\end{enumerate}
En d'autres termes, les valeurs \emph{positionnées} entre $\sigma(a)$ et $\sigma(b)$ ne sont pas \emph{comprises} entre $\sigma(a)$ et $\sigma(b)$.
\end{Proposition}

Cette proposition est juste une reformulation de la condition $\ell(\sigma \tau) = \ell(\sigma) +1$. En effet, si $\sigma(a) < \sigma(b)$ alors on augmente la longueur de $\sigma$ en les inversant. Si $a < i < b$ et $\sigma(i)<\sigma(a)$, l'inversion $(a,i)$ est conservée. De même, si $\sigma(i) > \sigma(b)$, l'inversion $(i,b)$ est conservée. Mais si $\sigma(a) < \sigma(i) < \sigma(b)$, alors la transposition $(a,b)$ rajoute les inversions $(a,i)$ et $(i,b)$ en plus de $(a,b)$ qui ne sont pas compensées. Par ailleurs, il n'y a pas d'inclusion des inversions ou des coinversions entre une permutation et son successeur. Par exemple, si $a < b < c$ et $\sigma(a) < \sigma(c) <\sigma(b)$, l'inversion $(b,c)$ devient une inversion $(a,c)$.  Par exemple, si $\sigma = 251436$, la transposition $(1,4)$ est une transposition de Bruhat. Le successeur de $\sigma$ est alors $\sigma' = 451236$. La permutation $\sigma$ comporte 5 inversions : $(1,3)$, $(2,3)$, $(2,4)$, $(2,5)$ et $(4,5)$. Son successeur en comporte bien 6: $(1,3)$, $(1,4)$, $(1,5)$, $(2,3)$, $(2,4)$, et $(2,5)$. Par contre, $(1,6)$ n'est pas une transposition de Bruhat pour $\sigma = 251436$ à cause des valeurs $4$ et $5$ qui sont à la fois comprises et positionnées entre les valeurs 2 et 6.

La correspondance entre l'ordre obtenu par la clôture transitive de la relation et l'ordre sur les sous-mots se prouve par récurrence sur la longueur des permutations en utilisant une propriété fondamentale, le \emph{lemme d'échange} décrit par Bourbaki \cite{bourbaki}.

\begin{Lemme}[Lemme d'échange]
\label{lem:}
Soit $\sigma$ une permutation et $s_i$ une transposition simple telle que $\ell(\sigma s_i) > \ell(\sigma)$, alors l'idéal inférieur engendré par $\sigma s_i$ se décompose en trois ensembles disjoints :
\begin{itemize}
\item $I_1 = \lbrace \mu \leq \sigma, \mu s_i \leq \sigma \rbrace$,
\item $I_2 = \lbrace \nu \leq \sigma, \nu s_i > \sigma \rbrace$,
\item $I_3 = \lbrace \eta > \sigma, \eta s_i \leq \sigma \rbrace$.
\end{itemize}
Dit autrement, il n'existe pas de permutation $\sigma s_i \geq \eta > \sigma$ telle que $\eta s_i > \sigma$.
\end{Lemme}

De la proposition~\ref{prop:bruhat-couv}, on déduit immédiatement que l'ordre est gradué. Les diagrammes de Hasse pour les tailles 3 et 4 sont donnés en figure \ref{fig:bruhat}. Contrairement aux ordres faibles, l'ordre de Bruhat n'est pas un treillis. Cela apparaît dès la taille 3 : on a que $\pmin \lbrace 231, 312 \rbrace = \emptyset$. Par ailleurs, on a vu que les ordres faibles étaient isomorphes l'un à l'autre par passage à l'inverse. Dans le cas de l'ordre de Bruhat, le passage à l'inverse est un \emph{automorphisme}, c'est-à-dire $\sigma \leq \mu \Leftrightarrow \sigma^{-1} \leq \mu^{-1}$, ce qui est clair par la définition. Cela se traduit par une propriété de symétrie dans toutes les propositions liées à l'ordre de Bruhat. Ainsi, on a vu en proposition \ref{prop:bruhat-couv} que la relation de couverture se décrit aussi bien par un produit à droite qu'à gauche.

\begin{figure}[ht]
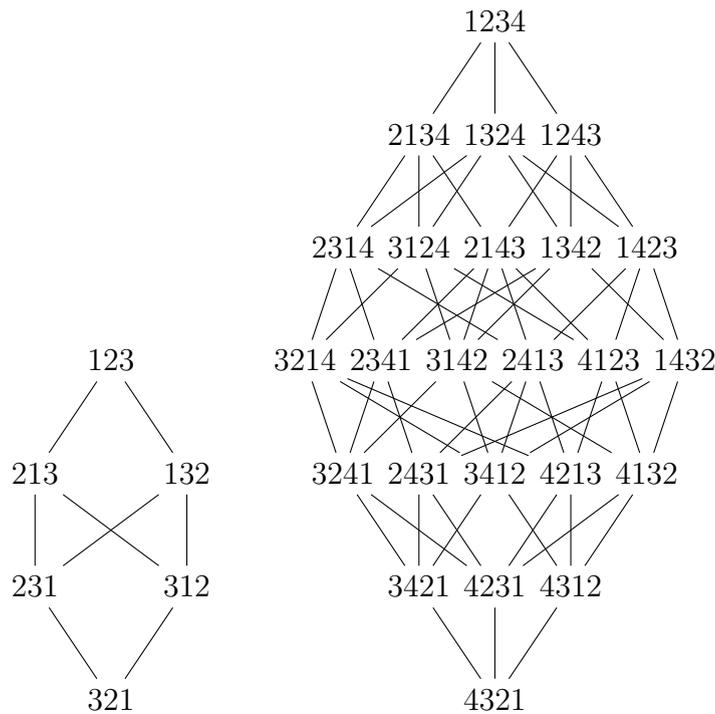

\centering
\begin{tabular}{cc}
\input{includes/figures/bruhat3.tex} &
\input{includes/figures/bruhat4.tex}
\end{tabular}
\caption{Ordres de Bruhat, tailles 3 et 4}
\label{fig:bruhat}
\end{figure}

\subsection{Comparaison des éléments}
\label{sub-sec:prelim_groupe_sym:bruhat:comp}
\index{permutations!bigrassmanniennes}

La comparaison des éléments dans l'ordre de Bruhat nous amène à considérer une nouvelle projection $\kbrho{k}{r}:=\krho{k} \circ \brho{r} = \brho{r} \circ \krho{k}$ bien définie car les opérations $\brho{r}$ et $\krho{k}$ commutent. Cette projection envoie les éléments de $\Sym{n}$ sur l'ensemble $\kbSym{n}{k}{r}$ des permutations ayant une unique descente en $k$ et un unique recul en $r$ auquel on adjoint l'identité.

\begin{align}
\kbrho{3}{4}(361425) &= \brho{4}(136245) = 125346 \\
&= \krho{3}(152346) = 125346
\end{align}

On appelle \emph{permutations bigrassmanniennes} les permutations possédant une unique descente et un unique recul (cf. figure \ref{fig:prelim_groupe_sym:bruhat:bigrassmanniennes}), c'est-à-dire l'intersection des bases des ordres droits et gauches. Un résultat de Deodhar nous dit que:

\begin{equation}
\label{eq:prelim_groupe_sym:bruhat:deodhar}
\sigma \leq \mu \Leftrightarrow \forall k,r \in \llbracket 1, n-1 \rrbracket,  \kbrho{k}{r}(\sigma) \leq \mu.
\end{equation}

Cela signifie que la base de l'ordre de Bruhat est incluse dans l'ensemble des permutations bigrassmanniennes. Les projecteurs $\krho{k}$ et $\brho{r}$ ainsi que les permutations bigrassmanniennes sont définis de façon générale sur les groupes de Coxeter et la base de l'ordre de Bruhat est toujours incluse dans l'ensemble des bigrassmanniennes. Dans le cas du type $A$, nous verrons que la base est constituée exactement des permutations bigrassmanniennes. 

\begin{figure}[ht]
\centering
\input{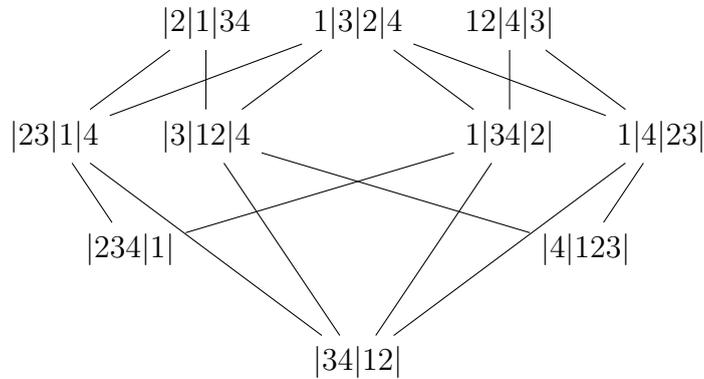}
\caption[Permutations bigrassmanniennes en taille 4]{Permutations bigrassmanniennes en taille 4, une permutation bigrassmannienne correspond à un échange de deux blocs de nombres consécutifs sur la permutation identité (représentés par les barres verticales). L'unique descente se situe avant la deuxième barre, l'unique recul avant la troisième barre.}
\label{fig:prelim_groupe_sym:bruhat:bigrassmanniennes}
\end{figure}

Historiquement, l'ordre de Bruhat pour le type $A$ a été décrit par Ehresmann comme la comparaison des clés des permutations \cite{Ehresmann}. 

\index{clé d'une permutation}
\begin{Definition}
\label{def:prelim_groupe_sym:cle}
La clé d'une permutation $\sigma$ est la suite de ses facteurs gauches réordonnés. On la note $\cle(\sigma)$.
\end{Definition}

Par exemple, la clé de la permutation $361425$ est

\begin{equation}
\begin{array}{cccccc}
3 & 6 & 6 & 6 & 6 & 6 \\ 
  & 3 & 3 & 4 & 4 & 5 \\
  &   & 1 & 3 & 3 & 4 \\
  &   &   & 1 & 2 & 3 \\
  &   &   &   & 1 & 2 \\
  &   &   &   &   & 1 \\
\end{array}
\end{equation}

On a alors la propriété suivante 

\begin{Proposition}
\label{prop:prelim_groupe_sym:bruhat:compcle}
\begin{equation}
\sigma \leq \mu \Leftrightarrow \cle(\sigma) \leq \cle(\mu)
\end{equation}
par la comparaison terme à terme.
\end{Proposition}

La proposition \ref{prop:prelim_groupe_sym:bruhat:compcle} admtet une autre écriture.

\begin{align}
\nonumber
\sigma \leq \mu \Leftrightarrow 
&(\forall r, 1 \leq r < n),(\forall k,  1 \leq k <n) \\
&\# \lbrace \sigma_i \leq r, i \leq k \rbrace \geq \# \lbrace \mu_i \leq r, i \leq k \rbrace.
\end{align} 
où $\sigma = \sigma_1 \dots \sigma_n$ et $\mu = \mu_1 \dots \mu_n$. Vue de cette façon, la comparaison des clés est une interprétation directe des projecteurs $\kbrho{k}{r}$. Par ailleurs, en plus de la symétrie par passage à l'inverse, l'ordre de Bruhat possède aussi une symétrie par retournement des permutations. Soit $\Ret{\sigma}$ le retourné de $\sigma$, alors 
\begin{equation}
\sigma \leq \mu \Leftrightarrow \Ret{\sigma} \geq \Ret{\mu}.
\end{equation} 
Plus précisément,
\begin{equation}
\# \lbrace \sigma_i \leq r, i \leq k \rbrace \geq \# \lbrace \mu_i \leq r, i \leq k \rbrace \Leftrightarrow \# \lbrace \sigma_i \leq r, i >k \rbrace \leq \# \lbrace \mu_i \leq r, i >k \rbrace
\end{equation}
car $\# \lbrace \sigma_i \leq r, i \leq k \rbrace = r - \# \lbrace \sigma_i \leq r, i > k \rbrace$. La proposition \ref{prop:prelim_groupe_sym:bruhat:compcle} peut alors s'écrire de façon plus générale comme

\begin{Proposition}
\label{prop:prelim_groupe_sym:bruhat:compcle_gen}
Soit $\sigma = \sigma_1 \dots \sigma_n$ et $\mu = \mu_1 \dots \mu_n$, alors $\sigma \leq \mu$ si et seulement si il existe $k$ tel que $\cle(\sigma_1 \dots \sigma_k) \leq \cle(\mu_1 \dots \mu_k)$ et $\cle(\sigma_n\sigma_{n-1} \dots \sigma_k) \geq \cle(\mu_n \mu_{n-1} \dots \mu_k)$.
\end{Proposition}

Cette symétrie est fondamentale dans notre travail. Dans la suite de cette thèse, en particulier pour les résultats du chapitre \ref{chap:polynomes_grothendieck}, nous utiliserons de façon poussée la comparaison des éléments dans l'ordre de Bruhat, et en particulier la proposition \ref{prop:prelim_groupe_sym:bruhat:compcle_gen} sur la comparaison des clés. Un exemple de comparaison des clés en partant de la gauche, de la droite, ou en coupant au milieu est donné figure~\ref{fig:prelim_groupe_sym:bruhat:compcle}.

\begin{figure}[ht]
\centering
\input{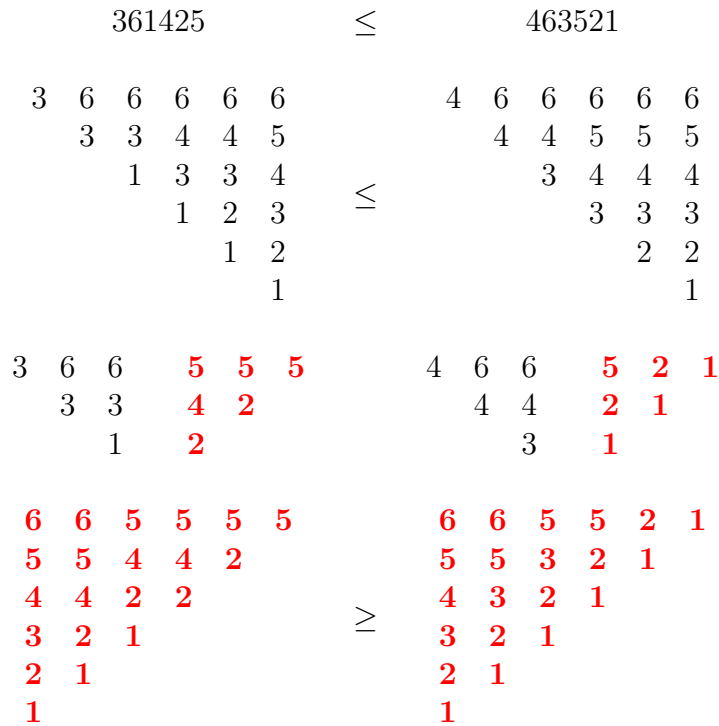}
\caption{Exemple de comparaison de clés}
\label{fig:prelim_groupe_sym:bruhat:compcle}
\end{figure}

\subsection{Borne supérieure et treillis enveloppant}
\label{sub-sec:prelim_groupe_sym:bruhat:sup}

La description en termes de clés donne une méthode pour calculer les bornes supérieures et inférieures d'une partie donnée. 

\begin{Proposition}
Une permutation $\sigma$ est la borne supérieure d'un ensemble de permutations si $\cle(\sigma)$ est le $Sup$ composante à composante des clés de l'ensemble.
\end{Proposition}

On en donne un exemple figure~\ref{fig:prelim_groupe_sym:bruhat:maxcle}. L'ordre de Bruhat n'est pas un treillis. En effet, l'ensemble des clés n'est pas stable par passage au Sup (voir deuxième exemple de la figure \ref{fig:prelim_groupe_sym:bruhat:maxcle}). On définit alors l'ensemble des \emph{triangles monotones} qui contient les clés et qui est stable par passage au Sup.

\begin{figure}[ht]
\centering
\input{includes/figures/exemple_max_cle}
\caption[Exemple de $Sup$ pour Bruhat]{Exemple de $Sup$ pour Bruhat, dans le deuxième exemple, le triangle Sup ne correspond pas à une clé de permutation car la deuxième colonne n'est pas incluse dans la troisième. En fait, on se restreint à une condition locale sur les triplets, le triplet problématique a été coloré en rouge. }
\label{fig:prelim_groupe_sym:bruhat:maxcle}
\end{figure}

\index{triangle monotone}
\begin{Definition}
\label{def:prelim_groupe_sym:triangle}
Un \emph{triangle monotone} est une suite de colonnes de tailles $1,2,\dots,n$ vérifiant des conditions de croissance large sur les ligne de gauche à droite, décroissance stricte sur les colonnes de haut en bas et décroissance large sur les diagonales de haut-gauche vers bas-droit. 
\end{Definition}

A un triangle $t$ de taille $n$ rempli avec les lettres $\lbrace 1, \dots, n \rbrace$, on peut toujours associer un ensemble de bigrassmanniennes $\mathcal{B}$ tel que $t = \pmax \mathcal{B}$ et on en déduit la proposition suivante.

\begin{Proposition}
\label{prop:prelim_groupe_sym:bruhat_env}
Le treillis enveloppant de l'ordre de Bruhat est isomorphe à l'ensemble des triangles monotones rempli avec les lettres $\lbrace 1, \dots, n \rbrace$ où l'ordre est la comparaison composante à composante. 
\end{Proposition}

\begin{Remarque}
Il existe une bijection bien connue entre les triangles monotones et les matrices à signes alternants. Le treillis enveloppant de l'ordre de Bruhat peut donc \^etre vu comme un treillis sur les matrices à signe alternants. Les permutations correspondent alors aux matrices de permutations (ne contenant que des 1) et les triangles qui ne sont pas des clés aux matrices contenant des $-1$.
\end{Remarque}

Le nombre de prédécesseurs d'un triangle monotone est le nombre de composantes du triangle qui peuvent \^etre diminuées de 1 en conservant les conditions de croissance. On a vu dans le paragraphe \ref{sub-sec:prelim_groupe_sym:bruhat:comp} que la base de l'ordre de Bruhat était comprise dans l'ensemble des bigrassmanniennes. En fait, les bigrassmanniennes correspondent exactement aux triangles possédant un unique prédécesseur et forment donc la base du treillis enveloppant et donc de l'ordre de Bruhat.

\subsection{Intervalles et cosets}
\label{sub-sec:prelim_groupe_sym:bruhat:coset}

L'ordre de Bruhat n'étant pas un treillis, l'intersection de deux intervalles n'est en général pas un intervalle. En effet, bien que l'intersection soit close par intervalle, elle peut posséder plusieurs éléments minimaux ou maximaux. Cependant, nous prouvons ici une propriété particulière de certains intervalles appelés \emph{cosets} dont nous nous servirons dans le chapitre \ref{chap:polynomes_grothendieck}. 

Pour $k$ donné, on étudie les classes d'équivalences du quotient de $\Sym{n}$ par $\Sym{k} \times \Sym{n-k}$, c'est-à-dire les permutations qui sont envoyées sur le même élément par la projection $\krho{k}$ (voir paragraphe \ref{sub-sec:prelim_groupe_sym:perm:treillis}). L'ensemble des permutations appartenant à une même classe d'équivalence est appelé un \emph{coset}. Tous les cosets sont isomorphes à $\Sym{k} \times \Sym{n-k}$ et forment donc chacun des intervalles de $\Sym{n}$. Un coset particulier s'écrit $\sigma(\Sym{k} \times \Sym{n-k})$ où $\sigma$ est un représentant quelconque de la classe.

\begin{Lemme}
\label{lem:prelim_groupe_sym:bruhat:coset}
Soient $\sigma$ et $\nu$ deux permutations de $\Sym{n}$ telles que $\nu > \sigma$. Alors, pour $k<n$, on a que l'intersection du coset $\sigma(\Sym{k} \times \Sym{n-k})$ et de l'intervalle $[\sigma, \nu]$ est un intervalle. 
\end{Lemme}

\begin{proof}
Il nous faut prouver que l'intersection possède un unique élément maximal. Les éléments du coset $\sigma(\Sym{k} \times \Sym{n-k})$ correspondent aux permutations $\mu$ telles que $\lbrace \mu_1, \dots \mu_k  \rbrace = \lbrace \sigma_1, \dots \sigma_k \rbrace$. Tout d'abord, nous décrivons la construction d'un élément particulier $\nu'$ appartenant à l'intersection puis nous prouvons qu'il est maximal.

Soit $V_1 = \lbrace \sigma_1, \dots, \sigma_k \rbrace$, alors $\nu_1' = \max(v \in V_1, v \leq \nu_1)$. Puis on définit $V_2 = V_1 \backslash \lbrace \nu_1' \rbrace$ et $\nu_2' = \max(v \in V_2, v \leq \nu_2)$. On continue la construction jusqu'à $\nu_k'$. Pour la partie droite de $\nu'$, on prend $V_n = \lbrace \sigma_{k+1}, \dots, \sigma_n \rbrace$ et $\nu_n' = \min(v \in V_n, v \geq \nu_n)$, puis $V_{n-1} = V_n \backslash \lbrace \nu_n' \rbrace$ et on continue la construction jusqu'à $\nu_{k+1}'$. Par exemple, pour $\sigma = 13245$, $k=2$ et $\nu = 54123$, on a $\nu' = 31524$. Le reste de la preuve sera fait uniquement sur la partie gauche de $\nu'$, la preuve sur la partie droite est complètement symétrique.

Tout d'abord, comme $\nu > \sigma$, la construction de l'élément $\nu'$ est toujours possible. On prouve que $V_i$ contient toujours un élément $v \leq \nu_i$. En effet, soit $v_{\min} = \min(v \in V_i)$, les éléments strictement plus petits que $v_{\min}$ dans $\lbrace \sigma_1, \dots, \sigma_n \rbrace$ ont donc  déjà été choisis et on a $\# \lbrace \sigma_j < v_{\min}, j \leq k \rbrace = \# \lbrace \nu_j' < v_{\min}, j < i \rbrace = \# \lbrace \nu_j< v_{\min}, j < i \rbrace$. Par ailleurs, comme $\sigma < \nu$, on a que $ \# \lbrace \sigma_j < v_{min}, j \leq k \rbrace \geq \# \lbrace \nu_j < v_{min}, j \leq k \rbrace$ et donc $\# \lbrace \nu_j < v_{min}, j \leq k \rbrace = \# \lbrace \nu_j< v_{\min}, j < i \rbrace$ ce qui signifie que $v_{\min} \leq \nu_i$.

Par ailleurs, il est clair par construction que $\nu'$ appartient au coset $\sigma(\Sym{k} \times \Sym{n-k})$ et que $\nu' \leq \nu$. Soit $\mu \in \sigma(\Sym{k} \times \Sym{n-k})$ tel que $\sigma \leq \mu \leq \nu$, prouvons que $\mu \leq \nu'$. Supposons que $\mu \nleq \nu'$, alors il existe $r < n$ et $i < k$ tel que 
\begin{equation}
\label{eq:prelim_groupe_sym:bruhat:coset_proof_1}
\# \lbrace \mu_j \leq r, j \leq i \rbrace < \# \lbrace \nu_j' \leq r, j \leq i \rbrace.
\end{equation}

Cela signifie en particulier qu'il existe $\mu_\ell > r$ avec $\ell \leq i$ tel que $\mu_\ell \notin \lbrace \nu_1', \dots, \nu_i' \rbrace$, c'est-à-dire que $\mu_\ell \in \lbrace \nu_{i+1}', \dots, \nu_k' \rbrace$. On choisit $\mu_\ell$ minimal. Pour $j \leq i$, on a que $\mu_\ell \in V_j$ et que $\mu_\ell \neq \nu_j'$ ce qui signifie que  $\nu_j' < \mu_\ell$ si et seulement si $\nu_j < \mu_\ell$, d'où
\begin{equation}
\label{eq:prelim_groupe_sym:bruhat:coset_proof_2}
\# \lbrace \nu_j < \mu_\ell, j\leq i \rbrace = \# \lbrace \nu_j' < \mu_\ell, j \leq i \rbrace.
\end{equation}

Par ailleurs, comme $\mu_\ell$ a été choisi minimal, que $\mu_\ell > r$ et $\mu_\ell \notin \lbrace \nu_1', \dots, \nu_i' \rbrace$ on a que
\begin{equation}
\label{eq:prelim_groupe_sym:bruhat:coset_proof_3}
\# \lbrace \mu_j < \mu_\ell, j \leq i \rbrace < \# \lbrace \nu_j' < \mu_\ell, j \leq i \rbrace.
\end{equation}

Or on a $\mu \leq \nu$ et donc $\# \lbrace \mu_j < \mu_\ell,  j \leq i \rbrace > \# \lbrace \nu_j < \mu_\ell, j\leq i \rbrace$ ce qui contredit \eqref{eq:prelim_groupe_sym:bruhat:coset_proof_2} et \eqref{eq:prelim_groupe_sym:bruhat:coset_proof_3}. D'où $\mu \leq \nu'$.
\end{proof}

Le Lemme \ref{lem:prelim_groupe_sym:bruhat:coset} peut être généralisé sous la forme suivante.

\begin{Lemme}
\label{lem:prelim_groupe_sym:bruhat:coset_gen}
Soient $1 \leq k_1 < k_2 < \dots < k_m < n$ et $\sigma$ et $\nu$ deux permutations de $\Sym{n}$ telles que $\sigma \leq \nu$. Alors l'intersection du coset $\sigma(\Sym{k_1} \times \Sym{k_2 - k_1} \times \Sym{k_3 - k_2} \times \dots \times \Sym{n - k_m})$ et de l'intervalle $[\sigma, \nu]$ forme un intervalle.
\end{Lemme}

\begin{proof}
La preuve se fait par récurrence étant donné que $\sigma (\Sym{k_1} \times \Sym{n- k_1}) \supset \sigma( \Sym{k_1} \times \Sym{k_2 - k1} \times \Sym{n - k_2}) \supset \dots \supset \sigma(\Sym{k_1} \times \Sym{k_2 - k_1} \times \dots \times \Sym{n - k_m})$.
\end{proof}

\cleardoublepage

\part{Polyn\^omes multivariés}
\label{part:polynomes}

\chapter{Action du groupe symétrique sur les polynômes}
\label{chap:polynomes_action}

Les bases de l'anneau des polyn\^omes étudiées dans ce chapitre apparaissent naturellement dans un problème ancien de géométrie algébrique, le calcul de Schubert. Le principe est de décomposer une variété géométrique en ce qu'on appelle des cellules de Schubert en fonction de leurs intersections avec un drapeau, c'est-à-dire un ensemble d'espaces vectoriels emboîtés. Le cas où la variété de départ est la Grassmannienne est en particulier décrit par Fulton \cite{Fulton}. Les intersections de cellules de Schubert peuvent alors être obtenues par un calcul algébrique dans l'anneau de cohomologie qui correspond ici à un quotient de l'anneau des fonctions symétriques. En particulier, les images des cellules de Schubert sont les \emph{fonctions de Schur}. Quand la variété de départ est une variété de drapeau, l'anneau de cohomologie est un quotient de l'anneau des polynômes non symétriques et l'image des cellules de Schubert est donné par les polynômes de Schubert \cite{LascouxSchub}. Malgré leur origine géométrique, ces polynômes se calculent de façon purement combinatoire et font intervenir les structures vues dans la partie~\ref{part:prelim}. 

En particulier, l'action du groupe symétrique joue un rôle fondamental. Les propriétés des polyn\^omes sont donc fortement liées à celles du groupe symétrique et en particulier aux structures d'ordres que nous avons définies. Nous expliquons ici comment certains opérateurs particuliers, les \emph{différences divisées}, permettent de retrouver les bases des polyn\^omes multivariés et le lien ténu qui existe entre ces opérateurs et les ordres sur les permutations, en particulier l'ordre de Bruhat. 

Dans tout ce chapitre, nous nous appuyons sur le travail de Lascoux et Schützenberger à qui l'on doit en particulier les polynômes de Schubert et de Grothendieck \cite{LascouxSchub, LascouxKBruhat, LascouxGroth}. Si les différences divisées étaient connues depuis Newton, ce n'est que récemment qu'elles ont été utilisées dans le contexte de la combinatoire et de la géométrie, d'abord par Demazure \cite{Dem1, Dem2} et Bernstein-Gelfand-Gelfand \cite{BGG} puis de façon plus systèmatique par Lascoux et Schützenberger \cite{LascouxDiffDiv}. On trouvera dans l'oeuvre de Lascoux de nombreuses explications sur les calculs à partir de différences divisées. On peut par exemple se reporter à l'historique des polyn\^omes de Schubert \cite{LascouxSchubHist}. L'article \cite{LascouxCoxeter} que nous avons déjà largement utilisé dans le chapitre \ref{chap:prelim_groupe_sym} fait le lien entre les différences divisées et les ordres sur les groupes de Coxeter. Par ailleurs, la thèse de Veigneau \cite{VEIGN} aborde la théorie d'une façon très similaire à la nôtre avec des explications détaillées de l'interprétation géométrique.

Dans le paragraphe \ref{sec:polynomes_action:ope}, nous définissons l'action de base du groupe symétrique sur les polynômes et introduisons les différences divisées. Le paragraphe \ref{sec:polynomes_action:sym} est un rapide survol de la théorie des fonctions symétriques et plus précisément de la construction des fonctions de Schur par différence divisée. Dans le paragraphe \ref{sec:polynomes_action:pol}, nous généralisons cette construction au cas des polynômes non symétriques. Enfin, dans le paragraphe \ref{sec:polynomes_action:0hecke}, nous étudions plus particulièrement les différences divisées isobares $\pi$ et $\hpi$, générateurs de l'algèbre 0-Hecke. Les relations entre ces opérateurs permettent de faire le lien avec l'ordre de Bruhat vu en partie \ref{part:prelim}. Le cadre formel que nous définissons ici sera la base de notre travail sur les polyn\^omes de Grothendieck dans le chapitre \ref{chap:polynomes_grothendieck}.

\section{Opérateurs sur les polynômes}
\label{sec:polynomes_action:ope}

\subsection{Action élémentaire sur les monômes}
\label{sub-sec:polynomes_action:ope:si}

Le groupe symétrique, et plus généralement les groupes de Coxeter, agissent fidèlement sur les vecteurs. Soit $v = \left[ v_1, \dots, v_n \right] \in \mathbb{Z}^n$ on a:

\begin{align}
v s_i &= \left[ v_1, \dots, v_{i+1}, v_i, \dots, v_n \right] &\text{ pour }1 \leq i < n, \\
v s_i^B &= \left[ v_1, \dots, -v_i, \dots, v_n \right] &\text{ pour }1 \leq i \leq n, \\
v s_i^D &= \left[ v_1, \dots, -v_i, -v_{i-1}, \dots, v_n \right] &\text{ pour }1 \leq i \leq n. 
\end{align}
où les opérateurs agissent sur leur gauche. On vérifie facilement que les relations définies dans les paragraphes \ref{sub-sec:prelim_groupe_sym:groupe_sym:def} et \ref{sub-sec:prelim_groupe_sym:groupe_sym:coxeter} sont bien respectées. On rappelle que le groupe $BC_n$ (resp. $D_n$) est engendré par $s_1, \dots, s_{n-1}$ et $s_n^B$ (resp. $s_n^D$). Cependant, on s'autorise aussi à utiliser $s_i^B$ (resp. $s_i^D$) pour $i\neq n$. On a vu que $\Sym{n}$ était le groupe des permutations, ce qui correspond bien à l'orbite du vecteur $\left[ 1, \dots, n \right]$ . Dans le cas de $BC_n$ (resp. $D_n$), on obtient les permutations signées (resp. signées avec un nombre pair de négatifs).

On se place maintenant dans l'espace $\mathbb{Z}\left[x_1, \dots, x_n \right]$ des polynômes multivariés. Pour $v \in \mathbb{Z}^n$, on écrit:

\begin{equation}
x^v = x_1^{v_1} x_2^{v_2} \dots x^{v_n}.
\end{equation}

Ainsi, les polynômes sont vus comme des sommes formelles de vecteurs (les exposants).
L'action du groupe symétrique sur les vecteurs nous donne alors une opération sur les polynômes. De la même façon, les groupes $BC$ et $D$ agissent sur les polynômes à exposants dans $\mathbb{Z}$. Pour $f$ un polynôme (éventuellement avec exposants négatifs), on a donc :
\begin{align}
f s_i &= f(x_1, \dots, x_{i+1}, x_i, \dots, x_n), \\
f s_i^B &= f(x_1, \dots, \frac{1}{x_i}, \dots, x_n), \\
f s_i^D &= f(x_1, \dots, \frac{1}{x_i}, \frac{1}{x_{i-1}}, \dots, x_n).
\end{align}

\subsection{Différences divisées}
\label{sub-sec:polynomes_action:ope:diffdiv}
\index{différence divisée}

On utilise les opérations élémentaires $s_i$  pour définir des familles d'opérateurs plus sophistiquées, les \emph{différences divisées}. La différence divisée $\partial_i$ s'écrit :

\begin{equation}
\label{eq:polynomes_action:partial}
f \partial_i = \frac{f - f s_i}{x_i - x_{i+1}}.
\end{equation}

La différence entre $f$ et $fs_i$ \emph{symétrise} le polynôme $f$. C'est-à-dire que si $g = f\partial_i$, on a $g s_i = g$. Par ailleurs, le résultat est toujours un polynôme car la symétrie fait que $f - fs_i$ s'annule pour $x_i = x_{i+1}$. Plus précisément, on peut écrire la différence divisée directement comme une somme.

\begin{align}
x^v\partial_i &= 0 &\text{ si } v_i = v_{i+1} \\
\label{eq:polynomes_action:partial_sum}
x^v\partial_i &= \sum_{k=v_{i+1}}^{v_{i}-1} x^{[\dots, k, v_i-k,\dots]} &\text{ si } v_i > v_{i+1}\\
x^v\partial_i &= \sum_{k=v_i}^{v_{i+1}-1} -x^{[\dots, k, v_{i+1} -k, \dots]} &\text{ si } v_i < v_{i+1}
\end{align}

Par exemple,
\begin{equation}
\label{eq:polynomes_action:exemple_diffdiv}
x_1^5x_2 ~ \partial_1 = x_{1} x_{2}^{4} + x_{1}^{2} x_{2}^{3} + x_{1}^{3} x_{2}^{2}  + x_{1}^{4} x_{2}.
\end{equation}

La famille d'opérateurs $(\partial_i)_{1 \leq i < n}$ satisfait des relations proches de celles des $s_i$. En effet, les relations de tresses \eqref{eq:prelim_groupe_sym:tresse1} et \eqref{eq:prelim_groupe_sym:tresse2} sont conservées et seule la relation quadratique est modifiée. On a:

\begin{align}
\label{eq:polynomes_action:partial_quad}
\partial_i ^2 &= 0.  \\
\label{eq:polynomes_action:partial_tresse1}
\partial_i \partial_j &= \partial_j \partial_i \text{ pour } |i-j| > 1 \\
\label{eq:polynomes_action:partial_tresse2}
\partial_i \partial_{i+1} \partial_i &= \partial_{i+1} \partial_i \partial_{i+1}
\end{align}

Il existe d'autres types de différences divisées. Nous présentons deux sortes de \emph{différences divisées isobares} qui, contrairement aux $\partial_i$, conservent le degré du polynômes :

\index{différence divisée!isobare}
\begin{align}
\label{eq:polynomes_action:pi}
\pi_i &:= x_i \partial_i, \\
\label{eq:polynomes_action:hpi}
\hpi_i &:= \partial_i x_i.
\end{align}
Ou de façon équivalente,
\begin{align}
f \pi_i &= \frac{x_i f - x_{i+1} (fs_i)}{x_i - x_{i+1}}, \\
f \hpi_i &= \frac{x_{i+1} (f - fs_i)}{x_i - x_{i+1}}.
\end{align}

Les familles $(\pi_i)_{1 \leq i < n}$ et $(\hpi_i)_{1 \leq i <n}$ vérifient elles aussi les relations de tresses ainsi que les relations quadratiques :

\begin{align}
\label{eq:polynomes_action:pi_quad}
\pi_i^2 &= \pi_i, \\
\label{eq:polynomes_action:hpi_quad}
\hpi_i^2 &= - \hpi_i.
\end{align}
Par ailleurs, on vérifie que
\begin{equation}
\label{eq:polynomes_action:pi_hpi}
\hpi_i = \pi_i - 1.
\end{equation}

Pour $s_{i_1} \dots s_{i_r}$ une décomposition réduite de $\sigma$, on écrit $\partial_\sigma := \partial_{i_1} \dots \partial_{i_r}$, $\pi_\sigma := \pi_{i_1} \dots \pi_{i_r}$ et $\hpi_\sigma := \hpi_{i_1} \dots \hpi_{i_r}$. Comme les différences divisées vérifient les relations de tresse, ces produits ne dépendent pas de la décomposition réduite choisie et sont donc bien définis.

Nous verrons dans le paragraphe \ref{sec:polynomes_action:pol} que les différences divisées sont utilisées pour définir des bases de l'anneau des polyn\^omes multivariés. Par ailleurs, les opérateurs $\pi_i$ et $\hpi_i$ sont liés à l'algèbre de 0-Hecke que nous présentons dans le paragraphe \ref{sec:polynomes_action:0hecke} et à l'ordre de Bruhat. Dans le chapitre \ref{chap:polynomes_grothendieck}, nous les étudions de façon plus approfondie et prouvons un résultat sur un produit mélangeant des opérateurs $\pi_i$ et $\hpi_i$.

\subsection{Types $B$, $C$ et $D$}
\label{sub-sec:polynomes_action:ope:bcd}

Les différences divisées $\partial_i$, $\pi_i$ et $\hpi_i$ sont définies à partir des générateurs $s_i$ du groupe symétrique. Nous avons vu dans le paragraphe \ref{sub-sec:prelim_groupe_sym:groupe_sym:coxeter} que le groupe symétrique $\Sym{n}$ appartient à une famille plus générale : les groupes de Coxeter. Dans le paragraphe \ref{sub-sec:polynomes_action:ope:si}, nous avons en particulier défini les actions des éléments $s_i^B$ et $s_i^D$ sur les vecteurs. De façon similaire, il existe aussi des différences divisées $\partial_i^B$, $\partial_i^C$ et $\partial_i^D$.

Les groupes de Coxeter sont générés par des réflexions que l'on peut réaliser géométriquement comme des systèmes de racines \cite{Coxeter}. Les vecteurs d'exposants s'interprètent comme des éléments de l'espace ambiant de ces systèmes de racines. La différence divisée $\partial_i$ appliquée à $x^v$ revient à sommer des éléments $x^{v- e_i - kr}$ où $e_i =[0, \dots, 0, 1, 0, \dots, 0]$ et $r = [0, \dots, -1 , 1 \dots, 0]$, comme on l'a vu en \eqref{eq:polynomes_action:partial_sum}. Le vecteur $r$ est une racine simple de l'espace ambiant du système de racine $A_{n-1}$. De façon générale les racines simples de $A_{n-1}$ sont les vecteurs $(r_i)_{1 \leq i < n}$ de $\ZZ^n$ où $r_i[i] = 1$, $r_i[i+1] = -1$ et où les autres valeurs sont nulles. On pose
\begin{align}
r_n^B &= [0, \dots, 0, 1] \\ 
r_n^C &= [0, \dots, 0, 2] \\
r_n^D &= [0, \dots, 0, 1, 1].
\end{align}
Les racines simples des espaces ambiants des systèmes de racines de type $B_n$ (resp. $C_n$ et $D_n$) sont les vecteurs $(r_i)_{1 \leq i < n}$ et $r_n^B$ (resp. $r_n^C$ et $r_n^D$). On a alors, par exemple 
\begin{align}
x_1 x_2^3 ~\partial_2^B &= \sum_{0 \leq k < 6} x^{[1,2] - k.[0,1]} \\
&= x_1 x_2^2 + x_1 x_2^1 + x_1 x_2^0 + x_1 x_2^{-1} + x_1 x_2^{-2} + x_1 x_2^{-3}, \\
x_1 x_2^3 ~\partial_2^C &= \sum_{0 \leq k < 3} x^{[1,2] - k.[0,2]} \\
&= x_1 x_2^{2} + x_1 x_2^{0} + x_1 x_2^{-2}, \\
x_1 x_2^3 ~\partial_2^D &=  \sum_{0 \leq k < 4} x^{[1,2] - k.[1,1]} \\
&= x_1 x_2^{2} + x_1^0 x_2^{1} + x_1^{-1} x_2^{0} + x_1^{-2} x_2^{-1}.
\end{align}
La valeur maximale de $k$ est donnée par le produit scalaire entre le vecteur et la co-racine de $r_i$ dans l'espace ambiant. De façon générale, on traduit ces sommes par
\begin{align}
\label{eq:polynomes_action:def-diffdiv-B}
\partial_n^B &= \frac{1 - s_n^B}{x_n^{\frac{1}{2}} - x_{n}^{-\frac{1}{2}}} \\
\label{eq:polynomes_action:def-diffdiv-B}
\partial_n^C &= \frac{1 - s_n^C}{x_n - x_n^{-1}} \\
\label{eq:polynomes_action:def-diffdiv-B}
\partial_n^D &= \frac{1 - s_n^D}{x_{n-1}^{-1} - x_n}.
\end{align}

Les différences divisées isobares peuvent aussi s'exprimer de façon formelle sur les espaces ambiants des systèmes de racines. On obtient alors les définitions suivantes
\begin{align}
\pi_i^B &= x_i^{1/2}\partial_i^B, &\hat{\pi}_i^B &= \partial_i^B x_i^{-1/2}, \\
\pi_i^C &= x_i \partial_i^C, &\hat{\pi}_i^C &= \partial_i^C x_i^{-1}, \\
\pi_i^D &= \left( 1 - \frac{s_i^D}{x_{i-1}x_i} \right) \frac{1}{1-\frac{1}{x_{i-1}x_i}}, &\hat{\pi_i}^D &= \frac{1-s_i^D}{x_{i-1}x_i-1}.
\end{align}

\section{Polyn\^omes et fonctions symétriques}
\label{sec:polynomes_action:sym}

\subsection{Définition et premières bases}
\label{sub-sec:polynomes_action:sym:def}
\index{fonctions!symétriques}

On dit qu'un polyn\^ome en $n$ variables est \emph{symétrique} s'il est invariant par l'action de $\Sym{n}$.  Par exemple, le polyn\^ome suivant sur $\lbrace x_1, x_2, x_3 \rbrace$ est symétrique :

\begin{equation}
\label{eq:polynomes_action:sym:exemple}
x_1 x_2 x_3^2 + x_1 x_2^2 x_3 + x_1^2 x_2 x_3.
\end{equation}

Si l'on considère maintenant des séries formelles sur un alphabet infini, les \emph{fonctions symétriques} sont les séries invariantes par l'action de $\Sym{n}$ pour tout $n$. On notera cet ensemble $Sym$ et sa projection sur les polynômes en $n$ variables $Sym_n$. 

Les fonctions symétriques sont un objet d'étude fondamental en combinatoire algébrique. Elles admettent plusieurs bases dont nous présentons ici quelques exemples. Le calcul des changements de base est un problème ancien, déjà abordé par Newton puis plus tard par Vandermonde, Fa\`a de Bruno, Cayley et Kostka. Elles apparaissent dans des problèmes liés à la fois à la combinatoire, l'algèbre et la géométrie. Nous nous contentons ici d'un très rapide survol des notions de base. Le premier chapitre du livre de Macdonald \cite{MAC} est une très bonne introduction à la théorie. 

\index{fonctions!monomiales}
\index{partition}
La base la plus naturelle des fonctions symétriques est celle des \emph{fonctions monomiales} obtenue en symétrisant les monômes. Soit $v$ un vecteur de taille $n$, le symétrisé de $x^v$ est la somme de tous les éléments distincts $x^{v \sigma}$ avec $\sigma \in \Sym{n}$. Par exemple, le symétrisé de $x_1 x_2^2 x_3$ est le polyn\^ome symétrique \eqref{eq:polynomes_action:sym:exemple}. Soit $\lambda$ une partition de longueur $n$, c'est-à-dire un vecteur tel que $\lambda_1 \geq \lambda_2 \geq \dots \geq \lambda_n \geq 0$, alors $m_\lambda$ est le symétrisé de $x^\lambda$. Tout polynôme symétrique peut s'écrire comme une somme de $m_\lambda$. Les $(m_\lambda)$ forment donc une base des polynômes symétriques. Par exemple, le polynôme \eqref{eq:polynomes_action:sym:exemple} est égal à $m_{2,1,1}(x_1,x_2,x_3)$.

\index{fonctions!élémentaires}
\index{fonctions!complètes}
On définit aussi deux autres bases : les fonctions \emph{élémentaires} et \emph{complètes}. Soit $r \in \mathbb{N}$, alors

\begin{align}
e_r &:= \sum_{i_1 < i_2 < \dots < i_r} x_{i_1} x_{i_2} \dots x_{i_r}, \\
h_r &:= \sum_{i_1 \leq i_2 \leq \dots \leq i_r} x_{i_1} x_{i_2} \dots x_{i_r}.
\end{align}
Par exemple,
\begin{align}
e_2(x_1,x_2,x_3) &= x_1 x_2 + x_1 x_3 + x_2 x_3,\\
h_2(x_1,x_2,x_3) &= x_1^2 + x_2^2 + x_3^2 +  x_1 x_2 + x_1 x_3 + x_2 x_3.
\end{align}

A présent, pour $\lambda$ une partition, on écrit $e_\lambda := e_{\lambda_1} e_{\lambda_2} \dots e_{\lambda_n}$ et de la m\^eme façon, $h_\lambda := h_{\lambda_1} h_{\lambda_2} \dots h_{\lambda_n}$. On prouve que les $e_\lambda$ ainsi que les $h_\lambda$ sont algébriquement indépendants et forment des bases des  fonctions symétriques.

Les \emph{partitions} sont les objets combinatoires indexant les bases des fonctions symétriques. Une partition $\lambda$ est un vecteur $(\lambda_1, \lambda_2, \dots, \lambda_k)$ tel que $\lambda_1 \geq \lambda_2 \geq \dots \geq \lambda_k$. La \emph{longueur} d'une partition $\ell(\lambda)$ est le nombre d'éléments dans le vecteur. La \emph{taille} d'une partition est donnée par la somme des valeurs $\lambda_1 + \lambda_2 + \dots + \lambda_n$. On représente souvent une partition par un \emph{diagramme de Ferrers} : un ensemble de cases tel que la ième ligne (en partant du bas) contienne $\lambda_i$ cases. Par exemple, la partition $(4,3,2,2)$ est représentée par 
\begin{equation}
{\syng{2,2,3,4}}.
\end{equation}
En lisant le diagramme selon les colonnes plutôt que selon les lignes, on obtient $\lambda' = (4,4,2,1)$ la \emph{partition conjuguée} de $\lambda$,
\begin{equation}
{\syng{1,2,4,4}}.
\end{equation}

\subsection{Fonctions de Schur}
\label{sub-sec:polynomes_action:sym:schur}
\index{fonctions!de Schur}

Nous avons vu dans le paragraphe \ref{sub-sec:polynomes_action:ope:diffdiv} que les opérateurs de différence divisée symétrisaient les polynômes. En effet, si $f = x_1^4 x_2 ^2 x_3$, alors

\begin{equation}
	f \partial_1 = x_{1}^{3} x_{2}^{2} x_{3} + x_{1}^{2} x_{2}^{3} x_{3}
\end{equation}
est un polynôme symétrique en $x_1$ et $x_2$. Si à présent on applique $\partial_2$, on obtient
\begin{equation}
	f \partial_1 \partial_2 = x_{1}^{3} x_{2} x_{3} + x_{1}^{2} x_{2}^{2} x_{3} + x_{1}^{2} x_{2} x_{3}^{2}
\end{equation}
qui est symétrique en $x_2$ et $x_3$. Cependant, on a perdu la symétrie entre $x_1$ et $x_2$, on peut la retrouver en réappliquant $\partial_1$. On obtient alors
\begin{equation}
	f \partial_1 \partial_2 \partial_1 = x_{1}^{2} x_{2} x_{3} + x_{1} x_{2}^{2} x_{3} + x_{1} x_{2} x_{3}^{2}.
\end{equation}

Pour obtenir un polynôme symétrique en 3 variables, on a appliqué $\partial_1 \partial_2 \partial_1 = \partial_\omega$ où $\omega$ est la permutation maximale de taille 3. Cette propriété est vraie quelque soit $n$ et les polyn\^omes symétriques obtenus sont appelés \emph{fonctions de Schur}. Celles-ci forment une base des fonctions symétriques. Plus précisément, soit $\lambda$ une partition de longueur $n$ et $\delta$ le vecteur $[ n-1, n-2,\allowbreak \dots, 1, 0]$, alors

\begin{align}
s_\lambda &:= x^{\lambda + \delta} \partial_\omega \\
&= \frac{x^{\lambda + \delta} \sum_{\sigma \in \Sym{n}} \varepsilon(\sigma) s_\sigma}{\prod_{1 \leq i < j \leq n} (x_i - x_j) }.
\end{align}

Le numérateur est \emph{l'antisymétrisé} de $x^{\lambda + \delta}$ : on applique toutes les permutations de $\Sym{n}$ pondérées par leur signature $\varepsilon(\sigma)$. Le dénominateur est le \emph{déterminant de Vanderrmonde}.

Les fonctions de Schur peuvent aussi être exprimées en tant que déterminants de fonctions élémentaires ou complètes. On a

\begin{align}
\label{eq:polynomes_action:SfromH}
s_\lambda = \mathrm{det}(h_{\lambda_i - i +j})_{1 \leq i,j \leq n}, \\
\label{eq:polynomes_action:SfromE}
s_\lambda = \mathrm{det}(e_{\lambda_i' - i +j})_{1 \leq i,j \leq m},
\end{align}
où $\lambda'$ est la partition conjuguée de $\lambda$ et $n = \ell(\lambda)$, $m=\ell(\lambda')$.
En particulier, 

\begin{align}
s_{(n)} &= h_n, \\
s_{(1^n)} &= e_n.
\end{align}

\subsection{Formule de Pieri et interprétation géométrique }
\label{sub-sec:polynomes_action:sym:pieri}
\index{formule de Pieri}

Les fonctions de Schur jouent un rôle fondamental en combinatoire algébrique. Elles correspondent en particulier aux caractères des représentations du groupe symétrique \cite{MAC}. Leur structure multiplicative est liée à une très jolie combinatoire sur les tableaux décrite par la règle de Littlewood-Richardson. Le cas spécifique de la multiplication d'une fonction de Schur par une fonction complète est donné par la \emph{formule de Pieri} \cite{MAC}. Soit $\mu$ une partition et $r \in \mathbb{N}$, alors

\begin{equation}
\label{eq:pieri-formula}
s_\mu h_r = \sum_\lambda s_\lambda
\end{equation}
sommé sur les partitions $\lambda$ égales à la partition $\mu$ augmentée d'une bande horizontale de taille $r$. Cela revient à ajouter des cases sur le  diagramme de Ferrers de la partition $\mu$ : au maximum une case par colonne et $r$ cases en tout. Par exemple, le produit de la fonction de Schur $s_{(2,2,1)}$ par la fonction complète $h_2$ s'écrit

\begin{align}
s_{\syng{1,2,2}} h_{\syng{2}} &= s_{\syng{1,2,2,2}} +  s_{\syng{1,1,2,3}} +  s_{\syng{2,2,3}} +  s_{\syng{1,2,4}} \\
s_{(2,2,1)} h_2 &= s_{(2,2,2,1)} + s_{(3,2,1,1)} + s_{(3,2,2)} + s_{(4,2,1)}.
\end{align}

Ce calcul a une importance particulière en géométrie algébrique. La géométrie énumérative est apparue vers le milieu du XIXème siècle. Son but est de calculer le nombre de solutions de problèmes de type intersection de variétés. Un cas classique est le \emph{calcul de Schubert}. La Grassmannienne $G(k,n)$ est l'ensemble des sous-espaces vectoriels de dimension $k$ d'un espace de dimension $n$. Soit maintenant un ensemble de sous-espaces vectoriels $V_i$ de dimension $i$ tel que $V_1 \subset V_2 \subset \dots V_n$, ou \emph{drapeau}. On décompose l'ensemble $G(k,n)$ en \emph{cellules de Schubert} en fonction de la dimension de $W \cap V_i$ pour $W \subset G(k,n)$ et $i=1,\dots, k$.  Les cellules de Schubert sont indexées par des partitions et notées $\Omega_\mu$. Une cellule $\Omega_\mu$ est incluse dans l'adhérence d'une cellule $\Omega_\nu$ si et seulement si $\mu \subset \nu$. Cela donne un ordre sur les cellules de Schubert introduit par Ehresmann et qu'on appelle aussi ordre de Bruhat. En effet, quand on ne travaille plus sur la Grassmannienne mais sur les variétés de drapeau, l'ordre sur les cellules de Schubert correspond alors à celui que nous avons vu dans le chapitre \ref{chap:prelim_groupe_sym}.

Un des buts de la géométrie algébrique est de résoudre des problèmes énumératifs comme celui que nous avons décrit par des méthodes algébriques. Ainsi, on montre que les adhérences $\Omega_\mu$ des cellules de Schubert peuvent être identifiées à des variétés algébriques et on les appelle \emph{variétés de Schubert}. Il est alors possible de définir une relation d'équivalence appelée homologie sur le $\mathbb{Z}$-module des combinaisons linéaires des sous-variétés de $G(k,n)$. On peut ainsi travailler dans l'anneau de cohomologie de $G(k,n)$. Le produit des classes des sous-variétés $U$ et $V$ de $G(k,n)$ dans cet anneau code l'information géométrique liée à l'intersection de $U$ et $V$. Plus précisément, la classe de cohomologie d'un point de $G(k,n)$ est toujours indexée par la partition $(k^{n-k})$, on la note $\sigma_{(k^{n-k})}$. Si $U$ et $V$ sont deux sous-variétés de dimensions complémentaires dont on note les classes respectives $[U]$ et $[V]$ alors

\begin{equation}
[U] [V] = m \sigma_{(k^{n-k})}
\end{equation}
où $m$ est le nombre de points d'intersection de $U$ et de $V$.

Ce produit s'interprète en réalité comme un produit de polynôme. Plus précisément, à une variété de Schubert $\Omega_\mu$, on fait correspondre une fonction de Schur $s_\mu$. L'anneau de cohomologie de $G(k,n)$ s'identifie alors au quotient $Sym/\mathcal{I}(k,n)$ où $\mathcal{I}(k,n)$ est l'idéal engendré par les fonctions de Schur $s_\nu$ telles que $\nu$ n'est pas contenue dans la partition $(k^{n-k})$. C'est dans ce contexte que la formule de Pieri a d'abord été établie. Ces problèmes géométriques motivent encore aujourd'hui les questions combinatoires liées aux fonctions symétriques ainsi qu'aux polynômes non symétriques.

Le premier cas non trivial est la Grassmannienne $G(2,4)$ des plans de dimensions 2 dans $\mathbb{C}^4$ qui peut en fait être identifiée à l'espace de module des droites projectives dans $\mathbb{P}^3(\mathbb{C})$. Le calcul suivant

\begin{equation}
s_{\syng{1,1}}^2 \equiv 1.s_{\syng{2,2}}
\end{equation}
nous confirme un résultat géométrique clair : il existe exactement une droite contenue dans deux plans en position générale. En effet, dans $G(2,4)$, la classe d'un plan est donnée par la partition $(1,1)$. Si ce résultat est relativement simple, ce système de calcul permet de résoudre des problèmes beaucoup plus complexes. En restant dans $G(2,4)$ on obtient par exemple qu'il existe deux droites s'appuyant sur quatre droites en position générale. Le résultat est donné par le calcul

\begin{equation}
s_{\syng{1}}^4 \equiv 2 s_{\syng{2,2}},
\end{equation}
la classe d'une droite s'appuyant sur une droite donnée étant $s_{\syng{1}}$. De la m\^eme façon, dans $G(2,5)$ (identifiée à à l'espace de module des droites projectives dans $\mathbb{P}^4$), le nombre de droites coupant six plans est donné par le coefficient de $s_{\syng{3,3}}$ dans le développement de $s_{\syng{1}}^6$. La règle de Littlewood-Richardson nous dit qu'il est égal au nombre de tableaux standard de forme $(3,3)$, c'est-à-dire 5.

\section{Polynômes non symétriques} 
\label{sec:polynomes_action:pol}

La théorie des fonctions symétriques a été largement étudiée en combinatoire. Nous avons vu ses motivations géométriques. Elle fait aussi le lien avec la théorie des représentations. L'étude par une approche similaire des polynômes non symétriques et de leur combinatoire est plus récente. Elle est liée à des problèmes géométriques qui généralisent ceux que nous avons abordés. Si pour les fonctions symétriques, les objets combinatoires de base sont les partitions et les tableaux, les problèmes du cas non symétrique font intervenir les permutations et les ordres du groupe symétrique par le biais des différences divisées. 

\subsection{Polynômes de Schubert}
\label{sub-sec:polynomes_action:pol:schub}
\index{polynômes!de Schubert}

Nous avons vu qu'en appliquant la différence divisée maximale $\partial_\omega$ à un monôme, on obtenait un polynôme symétrique. Si on applique à présent toutes les différences divisées $\partial_\sigma$, on symétrise partiellement le polynôme. Ces symétrisations partielles nous donnent une base des polynômes non symétriques : les polynômes de Schubert. On se place sur l'anneau des polynômes sur $\mathbb{Z}$ en deux alphabets commutatifs potentiellement infinis $\lbrace x_1, x_2, \dots \rbrace$ et $\lbrace y_1, y_2, \dots \rbrace$.

\begin{Definition}
\label{def:polynomes_action:schub}
Soit $\lambda$ une partition, on définit le polynôme de Schubert dominant $Y_\lambda$ par
\begin{equation}
Y_\lambda :=  \prod_{\substack{i = 1 \ldots n\\ j = 1 \ldots \lambda_i}} (x_i - y_j).
\end{equation}

Les polynômes de Schubert sont l'ensemble des images non nulles des polynômes dominants par les différences divisées en $x$.
\end{Definition}

\begin{figure}[ht]
\centering
\input{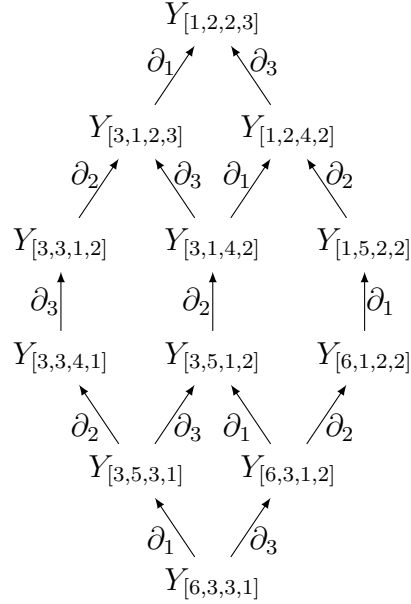}
\caption{Images par différences divisées du polynôme Schubert dominant $Y_{[6,3,3,1]}$}
\label{fig:polynmes_action:schub_images}
\end{figure}

Si l'on pose l'alphabet $y = 0$, alors pour $\lambda$ une partition, le polynôme dominant correspondant à $\lambda$ est simplement $x^{\lambda}$. Pour éviter les confusions, on parlera souvent de polynômes \emph{simples} quand $y=0$ et \emph{doubles} lorsqu'on utilise les deux alphabets. 

Un polynôme est donné par un polynôme dominant et un produit de différences divisées. En réalité, on indexe les polynômes par des vecteurs tels que pour $v$ avec  $v_i > v_{i+1}$, on a

\begin{equation}
\label{eq:polynomes_action:schub_rec}
Y_{\dots, v_{i+1}, v_i - 1, \dots} = Y_v \partial_i.
\end{equation}

Par exemple, sur la figure \ref{fig:polynmes_action:schub_images} on lit que
\begin{equation}
Y_{[1,2,2,3]} = Y_{[6,3,3,1]} \partial_1 \partial_2 \partial_3 \partial_2 \partial_1. \\
\end{equation}
On applique les différences divisées en suivant un chemin dans le graphe entre $Y_{[6,3,3,1]}$ et $Y_{[1,2,2,3]}$. Le résultat ne dépend pas du chemin suivi car les différences divisées vérifient les relations de tresses. Dans le cas des polynômes de Schubert simples, cela donne

\begin{align}
\nonumber
Y_{[1,2,2,3]} &= x_{1}^{3} x_{2}^{2} x_{3}^{2} x_{4} + x_{1}^{3} x_{2}^{2} x_{3} x_{4}^{2} + x_{1}^{3} x_{2} x_{3}^{2} x_{4}^{2} + x_{1}^{2} x_{2}^{3} x_{3}^{2} x_{4} + x_{1}^{2} x_{2}^{3} x_{3} x_{4}^{2} + x_{1}^{2} x_{2}^{2} x_{3}^{3} x_{4}\\
 \nonumber
  &+ 3~x_{1}^{2} x_{2}^{2} x_{3}^{2} x_{4}^{2} + x_{1}^{2} x_{2}^{2} x_{3} x_{4}^{3} + x_{1}^{2} x_{2} x_{3}^{3} x_{4}^{2} + x_{1}^{2} x_{2} x_{3}^{2} x_{4}^{3} + x_{1} x_{2}^{3} x_{3}^{2} x_{4}^{2} + x_{1} x_{2}^{2} x_{3}^{3} x_{4}^{2} \\
   &+ x_{1} x_{2}^{2} x_{3}^{2} x_{4}^{3}.
\end{align}
C'est la fonction de Schur $s_{(3,2,2,1)}$ développée sur $\lbrace x_1, x_2, x_3 \rbrace$. De façon générale, les polynômes de Schubert indexés par des vecteurs croissants sont des fonctions de Schur. 

Il arrive qu'en appliquant une différence divisée, on retombe selon \eqref{eq:polynomes_action:schub_rec} sur un polynôme dominant. Par exemple, on a $Y_{[6,3,3,1]} = Y_{[6,4,3,1]}\partial_2$. Le polynôme $Y_{[6,3,3,1]} $ a donc deux définitions : $x_1^6 x_2^3 x_3^3 x_4$ ou $x_1^6 x_2^4 x_3^3 x_4 \partial_2$. Cependant, on prouve que cela ne provoque pas d'incohérence.

Le vecteur indexant le polynôme peut être interprété comme le code de Lehmer d'une permutation. On peut donc aussi indexer les polynômes par une permutation plutôt que par un vecteur et la relation \eqref{eq:polynomes_action:schub_rec} s'écrit

\begin{equation}
\label{eq:polynomes_action:schub_perm}
Y_{\sigma} = Y_{\sigma s_i} \partial_i.
\end{equation}
où $\sigma$ est une permutation telle que $\sigma_i > \sigma_{i+1}$. Les différences divisées sont alors simplement des opérateurs formels de réordonnement sur les permutations. Le produit de différences divisées qu'on applique à un polynôme dominant pour obtenir un polynôme donné est un chemin dans l'ordre faible.

\begin{figure}[ht]
\centering
\input{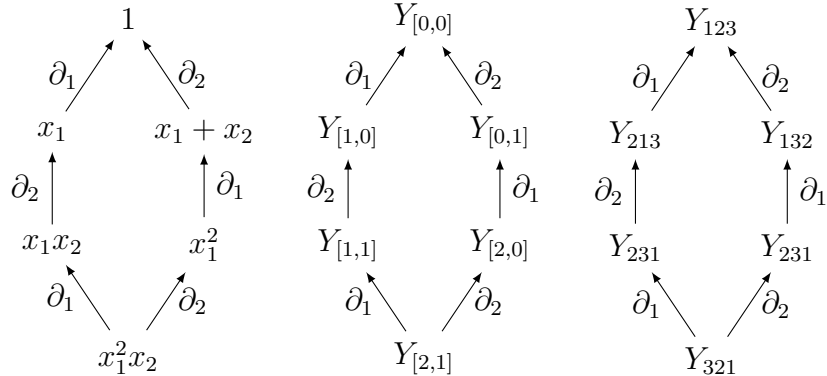}
\caption[Polynômes de Schubert indexés par $\Sym{3}$]{Polynômes de Schubert indexés par $\Sym{3}$. A gauche, les polynômes développés en somme de monômes, au milieu la notation avec codes de Lehmer en index et à droite la notation avec permutations en index.}
\label{fig:polynomes_action:schub_groupe}
\end{figure}

L'indexation par code de Lehmer permet de connaître le nombre de variables nécessaires (sur l'alphabet $x$) au développement du polynôme. En effet, si $v$ est un vecteur de taille $n$ sans 0 final, par construction $Y_v$ se développe sur $n$ variables. Plus précisément, l'ensemble $(Y_v)_{|v|\leq n}$ forme une base triangulaire des polynômes en $n$ variables. Par ailleurs, chaque polynôme est de degré homogène donné par la somme des valeurs du vecteur. Ces propriétés sont illustrées figure \ref{fig:polynomes_action:schub_transition} pour les polynômes en 2 variables de degré inférieur ou égal à 3. 

\begin{figure}[ht]
\centering

\begin{tabular}{|c|c|c|c|c|c|c|c|c|c|c|}
\hline
             & 1 & $x_1$ & $x_2$ & $x_1^2$ & $x_1 x_2$ & $x_2^2$ & $x_1^3$ & $x_1^2 x_2$ & $x_1 x_2^2$ & $x_2^3$ \\ \hline
 $Y_{[0,0]}$ & 1 &       &       &         &           &         &         &             &             & \\ \hline
 $Y_{[1,0]}$ & 0 & 1     &       &         &           &         &         &             &             & \\ \hline
 $Y_{[0,1]}$ & 0 & 1     & 1     &         &           &         &         &             &             & \\ \hline
 $Y_{[2,0]}$ & 0 & 0     & 0     & 1       &           &         &         &             &             & \\ \hline
 $Y_{[1,1]}$ & 0 & 0     & 0     & 0       & 1         &         &         &             &             & \\ \hline
 $Y_{[0,2]}$ & 0 & 0     & 0     & 1       & 1         & 1       &         &             &             & \\ \hline
 $Y_{[3,0]}$ & 0 & 0     & 0     & 0       & 0         & 0       & 1       &             &             & \\ \hline
 $Y_{[2,1]}$ & 0 & 0     & 0     & 0       & 0         & 0       & 0       & 1           &             & \\ \hline
 $Y_{[1,2]}$ & 0 & 0     & 0     & 0       & 0         & 0       & 0       & 1           & 1           & \\ \hline
 $Y_{[0,3]}$ & 0 & 0     & 0     & 0       & 0         & 0       & 1       & 1           & 1           & 1 \\ \hline
\end{tabular}
\caption[Matrice de transition des polynômes de Schubert]{Matrice de transition des polynômes de Schubert pour les vecteurs de taille 2 jusqu'en degré 3}
\label{fig:polynomes_action:schub_transition}
\end{figure}

Les polynômes de Schubert sont donc une généralisation des fonctions de Schur à la fois sur les polynômes non symétriques et par l'utilisation d'un double alphabet. Ils répondent aussi à un problème géométrique plus général. Nous avons vu que le produit des fonctions Schur correspondait au produit des variétés de Schubert de la Grassmannienne dans l'anneau de cohomologie. Au lieu de la Grassmanienne, considérons une suite de variétés linéaires emboitées
\begin{equation}
W_{k_1} \subset W_{k_2} \subset \dots \subset W_{k_n}
\end{equation}
de dimension $k_1 < \dots < k_n$. On l'appelle une \emph{variété de drapeau}, le drapeau est dit complet si $k_i = i$. 

Tout comme pour la Grassmannienne, on peut considérer les intersections d'éléments de la variété de drapeau avec un drapeau de référence. On doit à Ehresmann \cite{Ehresmann} la décomposition en cellules de Schubert des variétés de drapeaux. Pour $\mathcal{F}_n$, la variété des drapeaux complets de $\CC^n$, les cellules sont indexées par une suite de tableaux colonnes

\begin{equation}
\begin{array}{ccccc}
a_{11} & a_{21} & a_{31} & \dots & a_{n1} \\
       & a_{22} & a_{32} &       & a_{n2} \\
       &        & a_{33} &       & \dots   \\
       &        &        &       & a_{nn}
\end{array}
\end{equation}
qui sont par construction emboitées les unes dans les autres. C'est-à-dire, on a
 $\lbrace a_{11} \rbrace \subset \lbrace a_{21}, a_{22} \rbrace \subset \dots \lbrace a_{n1}, \dots, a_{nn} \rbrace $. On a vu dans le chapitre \ref{chap:prelim_groupe_sym} qu'on pouvait interpréter ce type de tableau comme des facteurs gauches réordonnés de permutations. Les cellules de Schubert sont donc indexées par des permutations et l'ordre d'inclusion est l'ordre de Bruhat vu au paragraphe \ref{sec:prelim_groupe_sym:bruhat}.

Comme dans le cas de la Grassmannienne, les anneaux d'homologie et de cohomologie des variétés de drapeaux peuvent être interprétés comme des quotients d'anneaux de polynômes, non symétriques cette fois. Plus précisément, l'anneau de cohomologie de la variété de drapeau $\mathcal{F}_n$ est isomorphe au quotient $\ZZ[x_1, \dots, x_n]/I^{+}$ où $I^{+}$ est l'idéal engendré par les fonctions symétriques sans termes constants en les variables $\lbrace x_1, \dots, x_n \rbrace$. Dans ce contexte, la variété de Schubert indexée par la permutation $\sigma$ correspond au polynôme de Schubert simple $Y_\sigma$. Les polynômes doubles permettent de travailler dans l'anneau de cohomologie équivariante, généralisation classique de la cohomologie. Nous donnerons des exemples de calcul et d'interprétations géométriques dans le chapitre \ref{chap:polynomes_sage}.

\subsection{Polynômes de Grothendieck}
\label{sub-sec:polynomes_action:pol:groth}
\index{polynômes!de Grothendieck}

Pour obtenir une information plus précise sur la géométrie d'une variété, on remplace parfois l'anneau de cohomologie par l'anneau de Grothendieck des classes d'isomorphismes de fibrés vectoriels. On parle dans ce cas de $K$-théorie et de $K$-théorie équivariante. Pour la variété de drapeau $\mathcal{F}^n$, les variétés de Schubert sont alors représentées par une autre base des polynômes multivariés, les polynômes de Grothendieck, introduits par Lascoux et Sch\"utzenberger \cite{LascouxGroth}.

\begin{Definition}
\label{def:polynomes_action:groth}
Soit $\lambda$ une partition, on définit le polynôme de Grothendieck dominant $G_\lambda$ par
\begin{equation}
\label{eq:polynomes_action:groth_dom}
G_\lambda :=  \prod_{\substack{i = 1 \ldots n\\ j = 1 \ldots \lambda_i}} (1 - \frac{y_j}{x_i}).
\end{equation}

Les polynômes de Grothendieck sont l'ensemble des images des polynômes dominants par les différences divisées $\pi_i$.
\end{Definition}

Tout comme pour les polynômes de Schubert, on indexe les polynômes de Grothendieck par des vecteurs. Pour $v$ tel que $v_i > v_{i+1}$, on a

\begin{equation}
\label{eq:polynomes_action:groth_rec}
G_{\dots, v_{i+1}, v_i - 1, \dots} = G_v \pi_i.
\end{equation}
Le vecteur correspond là aussi au code de Lehmer d'une permutation. On peut décider d'indexer directement par la permutation elle même. Dans ce cas, pour $\sigma$ telle que $\sigma_i > \sigma_{i+1}$, on obtient

\begin{equation}
\label{eq:polynomes_action:groth_perm}
G_{\sigma} = G_{\sigma s_i} \pi_i.
\end{equation}

L'algorithme de calcul d'un polynôme de Grothendieck est donc le même que celui d'un polynôme de Schubert, seuls l'opération et le développement du polynôme dominant changent. Par exemple, pour calculer $G_{[1,2,2,3]}$  on suivra dans le graphe dessiné figure \ref{fig:polynmes_action:schub_images} le même chemin que pour $Y_{[1,2,2,3]}$ et on obtient

\begin{equation}
G_{[1,2,2,3]} = G_{[6,3,3,1]} \pi_1 \pi_2 \pi_3 \pi_2 \pi_1.
\end{equation}

On utilisera parfois des polynômes de Grothendieck en un seul alphabet. Deux solutions sont alors possibles. On peut spécialiser l'alphabet $y$ en 1. Dans ce cas les polynômes de Grothendieck forment une base de l'anneau $\ZZ[\frac{1}{x_1}, \dots, \frac{1}{x_n}]$. Cette base est triangulaire pour les variables $(1 - \frac{1}{x_i})$. Une seconde solution consiste donc à effectuer un changement de variable $x_i = (1 - \frac{1}{x_i})$ pour obtenir une base de $\ZZ[x_1, \dots, x_n]$. On a illustré le changement de base dans ce dernier cas figure~\ref{fig:polynomes_action:groth_transition}.

\begin{figure}[ht]
\centering

\begin{tabular}{|c|c|c|c|c|c|c|c|c|c|}
\hline
             & 1 & $x_1$ & $x_1 x_2$ & $x_2$ & $x_1^2$ & $x_1^2 x_2$ & $x_1^2 x_2^2$ & $x_1 x_2^2$ & $x_2^2$ \\ \hline
 $G_{[0,0]}$ & 1 &       &           &       &         &            &               &              &       \\ \hline
 $G_{[1,0]}$ & 0 & 1     &           &       &         &            &               &              &       \\ \hline
 $G_{[1,1]}$ & 0 & 0     & 1         &       &         &            &               &              &       \\ \hline
 $G_{[0,1]}$ & 0 & 1     & -1        & 1     &         &            &               &              &       \\ \hline 
 $G_{[2,0]}$ & 0 & 0     & 0         & 0     & 1       &            &               &              &       \\ \hline 
 $G_{[2,1]}$ & 0 & 0     & 0         & 0     & 0       & 1          &               &              &       \\ \hline
 $G_{[2,2]}$ & 0 & 0     & 0         & 0     & 0       & 0          & 1             &              &       \\ \hline
 $G_{[1,2]}$ & 0 & 0     & 0         & 0     & 0       & 1          & -1            & 1            &       \\ \hline 
 $G_{[0,2]}$ & 0 & 0     & 1         & 0     & 1       & -1         & 0             & -1           & 1     \\ \hline
\end{tabular}
\caption[Matrice de transition des polynômes de Grothendieck]{Début de la matrice de transition des polynômes de Grothendieck simples pour les vecteurs de taille 2 (après le changement de variable $x_i = 1 - \frac{1}{x_i}$)}
\label{fig:polynomes_action:groth_transition}
\end{figure}

\subsection{Polynômes Clés}
\label{sub-sec:polynomes_action:pol:dem}
\index{polynômes!clés}
\index{caractères de Demazure}

Nous avons vu qu'un polynôme de Schubert dominant simple est donné par $x^\lambda$ avec $\lambda$ une partition. Les polynômes de Schubert sont ensuite obtenus par l'application des opérateurs $\partial_i$. Dans le cas des polynômes de Grothendieck, on applique les opérateurs $\pi_i$ tout en modifiant la définition des polynômes dominants. Il est aussi possible d'appliquer les opérateurs $\pi_i$ sur les monômes dominants de Schubert, on obtient alors les \emph{polynômes clés} ou \emph{caractères de Demazure}.

\begin{Definition}
\label{def:polynomes_action:dem}
Soit $\lambda$ une partition, on définit les polynômes clés dominants $K_\lambda$ et $\hK_\lambda$ par
\begin{equation}
\hK_\lambda = K_\lambda :=  x^\lambda
\end{equation}

Les polynômes clés $K$ et $\hK$ sont l'ensemble des images des polynômes dominants par les différences divisées isobares, respectivement $\pi_i$ et $\hpi_i$.
\end{Definition}

On obtient donc deux sortes de polynômes clés, $K$ et $\hK$. Ces polynômes ont été introduits à l'origine par Demazure et sont les caractères en théorie des représentations des \emph{modules de Demazure}.

Tout comme les polynômes de Schubert et de Grothendieck, on les indexe par des vecteurs. Soit un vecteur $v$ tel que $v_i > v_{i+1}$ alors

\begin{align}
K_{v s_i} &= K_v \pi_i, \\
\hK_{v s_i} &= \hK_v \hpi_i.
\end{align}

Les familles $(K_v)_{|v| \leq n}$ et $(\hK_v)_{|v|\leq n}$ forment des bases des polynômes en $n$ variables. Les polynômes $K_v$ sont égaux aux polynômes de Schubert sur une large classe de vecteurs : ceux qui correspondent aux permutations dites \emph{vexillaires}, c'est-à-dire évitant le motif 2143. La première permutation non vexillaire correspond au motif lui même, donc au code de Lehmer $[1,0,1]$,

\begin{align}
Y_{[1,0,1]} &= x_1 x_2 + x_1 x_3 + x_1^2, \\
K_{[1,0,1]} &= x_1 x_2 + x_1 x_3.
\end{align}

Nous étudierons en détail le changement de base des $K$ vers les $\hK$ dans le paragraphe \ref{sec:polynomes_action:0hecke}. Il est lié à l'algèbre 0-Hecke et à l'ordre de Bruhat, nous l'utilisons de façon poussée dans le chapitre \ref{chap:polynomes_grothendieck}. 

Les polynômes clés admettent aussi des définitions en types $B$, $C$ et $D$. Les éléments sont indexés par des vecteurs $v \in \ZZ^n$. On dit que $v$ est dominant pour les types $B$ et $C$ si $v = (v_1, \dots, v_n)$ et $v_1 \geq v_2 \geq \dots \geq v_n \geq 0$. Et on dit que $v$ est dominant pour le type $D$ si $v = (v_1, \dots, v_{n-1}, \pm v_n)$ avec $v_1 \geq v_2 \geq \dots \geq v_n \geq 0$. Alors pour $v$ dominant pour le type  $\diamondsuit = B,C$ ou $D$, on pose
\begin{equation}
K^\diamondsuit_v = \hK^\diamondsuit_v = x^v,
\end{equation}
et pour $v \in \ZZ^n$ tel que $v_i > v_{i+1}$ et $i<n$, 
\begin{align}
K^\diamondsuit_{v s_i} &= K^\diamondsuit_v \pi_i, \\
\hK^\diamondsuit_{v s_i} &= K^\diamondsuit_v \hpi_i.
\end{align}
Le cas $i=n$ est différent en fonction du type. Pour les types $B$ et $C$, on a
\begin{align}
K^B_{v s_n^B} &= K^B_v \pi_n^B & \hK^B_{v s_n^B} &= \hK^B_v \pi_n^B \\
K^C_{v s_n^C} &= K^C_v \pi_n^C & \hK^C_{v s_n^C} &= \hK^C_v \pi_n^C 
\end{align}
si $v_n >0$ et
\begin{align}
K^D_{v s_n^D} &= K^D_v \pi_n^D & \hK^D_{v s_n^D} &= \hK^D_v \pi_n^D
\end{align}
si $v_{n-1} + v_n >0$.

Par exemple, on calcule 
\begin{align}
K^B_{[2,-1,1]} &= K^B_{[2,1,-1]} \pi_2 \\ 
&= K^B_{[2,1,1]} \pi_3^B \pi_2 \\  
&= x_1^2 + x_1^2 x_2^{-1} x_3 + x_1^2 x_2 + x_1^2 x_2 x_3^{-1} + x_1^2 x_2 x_3 + x_1^2 x_3, \\
K^C_{[2,-1,1]} &= K^C_{[2,1,1]} \pi_3^B \pi_2 \\
&= x_1^2  + x_1^2 x_2^{-1} x_3 + x_1^2 x_2 x_3^{-1} + x_1^2 x_2 x_3, \\
K^D_{[2,-2,1]} &= K^D_{[2,1,-2]} \pi_2 \\
&=  K^D_{[2,2,-1]} \pi^D_3 \pi_2 \\
&= x_{1}^{2} x_{2}^{2}x_{3}^{-1} + x_{1}^{2} x_{2} + x_{1}^{2} x_2 x_{3}^{2} + x_{1}^{2} x_{3} + x_{1}^{2} x_{2}x_{3}^{-2} + x_{1}^{2}x_{2} \\
\nonumber
& + x_{1}^{2} x_2^{-2} x_{3} + x_{1}^{2}x_{3}^{-1}.
\end{align}

Les ensembles $(K^B_v)$, $(\hK^B_v)$, $(K^C_v)$, $(\hK^C_v)$, $(K^D_v)$ et $(\hK^D_v)$ pour $v \in \ZZ^n$ forment chacun des bases des polynômes multivariés à exposants dans $\ZZ^n$, aussi appelés \emph{polynômes de Laurent}. 

\section{Algèbre 0-Hecke et ordre de Bruhat}
\label{sec:polynomes_action:0hecke}
\index{algèbre!de 0-Hecke}
\index{ordre!de Bruhat}

Les opérateurs $\pi$ et $\hpi$ apparaissent dans un autre contexte, celui de l'algèbre de Hecke. Ils sont en effet les générateurs de ce qu'on appelle communément \emph{l'algèbre de Hecke à $q=0$} ou \emph{0-Hecke}, cas dégénéré de l'algèbre de Hecke classique. Le changement de base entre les deux familles de générateurs fait intervenir l'ordre de Bruhat et nous intéresse plus particulièrement. Par ailleurs, en faisant agir l'algèbre de 0-Hecke sur les vecteurs, les opérateurs $\pi$ et $\hpi$ s'interprètent simplement comme des opérateurs de réordonnement. De nombreuses questions liées aux polynômes peuvent alors s'interpréter formellement comme des opérations sur un espace vectoriel dont la base est indexée par les permutations. 

\subsection{Algèbre de 0-Hecke}
\label{sub-sec:polynomes_action:0hecke:def}

L'algèbre de Hecke de $\Sym{n}$, notée $\mathcal{H}(t_1, t_2)$, est une déformation de l'algèbre du groupe symétrique en fonction de deux paramètres scalaires $t_1$ et $t_2$. Elle est engendrée par une famille d'opérateurs $(T_i)_{1 \leq i <n}$ vérifiant les relations

\begin{align}
\label{eq:polynomes_action:Hecke_tresse1}
&T_iT_j = T_jT_i, &\text{ si } |i-j|>1, \\
\label{eq:polynomes_action:Hecke_tresse2}
&T_{i+1}T_iT_{i+1} = T_iT_{i+1}T_i, &\text{ si } 1  \leq i \leq n-2, \\
\label{eq:polynomes_action:Hecke_carree}
&(T_i - t_1)(T_i - t_2)=0.
\end{align} 

Elle conserve donc les mêmes relations de tresses que $\Sym{n}$ : \eqref{eq:prelim_groupe_sym:tresse1} et \eqref{eq:prelim_groupe_sym:tresse2}. Seule la relation quadratique est modifiée. Dans la littérature, on trouve souvent l'algèbre de Hecke paramétrée par un unique paramètre $q$, ce qui correspond à $t_2 = q$ et $t_1 = -1$ et donc à la relation quadratique

\begin{equation}
(T_i - q)(T_i + 1) = 0.
\end{equation} 

L'algèbre de Hecke est de dimension $n!$ et sa base est donnée par $(T_\sigma)_{\sigma  \in \Sym{n}}$ où pour $s_{i_1}\dots s_{i_m}$ une décomposition réduite de $\sigma$, on a
\begin{equation}
T_\sigma := T_{i_1} \dots T_{i_m}.
\end{equation}

Nous étudions ici le cas où $q=0$, qu'on appelle aussi l'algèbre de 0-Hecke. Si l'on pose $t_2 = 0$ et $t_1 = -1$, les $T_i$ vérifient la relation quadratique \eqref{eq:polynomes_action:pi_quad} et correspondent donc aux opérateurs de différence divisée $\pi_i$. Si l'on pose à présent

\begin{equation}
\label{eq:polynomes_action:0hecke_base}
\hpi_i = \pi_i -1,
\end{equation}
ces nouveaux opérateurs $\hpi_i$ correspondent aux opérateurs de même nom définis dans le paragraphe \ref{sub-sec:polynomes_action:ope:diffdiv}. De là, ils vérifient les relations de tresses \eqref{eq:polynomes_action:Hecke_tresse1} et \eqref{eq:polynomes_action:Hecke_tresse2}. Leur relation quadratique est donnée par \eqref{eq:polynomes_action:hpi_quad} et correspond à la relation quadratique de Hecke \eqref{eq:polynomes_action:Hecke_carree} pour $t_2 = 0$ et $t_1 = -1$. Si $t_2 = 0$, la spécialisation de $t_1$ en $\pm 1$ n'a donc pas d'impact structurel sur l'algèbre. Dans les deux cas, on l'appelle algèbre de $0$-Hecke et elle admet deux bases $\pi = (\pi_\sigma)_{\sigma \in \Sym{n}}$ et $\hpi = (\hpi_\sigma)_{\sigma \in \Sym{n}}$ dont la relation est donnée par \eqref{eq:polynomes_action:0hecke_base}.

\subsection{Changement de base}
\label{sub-sec:polynomes_action:0hecke:bases}

Le changement de base entre les générateurs $\pi_\sigma$ et $\hpi_\sigma$ est donné par l'ordre de Bruhat. Soit $s_{i_1}s_{i_2}\dots s_{i_m}$ une décomposition réduite d'une permutation $\sigma$, alors

\begin{align}
\hpi_\sigma &= \hpi_{i_1}\hpi_{i_2}\dots \hpi_{i_m}, \\
\label{eq:polynomes_action:HatPiToPi-half}
            &= (\pi_{i_1} - 1)(\pi_{i_2} - 1) \dots (\pi_{i_m} - 1).
\end{align}

On prouve la proposition suivante.

\begin{Proposition} \label{prop:polynomes_action:0hecke_bases}
Soit $\sigma \in \Sym{n}$, alors le développement de $\pi_\sigma$ (resp. $\hpi_\sigma$) dans la base $\hpi$ (resp. $\pi$) est donnée par
   \begin{align}
   		\label{eq:HatPiToPi}
		\hat{\pi}_\sigma &= \sum_{\mu \leq \sigma} (-1)^{\ell(\mu) - \ell(\sigma)} \pi_{\mu},\\
   		\label{eq:PiToHatPi}
		\pi_\sigma &= \sum_{\mu \leq \sigma} \hat{\pi}_\mu,
   \end{align}
où l'ordre $\mu \leq \sigma$ est l'ordre de Bruhat sur les permutations.
\end{Proposition}

En développant \eqref{eq:polynomes_action:HatPiToPi-half}, on obtient bien que $\hpi_\sigma$ est une somme de $\pi_\mu$ où une décomposition réduite de $\mu$ est un sous mot de la décomposition réduite de $\sigma$. C'est-à-dire qu'on a bien $\mu \leq \sigma$ pour l'ordre de Bruhat. Il reste à vérifier que les annulations dans le développement de \eqref{eq:polynomes_action:HatPiToPi-half} permettent bien d'obtenir chaque permutation avec coefficient $\pm 1$ en fonction de sa longueur. On trouve ce résultat dans \cite[Lemme 1.13]{LascouxGroth}, on peut le vérifier sur l'exemple suivant.

\begin{align}
	\hpi_{321} &= \hpi_1 \hpi_2 \hpi_1 \\
	           &= (1 - \pi_1)(1 - \pi_2)(1 - \pi_1) \\
	           &= 1 - \pi_1 - \pi_2 + \pi_2 \pi_1 - \pi_1 + \pi_1^2 + \pi_1\pi_2 - \pi_1 \pi_2 \pi_3 \\
	           &= 1 - \pi_1 - \pi_2 + \pi_2 \pi_1 + \pi_1\pi_2 - \pi_1 \pi_2 \pi_3
\end{align}

\subsection{Opérateurs de réordonnement}
\label{sub-sec:polynomes_action:0hecke:ope}

L'action des opérateurs $\pi$ et $\hpi$ sur les permutations s'interprète comme une opération de réordonnement. On définit deux espaces vectoriels formels sur les permutations dont les bases respectives sont appelées $K = (K_\sigma)_{\sigma \in \Sym{n}}$ et $\hK = (\hK_\sigma)_{\sigma \in \Sym{n}}$. L'action des ensembles $\pi$ et $\hpi$ sur respectivement $K$ et $\hK$ est donnée par

 \begin{equation}
   		\label{eq:polynomes_action:0hecke-pi}
			K_\sigma \pi_i = 
			\begin{cases}
				K_{\sigma s_i} \text{ si }\sigma_i > \sigma_{i+1}, \\
				K_\sigma \text{ sinon},
			\end{cases}
   \end{equation}

   \begin{equation}
   		\label{eq:polynomes_action:0hecke-hpi}
		\hat{K}_\sigma \hat{\pi}_i =
		\begin{cases}
			\hat{K}_{\sigma s_i} \text{ si }\sigma_i > \sigma_{i+1}, \\
			-\hat{K}_\sigma \text{ sinon}.
		\end{cases}
   \end{equation}
   
Cette action est bien compatible avec la définition des opérateurs, on dit qu'elle \emph{réordonne} la permutation. En effet, pour $\omega$ la permutation maximale, on a que $K_\omega \pi_\omega = K_{12 \dots n}$. L'application d'un opérateur $\pi_i$ (ou $\hpi_i$) revient à remonter d'une arête dans le graphe du permutoèdre. L'action de l'opérateur lorsque la permutation est déjà ordonnée est déterminée par les relations quadratiques des opérateurs $\pi$ et $\hpi$. On pose à présent que

\begin{equation}
K_\omega = \hK_\omega.
\end{equation}

Les ensembles $K$ et $\hK$ sont alors deux bases d'un même espace vectoriel et le changement de base est donné par la proposition \ref{prop:polynomes_action:0hecke_bases}, 

\begin{align}
   		\label{eq:HatKToK}
		\hat{K}_\sigma &= \sum_{\mu \geq \sigma} (-1)^{\ell(\mu) - \ell(\sigma)} K_\mu, \\
   		\label{eq:KToHatK}
		K_\sigma &= \sum_{\mu \geq \sigma} \hat{K}_\mu.
   \end{align}
   
Nous avons utilisé ici la même terminologie que pour les polynômes clés du paragraphe \ref{sub-sec:polynomes_action:pol:dem}. En effet, ces polynômes sont une des interprétations que l'on peut donner aux bases $K$ et $\hK$. Il suffit pour cela de poser

\begin{equation}
\label{eq:polynomes_action:0hecke-dem}
\hK_\omega = K_\omega = x_1^n x_2^{n-1} \dots x_n.
\end{equation}

On retrouve la définition des polynômes clés dominants (Définition \ref{def:polynomes_action:dem}). On s'est simplement restreint aux polynômes indexés par des vecteurs aux valeurs distinctes. Dans le cas général, il faut préciser l'action de $\pi_i$ quand $v_i = v_{i+1}$ : $K_v \pi_i = K_v$ et $\hK_v \hpi_i = 0$.

Cependant, les polynômes clés ne sont pas la seule interprétation des bases formelles $K$ et $\hK$. Au lieu de \eqref{eq:polynomes_action:0hecke-dem}, posons à présent

\begin{equation}
K_\omega = \prod_{i=1}^{n-1} \prod_{j=i}^{n-i} (1 - \frac{y_j}{x_i}).
\end{equation}

D'après la Définition \ref{def:polynomes_action:groth}, le polynôme $K_\omega$ correspond alors au polynôme de Grothendieck dominant indexé par la partition $\lambda = n-1, n-2, \dots, 1, 0$, c'est-à-dire par le code de la permutation maximale $\omega$. L'action des opérateurs formels $\pi$ sur $K$ est la même que celle des différences divisées $\pi$ sur les polynômes de Grothendieck indexés par des permutations \eqref{eq:polynomes_action:groth_perm}.

En fonction du sens donné à $K_\omega$, les bases formelles $K$ et $\hK$ peuvent donc être interprétées en termes de polynômes clés ou de polynômes de Grothendieck. Cela nous permet de prouver formellement des résultats qui s'appliqueront aux deux types de polynômes. Ce sera l'objet du chapitre \ref{chap:polynomes_grothendieck}.


\chapter{Formule de Pieri pour les polynômes de Grothendieck}
\label{chap:polynomes_grothendieck}
\index{polynômes!de Grothendieck}
\index{formule de Pieri}

Nous avons vu l'importance de l'anneau de Grothendieck dans les problème de géométrie algébrique. Cet anneau est isomorphe à un anneau de polynômes où les variétés de Schubert sont représentées par les polynômes de Grothendieck (cf. paragraphe \ref{sub-sec:polynomes_action:pol:groth}). La structure multiplicative de ces polynômes est donc particulièrement intéressante à étudier. Dans le cas de l'anneau de cohomologie de la Grassmannienne, cette structure est donnée par la formule de Pieri sur les fonctions de Schur que nous avons décrites dans le paragraphe \ref{sub-sec:polynomes_action:sym:pieri}. On cherche à obtenir une description combinatoire similaire pour les polynômes de Grothendieck. 

La formule de Pieri donne le produit d'une fonction de Schur par une fonction complète. Un équivalent pour les polynômes de Grothendieck est le produit 
\begin{equation}
G_\sigma G_{s_k}
\end{equation}
où $\sigma$ est une permutation donnée et $s_k$ une transposition simple. En termes de codes, cela donne
\begin{equation}
G_v G_{[0^{k-1},1]}
\end{equation}
où $v$ est le code de Lehmer de la permutation $\sigma$. Le polynôme $G_{s_k} = G_{[0^{k-1},1]}$ se développe simplement par les règles décrites dans le paragraphe \ref{sub-sec:polynomes_action:pol:groth}. On a que
\begin{equation}
G_{[0^{k-1},1]} = G_{[k,0^{k-1}]} \pi_1 \pi_2 \dots \pi_{k-1}.
\end{equation}

Calculons par exemple $G_{s_3}$,
\begin{align}
G_{[3,0,0]} \pi_1 \pi_2 &= (1- \frac{y_1}{x_1})(1 - \frac{y_2}{x_1})(1 - \frac{y_3}{x_1}) \pi_1 \pi_2\\
&= \left( 1 - \frac{e_1(y_1,y_2,y_3)}{x_1} + \frac{e_2(y_1,y_2,y_3)}{x_1^2}  - \frac{e_3(y_1,y_2,y_3)}{x_1^3} \right) \pi_1 \pi_2 \\
&= \left( 1 + \frac{e_2(y_1,y_2,y_3)}{x_1 x_2} - \frac{e_3(y_1,y_2,y_3)}{x_1^2 x_2} - \frac{e_3(y_1,y_2,y_3)}{x_1 x_2^2} \right) \pi_2 \\
&= 1 - \frac{y_1 y_2 y_3}{x_1 x_2 x_3}.
\end{align}

La fonction $e$ est la fonction symétrique élémentaire définie au paragraphe~\ref{sub-sec:polynomes_action:sym:def}. A chaque étape, les termes de degré $-1$ en $x_i$ sont annulés par l'opérateur $\pi_i$. De façon générale, on obtient donc 
\begin{equation}
G_{s_k} = 1 - \frac{y_1 \dots y_k}{x_1 \dots x_k}.
\end{equation}

\index{algèbre!de 0-Hecke}
Comme dans le cas de l'anneau de cohomologie, on travaille dans un quotient de l'anneau des polynômes par les fonctions symétriques. Plus précisément, on identifie les fonctions symétriques en $x$ à des fonctions symétriques en $y$. Dans \cite[Théorème 6.4]{LascouxGroth}, Lascoux décrit le produit $G_\sigma G_{s_k}$ dans ce quotient en termes des opérateurs $\pi$ et $\hpi$ de l'algèbre de 0-Hecke.

On rappelle que l'application $\mrho{k}$ définie au paragraphe \ref{sub-sec:prelim_groupe_sym:perm:treillis} envoie une permutation $\sigma$ sur l'élément maximal de son coset $\sigma (\Sym{k} \times \Sym{n-k})$.

\begin{Theoreme}{(Lascoux)}
\label{thm:polynomes_grothendieck:PieriGrothLascoux}
Soit $\sigma \in \Sym{n}$ et $1 \leq k < n$. Soit $\zeta = \zeta(\sigma,k)$ l'image de $\sigma$ par la projection $\mrho{k}$. Alors modulo l'idéal décrit ci-dessus, on a
\begin{equation}
\label{eq:polynomes_grothendieck:PieriGrothLascoux}
G_{\sigma} \frac{y_{\sigma_1}\cdots y_{\sigma_k} }{x_1\cdots x_k} 
= G_{\omega}  \hat{\pi}_{\omega \zeta} 
       \pi_{\zeta^{-1}\sigma}.
\end{equation}
\end{Theoreme}

Par exemple, pour $\sigma = 136254$ et $k=4$ on a $\zeta = 632154$ et
\begin{equation}
G_\sigma \frac{y_1 y_3 y_6 y_2}{x_1 x_2 x_3 x_4} = G_{654321} \hpi_3 \hpi_4 \hpi_5 \hpi_2 \hpi_3 \hpi_4 ~ \pi_3 \pi_2 \pi_1 \pi_2.
\end{equation}

Par le calcul (cf. paragraphe \ref{sub-sec:polynomes_sage:appli:groth}), on obtient une somme de 14 éléments avec multiplicités $\pm 1$

\begin{align}
\label{eq:polynomes_grothendieck:exemple}
G_\sigma \frac{y_1 y_3 y_6 y_2}{x_1 x_2 x_3 x_4} &= G_{136254} - G_{136452} - G_{136524} + G_{136542} - G_{146253} \\ \nonumber &+ G_{146352} + G_{146523} - G_{146532} - G_{156234} + G_{156243} \\ \nonumber &+ G_{156324} - G_{156342} - G_{156423} + G_{156432}.
\end{align}

Le théorème de Lascoux nous dit que l'étude du produit $G_\sigma G_{s_k}$ revient à un calcul formel sur les opérateurs $\pi$ et $\hpi$. Plutôt que de travailler sur les polynômes, on peut donc se placer dans les bases formelles $K$ et $\hK$ définies dans le paragraphe~\ref{sec:polynomes_action:0hecke}. Le calcul précédent revient donc à développer

\begin{equation}
\label{eq:polynomes_grothendieck:K-produit}
K_{\omega}  \hat{\pi}_{\omega \zeta} 
       \pi_{\zeta^{-1}\sigma}
\end{equation}

dans la base des $K$. Ce sera l'objet de ce chapitre. On verra que l'utilisation de la base $\hK$ est essentielle. Bien que les relations entre les opérateurs $\pi$ et $\hpi$ soient bien connues et relativement simples, il n'existe pas de résultat général permettant de développer un produit qui mélange les deux types d'opérateurs. Comme on l'a vu avec l'exemple \eqref{eq:polynomes_grothendieck:exemple}, on obtient empiriquement des coefficients $\pm 1$. Ce résultat s'explique par une interprétation combinatoire donnée dans un cadre plus général par Lenart et Postnikov \cite[Corollaire 8.2]{LenartPostnikov}.

\index{ordre!de Bruhat}
\begin{Theoreme}{(Lenart et Postnikov)}
\label{thm:polynomes_grothendieck:LenartPostnikov}
Soit $s_k$ une transposition simple. On construit la liste de transpositions suivante
$(r_1, \dots r_\ell) = ( (1,n), (2,n), \dots, (k,n), \allowbreak (1,n-1), \dots, (k,n-1), \dots, (1,k+1), \dots, (k,k+1))$. Alors

\begin{equation}
\label{eq:polynomes_grothendieck:LenartPostnikov}
G_{\sigma}G_{s_k} = G_{\sigma} - \frac{y_1 \dots y_k}{y_{\sigma_1} \dots y_{\sigma_k}} \sum_J (-1)^{\vert J \vert}  G_{w(J)}
\end{equation}
sommé sur les sous-ensembles $J=(j_1 < j_2 \dots < j_s)$ de $(1, \dots, \ell)$ tels que $\sigma \lessdot \sigma r_{j_1} \lessdot \sigma r_{j_1}r_{j_2} \lessdot \dots \lessdot \sigma r_{j_1} \dots r_{j_s} = w(J)$ soit une chaîne saturée pour l'ordre de Bruhat de $\sigma$ à $w(J)$. En d'autres termes, $J$ est un chemin dans le diagramme de Hasse de l'ordre de Bruhat entre $\sigma$ et $w(J)$.

La somme n'a pas d'annulation et les coefficients sont $\pm 1$ (chaque permutation est obtenue par au plus une chaîne).
\end{Theoreme}

Ce théorème nous dit que le développement de \eqref{eq:polynomes_grothendieck:K-produit} s'exprime en termes d'énumérations de chaînes sur l'ordre de Bruhat. Sottile et Lenart \cite{SottileLenart} avaient déjà obtenu un résultat similaire dans le cas des polynômes de Grothendieck simples. La démonstration de Lenart et Postnikov ne s'appuie pas sur le théorème de Lascoux \ref{thm:polynomes_grothendieck:PieriGrothLascoux} et leur résultat s'applique à tous les groupes de Coxeter. Dans ce chapitre, nous proposons une démonstration combinatoire de ce théorème dans le cas du type $A$. Surtout nous donnons un résultat plus fort qui permet d'exprimer la somme en termes non plus énumératifs mais structurels. 

\begin{Theoreme}
\label{thm:polynomes_grothendieck:interval}
Soit $\sigma \in \Sym{n}$ et $1 \leq k <n$, alors il existe une permutation $\eta(\sigma, k)$ construite explicitement à partir de $\sigma$ et $k$ telle que
\begin{equation}
G_{\sigma} \frac{y_{\sigma_1}\cdots y_{\sigma_k} }{x_1\cdots x_k} 
=  G_{\omega}  \hat{\pi}_{\omega \zeta} 
       \pi_{\zeta^{-1}\sigma} = \sum_{\sigma \leq \mu \leq \eta} (-1)^{\ell(\mu) - \ell(\sigma)} G_\mu.
\end{equation}
\end{Theoreme}

Le développement de \eqref{eq:polynomes_grothendieck:K-produit} correspond donc à une somme sur un intervalle de l'ordre de Bruhat. Dans l'exemple donné en \eqref{eq:polynomes_grothendieck:exemple}, la permutation $\eta$ est égale à $156432$ et la somme se fait sur l'intervalle $[136254, 156432]$. 

\index{polynômes!clés}
Le but premier de chapitre est de prouver le théorème \ref{thm:polynomes_grothendieck:interval} par un développement dans la base $K$ de l'expression \eqref{eq:polynomes_grothendieck:K-produit}. Les résultats énoncés ici ont été publiés dans \cite{Me_Grothendieck}. L'interprétation en termes de produit de polynômes est propre aux polynômes de Grothendieck. Cependant, le développement du produit des opérateurs $\hpi$ et $\pi$ comme une somme sur un intervalle peut être vu comme un calcul dans l'algèbre de 0-Hecke ou sur les polynômes clés. Une première partie de la preuve est donnée paragraphe \ref{sec:polynomes_grothendieck:cloture} où l'on démontre que l'ensemble de sommation est clos par intervalle. On se sert pour cela d'un développement dans la base $\hK$ puis d'un changement de base. Dans le paragraphe \ref{sec:polynomes_grothendieck:chaines}, nous reformulons le résultat du théorème \ref{thm:polynomes_grothendieck:LenartPostnikov} pour étudier de façon plus précise la structure des chaînes de l'ordre de Bruhat qui apparaissent dans la somme. On donnera en particulier une preuve directe du théorème \ref{thm:polynomes_grothendieck:LenartPostnikov}. Par cette nouvelle description, nous donnons de façon explicite la construction de la permutation $\eta(\sigma, k)$. Nous prouvons dans le paragraphe \ref{sec:polynomes_grothendieck:intervalle} qu'elle est bien l'unique élément maximal de l'ensemble de sommation et donc que la somme se fait sur un intervalle. Le paragraphe \ref{sec:polynomes_grothendieck:gen} est dédié à deux problèmes annexes : la variation de la permutation $\eta(\sigma,k)$ quand $k$ varie et la généralisation du développement de \eqref{eq:polynomes_grothendieck:K-produit} aux autres sous-groupes paraboliques de l'algèbre 0-Hecke.

\section{Clôture par intervalle}
\label{sec:polynomes_grothendieck:cloture}
\index{clos par intervalle}

Le théorème \ref{thm:polynomes_grothendieck:LenartPostnikov} nous dit que le développement de \eqref{eq:polynomes_grothendieck:K-produit} dans la base des $K$ est une somme de permutations avec coefficients $\pm 1$ en fonction de la longueur de la permutation. On peut donc directement considérer cette somme comme un ensemble. Le but de ce paragraphe est de prouver que l'ensemble est clos par intervalle.

\subsection{Développement sur la base $\hK$}
\label{sub-sec:polynomes_grothendieck:cloture:hK}

L'expression \eqref{eq:polynomes_grothendieck:K-produit} peut être développée dans la base $K$ ou $\hK$. En fait, le développement dans la base $\hK$ est beaucoup plus simple et nous apporte déjà des informations sur le résultat. La première partie du calcul de \eqref{eq:polynomes_grothendieck:K-produit} consiste en l'application des opérateurs $\hpi$. Cela se fait directement par la définition et on obtient le résultat dans les deux bases $K$ et $\hK$.

\begin{align}
\label{eq:polynomes_grothendieck:KomegaHatPi}
K_\omega \hat{\pi}_{\omega \zeta} &= \hat{K}_\omega \hat{\pi}_{\omega \zeta} \\
\label{eq:polynomes_grothendieck:HatKzeta}
&= \hat{K}_\zeta \\
\label{eq:polynomes_grothendieck:BigInterval}
&= \sum_{\mu \geq \zeta} (-1)^{\ell(\mu) - \ell(\zeta)} K_\mu .
\end{align}

Une décomposition réduite de $\omega \zeta$ correspond à un chemin dans l'ordre faible entre $\omega$ et $\zeta$. Pour le calcul \eqref{eq:polynomes_grothendieck:exemple}, c'est par exemple le chemin $s_3 s_4 s_5 s_2 s_3 s_4$ entre $654321$ et la permutation $\zeta = 632154$. On obtient donc bien $\hK_\zeta$ ou, par un changement de base vers $K$, une somme sur l'intervalle  $[\zeta, \omega]$. Dans la deuxième partie du calcul, on doit à présent appliquer les opérateurs $\pi$. On peut pour cela partir soit de \eqref{eq:polynomes_grothendieck:HatKzeta} dans la base $\hK$, soit de \eqref{eq:polynomes_grothendieck:BigInterval} dans la base $K$. Dans un premier temps, partons de \eqref{eq:polynomes_grothendieck:HatKzeta} et appliquons les opérateurs à $\hK_\zeta$. On obtient alors le développement de \eqref{eq:polynomes_grothendieck:K-produit} dans la base $\hK$.

\begin{Proposition}
\label{prop:polynomes_grothendieck:HatK}
Soient $\sigma$, $k$, et $\zeta$ définis comme dans le théorème \ref{thm:polynomes_grothendieck:PieriGrothLascoux}, alors 
\begin{equation}
\label{eq:polynomes_grothendieck:resultHatK}
K_\omega \hat{\pi}_{\omega \zeta} \pi_{\zeta^{-1} \sigma} = \sum_{\zeta \geq \mu \geq \sigma} \hat{K}_\mu.
\end{equation}
\end{Proposition}

\begin{proof} 

On a par définition que $\lbrace \zeta_1, \dots, \zeta_k \rbrace = \lbrace \sigma_1, \dots, \sigma_k \rbrace$. Cela signifie que l'intervalle $[\sigma, \zeta]$, bien qu'il soit inclus dans $\Sym{n}$, est en bijection avec un intervalle de $\Sym{k} \times \Sym{n-k}$. On appelle $\varphi$ la bijection du coset $\sigma(\Sym{k} \times \Sym{n-k})$ vers $\Sym{k} \times \Sym{n-k}$. On a que $\varphi(\sigma) = (\sigma', \sigma'')$ où $\sigma'$ et $\sigma''$ sont les standardisés de respectivement $\sigma_1 \dots \sigma_k$ et $\sigma_{k+1} \dots \sigma_n$. Par ailleurs, $\varphi(\zeta) = (\omega', \omega'')$ où $\omega'$ et $\omega''$ sont les permutations maximales de respectivement $\Sym{k}$ et $\Sym{n-k}$. Enfin, une décomposition réduite $v$ de $\zeta^{-1} \sigma$ est un chemin dans l'ordre faible entre $\zeta$ et $\sigma$. En particulier, elle ne contient pas la transposition $s_k$ et se décompose donc en deux chemins $v'$ et $v''$ qui commutent tels que $v'$ ne contient que des transpositions $s_i$ avec $i < k$ et $v''$ des transpositions $s_i$ avec $k < i <n$. 

\`A présent, posons $\phi$ l'application sur l'espace vectoriel engendré par $\hK$ telle que $\phi(\hK_\mu) = \hK_{\mu'} \otimes \hK_{\mu''}$ pour $\varphi(\mu) = (\mu',\mu'')$, alors

\begin{align}
\hK_{\zeta} \pi_{\zeta^{-1} \sigma} &= \phi^{-1} \left( \hK_\omega' \pi_{\omega' \sigma'} \otimes \hK_{\omega''} \pi_{\omega'' \sigma''} \right) \\
&= \phi^{-1} \left( K_{\sigma'} \otimes K_{\sigma''} \right) \\
&= \phi^{-1} \left( \sum_{\mu' \geq \sigma'} \hK_{\mu'} \otimes \sum_{\mu'' \geq \sigma''} \hK_{\mu''} \right).
\end{align}

On obtient une somme sur l'image par $\varphi^{-1}$ du produit entre les intervalles $[\sigma',\mu']$ et $[\sigma'',\mu'']$. Cette image correspond par définition à $[\sigma, \zeta]$.
\end{proof}

Prenons par exemple la permutation $\sigma = 15423$ et $k=3$. Dans ce cas, $\zeta = 54132$ et une décomposition réduite de $\zeta^{-1}\sigma$ est donnée par $s_2 s_1 s_4$. On remarque qu'elle ne contient pas $s_3$. L'intervalle $[\sigma, \zeta]$ est en bijection avec $[132, 321] \times [12, 21]$. On a

\begin{align}
\hK_{54132} \pi_2 \pi_1 \pi_4 &= \phi^{-1} \left( \hK_{321} \pi_2 \pi_1 \otimes \hK_{21} \pi_1 \right) \\
&= \phi^{-1} \left( (\hK_{132} + \hK_{231} + \hK_{312} + \hK_{321}) \otimes (\hK_{12} + \hK_{21}) \right) \\
&= \phi^{-1} (\hK_{132} \otimes \hK_{12}) + \phi^{-1} (\hK_{132} \otimes \hK_{21}) + \phi^{-1} (\hK_{231} \otimes \hK_{12}) \\ \nonumber &+ \phi^{-1} (\hK_{231} \otimes \hK_{21}) + \phi^{-1} (\hK_{312} \otimes \hK_{12}) + \phi^{-1} (\hK_{312} \otimes \hK_{21})\\ \nonumber  &+ \phi^{-1} (\hK_{321} \otimes \hK_{12}) + \phi^{-1} (\hK_{321} \otimes \hK_{21}) \\
&= \hK_{15423} + \hK_{15432} + \hK_{45123} + \hK_{45132} + \hK_{51423} + \hK_{51432}\\ \nonumber &+  \hK_{54123} + \hK_{54132}
\end{align}

\subsection{Changement de base}
\label{sub-sec:polynomes_grothendieck:cloture:K}
\index{ordre!de $k$-Bruhat}

Le développement de \eqref{eq:polynomes_grothendieck:K-produit} dans la base des $\hK$  s'obtient donc simplement par une somme sur un intervalle. Par application directe du changement de base de $\hK$ vers $K$ donné en \eqref{eq:HatKToK}, on obtient une première description du développement dans la base $K$. On a besoin pour cela d'introduire l'ordre de $k$-Bruhat, notion dérivée de l'ordre de Bruhat. 

\begin{Definition}
\label{def:polynomes_grothendieck:kbruhat}
Soit $1 \leq k <n$, une transposition $\tau = (a,b)$ est une $k$-transposition si $a \leq k < b$. On dit que $\tau$ est une $k$-transposition de Bruhat pour une permutation $\sigma$ si $\tau$ est une $k$-transposition et si $\sigma \tau$ est un successeur de $\sigma$ pour l'ordre de Bruhat. La permutation $\sigma \tau$ est alors un $k$-successeur de $\sigma$.
\end{Definition}

La notion de $k$-successeur est par définition plus contrainte que celle de successeur. Elle définit par transitivité un ordre qu'on appelle $k$-Bruhat. L'ordre de Bruhat est donc une extension de l'ordre de $k$-Bruhat (cf. paragraphe \ref{sub-sec:prelim_posets:posets:ext}). Cet ordre a des propriétés algébriques intéressantes, il a été introduit par \cite{LascouxKBruhat} et étudié de façon plus approfondie dans \cite{BergeronKbruhat}. Nous l'utiliserons à plusieurs reprises.

\begin{Proposition}
\label{prop:polynomes_grothendieck:complement}
Soit $\Succs_{\sigma}^{-k}$ l'ensemble des successeurs $\sigma'$ de $\sigma$ qui ne sont pas des $k$-successeurs. C'est-à-dire qu'on a $\sigma' = \sigma \tau$ avec $\tau$ une transposition non $k$-transposition. Alors, le développement de \eqref{eq:polynomes_grothendieck:K-produit} dans la base des $K$ est donné par
\begin{equation}
\label{eq:polynomes_grothendieck:complement}
\sum (-1)^{\ell(\nu) - \ell(\sigma)} K_\nu
\end{equation}
sommé sur les permutations $\nu$ telles que $\nu \geq \sigma$ et $\forall \sigma' \in \Succs_{\sigma}^{-k}$, $\nu \ngeq \sigma'$.
\end{Proposition}  
On somme donc sur des permutations $\nu \geq \sigma$ tel qu'il n'existe pas de chemin entre $\sigma$ et $\nu$ dont la première transposition ne soit pas une $k$-transposition. On verra par la suite que toutes les permutations $\nu$ sont en fait supérieures à $\sigma$ pour $k$-Bruhat. 

\begin{proof}
Par un changement de base sur \eqref{eq:polynomes_grothendieck:resultHatK}, on obtient

\begin{align}
\label{eq:polynomes_grothendieck:DoubleKSum}
\sum_{\zeta \geq \mu \geq \sigma} \hat{K}_\mu &= \sum_{\zeta \geq \mu \geq \sigma} \sum_{\nu \geq \mu} (-1)^{\ell(\nu) - \ell(\mu)}K_\nu \\
\label{eq:Expansioncnu}
&= \sum_{\nu \geq \sigma} c_\nu K_\nu
\end{align}
où
\begin{equation}
\label{eq:polynomes_grothencieck:cnu}
c_\nu = \sum_{\substack{\mu \leq \nu \\ \sigma \leq \mu \leq \zeta}} (-1)^{\ell(\nu) - \ell(\mu)}.
\end{equation}

Par \cite{VermatMobius}, $c_\nu$ est une somme sur la \emph{fonction de Möbius} de $\nu$ et $\mu$. Cette somme se fait sur l'intersection de l'intervalle $[\sigma, \nu]$ avec le coset $\sigma(\Sym{k} \times \Sym{n-k})$. Par le lemme \ref{lem:prelim_groupe_sym:bruhat:coset}, cette intersection est un intervalle. De ce fait, $c_\nu \neq 0$ uniquement quand l'intervalle est réduit à un élément $\sigma$. En d'autres termes, $c_\nu = 0$ si et seulement si il existe $\mu$ tel que $\sigma < \mu \leq \zeta$ et $\mu \leq \nu$. Il suffit de tester les éléments $\mu$ qui sont successeurs directs de $\sigma$, ce qui correspond à $\Succs_{\sigma}^{-k}$.
\end{proof}

De la proposition \ref{prop:polynomes_grothendieck:complement}, on déduit immédiatement le corollaire suivant.

\begin{Corollaire}
\label{thm:polynomes_grothendieck:cloture}
L'ensemble des permutations $\nu$ telles que $K_\nu$ apparaît dans le développement de \eqref{eq:polynomes_grothendieck:K-produit} dans la base $K$ est clos par intervalle.
\end{Corollaire}

\begin{proof}
Soit $\nu$ tel que $K_\nu$ apparaisse dans la somme \ref{eq:polynomes_grothendieck:complement}. Supposons qu'il existe $\sigma \leq \nu' \leq \nu$ telle que $K_{\nu'}$ n'apparaisse pas dans la somme. Alors, par la proposition \ref{prop:polynomes_grothendieck:complement}, il existe $\mu' \in \Succs_{\sigma}^{-k}$ avec $\mu' \leq \nu'$. Par transitivité, on a $\mu' \leq \nu$ ce qui contredit le fait que $K_\nu$ apparaisse dans la somme.
\end{proof}

Ce corollaire est une première étape dans la preuve du théorème \ref{thm:polynomes_grothendieck:interval}. Il reste à prouver que l'ensemble en question est bien un intervalle, c'est-à-dire qu'il comporte un unique élément maximal. Ce sera fait dans le paragraphe \ref{sec:polynomes_grothendieck:intervalle}. Par ailleurs, la proposition \ref{prop:polynomes_grothendieck:complement} permet de retrouver une partie du résultat de Lenart et Postnikov du théorème \ref{thm:polynomes_grothendieck:LenartPostnikov} : on obtient que les multiplicités des permutations dans le développement du produit sont $\pm 1$ en fonction de leur longueur.

\section{\'Enumération de chaînes}
\label{sec:polynomes_grothendieck:chaines}

Pour terminer la preuve du théorème \ref{thm:polynomes_grothendieck:interval}, il nous faut d'abord donner une description plus précise de l'énumération de chaînes de l'ordre de Bruhat du théorème~\ref{thm:polynomes_grothendieck:LenartPostnikov}. Nous reformulons donc ce théorème en réduisant la liste de transpositions à considérer et en introduisant un ordre sur les transpositions. Nous montrons d'abord que cette nouvelle formulation est équivalente à celle donnée par Lenart et Postnikov. Puis nous donnons une nouvelle preuve combinatoire du théorème~\ref{thm:polynomes_grothendieck:LenartPostnikov}. Enfin cette nouvelle formulation permet d'énoncer certaines propriétés qui seront essentielles à la preuve du théorème principal \ref{thm:polynomes_grothendieck:interval}.

\subsection{Description}
\label{sub-sec:polynomes_grothendieck:chaines:desc}

\begin{Definition}
\label{def:polynomes_grothendieck:Wsigma}
Soit $\sigma \in \Sym{n}$ et $1 \leq k <n$. Alors $\W_{\sigma, k}$ est la liste des $k$-transpositions de $\sigma$ munies d'un ordre total défini par:
\begin{equation}
\label{eq:polynomes_grothendieck:totalOrder}
(a,b) \tprec (a',b') \Leftrightarrow \sigma(a) > \sigma(a') \text{ ou } (a=a' \text{ et } \sigma(b) < \sigma(b'))
\end{equation}
\end{Definition}

La relation $\tprec$ dépend de la permutation $\sigma$ et nous pouvons donc comparer des transpositions uniquement au sein d'une même liste $\W_{\sigma, k }$. Nous verrons que $\tprec$ est en fait une extension linéaire d'un ordre partiel sur les transpositions qui lui ne dépend pas de $\sigma$.

La proposition suivante est une nouvelle formulation du théorème \ref{thm:polynomes_grothendieck:LenartPostnikov} comme nous le montrerons dans le paragraphe \ref{sub-sec:polynomes_grothendieck:chaines:ordre}. Nous en donnerons une preuve directe paragraphe \ref{sub-sec:polynomes_grothedieck:chaines:preuve}.

\begin{Proposition}
\label{prop:polynomes_grothedieck:ESigma}
Si $\W_{\sigma, k} = (\tau_1 \prec \tau_2 \prec \dots \prec \tau_m)$,  alors

\begin{equation}
\label{eq:polynomes_grothendieck:ESigmaResult}
K_\omega \hat{\pi}_{\omega \zeta} \pi_{\zeta^{-1} \sigma} = \E{\sigma,k}
\end{equation}
où
\begin{equation}
\label{eq:polynomes_grothendieck:ESigma}
\E{\sigma,k} := K_\sigma \cdot (1 - \tau_1) \cdot (1-\tau_2) \cdots (1-\tau_m) 
\end{equation}
et
\begin{equation}
\label{eq:polynomes_grothendieck:ProdKSigmaTau}
K_\mu \cdot \tau = \begin{cases}
K_{\mu\tau} \text{ si } \tau \text{ est une Bruhat transposition de } \mu, \\
0 \text{ sinon.}
\end{cases}
\end{equation}
\end{Proposition}

En d'autres termes, on considère $\W_{\sigma, k}$ comme un mot sur les transpositions. Alors $\E{\sigma,k}$ est une somme signée sur les sous-mots de $\W_\sigma$ qui sont des chemins partant de $\sigma$ dans l'ordre de Bruhat. Nous appellerons de tels sous-mots des \emph{sous-mots valides}. Par exemple, si $\sigma = 136254$ et $k=4$ (ce que nous noterons par la suite $\sigma = 1362 \vert 54$), alors $\W_\sigma = ((2,6),(2,5),(4,6),(4,5))$ et

\begin{equation}
\E{1362 \vert 54} = K_{1362 \vert 54} \cdot (1 - (2,6))\cdot(1 - (2,5))\cdot(1 - (4,6))\cdot(1 - (4,5)).
\end{equation}

Quand le produit est développé, $\E{\sigma,k}$ se dessine comme un sous-graphe de l'ordre de Bruhat (cf. figure \ref{fig:polynomes_grothendieck:ESigma}).

\begin{figure}[ht]
\centering
\input{includes/figures/groth_esigma}
\caption{L'ensemble $\E{1362 \vert 54}$}
\label{fig:polynomes_grothendieck:ESigma}
\end{figure}

Les coefficients des éléments $K_\mu$ dans $\E{\sigma, k}$ sont $\pm 1$ en fonction de la longueur de la permutation. Cela correspond bien à ce que l'on obtient dans le théorème~\ref{thm:polynomes_grothendieck:LenartPostnikov} ainsi qu'à la caractérisation que nous avons vue dans la proposition \ref{prop:polynomes_grothendieck:complement}. La somme $\E{\sigma, k}$ peut donc être considérée directement comme un ensemble de permutations et on parlera souvent de \emph{l'ensemble} $\E{\sigma, k}$.

Par ailleurs, on a introduit précédemment l'ordre de $k$-Bruhat. Par définition, une permutation $\mu$ de $\E{\sigma, k}$ est reliée à $\sigma$ par un chemin de $k$-transpositions. On a donc, $\mu$ supérieure à $\sigma$ pour l'ordre de $k$-Bruhat, ce que l'on note $\mu \geq_k \sigma$. Une caractérisation de cet ordre est donnée par \cite[Théorème 1.1.2]{BergeronKbruhat}.

\begin{Theoreme}[Bergeron, Sottile]
\label{thm:polynomes_grothendieck:KBruhat}
On a $\sigma \leq_k \mu$ si et seulement si
\begin{enumerate}[label=(\roman{*}), ref=(\roman{*})]
\item \label{cond:polynomes_grothendieck:kbruhat1} 
$(a \leq k \Rightarrow \sigma(a) \leq \mu(a))$ et
$(b > k \Rightarrow \sigma(b) \geq \mu(b))$,
\item \label{cond:polynomes_grothendieck:kbruhat2}
si $a<b$, $\sigma(a)<\sigma(b)$ et $\mu(a)>\mu(b)$ alors $ a \leq k <b$.
\end{enumerate}
\end{Theoreme}

Cela nous donne une propriété des permutations $\mu \in \E{\sigma, k}$ que nous utiliserons souvent,

\begin{align}
\forall a \leq k, &~~\mu(a) \geq \sigma(a), \\
\forall b > k, &~~\mu(b) \leq \sigma(b).
\end{align}

\subsection{Ordre partiel sur les $k$-transpositions}
\label{sub-sec:polynomes_grothendieck:chaines:ordre}

Dans le théorème \ref{thm:polynomes_grothendieck:LenartPostnikov} et la proposition~\ref{prop:polynomes_grothedieck:ESigma}, l'énumération est donnée en fonction d'une liste ordonnée de transpositions. A première vue, les ordres utilisés semblent différents. Cependant, ils sont tous les deux des extensions linéaires de l'ordre partiel suivant. 

\begin{Definition}
\label{def:polynomes_grothendieck:ordre-partiel}
L'ordre partiel $\ttprec$ sur les $k$-transpositions est défini par
\begin{align}
\label{eq:polynomes_grothendieck:ordre-partiel-1}
&(a,c) \ttprec (a,b) \text{ si } b<c, \\
\label{eq:polynomes_grothendieck:ordre-partiel-2}
&(a,c) \ttprec (b,c) \text{ si } a<b.
\end{align}
\end{Definition}

\begin{Lemme}
\label{lem:polynomes_grothendieck:extensions-lineaires}
L'ordre $\tprec$ de la définition \ref{def:polynomes_grothendieck:Wsigma} est une extension linéaire de $\ttprec$.
\end{Lemme}

\begin{proof}
C'est une conséquence directe du critère donné par la proposition~\ref{prop:prelim_groupe_sym:bruhat-critere} pour les transpositions de Bruhat. Soient $\tau \ttprec \tau'$ avec $\tau$ et $\tau'$ appartenant à un même $\W_{\sigma, k}$. Alors, $\tau$ et $\tau'$ sont des transpositions de Bruhat pour $\sigma$. Si $\tau = (a,c)$ et $\tau'=(a,b)$ avec $b<c$ cela signifie que $\sigma(c)<\sigma(b)$, c'est-à-dire $\tau \tprec \tau'$.  De même, si $\tau = (a,c)$ et $\tau'=(b,c)$ on a $\sigma(a) > \sigma(b)$ donc $\tau \tprec \tau'$.
\end{proof}

\begin{Lemme}
\label{lem:polynomes_grothendieck:extensions-lineaires2}
Soient $\tau, \theta$ deux $k$-transpositions et $\tprec'$ une extension linéaire de l'ordre $\ttprec$. Si $\tau = (a,d)$ et $\theta = (a,c)$ alors $\tau \tprec' \theta$ implique $c<d$. De même  si $\tau = (a,c)$ et $\theta = (b,c)$ alors $\tau \tprec' \theta$ implique $a<b$.
\end{Lemme}

La preuve est immédiate. Ce lemme est vrai en particulier pour l'ordre $\tprec$ sur les transpositions de $\W_{\sigma, k}$. On utilisera à plusieurs reprises les lemmes \ref{lem:polynomes_grothendieck:extensions-lineaires} et  \ref{lem:polynomes_grothendieck:extensions-lineaires2} dans le paragraphe \ref{sec:polynomes_grothendieck:intervalle} pour prouver le théorème \ref{thm:polynomes_grothendieck:interval}.

En donnant une liste ordonnée de $k$-transpositions, le théorème \ref{thm:polynomes_grothendieck:LenartPostnikov} définit lui aussi un ordre. Cet ordre avait en fait déjà été introduit dans \cite[Théorème~4.3]{SottileLenart} où le résultat était déjà donné pour les polynômes de Grothendieck simples. Il se décrit comme suit,

\begin{equation}
\label{eq:polynomes_grothendieck:lenart-order}
(a,b) \tprec' (c,d) \Leftrightarrow b>d \text{~ or ~}(b=d\text{~ and ~}a<c).
\end{equation}
C'est clairement une extension linéaire de $\ttprec$.

\begin{Proposition}
\label{prop:polynomes_grothendieck:wsigma-generalisation}
Soit $\tprec'$ une extension linéaire quelconque de $\ttprec$. On définit $\W_{\sigma,k}'$ et $\E{\sigma,k}'$ de la même façon que $\W_{\sigma,k}$ et $\E{\sigma,k}$ en remplaçant $\tprec$ par $\tprec'$. Alors, $\E{\sigma,k}' = \E{\sigma,k}$.
\end{Proposition}

\begin{proof}
Notons que si $\tau$ et $\theta$ sont deux $k$-transpositions non comparables par $\ttprec$, alors elles commutent. Ainsi, si $w'$ est un sous-mot de $\W_{\sigma,k}'$, on peut réordonner ses transpositions selon $\tprec$ au lieu de $\tprec'$. On obtient alors un mot $w$, sous-mot de $\W_{\sigma, k}$. Par ailleurs, si on interprète les mots de transpositions en termes de permutations, on a $w = w'$.
\end{proof}

Toutes les propriétés fondamentales de $\E{\sigma, k}$ viennent de l'ordre partiel sur les $k$-transpositions. L'ordre total n'est imposé que pour une facilité d'écriture et de construction. En particulier, tous les résultats que nous prouverons par la suite ne dépendent que de l'ordre partiel. On démontre en particulier le lemme suivant qui nous servira de nombreuses fois.

\begin{Lemme}
\label{lem:polynomes_grothendieck:non-inversions}
Soit $\mu \geq_k \sigma$ tel qu'un chemin entre $\sigma$ et $\mu$ soit donné par $w = \tau_1 \tprec' \dots \tprec' \tau_r$ où $\tprec'$ est une extension linéaire de $\ttprec$. Et soit $\theta =(a,b)$ une $k$-transposition telle que $\tau_r \tprec' \tau$. Alors

\begin{equation}
\label{eq:polynomes_grothendieck:non-inversions1}
\sigma(a) < \sigma(b) \Rightarrow \mu(a) < \mu(b)
\end{equation}
De façon plus forte, pour $c$ tel que $a < c < b$ alors

\begin{equation}
\label{eq:polynomes_grothendieck:non-inversions2}
\sigma(a) < \sigma(c) < \sigma(b) \Rightarrow \mu(a) < \mu(c) < \mu(b)
\end{equation}

\end{Lemme}

\begin{proof}
Prouvons d'abord la propriété plus faible \eqref{eq:polynomes_grothendieck:non-inversions1}. Il suffit de regarder une seule étape. Soit $\sigma$ telle que $\sigma(a) < \sigma(b)$ et $\tau$ une $k$-transposition de Bruhat pour $\sigma$ avec $\tau \tprec' \theta$. On veut prouver que $\sigma \tau (a) < \sigma \tau (b)$. Si $\tau = (a,c)$ alors par le lemme \ref{lem:polynomes_grothendieck:extensions-lineaires2} on a $c>b$. Comme $\tau$ est une transposition de Bruhat pour $\sigma$, cela signifie $\sigma(b) > \sigma(c)$. De même si $\tau = (c,b)$, on a $a < c$ et donc $\sigma(c)<\sigma(a)$.

Pour prouver \eqref{eq:polynomes_grothendieck:non-inversions2}, on se place d'abord dans le cas où $c \leq k$. Comme $(a,b) \ttprec (c,b)$, il est clair que $\mu(c) < \mu(b)$. Prouvons sur une étape $\tau$ que si $\sigma(a) < \sigma(c)$ alors $\sigma \tau (a) < \sigma \tau (c)$. On a que $\sigma \tau (c) \geq \sigma(c)$ donc seule l'action de $\tau$ sur $a$ nous intéresse. Si $\tau = (a,d)$ alors comme $a < c < d$, et que $\tau$ est une transposition de Bruhat, on a nécessairement $\sigma(c) > \sigma(d)$. La preuve pour le cas $c > k$ est exactement symétrique. 
\end{proof}

Il est clair qu'une chaîne de l'énumération donnée en proposition \ref{prop:polynomes_grothedieck:ESigma} peut être réordonnée pour correspondre à une chaîne du théorème \ref{thm:polynomes_grothendieck:LenartPostnikov}. Il n'est pas clair que l'opération inverse puisse être effectuée. En effet, le théorème \ref{thm:polynomes_grothendieck:LenartPostnikov} part d'une liste contenant toutes les $k$-transpositions alors que la proposition \ref{prop:polynomes_grothedieck:ESigma} n'utilise que les $k$-transpositions de Bruhat d'une certaine permutation. Cependant, la liste de $k$-transpositions du théorème \ref{thm:polynomes_grothendieck:LenartPostnikov} peut en fait être réduite à $\W_{\sigma, k}$.

\begin{Lemme}
\label{lem:polynomes_grothendieck:equiv-lenart}
Les transpositions apparaissant dans les chaînes du théorème \ref{thm:polynomes_grothendieck:LenartPostnikov} sont des transpositions de Bruhat pour la permutation $\sigma$.
\end{Lemme}

\begin{proof}
Soit $\tau_1 \tprec' \dots \tprec' \tau_\ell$ une chaîne du théorème \ref{thm:polynomes_grothendieck:LenartPostnikov} et supposons qu'il existe $\tau_i = (a,b)$ une transposition de la chaîne qui ne soit pas de Bruhat pour $\sigma$. Soit $w = \tau_1 \tprec' \dots \tprec' \tau_{i-1}$ et $\mu = \sigma w$. Comme $\tau$ n'est pas une transposition de Bruhat pour $\sigma$ alors il existe $c$ tel que $a < c < b$ et $\sigma(a) < \sigma(c) < \sigma(b)$. Par le lemme \ref{lem:polynomes_grothendieck:non-inversions}, on a alors $\mu(a) < \mu(c) < \mu(b)$ ce qui contredit le fait que $\tau_i$ soit une transposition de Bruhat pour $\mu$.
\end{proof}

\subsection{Preuve directe}
\label{sub-sec:polynomes_grothedieck:chaines:preuve}

Nous avons vu au paragraphe \ref{sec:polynomes_grothendieck:cloture} que 

\begin{equation}
\label{eq:polynomes_grothendieck:BigInterval2}
K_\omega \hat{\pi}_{\omega \zeta} = \sum_{\mu \geq \zeta} (-1)^{\ell(\mu) - \ell(\zeta)} K_\mu .
\end{equation}

La proposition \ref{prop:polynomes_grothedieck:ESigma} se prouve par récurrence sur la longueur de la permutation $\zeta^{-1} \sigma$ en appliquant les opérateurs $\pi$ sur la somme \eqref{eq:polynomes_grothendieck:BigInterval2}. Le cas initial est donné par le lemme suivant.

\begin{Lemme}
\label{lem:polynomes_grothendieck:EZeta}
\begin{equation}
\label{eq:EZeta}
\E{\zeta,k} = \sum_{\mu \geq \zeta} (-1)^{\ell(\mu) - \ell(\zeta)}K_\mu.
\end{equation}
\end{Lemme}

\begin{proof}[Démonstration lemme \ref{lem:polynomes_grothendieck:EZeta}]
Le lemme \ref{lem:polynomes_grothendieck:EZeta} revient à dire que l'ensemble $\E{\sigma, k}$ est égal à l'intervalle $[\zeta, \omega]$. Par construction, on a $\E{\zeta, k} \subset [\zeta, \omega]$. Il reste à prouver $[\zeta, \omega] \subset \E{\zeta, k}$.

Soit $\mu \in [\zeta, \omega]$, on prouve d'abord que $\mu \geq_k \zeta$. Soit $a \leq k$, la comparaison des clés gauches de $\mu$ et $\zeta$ nous dit que le facteur gauche réordonné de $\mu$ est plus grand valeur par valeur que le facteur gauche réordonné de $\zeta$. Dans le cas de $\zeta$, ce facteur est antidominant et on a donc $\mu(a) \geq \zeta(a)$. De même pour $b>k$, on a $\mu(b) \leq \zeta(b)$. La condition \ref{cond:polynomes_grothendieck:kbruhat1} du théorème \ref{thm:polynomes_grothendieck:KBruhat} est donc vérifiée. Par ailleurs si $a < b$ et $\zeta(a) < \zeta(b)$, alors $a \leq k < b$ et donc la condition \ref{cond:polynomes_grothendieck:kbruhat2} est aussi vérifiée.

Il existe donc un chemin dans l'ordre de Bruhat entre $\zeta$ et $\mu$ formé de $k$-transpositions. Il faut prouver que ce chemin est un sous mot de $\W_{\zeta, k}$. Pour cela on utilise l'algorithme décrit dans \cite[Algorithme 3.1.1]{BergeronKbruhat}. En fonction de deux permutations $\zeta \leq_k \mu$, un processus retourne une $k$-transposition $\tau$ telle que $\zeta \leq_k \mu \tau \leq_k \mu$. La transposition $\tau$ est donc le dernier élément d'un chemin entre $\zeta$ et $\mu$. On applique ensuite récursivement le processus sur $\mu \tau$. Le choix de $\tau$ est fait comme suit :
\begin{itemize}
\item choisir $a \leq k$ avec $\zeta(a)$ minimal pour $\zeta(a) < \mu(a)$,
\item choisir $b > k$ avec $\zeta(b)$ maximal pour $\mu(b) < \mu(a) \leq \zeta(b)$,
\end{itemize}
alors $\tau = (a,b)$. Bergeron et Sottile prouvent que $\mu$ est bien un successeur de $\mu \tau$ pour Bruhat. On a par construction $\zeta(a) < \zeta(b)$. Or $\zeta(i) < \zeta(a)$ pour tout $i$ tel que $a < i \leq k$ et $\zeta(i) > \zeta(b)$ pour tout $i$ tel que $k < i \leq b$. De là, $\tau$ est une transposition de Bruhat pour $\zeta$ et donc $\tau \in \W_{\zeta, k}$. Par ailleurs, comme on choisit à chaque étape $\zeta(a)$  minimal et $\zeta(b)$ maximal, on aura $\tau' \tprec \tau$ pour $\tau'$ choisie après $\tau$. Le chemin obtenu est donc bien un sous-mot de $\W_{\zeta, k}$.
\end{proof}

Prouver la proposition \ref{prop:polynomes_grothedieck:ESigma} revient maintenant à montrer

\begin{equation}
\label{eq:polynomes_grothendieck:rec}
\E{\sigma, k} = \E{\zeta, k} \pi_{\zeta^{-1} \sigma}.
\end{equation}

Il suffit pour cela de montrer une seule étape.

\begin{Proposition}
\label{prop:polynomes_grothendieck:InductiveStep}
Soit $\sigma \in \Sym{n}$ et on suppose que l'hypothèse de récurrence \eqref{eq:polynomes_grothendieck:rec} est vérifiée. Soit $s_i$ une transposition simple telle que $i \neq k$ et $\sigma s_i < \sigma$, alors

\begin{equation}
\label{eq:polynomes_grothendieck:InductiveStep}
\E{\sigma,k} \pi_i = \E{\sigma s_i,k}.
\end{equation}

\end{Proposition}

Remarquons d'abord une première propriété.

\begin{Lemme}
$\E{\sigma s_i, k} \cap \E{\sigma,k} = \emptyset$, et plus précisément $ \forall \nu \in \E{\sigma s_i, k}$, $\nu \ngtr \sigma$.
\end{Lemme}
\begin{proof}
Si $i < k$, soit $p_\mu = \# \lbrace \mu_j \geq \sigma_i, j \leq i \rbrace$, le nombre de valeurs supérieures ou égales à $\sigma_i$ dans le facteur de taille $i$ d'une permutation. Le nombre $p_\mu$ est constant sur $\E{\sigma s_i, k}$ et égal à $p_{\sigma s_i} = p_{\sigma} -1$. En effet, toute transposition $(a,b)$ avec $a \leq i < k < b$ telle que $\sigma s_i(b) > \sigma(i)$ n'est pas une transposition de Bruhat pour $\sigma s_i$. On en conclut que $\nu \ngeq \sigma$ pour $\nu \in \E{\sigma s_i, k}$. Un raisonnement symétrique peut être fait si $ i > k$.
\end{proof}

La proposition \ref{prop:polynomes_grothendieck:InductiveStep} est une conséquence des deux lemmes suivant.

\begin{Lemme}
\label{lem:polynomes_grothendieck:ESigmaPiNonZeroImage}
On a les implications suivantes :
\begin{enumerate}
\item \label{lem:polynomes_grothendieck:ESigmaPiNonZeroImage-LtoR}
Soit $w = \tau_{i_1} \dots \tau_{i_r}$ un sous-mot valide de $\W_{\sigma, k}$ et tel que pour $\mu = \sigma w$, alors $\mu s_i \ngtr \sigma$. Alors, $w' = (s_i\tau_{i_1}s_i)(s_i\tau_{i_2}s_i)\dots (s_i\tau_{i_r}s_i)$
est un sous-mot valide de $\W_{\sigma s_i, k}$.

\item \label{lem:polynomes_grothendieck:ESigmaPiNonZeroImage-RtoL}
Inversement, si $w'=t_1 \dots t_r$ est un sous-mot valide de $W_{\sigma s_i, k}$ alors 
$w = (s_i t_1 s_i) \dots (s_i t_r s_i)$
est un sous-mot valide de $W_{\sigma, k}$
\end{enumerate}
\end{Lemme}

\begin{Lemme}
\label{lem:polynomes_grothendieck:ESigmaPiZeroImage}
Soit $\mu \in \E{\sigma, k}$ tel que $\mu s_i > \sigma$, alors $\mu s_i \in \E{\sigma,k}$.
\end{Lemme}

\begin{proof}[Démonstration du lemme \ref{lem:polynomes_grothendieck:ESigmaPiNonZeroImage}] 
Cette démonstration utilise principalement des arguments de comparaisons de clés dans l'ordre de Bruhat, en particulier la proposition \ref{prop:prelim_groupe_sym:bruhat:compcle_gen}. Pour faciliter l'écriture, nous notons $\sigma(1, \dots, i)$ le facteur gauche de taille $i$ de $\sigma$ et $\sigma(1, \dots, i) \leq \mu(1, \dots, i)$ signifie que le facteur gauche réordonné de $\sigma$ est plus petit terme à terme que le facteur gauche réordonné de $\mu$.

L'implication (\ref{lem:polynomes_grothendieck:ESigmaPiNonZeroImage-RtoL}) est immédiate. Si $t$ est une $k$-transposition de Bruhat pour $\sigma s_i$, il est clair que $s_i t s_i$ est une $k$-transposition de Bruhat pour $\sigma$. Par ailleurs, si $t \tprec t'$ dans $\W_{\sigma s_i, k}$ alors $ s_i t s_i \tprec s_i t' s_i$ dans $\W_{\sigma, k}$ car $t$ et $s_i t s_i$ agissent sur les mêmes valeurs. Donc tout sous-mot valide $t_1 \dots t_r$ de $\W_{\sigma s_i, k}$ donne un sous-mot valide $(s_i t_1 s_i) \dots (s_i t_r s_i)$ de $\W_{\sigma, k}$.

 \`A présent, soit  $w = \tau_{i_1} \dots \tau_{i_r}$ un sous-mot valide de $\W_{\sigma, k}$ tel que pour $\mu = \sigma w$, alors $\mu s_i \ngtr \sigma$. Sans perte de généralité, on supposera que $i<k$. La preuve est symétrique pour $i>k$.
 
On commence par prouver que pour toute permutation $\tilde{\mu}$ de la chaîne donnée par $w$ entre $\sigma$ et $\mu$, alors $\tilde{\mu}s_i \ngtr \sigma$. On sait que $\mu s_i \ngtr \sigma$, ce qui signifie qu'il existe au moins un facteur gauche de $\mu s_i$ qui n'est pas plus grand que le facteur correspondant de $\sigma$. Par ailleurs, comme $\mu \geq \sigma$ tous les facteurs gauches de $\mu$ sont plus grands que les facteurs gauches correspondants de $\sigma$. Le seul facteur gauche qui diffère entre $\mu$ et $\mu s_i$ est le facteur gauche de taille $i$. On a donc
\begin{align}
&\mu s_i (1, \dots, i) \ngeq \sigma(1, \dots, i)\text{~ et }\\
&\mu s_i (1, \dots, i) = \lbrace \mu(1), \dots, \mu(i-1), \mu(i+1) \rbrace.
\end{align}
Par ailleurs comme $\tilde{\mu} <_k \mu$, on a $\tilde{\mu}(j) \leq \mu(j)$ pour tout 
$j\leq k$ et donc $\tilde{\mu} s_i (1, \dots, i) \ngeq \sigma(1, \dots, i)$.

A partir de ce résultat, on prouve que pour toute transposition $\tau$ de $w$, on a $\sigma \tau(i) > \sigma \tau (i+1)$. On sait que $\sigma(i) > \sigma(i+1)$ et $i<k$, le seul cas à considérer est donc celui où $\tau = (i+1,b)$. On veut prouver que $\sigma(b) < \sigma(i)$. On suppose qu'une telle transposition appartient à $w$ et qu'elle relie une permutation $\nu$ à une permutation $\nu'$. La permutation $\nu'$ est dans la chaîne donnée par $w$ entre $\sigma$ et $\mu$, par le résultat précédent, on a donc $\nu' s_i \ngtr \sigma$ et plus précisément $ \lbrace \nu'(1), \dots, \nu'(i-1), \nu'(i+1) \rbrace \ngeq \lbrace \sigma(1), \dots, \sigma(i) \rbrace$. Comme $\nu' >_k \sigma$, on a que $\nu'(j) \geq \sigma(j)$ pour $j \leq i-1$. Nécessairement, $\nu'(i+1) < \sigma(i)$ pour avoir $\nu' s_k \ngtr \sigma$. Cela nous donne que $\nu(b) < \sigma(i)$ car $\nu(b) = \nu'(i+1)$.

Supposons à présent que $\sigma(b) > \sigma(i)$. Cela signifie que la valeur en $b$ a été modifiée par $w$ car on a vu que $\nu(b) < \sigma(i)$. Il existe donc une transposition $(c,b) \tprec \tau$ dans $w$ telle que  $\sigma(c) < \sigma(i) < \sigma(b)$. Pour que $(c,b)$ soit une transposition de Bruhat, il faut $c > i$. Cependant, on a $(c,b) \tprec \tau = (i+1,b)$ et donc $c<i+1$ (par le lemme \ref{lem:polynomes_grothendieck:extensions-lineaires2}). On arrive à une contradiction et donc $\sigma(b) < \sigma(i)$.

De cette propriété, on déduit facilement que pour chaque transposition $\tau$ de $w$, la transposition $s_i \tau s_i$ est une transposition de Bruhat pour $\sigma s_i$. Et donc, $(s_i \tau_{i_1} s_i) \dots (s_i \tau_{i_r} s_i)$ est un sous-mot de $\W_{\sigma s_i, k}$. Par ailleurs, ce sous-mot est valide car si $\nu$ et $\nu' = \nu \tau$ sont deux permutations de la chaîne, alors $\nu'(i+1) < \nu(i)$ ce qui fait que $s_i \tau s_i$ est une transposition de Bruhat pour $\nu s_i$. 
\end{proof}

\begin{proof}[Démonstration du lemme \ref{lem:polynomes_grothendieck:ESigmaPiZeroImage}]
On utilise la proposition \ref{prop:polynomes_grothendieck:complement} qui est vraie pour $\E{\sigma,k}$ par l'hypothèse de récurrence. On montre que si $\mu \in \E{\sigma, k}$ avec $\mu s_i > \sigma$, il n'existe pas de permutation $\sigma' \in \Succs_{\sigma}^{-k}$ telle que $\sigma' \leq \mu s_i$. Cela signifie que $\mu s_i \in \E{\sigma, k}$. La propriété est triviale pour $\mu > \mu s_i$, on se place donc dans le cas inverse $\mu < \mu s_i$.

Supposons alors qu'il existe $\sigma' = \sigma \tau$ un successeur de $\sigma$ tel que $\tau = (a,b)$ ne soit pas une $k$-transposition et tel que $\sigma' \leq \mu s_i$. Comme précédemment, on suppose $i<k$. Tout d'abord, on a nécessairement $a \leq i < b$. En effet, on sait que $\mu s_i \geq \sigma'$ et que $\mu \ngeq \sigma'$, or le seul facteur gauche qui diffère entre $\mu$ et $\mu s_i$ est celui de taille $i$. On en déduit que $\mu(1, \dots, i) \ngeq \sigma'(1, \dots, i)$. Or $\mu(1, \dots, i) \geq \sigma(1, \dots, i)$ donc $\sigma(1, \dots, i) \neq \sigma'(1, \dots, i)$.

\`A présent, montrons

\begin{equation}
\label{eq:polynomes_grothendieck:proof-ESigmaPiZeroImage-eg1}
\forall j<b, (\sigma(j) < \sigma(b)  \Leftrightarrow \mu(j) < \sigma(b)).
\end{equation}

Comme on a $\mu(j) \geq \sigma(j)$ pour $j\leq k$, l'implication de droite à gauche est claire. Supposons à présent qu'il existe $c < b$ avec $\sigma(c) < \sigma(b)$ et $\mu(c) \geq \sigma(b)$. Cela signifie qu'on a appliqué une transposition $(c,d)$ avec $c < b <d$ et $\sigma(c) < \sigma(b) < \sigma(d)$ ce qui n'est pas possible.

Par ailleurs, on prouve aussi
\begin{equation}
\label{eq:polynomes_grothendieck:proof-ESigmaPiZeroImage-eg2}
 \forall j \leq a, (\mu(j) < \sigma(b) \Leftrightarrow  \mu s_i(j) < \sigma(b)) .
\end{equation}

Pour $a<i$, on a $\mu(1, \dots, j) = \mu s_i(1, \dots j)$ et donc \eqref{eq:polynomes_grothendieck:proof-ESigmaPiZeroImage-eg2} est vraie. Il reste à considérer $a=i$, dans ce cas $\sigma(i) = \sigma(a) < \sigma(b)$ et donc $\mu(i) <\sigma(b)$ par \eqref{eq:polynomes_grothendieck:proof-ESigmaPiZeroImage-eg1}. De même, comme $\sigma(i+1) < \sigma(i) < \sigma(b)$ on a aussi $\mu s_i (i) = \mu(i+1) < \sigma(b)$ par \eqref{eq:polynomes_grothendieck:proof-ESigmaPiZeroImage-eg1}. En effet, comme $a=i$ et $(a,b) \neq s_i$ alors $i+1 < b$. 

\`A présent par \eqref{eq:polynomes_grothendieck:proof-ESigmaPiZeroImage-eg1} et \eqref{eq:polynomes_grothendieck:proof-ESigmaPiZeroImage-eg2}, on a

\begin{align}
\# \lbrace j \leq a~;~\mu s_i(j) < \sigma(b) \rbrace &= \#  \lbrace j \leq a~;~\sigma(j) < \sigma(b) \rbrace \\
&= \#  \lbrace j \leq a~;~\sigma'(j) < \sigma(b) \rbrace + 1,
\end{align}
ce qui contredit le fait $\mu s_i \geq \sigma'$. 
\end{proof}

\begin{proof}[Démonstration de la proposition \ref{prop:polynomes_grothendieck:InductiveStep}]
Une conséquence du lemme \ref{lem:polynomes_grothendieck:ESigmaPiNonZeroImage} est que

\begin{equation}
\E{\sigma s_i, k} = \left( \sum_{ \substack{\mu \in \E{\sigma,k} \\ \mu s_i \ngtr \sigma}} (-1)^{\ell(\mu) - \ell(\sigma)} K_\mu \right) \pi_i. 
\end{equation}

En effet, pour tous les éléments $K_\mu$ de la somme ci-dessus, on a $K_\mu \pi_i = K_{\mu s_i}$ car $\mu s_i \ngtr \sigma$ implique $\mu s_i < \mu$. L'implication (\ref{lem:polynomes_grothendieck:ESigmaPiNonZeroImage-LtoR}) du lemme \ref{lem:polynomes_grothendieck:ESigmaPiNonZeroImage} nous dit que les $K_{\mu s_i}$ sont bien dans la somme $\E{\sigma s_i, k}$ et l'implication (\ref{lem:polynomes_grothendieck:ESigmaPiNonZeroImage-RtoL}) nous dit qu'elle ne contient pas d'autres éléments. \`A présent,

\begin{equation}
\E{\sigma,k} \pi_i = \E{\sigma s_i, k} + \sum_{\substack{\mu \in \E{\sigma,k} \\ \mu s_i > \sigma}} (-1)^{ \ell(\mu) - \ell(\sigma)} K_\mu \pi_i.
\end{equation} 

Par le lemme \ref{lem:polynomes_grothendieck:ESigmaPiZeroImage}, la seconde partie de la somme est égale à 0. En effet, chaque élément $K_\mu$ est couplé à un élément $K_{\mu s_i}$ qui apparaît lui aussi dans la somme avec un signe opposé. Pour chaque couple, on a alors $(K_\mu - K_{\mu s_i})\pi_i = 0$ par définition de l'opérateur $\pi_i$.
\end{proof}

Pour illustrer cette preuve, on peut dessiner l'exemple suivant (figure \ref{fig:polynomes_grothendieck:ESigmaPi_i}). Chaque élément de $\E{\sigma, k}$ (sur la gauche) est couplé par $s_i$ soit avec un élément de $\E{\sigma s_i, k}$ (sur la droite), soit avec un autre élément de $\E{\sigma, k}$.

\begin{figure}[ht]
\centering
\input{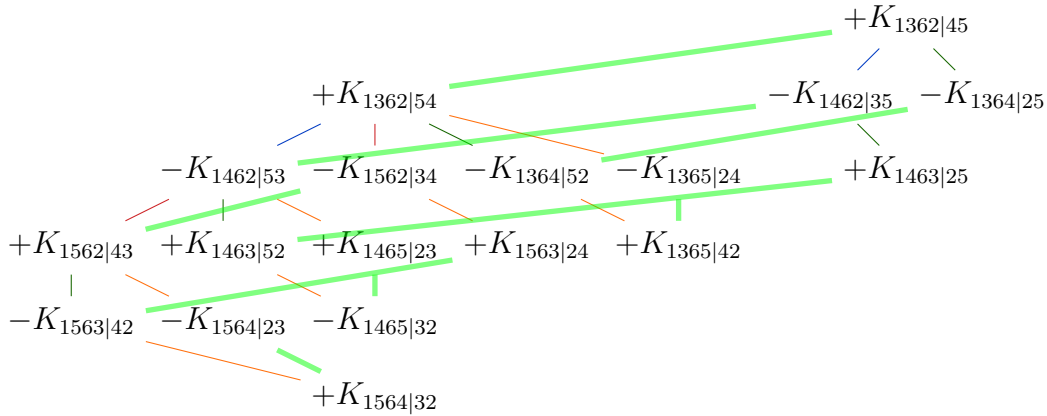}
\caption{Illustration du calcul $\E{1362|54} \pi_4 = \E{1362|45}$}
\label{fig:polynomes_grothendieck:ESigmaPi_i}
\end{figure}

\begin{figure}[ht]
\centering
\input{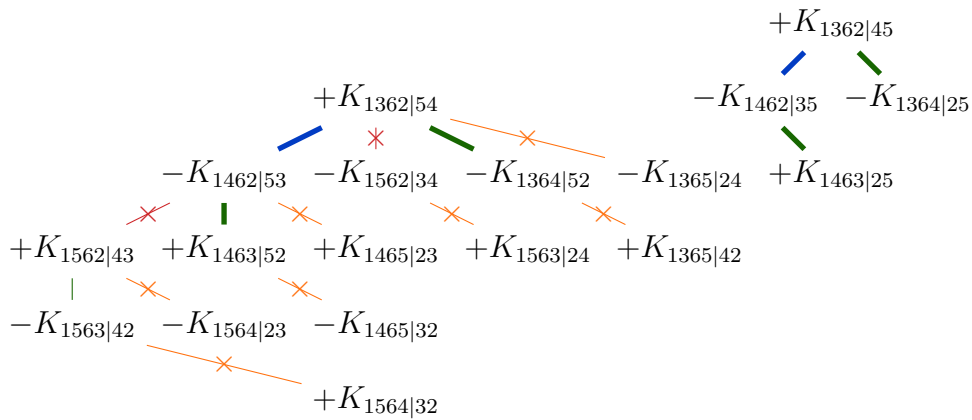}
\caption{Branches "coupées" dans $\E{1362|54} \pi_4 = \E{1362|45}$}
\label{fig:polynomes_grothendieck:cuttingBranches}
\end{figure}

Par ailleurs $\E{\sigma s_i, k}$ est en fait un sous-arbre de $\E{\sigma, k}$ après conjugaison des transpositions. Les transpositions de $\W_{\sigma, k} = (\tau_1, \dots \tau_m)$ forment les branches de l'arbre. Le lemme \ref{lem:polynomes_grothendieck:ESigmaPiNonZeroImage} nous dit que $\W_{\sigma s_i, k}$ est un sous-mot conjugué de $\W_{\sigma, k}$: seules les transpositions $\tau$ telles que $s_i \tau s_i$ soit toujours une transposition de Bruhat pour $\sigma s_i$ sont conservées. Cela revient à "couper" les branches de $\E{\sigma, k}$ qui correspondent à des transpositions non conservées (cf. figure \ref{fig:polynomes_grothendieck:cuttingBranches}).

\section{Structure d'intervalle}
\label{sec:polynomes_grothendieck:intervalle}

Nous savons par le corollaire \ref{thm:polynomes_grothendieck:cloture} que l'ensemble $\E{\sigma, k}$ est clos par intervalle. Pour prouver le théorème \ref{thm:polynomes_grothendieck:interval}, il faut prouver que $\E{\sigma, k}$ possède un unique élément maximal. La suite de transpositions $\W_{\sigma, k}$ est toujours un chemin valide à partir de $\sigma$ dans l'ordre de Bruhat. Et donc, l'élément maximal de $\E{\sigma, k}$ ne peut être que la permutation $\sigma\W_{\sigma, k}$, c'est-à-dire la permutation où l'on a appliqué toutes les transpositions. C'est la seule permutation de longueur maximale $\ell(\sigma) + |\W_{\sigma,k}|$. La démonstration du théorème \ref{thm:polynomes_grothendieck:interval} sera donc achevée par la preuve du lemme suivant.

\begin{Lemme}
\label{lem:polynomes_grothendieck:intervalMax}
\begin{equation}
\label{eq:polynomes_grothendieck:intervalMax}
\forall \mu \in \E{\sigma, k},~~ \mu \leq \eta(\sigma, k)
\end{equation}
où $\eta(\sigma, k) = \sigma \W_{\sigma, k}$.
\end{Lemme}

Nous prouvons ce lemme paragraphe \ref{sub-sec:polynomes_grothendieck:intervalle:preuve}. Nous avons d'abord besoin d'étudier plus précisément la structure de l'ensemble $\E{\sigma, k}$.

\subsection{Sous-mots compatibles}
\label{sub-sec:polynomes_grothendieck:intervalle:compatible}

\begin{Definition}
\label{def:polynomes_grothendieck:compatible}
On dit qu'un sous-mot $w$ de $\W_{\sigma, k}$ n'est pas \emph{compatible} si:
\begin{enumerate}
\item $w$ contient le sous-mot $(a,c)(b,d)$ avec $a<b<c<d$,
\label{enum:polynomes_grothendieck:compatible1}
\item $(a,d) \in \W_{\sigma, k}$ et $(a,d) \notin  w$.
\label{enum:polynomes_grothendieck:compatible2}
\end{enumerate}
Sinon, $w$ est dit compatible.
\end{Definition}

Par exemple, pour $\W_{1372|654} =((2,7),(2,6),(2,5),(4,7),(4,6),(4,5))$, le mot $w = ((2,6),\allowbreak(2,5),(4,7))$ n'est pas compatible. En effet, il contient le sous-mot $(2,6)(4,7)$ mais pas la transposition $(2,7) \in \W_{\sigma, k}$. Notons que si $W_{\sigma, k}$ contient 3 transpositions $(a,d),(a,c)$ et $(b,d)$ elles seront toujours dans l'ordre $(a,d) \tprec (a,c) \tprec (b,d)$. 

Nous donnons à présent la proposition clé dans la preuve du  théorème \ref{thm:polynomes_grothendieck:interval}. Elle permet d'énumérer les permutations de $\E{\sigma, k}$ directement à partir de $\W_{\sigma, k}$ sans tester la compatibilité avec l'ordre de Bruhat.

\begin{Proposition}
\label{prop:polynomes_grothendieck:validSubwords}
Si $w$ est un sous-mot de $\W_{\sigma,k}$ alors $w$ est un sous-mot valide (c'est-à-dire $\sigma w \in \E{\sigma,k}$) si et seulement si $w$ est compatible au sens de la définition~\ref{def:polynomes_grothendieck:compatible}.
\end{Proposition}

La démonstration de cette proposition utilise des propriétés déjà démontrées de $\E{\sigma, k}$. Nous nous servons en particulier de la caractérisation de l'ordre de $k$-Bruhat donnée dans le théorème \ref{thm:polynomes_grothendieck:KBruhat} ainsi que du lemme \ref{lem:polynomes_grothendieck:non-inversions} sur la conservation de l'ordre des valeurs dans $\E{\sigma, k}$.

\begin{proof}
Prouvons d'abord que si $w =  \tau_{i_1}  \dots \tau_{i_r}$ est sous-mot non valide de $\W_{\sigma, k}$ alors $w$ est aussi non compatible. Soit $w_1$ le facteur gauche de $w$ de longueur maximale tel que $w_1$ soit un sous-mot valide. Alors $w = w_1 \tau w_1'$ avec $\tau=(b,d) \in \W_{\sigma,k}$ une transposition qui n'est pas de Bruhat pour $\sigma w_1$. Le lemme~\ref{lem:polynomes_grothendieck:non-inversions} nous donne que $\sigma w_1 (b) < \sigma w_1 (d)$. Alors, il existe $c$ tel que $b<c<d$ et 

\begin{equation}
\label{eq:polynomes_grothendieck:proof-valid-subword-conflit}
\begin{array}{ccccc}
 \sigma w_1(b)& < &\sigma w_1 (c)& < &\sigma w_1 (d)\\
 \vgeq        &   &              &   & \vleq         \\
 \sigma (b)   &   &              &   & \sigma(d)
\end{array}
\end{equation}

Tout d'abord, utilisons la caractérisation de l'ordre de $k$-Bruhat du théorème~\ref{thm:polynomes_grothendieck:KBruhat} pour prouver que $c>k$. On sait que $\tau$ est une transposition de Bruhat pour $\sigma$. Cela signifie que $\sigma(c) < \sigma(b)$ ou bien $\sigma(c) > \sigma(d)$. Dans le premier cas, on lis sur \eqref{eq:polynomes_grothendieck:proof-valid-subword-conflit} que $\sigma(c) < \sigma w_1 (c)$ ce qui par $k$-Bruhat nous donne $c\leq k$. Dans ce cas, comme $\sigma(c) < \sigma(b)$, toute transposition $(c,*)$ est supérieure à $\tau = (b,d)$ et donc $\sigma w_1 (c) = \sigma(c)$ ce qui n'est pas possible. On a donc $\sigma(c)>\sigma(d) \geq \sigma w_1 (d) > \sigma w_1 (c)$ ce qui implique $c>k$.

Soit $(a,c)$ la dernière transposition à agir sur la position $c$. On sait que $(a,c)$ existe car $\sigma(c) \neq \sigma w_1(c)$. On a $w_1 = w_2 (a,c) w_2'$. Par ailleurs comme $\sigma(c) > \sigma(d)$ et $(a,c) \prec (b,d)$ alors $a \neq b$ et on a

\begin{equation}
\label{eq:polynomes_grothendieck:proof-valid-subword-main}
\begin{array}{ccccc}
\sigma w_2 (d) & > & \sigma w_2 (a) & > & \sigma w_2 (b) \\
\vleq          &   & \vgeq          &   & \veq \\
\sigma(d) & & \sigma(a) & & \sigma(b)
\end{array}
\end{equation}

Les relations verticales viennent de l'ordre de $k$-Bruhat entre $\sigma$ et $\sigma w_2$. Les relations horizontales découlent de \eqref{eq:polynomes_grothendieck:proof-valid-subword-conflit}. De là, nous affirmons que

\begin{align}
\label{eq:polynomes_grothendieck:proof-valid-subword-values}
& \sigma(c) > \sigma(d) > \sigma (a) > \sigma (b) \text{~ et} \\
\label{eq:polynomes_grothendieck:proof-valid-subword-positions}
& a < b < c < d.
\end{align}

On a déjà prouvé $\sigma(c) > \sigma(d)$ et on peut lire dans \eqref{eq:polynomes_grothendieck:proof-valid-subword-main} que $\sigma(d) > \sigma(a)$. Par ailleurs, $(a,c) \tprec (b,d)$ nous donne $\sigma(a) > \sigma(b)$ et donc \eqref{eq:polynomes_grothendieck:proof-valid-subword-values} est vraie. Pour prouver \eqref{eq:polynomes_grothendieck:proof-valid-subword-positions} il suffit de montrer $a<b$. Comme $(b,d)$ est une transposition de Bruhat pour $\sigma$ alors à cause de \eqref{eq:polynomes_grothendieck:proof-valid-subword-values} on ne peut pas avoir $b<a<d$.

Par \eqref{eq:polynomes_grothendieck:proof-valid-subword-values} et \eqref{eq:polynomes_grothendieck:proof-valid-subword-positions}, on a que $(a,c)(b,d)$ est un sous-mot de $w$ répondant à la condition \ref{enum:polynomes_grothendieck:compatible1} de la définition \ref{def:polynomes_grothendieck:compatible}. \`A présent, $(a,d)$ est aussi une transposition de Bruhat pour $\sigma$. En effet, s'il existe $a < x < d$ tel que $\sigma(a) < \sigma(x) < \sigma(d)$ alors soit $x<c$ et $(a,c)$ n'est pas une transposition de Bruhat, soit $x > c > b$ et $(b,d)$ n'est pas une transposition de Bruhat. Donc $(a,d) \in \W_{\sigma, k}$ et comme  $\sigma w_2 (d) > \sigma(a)$ alors $(a,d) \notin w$ et la condition \ref{enum:polynomes_grothendieck:compatible2} de la définition \ref{def:polynomes_grothendieck:compatible} est aussi satisfaite. On a bien que $w$ n'est pas un sous-mot compatible.

\`A présent, soit $w$ un sous-mot non compatible de $\W_{\sigma, k}$, on prouve que $w$ n'est pas un sous-mot valide. Par définition, $w$ contient au moins un sous-mot $(a,c)(b,d)$ satisfaisant les conditions \eqref{enum:polynomes_grothendieck:compatible1} et \eqref{enum:polynomes_grothendieck:compatible2} de la définition \ref{def:polynomes_grothendieck:compatible}. On choisit celui dont la transposition $(a,c)$ est minimale. On a $w = w_1 (b,d) w_1'$, prouvons que si $w_1$ est un sous-mot valide alors $(b,d)$ n'est pas une transposition de Bruhat pour $\sigma w_1$. On a $(a,c) \in w_1$, c'est-à-dire $w_1 = w_2 (a,c) w_2'$ et l'on montre

\begin{equation}
\begin{array}{ccccc}
\sigma(b)       & < &\sigma(a)       & < & \sigma(d) \\
\veq            &   & \vleq          &   & \vgeq     \\
\sigma w_2 (b)  & < & \sigma w_2 (a) & < & \sigma w_2 (d).
\end{array}
\end{equation}

La relation $\sigma(b) < \sigma(a) < \sigma(d)$ est immédiate car par définition $a<b<d$ et $(a,d) \tprec (b,d)$. Les relations  $\sigma(a) \leq \sigma w_2 (a)$ et $\sigma(d) \geq \sigma w_2(d)$  sont données par $k$-Bruhat. Par ailleurs, toute transposition agissant sur $b$ est supérieure à $(a,c)$ donc la valeur en position $b$ n'a pas été modifiée. Il reste à prouver $\sigma w_2 (a) < \sigma w_2 (d)$. Toute transposition $\tau \in w_2$ est telle que $\tau \tprec (a,d)$. En effet, si $(a,x) \in w$ avec $(a,d) \tprec (a,x) \tprec (a,c)$ on a par le lemme \ref{lem:polynomes_grothendieck:extensions-lineaires2} que $c<x<d$ et donc $(a,x)(b,d)$ satisfait les conditions \eqref{enum:polynomes_grothendieck:compatible1} et \eqref{enum:polynomes_grothendieck:compatible2} de la définition \ref{def:polynomes_grothendieck:compatible}. On a choisi $(a,c)$ minimale donc une telle transposition $(a,x)$ n'existe pas. On peut alors appliquer le lemme~\ref{lem:polynomes_grothendieck:non-inversions} et on a $\sigma w_2 (a) < \sigma w_2 (d)$. Si on pose $\sigma' = \sigma w_2 (a,c)$ on a maintenant

\begin{equation}
\label{eq:proof-valid-subword-conflit2}
b<c<d \text{~ and ~} \sigma'(b) < \sigma'(c) < \sigma'(d).
\end{equation}

Par le lemme \ref{lem:polynomes_grothendieck:non-inversions}, comme $w_2'$ est un chemin valide dont les transpositions sont inférieures à $(b,d)$, cette relation est préservée. La transposition $(b,d)$ n'est dont pas une transposition de Bruhat pour la permutation $\sigma w_1$.
\end{proof}

\subsection{Preuve du résultat}
\label{sub-sec:polynomes_grothendieck:intervalle:preuve}

La proposition \ref{prop:polynomes_grothendieck:validSubwords} nous donne une caractérisation des listes de transpositions apparaissant dans $\E{\sigma, k}$, c'est l'outil essentiel pour démontrer le lemme \ref{lem:polynomes_grothendieck:intervalMax}. 

\begin{proof}[Démonstration du lemme \ref{lem:polynomes_grothendieck:intervalMax}]
Soit $\mu \in \E{\sigma,k}$ tel que $\mu \neq \eta(\sigma, k)$. On prouve qu'il existe $\tilde{\mu} \in \E{\sigma,k}$ tel que $\tilde{\mu}$ soit un successeur direct de $\mu$ pour l'ordre de Bruhat. C'est suffisant pour prouver le lemme \ref{lem:polynomes_grothendieck:intervalMax} car si $\tilde{\mu} \neq \eta(\sigma, \mu)$, on peut à nouveau appliquer l'algorithme pour obtenir un successeur de $\tilde{\mu}$. A chaque étape, la longueur de la permutation augmente de 1 et donc le processus se termine au bout de $\ell(\eta) - \ell(\mu)$ itérations.

Soit $\mu = \sigma w$ avec $w$ un sous-mot valide de $\W_{\sigma, k}$ et $\mu \neq \eta$. On définit alors la transposition $\tau$ comme étant la première transposition de $\W_{\sigma, k}$ à ne pas apparaître dans $w$. C'est-à-dire qu'on a $\tau \in \W_{\sigma,k}$, $\tau \notin w$ et $\tau$ minimale. Cette transposition existe toujours car $\mu \neq \eta$ ce qui signifie que $w \neq \W_{\sigma, k}$. Le chemin $w$ s'écrit alors $w = uv$ avec $u$ un facteur gauche de $\W_{\sigma,k}$ et tel que $\tau$ est supérieure à toutes les transpositions de $u$ et inférieure à toutes celles de $v$. On pose alors $\tilde{w} = u \tau v$. Notons que $u\tau$ est alors aussi un facteur gauche de $\W_{\sigma,k}$ et que $\tilde{w}$ est un sous-mot de $\W_{\sigma,k}$.

Prouvons à présent que $\tilde{w}$ est un sous-mot valide de $\W_{\sigma,k}$. On utilise pour cela la proposition \ref{prop:polynomes_grothendieck:validSubwords}. Supposons que $\tilde{w}$ soit un sous-mot non compatible. On a alors que $\W_{\sigma,k}$ contient $(a,d) \tprec (a,c) \tprec (b,d)$ avec $(a,d) \notin \tilde{w}$ et $(a,c)(b,d)$ un sous-mot de $\tilde{w}$. Ces conditions sont aussi vérifiées pour $w$. En effet, comme $(a,d) \notin \tilde{w}$ alors nécessairement $(a,d) \notin w$. Par ailleurs, comme $u \tau$ est un facteur gauche de $\W_{\sigma,k}$ et $(a,d) \notin u \tau$, on a $\tau \tprec (a,d) \tprec (a,c) \tprec (b,d)$. Le sous-mot $(a,c)(b,d)$ se retrouve donc aussi dans $w$. De là, on déduit que $\tilde{w}$ est compatible, c'est-à-dire valide et on pose $\tilde{\mu} = \sigma \tilde{w} \in \E{\sigma,k}$.

Par ailleurs, $\ell(\tilde{\mu}) = \ell(\sigma) + \vert \tilde{w} \vert = \ell(\sigma) + \vert w \vert +1 = \ell(\mu) + 1$ et $\tilde{\mu} = \sigma u \tau v = \sigma u v v^{-1} \tau v = \mu v^{-1} \tau v = \mu \theta$ où $\theta$ est le conjugué d'une transposition, c'est-à-dire une transposition. De là, $\tilde{\mu}$ est un successeur direct de $\mu$.
\end{proof}

Ci-dessous, un exemple du processus $\mu \rightarrow \tilde{\mu}$ décrit dans la preuve appliqué à $\mu = 1365 \vert 42 \in \E{1362 \vert 54}$ représenté en figure \ref{fig:polynomes_grothendieck:ESigma}.

\begin{align}
\mu = \mu_0 &= (1362 \vert 54)(46)(45) =  1365 \vert 42 \\
\mu_1:= \tilde{\mu}_0 &= (1362 \vert 54)(26)(46)(45) = 1465 \vert 32 \\
\mu_2:= \tilde{\mu}_1 &=  (1362 \vert 54)(26)(25)(46)(45) = 1564 \vert 32 = \eta
\end{align} 

\begin{Remarque}
\label{rqe:polynomes_grothendieck:permutation-pattern}
Si $\W_{\sigma,k}$ contient le sous-mot $(a,d)(a,c)(b,d)$ alors on peut trouver des sous-mots non compatibles. On dit alors que $\W_{\sigma,k}$ contient un \emph{motif de conflit}. En particulier, comme on a $\sigma(b) < \sigma(a) < \sigma(d) < \sigma(c)$, cela implique que $\sigma$ contient le motif de permutation $21\vert43$. Cependant, l'implication inverse n'est pas vraie. Par exemple, la permutation $213 \vert 54$ contient le motif $21\vert43$ mais comme $(1,5)$, $(1,4)$ et $(2,5)$ ne sont pas des transpositions de Bruhat, $\W_{\sigma,k}$ ne contient pas le motif de conflit.
\end{Remarque}

\begin{Remarque}
\label{rqe:polynomes_grothendieck:interval-size}
La proposition \ref{prop:polynomes_grothendieck:validSubwords} nous donne aussi une information sur la taille de l'intervalle. Si $\W_{\sigma,k}$ est de taille $m$ et ne contient pas de motifs de conflit, alors $\vert \E{\sigma,k} \vert = 2^m$. Par ailleurs, le nombre de sous-mots non compatibles liés à un trio donné $(a,d) \prec (a,c) \prec (b,d)$ est $2^{m-3}$. De là, si $\W_{\sigma, k}$ ne contient qu'un seul motif de conflit, on a $\vert \mathfrak{E}_\sigma \vert = 2^m - 2^{m-3}$. Par exemple, $\vert \mathfrak{E}_{1362 \vert 54} \vert = 2^4 - 2^1 = 14$. De façon plus générale, on peut appliquer un algorithme d'inclusion-exclusion. Si le nombre de motifs de conflits est élevé, l'algorithme d'inclusion-exclusion peut prendre beaucoup plus de temps que le calcul direct de $\E{\sigma, k}$. Par exemple, la permutation  $4321 \vert 8765$  contient 36 motifs de conflits ce qui suppose de calculer des milliards d'intersections de ces motifs quand la taille de $\E{4321 \vert 8765}$ n'est que de 6092. Mais dans les cas de permutations de grande taille avec peu de motifs de conflit, l'algorithme est efficace. 
\end{Remarque}

\section{Quelques généralisations}
\label{sec:polynomes_grothendieck:gen}

\subsection{Variation du paramètre $k$}
\label{sub-sec:polynomes_grothendieck:gen:var-k}

Il est intéressant d'étudier la variation de la permutation $\eta(\sigma,k)$ lorsque le paramètre $k$ varie de 1 à $n$. A partir de la proposition \ref{prop:polynomes_grothedieck:ESigma}, on obtient un algorithme qui retourne $\eta(\sigma,k)$ à partir de $\eta(\sigma, k-1)$ et $\sigma$.

\begin{Proposition}
Soit $\sigma \in \Sym{n}$ et $1 \leq k \leq n$. Par convention, on pose $\eta(\sigma,0)= \eta(\sigma,n)  = \sigma$.  Soient $a_1 < a_2 < \dots < a_m < k $ les positions telles que $(a_i,k)$ soit une transposition de Bruhat pour $\sigma$. Soient $ k < b_1 < b_2 < \dots < b_{m'}$ les positions telles que $(k,b_i)$ soit une transposition de Bruhat pour $\sigma$. Alors $\eta(\sigma,k) = \eta(\sigma, k-1)\rho$ où $\rho$ est le cycle $(a_1, \dots, a_m, k, b_1, \dots, b_{m'})$.
\end{Proposition}

\begin{proof}
Dans cette preuve, la liste $\W_{\sigma,k}$ est toujours donnée à commutation près. C'est-à-dire que que nous n'utilisons plus systématiquement l'ordre $\tprec$ mais des extensions linéaires de l'ordre $\ttprec$ défini au paragraphe \ref{sub-sec:polynomes_grothendieck:chaines:ordre}. En effet, nous avons besoin d'ordonner les transpositions $(a,b)$ prioritairement selon $a$ ou $b$ selon la circonstance.

Regardons d'abord les cas extrêmes où $k=1$ ou $k=n$. Si $k=1$, alors $\W_{\sigma,1} = ((1,b_{m'}),\allowbreak(1,b_{m'-1}), \dots, (1,b_1))$ par définition. Le produit des transpositions est égal au cycle $\rho = (1, b_1, \dots,\allowbreak b_{m'})$. On a bien $\eta(\sigma,1) = \sigma \W_{\sigma, 1} = \eta(\sigma, 0) \rho$. Si $k=n$, alors $\W_{\sigma, n-1} = ((a_1,n),(a_2,n), \dots, (a_m,n))$ et $\W_{\sigma,n-1}^{-1}$ est égal au cycle $\rho = (a_1, \dots, a_m, n)$. On a bien $\eta(\sigma,n) = \sigma = \sigma \W_{\sigma, n-1} \W_{\sigma, n-1}^{-1} = \eta(\sigma, n-1)\rho$.

\`A présent, dans le cas où $1 < k < n$, alors $\W_{\sigma, k-1}$ et $\W_{\sigma,k}$ comporte un facteur commun de transpositions qu'on note $w$. Ce sont les transpositions qui n'agissent pas sur $k$. On peut alors écrire $\W_{\sigma, k-1} = w (a_1,k) \dots (a_m,k)$ et $\W_{\sigma, k} = w (k,b_{m'})\dots (k,b_1)$. Ce qui donne, $\W_{\sigma, k} = \W_{\sigma, k-1} (a_m,k) (a_{m-1}, k) \dots (a_1,k)\allowbreak(k,b_{m'})\allowbreak(k,b_{m'-1})\dots(k,b_1) = \W_{\sigma, k-1} \rho$.
\end{proof}

Nous donnons un exemple de cet algorithme appliqué à la permutation $\sigma = 1362547$ figure \ref{fig:polynomes_grothendieck:variation_k}. A chaque étape, $k \rightarrow k+1$ on obtient un cycle qui envoie $\eta(\sigma,k)$ vers $\eta(\sigma,k+1)$. En faisant varier $k$ de $0$ à $n$, on associe donc à la permutation $\sigma$ une décomposition de l'identité en un produit de $n$ cycles. Nous n'avons pas trouvé d'autres exemples où cette décomposition apparaît mais il nous a paru intéressant de la signaler. On trouvera en figure 
\ref{fig:polynomes_grothendiecl:variation_k-3-4} la décomposition de l'unité en produit de cycles associées aux permutations de taille 3 et 4.

La décomposition associée à une permutation $\sigma$, notée $c(\sigma)$, dépend de ses transpositions de Bruhat, c'est-à-dire de ses successeurs. Chaque transposition $(a,b)$ est utilisée deux fois : quand $k=a$ et quand $k=b$. Si $c(\sigma) = (c_1,\dots, c_n)$ où $c_1, \dots, c_n$ sont des cycles, on a $|c_1| + \dots + |c_n| = 2 \times |\Succs(\sigma)| + k$ où $|\Succs(\sigma)|$ est le nombre de successeurs de $\sigma$. Par ailleurs, les cycles vérifient une certaine symétrie : si $i$ apparaît dans le cycle $c_j$ alors $j$ apparaît dans le cycle $c_i$. Et par construction, on a toujours $i \in c_i$. Enfin, si $(i,j)$ est une transposition de Bruhat pour $\sigma$ alors $(\sigma_i,\sigma_j)$ l'est pour $\sigma^{-1}$. De là si pour $\sigma$ on a $c_i = (a_1, \dots, a_m)$ alors pour $\sigma^{-1}$ le cycle $c_{\sigma_i}$ sera composé des valeurs $\sigma(a_1), \dots, \sigma(a_m)$ réordonnées.

\begin{figure}[ht]
\centering

\begin{tabular}{ccl}
$k$  & Permutation  & cycle vers $k+1$    \\
 0   & 1362547      & $(1,2,4)$           \\
 1   & 3261547      & $(1,2,3,5,6)$       \\
 2   & 2651437      & $(2,3,7)$           \\
 3   & 2571436      & $(1,4,5,6)$         \\
 4   & 1574326      & $(2,4,5,7)$         \\
 5   & 1473625      & $(2,4,6,7)$         \\
 6   & 1372654      & $(3,5,6,7)$         \\
 7   & 1362547
\end{tabular}
\caption[Exemple de l'algorithme de variation de $k$ pour $1362547$]{Exemple de l'algorithme de variation de $k$ pour $1362547$, en appliquant le cycle sur la droite de la permutation ligne $k$, on obtient celle de la ligne $k+1$.}
\label{fig:polynomes_grothendieck:variation_k}
\end{figure}

\begin{figure}[ht!]
\centering
\input{includes/figures/variation_k_3-4}
\caption{Décomposition de l'identité en produits de cycles pour les permutations de tailles 3 et 4}
\label{fig:polynomes_grothendiecl:variation_k-3-4}
\end{figure}

\subsection{Autres sous-groupes paraboliques}
\label{sub-sec:polynomes_grothendieck:gen:parabolic-subgroups}
\index{algèbre!de 0-Hecke}
\index{sous-groupes paraboliques}
\index{polynômes!clés}

Un \emph{sous-groupe parabolique} d'un groupe de Coxeter $G$ est un sous-groupe de $G$ engendré par un ensemble restreint des générateurs de $G$. Lorsque l'on calcule $\E{\zeta, k} \pi_{\zeta^{-1} \sigma}$, on applique des opérateurs $\pi_i$ tels que $i \neq k$. Le produit $\pi_{\zeta^{-1} \sigma}$ est un élément du sous-groupe parabolique de l'algèbre de 0-Hecke engendré par $ \lbrace \pi_1, \dots, \pi_{k-1}, \pi_{k+1}, \dots, \pi_n \rbrace$. On peut généraliser ce calcul aux sous-groupes paraboliques où il manque plus d'un seul générateur.

\begin{Definition}
\label{def:polynomes_grothendieck:gen-blocks}
Soit $\sigma \in \Sym{n}$ et $1 \leq k_1 < k_2 < \dots < k_m < n$. Par convention, on posera $k_0 = 0$. On définit le bloc $\bblock_i(\sigma)$ pour $0 < i \leq m$ comme étant le mot $[\sigma(k_{i-1}+1), \sigma(k_{i-1}+2), \dots, \sigma(n)]$ et $\Ww_i$ est la liste de transpositions

\begin{align}
\nonumber
(&(k_{i-1} +1, n), (k_{i-1}+1, n-1), \dots, (k_{i-1}+1, k_i+1),\\
\nonumber
 &(k_{i-1} +2, n), (k_{i-1}+2, n-1), \dots, (k_{i-1}+2, k_i+1),\\
\nonumber
 &\dots,\\
 &(k_i, n), (k_i, n-1), \dots, (k_i, k_i+1)).
\end{align}

La liste $\Ww_{(k_1, \dots, k_m)}$ est la concaténation des listes $\Ww_i$ pour $m \geq i \geq 1$, c'est-à-dire  $\Ww_{(k_1, \dots, k_m)} := (\Ww_m, \dots, \Ww_1)$.
\end{Definition}

Par exemple, si $\sigma = 43283657$  et $k_1 = 2$, $k_2 = 3$, $k_3 = 6$ (ce qu'on notera $\sigma = 43 \vert 2 \vert 836 \vert 57$), alors $\bblock_1 =  43 \vert 283657$, $\bblock_2 = 2 \vert 83657$, $\bblock_3 = 836 \vert 57$ et $ \Ww_3 = ( (4,8),(4,7), (5,8), (5,7),(6,8),(6,7))$, $\Ww_2 = ((3,8),(3,7),(3,6),(3,5),(3,4))$, $\Ww_1 = ((1,8), (1,7),(1,6),(1,5),(1,4),(1,3),(2,8),(2,7),\allowbreak (2,6),(2,5),(2,4),(2,3))$. 

\begin{Proposition}
\label{prop:polynomes_grothendieck:gen-esigma}
Soit $\sigma \in \Sym{n}$ et $1 \leq k_1 < k_2 < \dots < k_m < n$. Soit $\zeta = \zeta(\sigma, k_1, \dots, k_m)$ l'élément maximal du coset $\sigma (\Sym{k_1} \times \Sym{k_2 - k_1} \times \Sym{k_3 - k_2} \times \dots \times \Sym{n - k_m})$. Alors

\begin{equation}
\label{eq:polynomes_grothendieck:gen-esigma}
K_{\omega} \hat{\pi}_{\omega \zeta} \pi_{\zeta^{-1} \sigma} = \E{\sigma, (k_1, \dots, k_m)}
\end{equation}
où 
\begin{equation}
\E{\sigma, (k_1, \dots, k_m)}  = \sum_{w} (-1)^{\vert w \vert} K_{\sigma w}
\end{equation}
sommé sur les sous-mots $w$ de la liste $\Ww_{(k_1, \dots, k_m)}$ qui correspondent à des chemins dans l'ordre de Bruhat à partir de $\sigma$.

Par ailleurs, la somme n'a pas d'annulations, les coefficients sont $\pm 1$ et l'ensemble de sommation est clos par intervalle.
\end{Proposition}

Remarquons que dans le cas général, la liste $\Ww_{(k_1,\dots,k_m)}$ ne dépend pas de la permutation $\sigma$. Quand $m=1$, la liste $\Ww_{k}$ correspond à la liste de $k$-transposition donnée dans le théorème \ref{thm:polynomes_grothendieck:LenartPostnikov} de Lenart et Postnikov à commutations près (on utilise une autre extension linéaire de $\ttprec$). On a vu dans le lemme \ref{lem:polynomes_grothendieck:equiv-lenart} que cette liste pouvait être réduite à $\W_{\sigma, k}$. C'est grâce à cette réduction que l'on obtient les caractéristiques très particulières de $\E{\sigma,k}$ qui nous ont permis de prouver le théorème \ref{thm:polynomes_grothendieck:interval}. Dans le cas général, une telle réduction n'est pas possible. L'ensemble de sommation, bien que clos par intervalle, ne comporte pas un unique élément maximal. Par exemple, pour $\sigma = 25 \vert 14 \vert 63$, l'ensemble $\E{25 \vert 14 \vert 63}$ comporte deux éléments de longueur maximale :  $362541 = 251463 (3,6)(4,5)(1,3)(2,4)$ et $461532 = 251463 (4,5)(1,6)(1,5)(2,4)$.

Prouvons d'abord la deuxième partie de la proposition, c'est-à-dire que la somme est sans annulation, avec des coefficients $\pm 1$ et sur un ensemble clos par intervalle. On énonce pour cela une généralisation de la proposition \ref{prop:polynomes_grothendieck:complement}.

\begin{Proposition}
\label{prop:polynomes_grothendieck:gen-complement}
Soit $\Succs_{\sigma}^{-(k_1, \dots, k_m)}$ l'ensemble des successeurs $\sigma' = \sigma \tau$ de $\sigma$ tels que $\tau$ ne soit pas une $k_i$-transpositions pour $1 \leq i \leq m$. On a alors

\begin{equation}
K_{\omega} \hat{\pi}_{\omega \zeta} \pi_{\zeta^{-1} \sigma}  = \sum (-1)^{\ell(\nu) - \ell(\sigma)} K_\nu
\end{equation}
sommé sur les permutations $\nu$ telle que $\nu \geq \sigma$ et $\forall \sigma' \in \Succs_{\sigma}^{-(k_1, \dots, k_m)}$, $\nu \ngeq \sigma'$.

En particulier, l'ensemble de sommation est clos par intervalle.
\end{Proposition}

\begin{proof}
De façon similaire au cas $m=1$, on obtient que

\begin{equation}
K_{\omega} \hpi_{\omega \zeta} \pi_{\zeta^{-1} \sigma} = \sum_{\zeta \geq \mu \geq \sigma} \hK_\mu.
\end{equation}
La preuve est similaire à celle de la proposition \ref{prop:polynomes_grothendieck:HatK}. Il suffit de remarquer que $[\sigma, \zeta]$ est inclus dans le coset $\sigma(\Sym{k_1} \times \Sym{k_2 - k_1} \times \dots \Sym{n - k_m})$ et que le calcul peut donc s'effectuer dans l'espace vectoriel engendré par $(K_\mu)_{\mu \in \Sym{k_1}} \otimes (K_\mu)_{\mu \in \Sym{k_2 - k_1}} \otimes \dots (K_\mu)_{\mu \in \Sym{n - k_m}}$. Par un changement de base, on obtient alors le résultat voulu avec un argument similaire à celui de la proposition \ref{prop:polynomes_grothendieck:complement}. On utilise cette fois le lemme généralisé \ref{lem:prelim_groupe_sym:bruhat:coset_gen} sur l'intersection d'un coset et d'un intervalle. 
\end{proof}

\begin{figure}[t]
\centering
\input{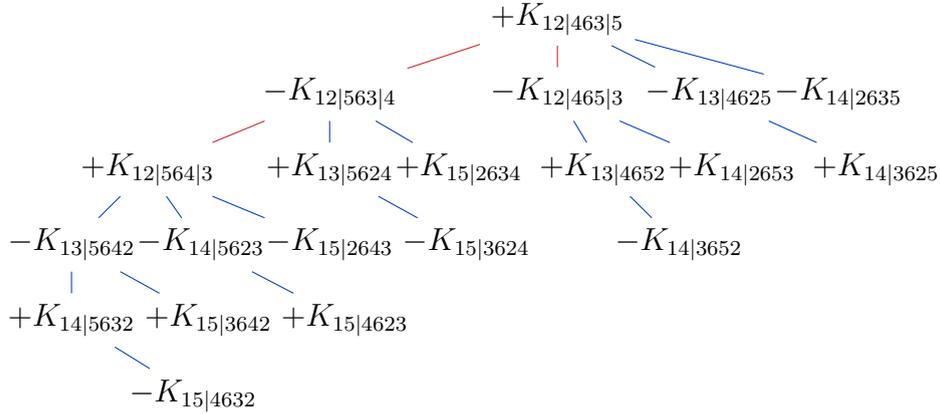}
\caption[L'ensemble $\E{12 \vert 463 \vert 5}$]{L'ensemble $\E{12 \vert 463 \vert 5}$, le calcul en rouge correspond à l'application des sous mots de $\Ww_2$ sur le bloc $\bblock_2 = 463 \vert 5$. Sur chaque élément, on applique ensuite (en bleu) les sous-mots de $\Ww_1$.}
\label{fig:polynomes_grothendieck:ESigmaMultiK}
\end{figure}

\begin{proof}[Démonstration de la proposition \ref{prop:polynomes_grothendieck:gen-esigma} ]
La preuve se fait par récurrence sur $m$. Le cas $m=1$ correspond à la proposition \ref{prop:polynomes_grothedieck:ESigma}. \`A présent, soit $\zeta'$ l'élément maximal de  $\sigma (\Sym{k_1} \times \Sym{k_2 - k_1} \times \dots \times \Sym{n - k_{m-1}})$. Les coinversions de $\zeta'$ contiennent les coinversions de $\zeta$ et donc $\zeta' \geq \zeta$ pour l'ordre faible droit. Cela signifie qu'une décomposition réduite de $\omega \zeta'$ est un préfixe d'une décomposition réduite de $\omega \zeta$ et on a

\begin{equation}
K_{\omega} \hat{\pi}_{\omega \zeta} \pi_{\zeta^{-1} \sigma} = K_{\omega} \hat{\pi}_{\omega \zeta'} \hat{\pi}_{\zeta'^{-1} \zeta} \pi_{\zeta^{-1} \sigma}.
\end{equation}

Tous les opérateurs $\hpi_i$ dans $\hpi_{\zeta'^{-1} \zeta}$ sont tels que $i > k_{m-1}$. Par ailleurs, il n'y a pas d'opérateur $\pi_{k_{m-1}}$ dans le produit $\pi_{\zeta^{-1} \sigma}$ et donc les opérateurs $\pi_i$ avec $i<k_{m-1}$ commutent avec les opérateurs $\pi_i$ avec $i>k_{m-1}$ ainsi qu'avec les opérateurs $\hpi_{\zeta'^{-1} \zeta}$. On peut donc écrire

\begin{equation}
K_{\omega} \hat{\pi}_{\omega \zeta} \pi_{\zeta^{-1} \sigma} = (K_{\omega} \hat{\pi}_{\omega \zeta'}  \pi_{\zeta^{-1} \sigma}^{(<k_{m-1})})(\hat{\pi}_{\zeta'^{-1} \zeta} \pi_{\zeta^{-1} \sigma}^{(>k_{m-1})}).
\end{equation}

Soit $\tilde{\sigma}$ la permutation donnée par $[\sigma(1), \sigma(2),\dots, \sigma(k_{m-1}), \zeta'(k_{m-1} +1), \zeta'(k_{m-1}+2), \dots, \zeta'(n)]$. On a alors $\pi_{\zeta^{-1} \sigma}^{(<k_{m-1})} = \pi_{\zeta'^{-1} \tilde{\sigma}}$. Par ailleurs, le produit $(\hpi_{\zeta'^{-1} \zeta} \pi_{\zeta^{-1} \sigma}^{(>k_{m-1})})$ agit seulement sur le bloc $\bblock_m(\sigma)$. On peut ignorer le facteur gauche de taille $k_{m-1}$ de $\sigma$ pour utiliser la proposition \ref{prop:polynomes_grothedieck:ESigma}. On a alors

\begin{equation}
K_{\tilde{\sigma}} \hpi_{\zeta'^{-1} \zeta} \pi_{\zeta^{-1} \sigma}^{(>k_{m-1})} = \sum_u (-1)^{|u|} K_{\sigma u}
\end{equation}
sommé sur les sous-mots $u$ de $\Ww_m$ valides en partant de $\sigma$. Si l'on exprime ce résultat en termes d'opérateurs, cela donne

\begin{equation}
 \hpi_{\zeta'^{-1} \zeta} \pi_{\zeta^{-1} \sigma}^{(>k_{m-1})} = \sum_u (-1)^{|u|} \pi_{\tilde{\sigma}^{-1} \sigma u} .
\end{equation}

C'est une somme sur les chemins entre $\tilde{\sigma}$ (qui sur le bloc $\bblock_m$ correspond à la permutation maximale) et les permutations $\sigma u$. On a donc

\begin{equation}
K_{\omega} \hat{\pi}_{\omega \zeta} \pi_{\zeta^{-1} \sigma} = \sum_{u} (-1)^{\vert u \vert} K_\omega \hat{\pi}_{\omega \zeta'} \pi_{\zeta'^{-1} \tilde{\sigma}} \pi_{\tilde{\sigma}^{-1} \sigma u}
\end{equation}
sommé sur les sous-mots valides de $\Ww_m$. Comme $\pi_{\zeta'^{-1} \tilde{\sigma}}$ et $ \pi_{\tilde{\sigma}^{-1} \sigma u}$ contiennent seulement des opérateurs $\pi_i$ avec respectivement $i<k_{m-1}$ et $i>k_{m-1}$, leur produit est toujours un produit réduit et donc

\begin{equation}
K_{\omega} \hat{\pi}_{\omega \zeta} \pi_{\zeta^{-1} \sigma} = \sum_{u} (-1)^{\vert u \vert} K_\omega \hat{\pi}_{\omega \zeta'} \pi_{\zeta'^{-1} \sigma u}.
\end{equation}
ce qui donne \eqref{eq:polynomes_grothendieck:gen-esigma} par récurrence.
\end{proof}

Par exemple, voyons le calcul détaillé pour $\sigma = 12 \vert 463 \vert 5$ ($k_1 = 2$, $k_2 = 5$). On a $\zeta = 21 \vert 643 \vert 5$ et $\zeta' = 21 \vert 6543$. Le résultat est illustré figure \ref{fig:polynomes_grothendieck:ESigmaMultiK}.

\begin{align}
K_{\omega} \hat{\pi}_{\omega \zeta} \pi_{\zeta^{-1} \sigma}  &= K_{654321} \hat{\pi}_4  \hat{\pi}_3  \hat{\pi}_2  \hat{\pi}_1 \hat{\pi}_5 \hat{\pi}_4 \hat{\pi}_3 \hat{\pi}_2 \hat{\pi}_4  \hat{\pi}_5~ \pi_1 \pi_3 \\
&= (K_{654321} \hat{\pi}_4  \hat{\pi}_3  \hat{\pi}_2  \hat{\pi}_1 \hat{\pi}_5 \hat{\pi}_4 \hat{\pi}_3 \hat{\pi}_2 ~ \pi_1)  (\hat{\pi}_4  \hat{\pi}_5 ~ \pi_3) \\
&= \hat{K}_{216543} \pi_1 (\pi_4 \pi_5 \pi_3 - \pi_4\pi_3 - \pi_5\pi_3 + \pi_3) \\
&= \mathfrak{E}_{12 \vert 4635} - \mathfrak{E}_{12 \vert 4653} - \mathfrak{E}_{12 \vert 5634} + \mathfrak{E}_{12 \vert 5643}
\end{align}

Nous ne connaissons pas d'interprétations géométriques de ce calcul en termes de produit de polynômes. Cependant, le résultat est intéressant en tant que tel soit comme un développement dans l'algèbre 0-Hecke soit comme application de ces opérateurs sur les polynômes clés ou les polynômes de Grothendieck.


\chapter{Implantation des bases des polynômes en Sage}
\label{chap:polynomes_sage}

Créé en 2005, le logiciel Sage  \cite{SAGE_WEBSITE} a pris une place importante au sein des systèmes de calcul formel pour la recherche en mathématique et plus particulièrement en combinatoire. Il a pour but de proposer une alternative libre et gratuite aux logiciels traditionnels tels que Maple ou Magma. Sous licence GPL, il s'ajoute à la grande famille des logiciels libres développés dans la philosophie du système d'exploitation Linux. Le développement libre n'est pas nouveau en mathématiques et Sage se propose aussi de réunir au sein d'un même logiciel les outils des différentes communautés. On y trouve entre autres des programmes tels que GAP et Symmetrica. Son modèle de développement est décentralisé, basé sur la communauté. Chaque utilisateur peut devenir développeur et proposer des modifications, corrections ou ajouts au logiciel. Ces propositions sont envoyées sous formes de \emph{patchs} et validées par d'autres développeurs avant d'être incorporées à une nouvelle version du logiciel. Dans le domaine du calcul formel traditionnel, Sage a atteint un niveau comparable aux logiciels classiques. Il a été ajouté récemment à la liste des logiciels de calculs formels admis au concours de l'agrégation. On pourra lire \cite{SAGE_BOOK} pour se familiariser avec ses principales fonctionnalités.

Le modèle de développement de Sage trouve un intérêt particulier en combinatoire algébrique. La recherche en combinatoire est souvent basée sur une exploration préliminaire par le calcul informatique. Les outils nécessaires font appel à de nombreuses branches des mathématiques : algèbre linéaire, théorie des groupes, théorie des représentations... Le développement d'une base commune par les différentes communautés est donc un atout majeur. Avant même l'arrivée de Sage, un projet \emph{Mupad-Combinat} avait été mis en place au sein du logiciel Mupad. Cependant, les possibilités de développement ont été limitées par les contraintes liées au logiciel, en particulier car il était propriétaire. En 2008, la communauté \emph{combinat} décide de rejoindre Sage devenant alors \emph{Sage-combinat} \cite{SAGE_COMBINAT}. Concrètement, Sage-combinat est un ensemble de patchs qui visent à améliorer l'intégration de la combinatoire dans Sage. Sage-combinat est aussi et surtout une communauté de chercheurs / développeurs qui créent, testent et intègrent ces patchs.

Il existait déjà dans Sage une implantation classique des polynômes multivariés en tant qu'expressions formelles en plusieurs variables. Cette implantation ne répond pas aux besoins de calculs sur des bases multiples telles que Schubert, Grothendieck ou les polynômes clés. En effet, on veut pouvoir travailler formellement sur les éléments de ces bases, leur appliquer des opérateurs, les convertir d'une base à l'autre et les développer sur la base des monômes seulement si nécessaire. Pour ce faire, nous considérons un polynôme dans une base quelconque comme une somme formelle de vecteurs. L'interprétation de ce vecteur est donné par la base : sur les monômes le vecteur est un exposant, sur les autres bases c'est ce que nous avons décrit dans le paragraphe \ref{sec:polynomes_action:pol}. Cela nous permet aussi de travailler en un nombre quelconque de variables, les variables n'étant plus données chacune comme des éléments formels mais seulement par la taille du vecteur. Cette approche n'est pas nouvelle, on la trouve en particulier au sein du logiciel ACE développé sous Maple \cite{VEIGN} ou en partie dans le logiciel Symmetrica. Cependant, il n'existait pas encore de version en Sage, et donc accessible à tous, d'une telle implantation. Ce faisant, nous nous inscrivons ici au sein d'un projet plus important de développement de la combinatoire dans Sage. Notre implantation a fait l'objet de plusieurs présentation et d'une publication \cite{Me_Polynomes}.

Outre les aspects philosophiques qui nous ont fait choisir un logiciel libre plutôt que propriétaire comme Maple, le choix de Sage est aussi d'ordre technique. Sage est développé en python et utilise une architecture orientée objet. Toute notre implantation est basée sur cette architecture, utilisant de façon avancée l'héritage et la classification. En particulier, nous utilisons le modèle de développement en \emph{Catégories} / \emph{Parents} / \emph{\'Eléments} mis en place par Nicolas Thiéry \cite{SAGE_categories}. Par ailleurs, nous avons pu profiter des implantations déjà existantes en Sage comme les groupes de Coxeter ou les modules libres. 

Dans le paragraphe \ref{sec:polynomes_sage:demo}, nous commençons par présenter les fonctionnalités de base du logiciel : création d'un polynôme, application d'un opérateur, changement de base, etc. Le paragraphe \ref{sec:polynomes_sage:archi} nous sert à expliquer l'architecture globale du projet et nos choix d'implantation. Enfin, paragraphe \ref{sec:polynomes_sage:appli}, nous donnons des exemples d'applications avancées, et en particulier, les calculs du chapitre \ref{chap:polynomes_grothendieck}. 

\section{Fonctionnalités, exemples d'utilisation}
\label{sec:polynomes_sage:demo}

\subsection{Installation et distribution du logiciel}
\label{sub-sec:polynomes_sage:demo:install}

Le logiciel est disponible en tant que patch additionnel à Sage \cite{SAGE_Polynomials} . On peut le télécharger sur sa page dédiée : \url{http://trac.sagemath.org/sage_trac/ticket/6629} et l'installer. Les étapes sont les suivantes :
\begin{enumerate}
    \item Installer Sage.
    \item Télécharger le patch contenant le programme.
    \item Effectuer une copie de votre version de sage avec la commande 
    
    {\tt sage -clone polynomials}.
    \item Lancer sage puis sur la line de commande sage, taper 
    
    {\tt hg\_sage.apply("nom\_du\_patch.patch")}.
    \item Quitter sage (quit) puis lancer la commande 
    
    {\tt sage -br}.
\end{enumerate} 

Le patch est en cours de révision. Il doit être approuvé par d'autres utilisateurs / développeurs et sera ensuite intégré au logiciel Sage. Il n'y aura alors plus d'installation spécifique et les fonctionnalités seront présentes sur l'installation par défaut de Sage. 

Il est aussi possible de bénéficier du patch en installant la suite complète des patchs de \emph{sage-combinat} par la commande {\tt sage - combinat install}. Ces patchs sont expérimentaux et n'ont pas encore fait l'objet d'un processus de validation.

\subsection{Création d'un polynôme, application d'opérateurs}
\label{sub-sec:polynomes_sage:demo:pol}

On crée d'abord l'objet de base représentant l'algèbre des polynômes.
\begin{lstlisting}
sage: A = AbstractPolynomialRing(QQ)
sage: A
The abstract ring of multivariate polynomials on x over Rational Field
\end{lstlisting}
On définit ensuite la base des monômes qui va nous servir à créer des éléments.
\begin{lstlisting}
sage: m = A.monomial_basis(); m
The ring of multivariate polynomials on x over Rational Field on the monomial basis
\end{lstlisting}
On peut alors créer un polynôme à partir de {\tt m}.
\begin{lstlisting}
sage: pol = m[1,1,2] + m[2,3]; pol
x[1, 1, 2] + x[2, 3, 0]
\end{lstlisting}
L'élément {\tt x[1,1,2]} signifie $x^{(1,1,2)}=x_1^1x_2^1x_3^2$ comme on peut le voir en transformant le polynôme en expression symbolique.
\begin{lstlisting}
sage: pol.to_expr()
x1^2*x2^3 + x1*x2*x3^2
\end{lstlisting}
Il n'est pas nécessaire de préciser à l'avance le nombre de variables : il sera calculé en fonction de la taille du vecteur. L'objet polynôme connaît son nombre de variable à travers son objet \emph{parent}.
\begin{lstlisting}
sage: pol.parent()
The ring of multivariate polynomials on x over Rational Field with 3 variables on the monomial basis
sage: pol = pol.change_nb_variables(4)
sage: pol
x[1, 1, 2, 0] + x[2, 3, 0, 0]
sage: pol.parent()
The ring of multivariate polynomials on x over Rational Field with 4 variables on the monomial basis
\end{lstlisting}
Un polynôme est toujours vu comme une somme formelle de vecteurs. Il ne peut pas être factorisé. Si l'on multiplie deux polynômes, le résultat sera toujours donné sous forme développée.
\begin{lstlisting}
sage: pol * pol
x[2, 2, 4, 0] + 2*x[3, 4, 2, 0] + x[4, 6, 0, 0]
\end{lstlisting}

On peut à présent appliquer des opérateurs de différences divisées à un polynôme. Par exemple, on trouve ci-dessous le calcul \eqref{eq:polynomes_action:exemple_diffdiv}.
\begin{lstlisting}
sage: pol = m[5,1]; pol
x[5, 1]
sage: pol.divided_difference(1)
x[1, 4] + x[2, 3] + x[3, 2] + x[4, 1]
sage: pol.isobaric_divided_difference(1)
x[1, 5] + x[2, 4] + x[3, 3] + x[4, 2] + x[5, 1]
sage: pol.hat_isobaric_divided_difference(1)
x[1, 5] + x[2, 4] + x[3, 3] + x[4, 2]
\end{lstlisting} 
Même si le polynôme n'est défini que sur deux variables, on peut appliquer une différence divisée qui fait intervenir la troisième variable.
\begin{lstlisting}
sage: pol.isobaric_divided_difference(2)
x[5, 1, 0] + x[5, 0, 1]
\end{lstlisting}
Certaines méthodes permettent d'appliquer une série d'opérations.
\begin{lstlisting}
sage: pol = m[1,4,2] + m[5,5,1]
sage: pol.apply_reduced_word([1,2])
-x[1, 2, 2] + x[3, 1, 1]
sage: pol.apply_reduced_word([1,2],method="pi")
-x[2, 2, 3] - x[2, 3, 2] - x[3, 2, 2] + x[5, 1, 5] + x[5, 2, 4] + x[5, 3, 3] + x[5, 4, 2] + x[5, 5, 1]
sage: pol.apply_reduced_word([1,2],method="hatpi")
-x[1, 2, 4] - x[1, 3, 3] - x[2, 2, 3]
\end{lstlisting}
Ici on a appliqué au polynôme les opérateurs respectifs $\partial_1 \partial_2$, $\pi_1 \pi_2$ et $\hpi_1 \hpi_2$. On peut par exemple vérifier la relation de tresse.
\begin{lstlisting}
sage: pol.apply_reduced_word([1,2,1])
x[1, 1, 2] + x[1, 2, 1] + x[2, 1, 1]
sage: pol.apply_reduced_word([2,1,2])             
x[1, 1, 2] + x[1, 2, 1] + x[2, 1, 1]
\end{lstlisting}
Il est aussi possible de mélanger différents types d'opérateurs. Voici la commande pour appliquer $\partial_2$ puis $\pi_1$.
\begin{lstlisting}
sage: pol.apply_composed_morphism([("d",2),("pi",1)])
x[1, 5, 4] - x[2, 2, 2] + x[2, 4, 4] + x[2, 5, 3] + x[3, 3, 4] + x[3, 4, 3] + x[3, 5, 2] + x[4, 2, 4] + x[4, 3, 3] + x[4, 4, 2] + x[4, 5, 1] + x[5, 1, 4] + x[5, 2, 3] + x[5, 3, 2] + x[5, 4, 1]
\end{lstlisting}
Enfin, on peut appliquer des opérateurs de type $B$, $C$ ou $D$ que nous avons vus paragraphe \ref{sub-sec:polynomes_action:ope:bcd}.
\begin{lstlisting}
pol.divided_difference(1,"B")
x[-5, 5, 1] + x[-4, 5, 1] + x[-3, 5, 1] + x[-2, 5, 1] + x[-1, 4, 2] + x[-1, 5, 1] + x[1, 5, 1] + x[2, 5, 1] + x[3, 5, 1] + x[4, 5, 1] + x[0, 4, 2] + x[0, 5, 1]
sage: pol.divided_difference(1,"C")
x[-4, 5, 1] + x[-2, 5, 1] + x[2, 5, 1] + x[4, 5, 1] + x[0, 4, 2] + x[0, 5, 1]
sage: pol.divided_difference(2,"D")
x[-4, -5, 1] + x[-3, -4, 1] + x[-3, -1, 2] + x[-2, -3, 1] + x[-2, 0, 2] + x[-1, -2, 1] + x[-1, 1, 2] + x[1, 3, 2] + x[1, 0, 1] + x[2, 1, 1] + x[3, 2, 1] + x[4, 3, 1] + x[5, 4, 1] + x[0, -1, 1] + x[0, 2, 2]
\end{lstlisting}

Quelque soit le nombre de variables du polynôme, les opérateurs de types $B$, $C$ et $D$ sont toujours respectivement $\partial_i^B$, $\partial_i^C$ et $\partial_i^D$ définis en \eqref{eq:polynomes_action:def-diffdiv-B}, \eqref{eq:polynomes_action:def-diffdiv-B} e \eqref{eq:polynomes_action:def-diffdiv-B}, même quand $i \neq n$. On peut mélanger des opérateurs de types différents, par exemple, ici $\partial_1 \pi_1^B \pi_1^C$.
\begin{lstlisting}
sage: pol.apply_composed_morphism([("d",1),("pi",1,"B"),("pi",2,"C")])
-x[-3, -1, 2] - x[-3, 1, 2] - x[-2, -2, 2] - x[-2, -1, 2] - x[-2, 1, 2] - x[-2, 2, 2] - x[-2, 0, 2] - x[-1, -3, 2] - x[-1, -2, 2] - 2*x[-1, -1, 2] - 2*x[-1, 1, 2] - x[-1, 2, 2] - x[-1, 3, 2] - x[-1, 0, 2] - x[1, -3, 2] - x[1, -2, 2] - 2*x[1, -1, 2] - 2*x[1, 1, 2] - x[1, 2, 2] - x[1, 3, 2] - x[1, 0, 2] - x[2, -2, 2] - x[2, -1, 2] - x[2, 1, 2] - x[2, 2, 2] - x[2, 0, 2] - x[3, -1, 2] - x[3, 1, 2] - x[0, -3, 2] - x[0, -2, 2] - 2*x[0, -1, 2] - 2*x[0, 1, 2] - x[0, 2, 2] - x[0, 3, 2] - x[0, 0, 2]
\end{lstlisting}
Cependant, il est possible d'indexer les monômes directement par des éléments de l'espace ambiant du système de racines d'un groupe de Coxeter. Dans ce cas, le type du groupe est contenu dans l'objet polynôme et les opérateurs utilisés sont ceux du groupe.
\begin{lstlisting}
sage: mb = A.ambient_space_basis("B")
sage: mb
The ring of multivariate polynomials on x over Rational Field on the Ambient space basis of type B
sage: polb = mb[1,4,2] + mb[5,5,1]
sage: polb.group_type()
'B'
sage: polb.divided_difference(1)
-x(1, 3, 2) - x(2, 2, 2) - x(3, 1, 2)
sage: polb.divided_difference(3)
x(1, 4, 0) + x(1, 4, -2) + x(1, 4, -1) + x(1, 4, 1) + x(5, 5, 0) + x(5, 5, -1)
\end{lstlisting}
Ici, la première opération réalisée est la différence divisée classique $\partial_1$ et la seconde est $\partial_3^B$. En effet comme on l'a vu paragraphe \ref{sub-sec:polynomes_action:ope:bcd}, en type $B$, seule la racine simple $r_n$ est différente des racines de type $A$.

Il est aussi possible de définir soi-même un opérateur agissant sur les exposants.
\begin{lstlisting}
sage: def affine(self,key): return self.divided_difference_on_basis(key) - self.si_on_basis(key)
 sage: m.add_operator("a",affine)
sage: pol
x[1, 4, 2] + x[5, 5, 1]
sage: pol.apply_morphism(1,method="a")
-x[1, 3, 2] - x[2, 2, 2] - x[3, 1, 2] - x[4, 1, 2] - x[5, 5, 1]
sage: pol.divided_difference(1) - pol.si(1)
-x[1, 3, 2] - x[2, 2, 2] - x[3, 1, 2] - x[4, 1, 2] - x[5, 5, 1]
\end{lstlisting}
Ici, on a défini la méthode {\tt affine} comme la différence entre la différence divisée et l'opérateur $s_i$ qui échange deux variables. En l'ajoutant à la base $m$, on a créé une  famille d'opérateurs $a_i = \partial_i - s_i$. On peut par exemple vérifier que ce nouvel opérateur vérifie les relations de tresses.
\begin{lstlisting}
sage: pol.apply_reduced_word([1,2,1],method="a") == pol.apply_reduced_word([2,1,2],method="a") 
True
\end{lstlisting}

\subsection{Schubert, Grothendieck, polynômes clés et autres bases}
\label{sub-sec:polynomes_sage:demo:bases}

On a déjà montré différentes bases dans les exemples précédents. Les monômes peuvent être indexés soit par des vecteurs, soit par des éléments de l'espace ambiant d'un système de racine. 
\begin{lstlisting}
sage: A = AbstractPolynomialRing(QQ)
sage: m = A.monomial_basis(); m
The ring of multivariate polynomials on x over Rational Field on the monomial basis
sage: ma = A.ambient_space_basis("A"); ma
The ring of multivariate polynomials on x over Rational Field on the Ambient space basis of type A
sage: mb = A.ambient_space_basis("B"); mb
The ring of multivariate polynomials on x over Rational Field on the Ambient space basis of type B
sage: mc = A.ambient_space_basis("C"); mc
The ring of multivariate polynomials on x over Rational Field on the Ambient space basis of type C
sage: md = A.ambient_space_basis("D"); md
The ring of multivariate polynomials on x over Rational Field on the Ambient space basis of type D
sage: pol = m[2,2,3]; pol
x[2, 2, 3]
sage: mb(pol)
x(2, 2, 3)
\end{lstlisting}

Les autres bases sont définies à partir de l'algorithme qui développe un vecteur en somme de monômes. Voici par exemple les polynômes de Schubert.
\begin{lstlisting}
sage: Schub = A.schubert_basis_on_vectors()
sage: Schub
The ring of multivariate polynomials on x over Rational Field on the Schubert basis of type A (indexed by vectors)
sage: pol = Schub[2,1,2] + Schub[1,3,2]
sage: pol
Y(1, 3, 2) + Y(2, 1, 2)
sage: pol.expand()
x(1, 3, 2) + x(2, 1, 2) + x(2, 2, 1) + x(2, 2, 2) + x(2, 3, 1) + x(3, 1, 1) + x(3, 1, 2) + x(3, 2, 1)
\end{lstlisting}
La méthode {\tt expand} utilise l'algorithme que nous avons vu au paragraphe \ref{sub-sec:polynomes_action:pol:schub}. \`A partir du développement d'un polynôme de Schubert, des méthodes préexistantes en Sage permettent d'obtenir le morphisme inverse d'une somme de monômes vers un polynôme de Schubert.
\begin{lstlisting}
sage: Schub(m[2,4,1] + m[5,5,2])
Y(2, 4, 1) - Y(3, 3, 1) - Y(4, 2, 1) + Y(5, 5, 2)
sage: pol^2       
Y(2, 6, 4) + 2*Y(3, 4, 4) + 2*Y(3, 5, 3) + Y(3, 5, 4) + Y(3, 6, 3) + Y(4, 2, 4) + Y(4, 3, 3) + 2*Y(4, 4, 3) + Y(4, 4, 4) + 2*Y(4, 5, 2) + Y(4, 5, 3) + Y(5, 2, 3) + 2*Y(5, 2, 4) + 2*Y(5, 3, 3) + Y(5, 5, 2) + Y(6, 2, 2) + 2*Y(6, 2, 3)
\end{lstlisting}
Pour effectuer la multiplication, le programme développe les deux polynômes dans la base des monômes et transforme ensuite le résultat en polynôme de Schubert. On peut aussi appliquer les opérateurs de différences divisées directement aux polynômes de Schubert.
\begin{lstlisting}
sage: pol  
Y(1, 3, 2) + Y(2, 1, 2)
sage: pol.divided_difference(1)
Y(1, 1, 2)
sage: pol.isobaric_divided_difference(1)
Y(1, 2, 2) + Y(1, 3, 1) + Y(1, 3, 2)
sage: pol.hat_isobaric_divided_difference(1)
Y(1, 2, 2) + Y(1, 3, 1) - Y(2, 1, 2)
\end{lstlisting}
Par défaut, le programme applique l'opérateur sur le polynôme développé avant de l'exprimer dans la base des Schubert. Cependant, si une méthode a été implantée directement pour les polynômes de Schubert, c'est elle qui sera utilisée. Ici, c'est le cas de la différence divisée $\partial_i$. 

Pour les polynômes de Grothendieck simples, on a implanté deux bases : celle où les variables $y_i$ sont spécialisées à 1 et où les exposants des $x$ sont négatifs et celle où un changement de variable a été appliqué pour obtenir des exposants positifs (cf. paragraphe \ref{sub-sec:polynomes_action:pol:groth}). 
\begin{lstlisting}
sage: Grothp = A.grothendieck_positive_basis_on_vectors(); Grothp
The ring of multivariate polynomials on x over Rational Field on the Grothendieck basis of type A, with positive exposants (indexed by vectors) 
sage: Grothn = A.grothendieck_negative_basis_on_vectors(); Grothn
The ring of multivariate polynomials on x over Rational Field on the Grothendieck basis of type A with negative exposants (indexed by vectors)
sage: pol1 = Grothp[1,3,2] + Grothp[2,1,2]; pol1
G(1, 3, 2) + G(2, 1, 2)
sage: pol2 = Grothn[1,3,2] + Grothn[2,1,2]; pol2
G(1, 3, 2) + G(2, 1, 2)
sage: pol1.expand()
x(1, 3, 2) + x(2, 1, 2) + x(2, 2, 1) + x(2, 3, 1) - 2*x(2, 3, 2) + x(3, 1, 1) - x(3, 2, 2) - x(3, 3, 1) + x(3, 3, 2)
sage: pol2.expand()
2*x(0, 0, 0) + x(-3, -3, -2) - x(-3, -3, -1) + x(-3, -2, 0) - 2*x(-3, -2, -2) + x(-3, -2, -1) - x(-3, -1, 0) + x(-3, -1, -2) + x(-2, 0, 0) + x(-2, -3, 0) - x(-2, -3, -2) - 5*x(-2, -2, 0) + 5*x(-2, -2, -1) + 3*x(-2, -1, 0) + 2*x(-2, -1, -2) - 5*x(-2, -1, -1) - x(-2, 0, -2) - 3*x(-1, 0, 0) - x(-1, -3, 0) + x(-1, -3, -1) + 4*x(-1, -2, 0) + 3*x(-1, -2, -2) - 7*x(-1, -2, -1) - 4*x(-1, -1, -2) + 4*x(-1, -1, -1) + x(-1, 0, -2) + 2*x(-1, 0, -1) - x(0, -2, -2) + x(0, -2, -1) - 2*x(0, -1, 0) + x(0, -1, -2) + x(0, -1, -1) - 2*x(0, 0, -1)
sage: polexp = pol1.expand()
sage: polexp.subs_var([(i, 1 - A.var(i)^(-1)) for i in xrange(1,4)]) == pol2
True
\end{lstlisting}
Le changement de base de la base des monômes vers les polynômes de Grothendieck n'est définie que pour la base en exposants positifs. 
\begin{lstlisting}
sage: Grothp(m[2,4,1] + m[5,5,2])
G(2, 4, 1) - G(3, 3, 1) + G(3, 4, 1) - G(4, 2, 1) + G(4, 4, 1) + G(5, 5, 2)
sage: pol1^2
G(2, 6, 4) + 2*G(3, 4, 4) + 2*G(3, 5, 3) - G(3, 5, 4) + G(3, 6, 3) - 2*G(3, 6, 4) + G(4, 2, 4) + G(4, 3, 3) - G(4, 3, 4) + 2*G(4, 4, 3) - G(4, 4, 4) + 2*G(4, 5, 2) - 3*G(4, 5, 3) - G(4, 6, 3) + G(4, 6, 4) + G(5, 2, 3) + G(5, 2, 4) + G(5, 3, 3) - 3*G(5, 3, 4) - 2*G(5, 4, 3) + 2*G(5, 4, 4) - G(5, 5, 2) + 2*G(5, 5, 3) - G(5, 5, 4) - G(5, 6, 2) + 2*G(5, 6, 3) - G(5, 6, 4) + G(6, 2, 2) - G(6, 2, 4) - 3*G(6, 3, 3) + 3*G(6, 3, 4) - 2*G(6, 4, 2) + 4*G(6, 4, 3) - 2*G(6, 4, 4) + 2*G(6, 5, 2) - 4*G(6, 5, 3) + 2*G(6, 5, 4)
\end{lstlisting}

Les polynômes clés sont les seuls qui soient implantés pour les types $B$, $C$ et $D$ selon les définitions données au paragraphe \ref{sub-sec:polynomes_action:pol:dem}. On peut vérifier les calculs donnés en exemple dans cette section. 
\begin{lstlisting}
sage: K = A.demazure_basis_on_vectors(); K
The ring of multivariate polynomials on x over Rational Field on the Demazure basis of type A (indexed by vectors)
sage: Kb = A.demazure_basis_on_vectors("B"); Kb
The ring of multivariate polynomials on x over Rational Field on the Demazure basis of type B (indexed by vectors)
sage: Kc = A.demazure_basis_on_vectors("C"); Kc
The ring of multivariate polynomials on x over Rational Field on the Demazure basis of type C (indexed by vectors)
sage: Kd = A.demazure_basis_on_vectors("C"); Kd
The ring of multivariate polynomials on x over Rational Field on the Demazure basis of type C (indexed by vectors)
sage: K[1,0,1].expand()
x(1, 1, 0) + x(1, 0, 1)
sage: Kb[2,-1,1].expand()
x(2, 0, 0) + x(2, -1, 1) + x(2, 1, 0) + x(2, 1, -1) + x(2, 1, 1) + x(2, 0, 1)
sage: Kc[2,-1,1].expand()
x(2, 0, 0) + x(2, -1, 1) + x(2, 1, -1) + x(2, 1, 1)
sage: Kd[2,-2,1].expand()
x(2, -2, 1) + x(2, -1, 0) + x(2, -1, 2) + 2*x(2, 1, 0) + x(2, 1, -2) + x(2, 1, 2) + x(2, 2, -1) + x(2, 2, 1) + x(2, 0, -1) + 2*x(2, 0, 1)
\end{lstlisting}
A chaque fois, le polynôme est développé dans la base {\tt ma}, {\tt mb}, {\tt mc} ou {\tt md}. 
\begin{lstlisting}
sage: Kb[2,-1,1].expand().parent()
The ring of multivariate polynomials on x over Rational Field with 3 variables on the Ambient space basis of type B
sage: Kc[2,-1,1].expand().parent()
The ring of multivariate polynomials on x over Rational Field with 3 variables on the Ambient space basis of type C
\end{lstlisting}

Les changements de bases des polynômes de Schubert, Grothendieck ou des polynômes clés se font non seulement vers la base monomiale mais aussi entre bases. Le passage par la base monomiale est fait de façon implicite par le programme.
\begin{lstlisting}
sage: K(Schub[1,0,1])  
K(1, 0, 1) + K(2, 0, 0)
sage: Schub(Grothp[1,4,2])
Y(1, 4, 2) - 2*Y(2, 4, 2) - Y(3, 4, 1) + Y(3, 4, 2)
\end{lstlisting}

Enfin, il est possible de définir sa propre base. En effet, une base est entièrement donnée par l'algorithme de développement d'un élément. On définit donc simplement cette fonction. Voici par exemple une copie des polynômes de Schubert.
\begin{lstlisting}
sage: def schubert_on_basis(v, basis, call_back):
...     for i in xrange(len(v)-1):
...         if(v[i]<v[i+1]):
...             v[i], v[i+1] = v[i+1] + 1, v[i]
...             return call_back(v).divided_difference(i+1)
...     return basis(v)
\end{lstlisting}
Le paramètre {\tt v} est le vecteur indexant l'élément, {\tt basis} est la base sur laquelle on est en train de développer et {\tt call\_back} est l'appel récursif de la méthode.
\begin{lstlisting}
sage: myBasis = A.linear_basis_on_vectors("A","MySchub","Y",schubert_on_basis)
sage: pol = myBasis[2,1,2] + myBasis[1,3,2]
sage: pol.expand()
x(1, 3, 2) + x(2, 1, 2) + x(2, 2, 1) + x(2, 2, 2) + x(2, 3, 1) + x(3, 1, 1) + x(3, 1, 2) + x(3, 2, 1)
sage: Schub(pol)
Y(1, 3, 2) + Y(2, 1, 2)
\end{lstlisting}

\subsection{Polynômes doubles}
\label{sub-sec:polynomes_sage:demo:double}

Un polynôme en deux ensembles de variables $x$ et $y$ est un polynôme en $x$ dont les coefficients sont des polynômes en $y$. Nous avons défini une nouvelle classe pour en faciliter l'utilisation
\begin{lstlisting}
sage: D = DoubleAbstractPolynomialRing(QQ); D
The abstract ring of multivariate polynomials on x over The abstract ring of multivariate polynomials on y over Rational Field
sage: D.an_element()
y[0]*x[0, 0, 0] + 2*y[0]*x[1, 0, 0] + y[0]*x[1, 2, 3] + 3*y[0]*x[2, 0, 0]
sage: mx = D.monomial_basis(); mx
The ring of multivariate polynomials on x over The abstract ring of multivariate polynomials on y over Rational Field on the monomial basis
sage: my = D.coeffs_ring().monomial_basis(); my
The ring of multivariate polynomials on y over Rational Field on the monomial basis
sage: pol = my[1,2] * mx[2,1,3] + my[1,2] * mx[1,4,2] + my[3,1] * mx[2,1,3]
sage: pol
(y[1,2])*x[1, 4, 2] + (y[1,2]+y[3,1])*x[2, 1, 3]
\end{lstlisting}
On peut facilement échanger le rôle de $x$ et $y$.
\begin{lstlisting}
sage: pol.swap_coeffs_elements()
(x[1,4,2]+x[2,1,3])*y[1, 2] + (x[2,1,3])*y[3, 1]
sage: pol.swap_coeffs_elements().parent()
The ring of multivariate polynomials on y over The abstract ring of multivariate polynomials on x over Rational Field with 2 variables on the monomial basis
\end{lstlisting}
On peut changer la base aussi bien des variables $x$ que $y$.
\begin{lstlisting}
sage: Schubx = D.schubert_basis_on_vectors()
sage: Schuby = D.coeffs_ring().schubert_basis_on_vectors()
sage: pol = my[1,2] * mx[2,1,3] + my[1,2] * mx[1,4,2] + my[3,1] * mx[2,1,3]
sage: pol
(y[1,2])*x[1, 4, 2] + (y[1,2]+y[3,1])*x[2, 1, 3]
sage: Schubx(pol)
(y[1,2])*Yx(1, 4, 2) + (y[1,2]+y[3,1])*Yx(2, 1, 3) + (-y[1,2]-y[3,1])*Yx(2, 2, 2) + (-y[1,2]-y[3,1])*Yx(2, 3, 1) + (-y[1,2])*Yx(2, 3, 2) + (-y[1,2])*Yx(2, 4, 1) + (-y[1,2]-y[3,1])*Yx(3, 1, 2) + (y[1,2]+y[3,1])*Yx(3, 2, 1) + (-y[1,2]-y[3,1])*Yx(4, 1, 1) + (-y[1,2])*Yx(4, 1, 2) + (y[1,2])*Yx(4, 2, 1)
sage: pol.change_coeffs_bases(Schuby)
(Yy(1,2)-Yy(2,1))*x[1, 4, 2] + (Yy(1,2)-Yy(2,1)+Yy(3,1))*x[2, 1, 3]
\end{lstlisting}

Enfin, on peut utiliser les bases de Schubert et Grothendieck doubles que nous avons définies dans les paragraphes \ref{sub-sec:polynomes_action:pol:schub} et \ref{sub-sec:polynomes_action:pol:groth}.
\begin{lstlisting}
sage: DSchub = D.double_schubert_basis_on_vectors(); DSchub
The ring of multivariate polynomials on x over The abstract ring of multivariate polynomials on y over Rational Field on the Double Schubert basis of type A (indexed by vectors)
sage: pol = DSchub[1,2,1]; pol
y[0]*YY(1, 2, 1)
sage: pol.expand()
(y(3,1,0)+y(3,0,1))*x(0, 0, 0) + (-y(2,1,0)-y(2,0,1)-y(3,0,0))*x(1, 0, 0) + (y(1,1,0)+y(1,0,1)+2*y(2,0,0))*x(1, 1, 0) + (-2*y(1,0,0)-y(0,1,0)-y(0,0,1))*x(1, 1, 1) + (-y[1])*x(1, 2, 0) + y[0]*x(1, 2, 1) + (y(1,1,0)+y(1,0,1)+y(2,0,0))*x(1, 0, 1) + y[2]*x(2, 0, 0) + (-y[1])*x(2, 1, 0) + y[0]*x(2, 1, 1) + (-y[1])*x(2, 0, 1) + (-y(2,1,0)-y(2,0,1)-y(3,0,0))*x(0, 1, 0) + (y(1,1,0)+y(1,0,1)+y(2,0,0))*x(0, 1, 1) + y[2]*x(0, 2, 0) + (-y[1])*x(0, 2, 1) + (-y(2,1,0)-y(2,0,1))*x(0, 0, 1)
sage: pol = my[1,1] * mx[1,2,1]; pol
(y[1,1])*x[1, 2, 1]
sage: DSchub(pol)
(y(2,3,1))*YY(0, 0, 0) + (-y(2,2,1))*YY(1, 0, 0) + (y(1,1,1))*YY(1, 1, 1) + (y[1,1])*YY(1, 2, 1) + (-y(2,2,0))*YY(1, 0, 1) + (-y[1,1])*YY(2, 1, 1) + (-y[2,1])*YY(2, 0, 1) + (y(2,2,1))*YY(0, 1, 0) + (y(2,1,1)+y(2,2,0))*YY(0, 1, 1) + (y[2,2])*YY(0, 2, 0) + (y[2,1])*YY(0, 2, 1) + (y(2,3,0))*YY(0, 0, 1)
sage: DGroth = D.double_grothendieck_basis_on_vectors(); DGroth
The ring of multivariate polynomials on x over The abstract ring of multivariate polynomials on y over Rational Field on the Double Grothendieck basis of type A (indexed by vectors)
sage: pol = DGroth[1,2,1]; pol
y[0]*GG(1, 2, 1)
sage: pol.expand()
y[0]*x(0, 0, 0) + (-y(2,1,1))*x(-2, -2, 0) + (y(3,1,1))*x(-2, -2, -1) + (y(1,1,1))*x(-2, -1, 0) + (-y(2,1,1))*x(-2, -1, -1) + (-y[1])*x(-1, 0, 0) + (y(1,1,1))*x(-1, -2, 0) + (-y(2,1,1))*x(-1, -2, -1) + (y(2,0,0)-y(0,1,1))*x(-1, -1, 0) + (y(1,1,1)-y(3,0,0))*x(-1, -1, -1) + y[2]*x(-1, 0, -1) + (-y[1])*x(0, -1, 0) + y[2]*x(0, -1, -1) + (-y[1])*x(0, 0, -1)
\end{lstlisting}

\section{Architecture du logiciel}
\label{sec:polynomes_sage:archi}

\subsection{Catégorie / Parent / \'Elément}
\label{sub-sec:polynomes_sage:archi:cat}

Une notion de base en programmation objet est celle \emph{d'héritage}.  Si plusieurs objets utilisent une même méthode, on peut les faire hériter d'un objet commun et implanter cette méthode dans ce parent. C'est ce qu'on appelle la factorisation du code. Elle évite de répéter plusieurs fois le même algorithme. Quand les objets représentent des notions mathématiques, la notion d'héritage n'est plus suffisante. En effet, non seulement on veut pouvoir donner des méthodes communes à tous les polynômes mais on veut pouvoir travailler sur l'ensemble des polynômes en tant qu'objet. C'est à ce niveau qu'interviennent les notions \emph{d'éléments} et de \emph{parents}. 

\begin{lstlisting}
sage: from sage.structure.element import Element
sage: from sage.structure.parent import Parent
sage: A = AbstractPolynomialRing(QQ)
sage: pol = A.an_element()
sage: pol
x[0, 0, 0] + 2*x[1, 0, 0] + x[1, 2, 3] + 3*x[2, 0, 0]
sage: isinstance(pol,Element)
True
sage: pol.parent()
The ring of multivariate polynomials on x over Rational Field with 3 variables on the monomial basis
sage: isinstance(pol.parent(),Parent)
True
\end{lstlisting}

Concrètement, {\tt Element} et {\tt Parent} sont des objets de Sage. Lorsqu'on définit la classe des {\tt Polynomes}, c'est-à-dire l'objet représentant l'ensemble des polynômes, celui-ci hérite de {\tt Parent}. Les éléments de type {\tt Polynome} (sans "s") héritent de {\tt Element}. Le lien qui unit les classes {\tt Polynomes} et {\tt Polynome} n'est pas un lien d'héritage. Cependant, on peut définir des méthodes communes à tous les polynômes  au sein de la classe {\tt Polynomes} : elles seront ajoutées dynamiquement aux éléments.

\begin{lstlisting}
sage: type(pol)
sage.combinat.multivariate_polynomials.monomial.FiniteMonomialBasis_with_category.element_class
sage: type(pol.parent())
sage.combinat.multivariate_polynomials.monomial.FiniteMonomialBasis_with_category
\end{lstlisting}
Ici, la classe {\tt Polynomes} est {\tt FiniteMonomialBasis} et les éléments polynômes sont de type {\tt FiniteMonomialBasis.element\_class}.  On remarque que le nom de la classe est en fait {\tt FiniteMonomialBasis\_with\_category}, cette classe a été créée dynamiquement pour permettre au parent de récupérer les méthodes de sa \emph{catégorie}.

En effet, il arrive que des algorithmes soient communs à plusieurs types de parents qui vérifient des propriétés mathématiques communes. Par exemple, un module libre dont la base est indexée par des objets combinatoires est considéré dans Sage comme un {\tt Parent}. Si ce module est en fait une algèbre, il suffit de définir le produit sur les éléments de la base. Ainsi, dans la catégorie {\tt Algèbre} on implémente un algorithme général qui s'appliquera à tous les parents utilisant la catégorie. Un parent possède en général de très nombreuses catégories.
\begin{lstlisting}
sage: pol.parent().categories()
[The category of bases of The abstract ring of multivariate polynomials on x over Rational Field with 3 variables where algebra tower is The ring of multivariate polynomials on x over Rational Field on the monomial basis,
 Category of realizations of The abstract ring of multivariate polynomials on x over Rational Field with 3 variables,
 Category of realizations of sets,
 Category of graded algebras with basis over Rational Field,
 Category of graded modules with basis over Rational Field,
 Category of graded algebras over Rational Field,
 Category of graded modules over Rational Field,
 Category of algebras with basis over Rational Field,
 Category of modules with basis over Rational Field,
 Category of algebras over Rational Field,
 Category of rings,
 Category of rngs,
 Category of vector spaces over Rational Field,
 Category of modules over Rational Field,
 Category of bimodules over Rational Field on the left and Rational Field on the right,
 Category of left modules over Rational Field,
 Category of right modules over Rational Field,
 Category of commutative additive groups,
 Category of semirings,
 Category of commutative additive monoids,
 Category of commutative additive semigroups,
 Category of additive magmas,
 Category of monoids,
 Category of semigroups,
 Category of magmas,
 Category of sets,
 Category of sets with partial maps,
 Category of objects]
\end{lstlisting}

\subsection{Multibases, multivariables}
\label{sub-sec:polynomes_sage:archi:multi}

Un des objectifs du programme est de pouvoir travailler avec plusieurs bases. On définit pour cela un parent abstrait qui ne possédera pas directement d'éléments mais des \emph{réalisations}, c'est-à-dire d'autre parents qui représenteront les différentes bases (cf. figure \ref{fig:polynomes_sage:multibases}).

\begin{figure}[ht]
\centering
\input{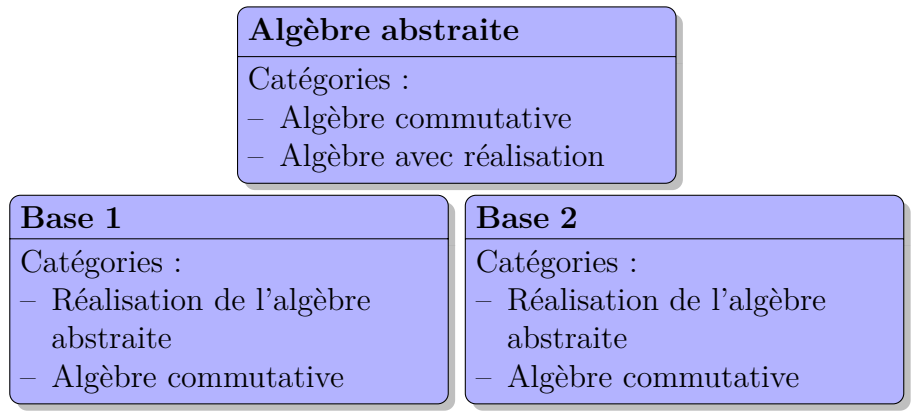}
\caption{Structure d'algèbre abstraite avec plusieurs bases}
\label{fig:polynomes_sage:multibases}
\end{figure}

Les changements de bases sont des objets {\tt Morphisme} qui possèdent une méthode retournant un objet de la base 2 à partir d'un objet de la base 1. On crée ces morphismes au moment de la création des bases et on les enregistre comme des \emph{conversions de type}. Tout objet parent possède une méthode {\tt call}. On appelle la méthode {\tt call} du parent avec en argument un objet d'un autre type. La méthode vérifie alors s'il existe une conversion possible entre les deux types : elle utilise pour cela le graphe créé par les conversions enregistrées.

\begin{lstlisting}
sage: A = AbstractPolynomialRing(QQ)
sage: pol = A.an_element(); pol
x[0, 0, 0] + 2*x[1, 0, 0] + x[1, 2, 3] + 3*x[2, 0, 0]
sage: Schub = A.schubert_basis_on_vectors()
sage: Schub(pol)
Y(0, 0, 0) + 2*Y(1, 0, 0) + Y(1, 2, 3) - Y(1, 3, 2) + 3*Y(2, 0, 0) - Y(2, 1, 3) + Y(2, 3, 1) + Y(3, 1, 2) - Y(3, 2, 1) + Y(4, 1, 1)
sage: Schub.has_coerce_map_from(pol.parent())
True
\end{lstlisting}

Cette architecture multi-bases est classique dans Sage et utilisée par de nombreuses implantations algébriques. Cependant, dans le cas des polynômes, nous devons gérer une difficulté supplémentaire : la gestion du multivarié. En effet, un des besoins du programme est de pouvoir travailler en un nombre quelconque de variables sans le préciser \emph{à priori}. Pour cela, on utilise à nouveau le système de conversion et des algèbres abstraites.

\begin{lstlisting}
sage: A = AbstractPolynomialRing(QQ); A
The abstract ring of multivariate polynomials on x over Rational Field
sage: m = A.monomial_basis(); m
The ring of multivariate polynomials on x over Rational Field on the monomial basis
sage: m3 = m.finite_basis(3); m3
The ring of multivariate polynomials on x over Rational Field with 3 variables on the monomial basis
sage: F3 = A.finite_polynomial_ring(3); F3
The abstract ring of multivariate polynomials on x over Rational Field with 3 variables
sage: pol = A.an_element(); pol
x[0, 0, 0] + 2*x[1, 0, 0] + x[1, 2, 3] + 3*x[2, 0, 0]
sage: pol.parent() == m3
True
sage: pol.parent() == F3.monomial_basis()
True
\end{lstlisting}

Les réalisations directes de l'algèbre abstraite des polynômes sont aussi des algèbres abstraites. Pour une réalisation concrète, il faut la donnée d'une \emph{base} et d'un \emph{nombre de variables}. Chaque parent concret est créé à la volée par son parent abstrait au moment de la création du polynôme. On crée aussi des conversions entre les bases concrètes en différents nombres de variables.

\begin{lstlisting}
sage: A = AbstractPolynomialRing(QQ)
sage: m = A.monomial_basis()   
sage: m3 = m.finite_basis(3)    
sage: m4 = m.finite_basis(4)
sage: m3.has_coerce_map_from(m4)
False
sage: m4.has_coerce_map_from(m3)
True
\end{lstlisting}
Sur cet exemple, on peut lire qu'il existe une conversion automatique des polynômes en 3 variables vers les polynômes en 4 variables mais pas l'inverse. On a illustré cette structure dans la figure \ref{fig:polynomes_sage:multivars}.

\begin{figure}[ht]
\centering
\input{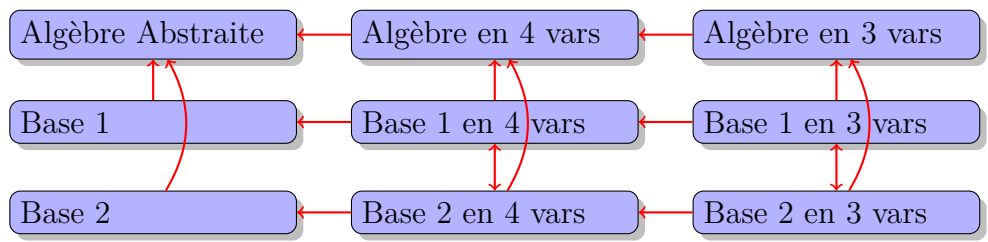}
\caption[Structure d'algèbre abstraite avec plusieurs bases et plusieurs variables]{Structure d'algèbre abstraite avec plusieurs bases et plusieurs variables. Les conversions sont représentées par des flèches rouges.}
\label{fig:polynomes_sage:multivars}
\end{figure}
Cette structure interne est transparente. Les conversions automatiques permettent une utilisation fluide du programme.

Notre implantation comporte nativement les bases des polynômes décrites dans le chapitre \ref{chap:polynomes_action}. Cependant, il est tout à fait possible de rajouter des bases. En effet, en dehors des bases monomiales, toutes les bases sont définies sur le même principe. Elles héritent de la classe {\tt LinearBasisOnVectors} et leurs éléments sont indexés par les éléments de l'espace ambiant du système de racine d'un groupe de Coxeter (qu'on peut tout simplement considérer comme des vecteurs). Seule est définie la fonction qui à un vecteur associe son développement en monôme. \`A partir de cette fonction, on crée le morphisme qui envoie les éléments de la nouvelle base sur les monômes. En créant un élément de la classe {\tt LinearBasisOnVectors}, on crée une nouvelle base dotée de toute la structure décrite précédemment. On en a donné un exemple paragraphe \ref{sub-sec:polynomes_sage:demo:bases}. Il est possible de spécifier que le changement de base vers les monômes est triangulaire et dans ce cas, l'inversion se fera automatiquement et le morphisme inverse sera ajouté comme conversion. Malheureusement, la façon dont sont implantés les morphismes en Sage ne permet pas encore d'inverser les matrices graduées non triangulaires car Sage ne précalcule pas la matrice. 

Grâce au graphe de conversion de Sage, bien qu'on ne définisse que la conversion d'une base donnée vers les monômes, toutes les conversions entre bases se font automatiquement. Pour créer une nouvelle base, l'utilisateur n'a donc qu'à écrire le code entièrement spécifique à sa base et pourra utiliser l'ensemble du système mis en place.

\subsection{Opérateurs : python, un langage dynamique}
\label{sub-sec:polynomes_sage:archi:ope}

Les objets fondamentaux dans tous les changements de bases sur lesquels nous travaillons sont les opérateurs de différences divisées que nous avons définis au paragraphe \ref{sec:polynomes_action:ope}. Un opérateur se définit par une fonction qui prend en paramètre un vecteur et un entier $i$ et qui lui associe une somme de monômes. A partir de ces fonctions nous créons des objets {\tt Morphisme} de Sage qu'on peut appliquer à un polynôme. Le fait d'utiliser un objet {\tt Morphisme} plutôt qu'une simple méthode permet d'utiliser des fonctions préexistantes comme la multiplication des morphismes et la création à partir d'une méthode sur la base.

La seule donnée dont nous avons besoin est donc la fonction définissant l'opérateur. Pour des raisons techniques, cette fonction est définie au sein d'une classe interne à chaque base {\tt \_divided\_difference\_wrapper}. Une instance de cette classe contient le paramètre $i$, la base dans laquelle vit le polynôme et éventuellement d'autres paramètres nécessaires à la différence divisées (par exemple, le type). On détaille ici le processus sur un exemple.

\begin{lstlisting}
sage: A = AbstractPolynomialRing(QQ)
sage: m = A.monomial_basis()
sage: m3 = m.finite_basis(3)
sage: w = m._divided_difference_wrapper(m3,2)
sage: pol = m[2,1,4]; pol
x[2, 1, 4]
sage: key = list(pol)[0][0]; key
[2, 1, 4]
sage: w.divided_difference_on_basis(key)
-x[2, 1, 3] - x[2, 2, 2] - x[2, 3, 1]
sage: w.isobaric_divided_difference_on_basis(key)
-x[2, 2, 3] - x[2, 3, 2]
\end{lstlisting}
Pour créer {\tt w}, on a donné en paramètre le parent du polynôme {\tt m3} et {\tt i=2}. L'objet {\tt w} contient les méthodes sur la base utilisées pour créer les morphismes. 

En fait, l'ensemble des méthodes est défini uniquement pour la base des monômes indexés par des éléments de l'espace ambiant d'un système de racine. Pour calculer la différence divisée, on utilise les opérations déjà existantes sur les groupes de Coxeter en se basant sur les formules décrites paragraphe \ref{sub-sec:polynomes_action:ope:bcd}. Cette méthode permet un algorithme qui ne dépend pas du type.

\begin{lstlisting}
@cached_method
def divided_difference_on_basis(self,key):
	i = self._i
	keys = self._module.basis().keys()
	n = key.scalar(keys.simple_coroot(i))
	if n >= 0:
		return self._module.sum_of_monomials((keys(key-(j)*keys.simple_root(i)-keys.basis()[i-1]) for j in xrange(n)))
	else:
		return -self.divided_difference_on_basis(keys.simple_reflection(i)(key))
\end{lstlisting}
On remarque que l'on utilise ici le mot clé {\tt cached\_method}, c'est ce qui correspond à l'option {\tt remember} de Maple. Elle permet de ne pas effectuer plusieurs fois le même calcul.

Les autres bases possèdent aussi une classe {\tt \_divided\_difference\_wrapper} mais font le plus souvent appel aux méthodes de la base monomiale de l'espace ambiant. Pour créer un  opérateur, on vérifie d'abord si la méthode correspondante existe au sein de {\tt \_divided\_difference\_wrapper} et si oui on l'utilise, si non, on effectue une conversion du polynôme pour pouvoir utiliser la différence divisée par défaut. Voilà par exemple la méthode spécifique de la différence divisée des polynômes de Schubert.

\begin{lstlisting}
@cached_method
def divided_difference_on_basis(self,key):
	i = self._i
	if(key[i-1] > key[i]):
		key2 = [key[j] for j in xrange(self._module.nb_variables())]
		key2[i-1], key2[i] = key2[i], key2[i-1]-1
		return self._module(key2)
	return self._module.zero()
\end{lstlisting}
Dans le cas d'un polynôme de Schubert, la différence divisée correspond par définition à l'échange de deux composantes du vecteur. Il n'est donc pas nécessaire de développer le polynôme pour l'appliquer.

Comme nous l'avons indiqué dans le titre de cette section, Python est un langage dynamique. Cela signifie que l'on peut rajouter des méthodes "à la volée" lors de l'exécution du programme. C'est ce qui permet d'ajouter des opérateurs avec la méthode {\tt add\_operator} dont nous avons donné un exemple paragraphe~\ref{sub-sec:polynomes_sage:demo:pol}. Le code se résume la fonction suivante.

\begin{lstlisting}
def add_operator(self, name, method):
	setattr(self._divided_difference_wrapper, name +"_on_basis", method)
\end{lstlisting}
La méthode prend en paramètre un nom qui sera celui à partir duquel on appellera l'opérateur et une méthode du type de celles que nous avons présentées plus haut. On ajoute simplement un nouvel attribut à la classe interne de la base.

\section{Applications avancées}
\label{sec:polynomes_sage:appli}

\subsection{Degrés projectifs des variétés de Schubert}
\label{sub-sec:polynomes_sage:appli:schubert}

Comme nous l'avons vu au paragraphe \ref{sub-sec:polynomes_action:pol:schub}, le produit sur les polynômes de Schubert s'interprète géométriquement en tant que produit dans l'anneau de cohomologie de la variété de drapeau. Dans \cite{VEIGN}, Veigneau calcule avec le logiciel ACE les degrés projectifs des variétés de Schubert. On peut effectuer un calcul similaire avec notre implantation.

Le degré projectif $d(X)$ d'une sous-variété $X \subset \PP^M$ de codimension $k$ est le nombre d'intersections entre $X$ et un hyperplan générique de dimension $k$. Soit $\sigma \in \Sym{n}$ et $X_\sigma$ une cellule de Schubert de la variété de drapeau $\mathcal{F}_n = \mathcal{F}(\CC^n)$ plongé dans l'espace projectif $\PP^M$ par le plongement de Plücker (avec $M= 2^N-1$ où $N=\frac{n(n-1)}{2}$ est la dimension de $\mathcal{F}_n$). Le degré projectif de la variété $X_\sigma$, $d(X_\sigma)$ est donné par un calcul sur les polynômes de Schubert. Soit $v$, le code de Lehmer de $\sigma$ et $h =  (n-1)x_1 + (n-2)x_2 + \ldots + x_{n-1}$ la classe d'une section hyperplane. On a que $d(X_\sigma)$ est le coefficient de $Y_{[n-1,n-2,\dots, 0]}$ dans $Y_v h^{N - \ell(\sigma)}$ \cite{Chern}. Il est donné par la fonction suivante.

\begin{lstlisting}
def proj_deg(perm):
    n = len(perm)
    d = n*(n-1)/2 - perm.length()
    
    # we create the polynomial ring and the bases
    A = AbstractPolynomialRing(QQ)
    Schub = A.schubert_basis_on_vectors()
    
    # we compute the product
    h = sum( [(n-i) * A.var(i) for i in xrange(1,n)])
    res = Schub( h**d * Schub(perm.to_lehmer_code()))
    
    # we return the coefficient
    key = [n-i for i in xrange(1,n+1)]
    return res[key]
\end{lstlisting}

Par exemple, si $\sigma = 2143$, le degré projectif de la variété de Schubert $X_\sigma$ est~78.

\begin{lstlisting}
sage: p = Permutation([2,1,4,3])   
sage: proj_deg(p)               
78
\end{lstlisting}

On peut aussi effectuer le calcul directement et lire le résultat sur le polynôme.

\begin{lstlisting}
sage: A = AbstractPolynomialRing(QQ)
sage: Schub = A.schubert_basis_on_vectors()
sage: m = A.monomial_basis()
sage: pol = Schub( (3*m[1] + 2*m[0,1] + m[0,0,1])^4 * Schub[1,0,1,0])
sage: pol
8*Y(1, 1, 4, 0) + 23*Y(1, 2, 3, 0) + 24*Y(1, 3, 2, 0) + 39*Y(1, 4, 1, 0) + 15*Y(1, 5, 0, 0) + Y(1, 0, 5, 0) + 48*Y(2, 1, 3, 0) + 101*Y(2, 2, 2, 0) + 117*Y(2, 3, 1, 0) + 84*Y(2, 4, 0, 0) + 12*Y(2, 0, 4, 0) + 173*Y(3, 1, 2, 0) + 78*Y(3, 2, 1, 0) + 147*Y(3, 3, 0, 0) + 53*Y(3, 0, 3, 0) + 283*Y(4, 1, 1, 0) + 171*Y(4, 2, 0, 0) + 96*Y(4, 0, 2, 0) + 93*Y(5, 1, 0, 0) + 176*Y(5, 0, 1, 0) + 80*Y(6, 0, 0, 0)
sage: pol[3,2,1,0]
78
\end{lstlisting}

On peut utiliser la fonction pour calculer tous les degrés projectifs pour toutes les permutations de taille 4.

\begin{lstlisting}
sage: degrees = {}
sage: for perm in Permutations(4):
....:     degrees[perm] = proj_deg(perm)
....:     
sage: degrees
{[2, 1, 4, 3]: 78, [1, 3, 4, 2]: 48, [3, 2, 4, 1]: 3, [3, 1, 2, 4]: 48, [4, 2, 1, 3]: 3, [1, 4, 2, 3]: 46, [3, 2, 1, 4]: 16, [4, 1, 3, 2]: 3, [2, 3, 4, 1]: 6, [3, 4, 2, 1]: 1, [1, 2, 3, 4]: 720, [1, 3, 2, 4]: 280, [2, 4, 3, 1]: 3, [2, 3, 1, 4]: 46, [3, 4, 1, 2]: 2, [4, 2, 3, 1]: 1, [1, 4, 3, 2]: 16, [4, 1, 2, 3]: 6, [2, 4, 1, 3]: 12, [4, 3, 1, 2]: 1, [4, 3, 2, 1]: 1, [3, 1, 4, 2]: 14, [2, 1, 3, 4]: 220, [1, 2, 4, 3]: 220}
\end{lstlisting}

\subsection{Déterminants de fonctions de Schur}
\label{sub-sec:polynomes_sage:appli:schur}

Les polynômes de Schubert Grassmannien sont ceux indexés par un vecteur $v$ tel que $v_1 \leq v_2 \leq \dots \leq v_n$. On a vu dans le chapitre \ref{chap:polynomes_action} que ce sont des polynômes symétriques et qu'ils correspondent aux fonctions de Schur dans le cas des Schubert simples. Plus précisément, la matrice de transition entre les polynômes de Schubert grassmanniens doubles et les fonctions de Schur est unitriangulaire.

\begin{lstlisting}
sage: A = AbstractPolynomialRing(QQ)
sage: Schub =  A.schubert_basis_on_vectors()
sage: pol = Schub[1,2]
sage: pol.expand()
x(1, 2) + x(2, 1)
sage: 
sage: D = DoubleAbstractPolynomialRing(QQ)
sage: DSChub = D.double_schubert_basis_on_vectors()
sage: pol = DSchub[1,2]
sage: pol
y[0]*YY(1, 2)
sage: Schub = D.schubert_basis_on_vectors()
sage: Schub(pol)
y[0]*Yx(1, 2) + (-y(2,1,0)-y(2,0,1))*Yx(0, 0) + (-y(1,0,0)-y(0,1,0)-y(0,0,1))*Yx(1, 1) + (y(1,1,0)+y(1,0,1)+y(2,0,0))*Yx(0, 1) + (-y[1])*Yx(0, 2)
\end{lstlisting}
Cela nous permet de calculer des déterminants de fonctions de Schur en les remplaçant par des polynômes de Schubert et en spécialisant arbitrairement les variables $y$. Par exemple, on peut calculer
\begin{equation}
\vert s_\mu(A) \vert_{\mu \subseteq \syng{1,1}},
\end{equation}
où $A \in [ \lbrace x_1, x_2 \rbrace, \lbrace x_1, x_3 \rbrace, \lbrace x_2, x_3 \rbrace ]$ et prouver que l'on obtient $\prod_{j>i}(x_j-x_i)$.  Tout d'abord, remplaçons $s_\mu$ par
\begin{equation}
\vert Y_u(A,y) \vert_{u = 00, 01, 11}
\end{equation}
et spécialisons en $y_1=x_1$, $y_2=x_2$. Dans ce cas, le déterminant devient
\begin{equation}
\label{eq:polynomes_sage:ex-det-schur}
\left|
\begin{array}{lll}
1 & 1 & 1 \\
0 & x_3 - x_2 & x_3 - x_1 \\
0 & 0 & (x_3-x_1)(x_2-x_1)
\end{array}
\right|
\end{equation}
et nous donne le résultat. La fonction suivante calcule cette matrice.

\begin{lstlisting}
def compute_matrix(variables, alphabets, indices):
    
    n = len(indices)
    
    #Initial definitions
    D = DoubleAbstractPolynomialRing(QQ)
    DSchub = D.double_schubert_basis_on_vectors()
    
    result_matrix = []
    
    for u in indices:
        line = []
        
        #the expansion on the double schubert will allow us to compute the result
        pu = DSchub(u).expand()
        
        for a in alphabets:
        	#we apply our polynomial on alphabets and specialize the y 
            coeff = pu.to_expr(alphabet = a, alphabety = variables)
            line.append(coeff)
        result_matrix.append(line)
    
    return Matrix(result_matrix)
\end{lstlisting}
On peut vérifier le résultat de \eqref{eq:polynomes_sage:ex-det-schur}.
\begin{lstlisting}
sage: var('x1,x2,x3')       
(x1, x2, x3)
sage: variables = (x1,x2,x3)
sage: alphabets = [[x1,x2],[x1,x3],[x2,x3]]
sage: indices =  [[0,0],[0,1],[1,1]]
sage: res = compute_matrix(variables, alphabets, indices)     
sage: res
[1                            1                            1]
[0                     -x2 + x3                     -x1 + x3]
[0                            0 x1^2 - x1*x2 - x1*x3 + x2*x3]
sage: det = res.determinant()
sage: det
-(x2 - x3)*(x1^2 - x1*x2 - x1*x3 + x2*x3)
sage: factor(det)
-(x2 - x3)*(x1 - x3)*(x1 - x2)
\end{lstlisting}

\subsection{Produit des polynômes de Grothendieck}
\label{sub-sec:polynomes_sage:appli:groth}

Notre implantation peut être utilisée pour effectuer les calculs nécessaires au chapitre \ref{chap:polynomes_grothendieck}. Elle permet de vérifier le théorème \ref{thm:polynomes_grothendieck:interval} sur des exemples.

Tout d'abord, vérifions sur un exemple le théorème \ref{thm:polynomes_grothendieck:PieriGrothLascoux}. On a besoin d'effectuer un produit de polynôme et donc de convertir depuis la base des monômes vers la base  Grothendieck. Dans notre implantation, cela ne peut se faire que sur les polynômes de Grothendieck simples. On calcule $G_\sigma(1 - x_1)(1 - x_2)(1 - x_3)(1 - x_4)$ (le changement de variable remplace $x_i^{-1}$ par $1 - x_i$). Par ailleurs, on travaille dans un espace quotienté par un idéal et on supprime donc les termes inutiles.

\begin{lstlisting}
def apply_operators(perm, k):
    perm = Permutation(perm)
    code = perm.to_lehmer_code()
    
    A = AbstractPolynomialRing(QQ)
    Groth = A.grothendieck_positive_basis_on_vectors()
    pol= Groth( code )
    
    factor = prod( [A.one() - A.var(i+1) for i in xrange(k)] )
    
    pol =  pol*factor
    
    #suppression des termes inutiles
    res = pol.parent().zero()
    for (k,c) in pol:
        m = pol.nb_variables() - 1
        for i in xrange(pol.nb_variables()):
            if(k[i]>m): break 
            m = m-1
        else:
            res+= pol.parent().term(k,c)
            
    return res
\end{lstlisting}

Pour $\sigma = 136254$, on obtient
\begin{lstlisting}
sage: perm = Permutation([1,3,6,2,5,4])  
sage: res1 = apply_operators(perm, 4)
sage: res1
-G(0, 1, 3, 1, 1, 0) - G(0, 1, 3, 2, 0, 0) + G(0, 1, 3, 2, 1, 0) + G(0, 1, 3, 0, 1, 0) + G(0, 2, 3, 1, 1, 0) + G(0, 2, 3, 2, 0, 0) - G(0, 2, 3, 2, 1, 0) - G(0, 2, 3, 0, 1, 0) - G(0, 3, 3, 0, 0, 0) + G(0, 3, 3, 1, 0, 0) - G(0, 3, 3, 1, 1, 0) - G(0, 3, 3, 2, 0, 0) + G(0, 3, 3, 2, 1, 0) + G(0, 3, 3, 0, 1, 0)
\end{lstlisting}
Soit en terme de permutations
\begin{lstlisting}
from sage.combinat import permutation
sage: perms = [ permutation.from_lehmer_code(key) for (key,c) in res1.elements()]
sage: perms
[[1, 5, 6, 3, 4, 2],
 [1, 4, 6, 5, 3, 2],
 [1, 3, 6, 2, 5, 4],
 [1, 5, 6, 2, 4, 3],
 [1, 5, 6, 3, 2, 4],
 [1, 5, 6, 4, 3, 2],
 [1, 3, 6, 5, 2, 4],
 [1, 3, 6, 4, 5, 2],
 [1, 4, 6, 5, 2, 3],
 [1, 4, 6, 2, 5, 3],
 [1, 5, 6, 2, 3, 4],
 [1, 5, 6, 4, 2, 3],
 [1, 4, 6, 3, 5, 2],
 [1, 3, 6, 5, 4, 2]]
\end{lstlisting}
On peut vérifier que ce sont bien les mêmes qui apparaissent dans le calcul suivant où l'on utilise le résultat du théorème \ref{thm:polynomes_grothendieck:PieriGrothLascoux}.
\begin{lstlisting}
sage: A = AbstractPolynomialRing(QQ)
sage: K = A.demazure_basis_on_vectors()
sage: pol = K[6,5,4,3,2,1]
sage: res2 = pol.apply_reduced_word([3,4,5,2,3,4],method="hatpi")
sage: res2 = res2.apply_reduced_word([3,2,1,2],method="pi")
sage: res2
K(1, 3, 6, 2, 5, 4) - K(1, 3, 6, 4, 5, 2) - K(1, 3, 6, 5, 2, 4) + K(1, 3, 6, 5, 4, 2) - K(1, 4, 6, 2, 5, 3) + K(1, 4, 6, 3, 5, 2) + K(1, 4, 6, 5, 2, 3) - K(1, 4, 6, 5, 3, 2) - K(1, 5, 6, 2, 3, 4) + K(1, 5, 6, 2, 4, 3) + K(1, 5, 6, 3, 2, 4) - K(1, 5, 6, 3, 4, 2) - K(1, 5, 6, 4, 2, 3) + K(1, 5, 6, 4, 3, 2)
\end{lstlisting}
On généralise ce second calcul dans une fonction.
\begin{lstlisting}
def product_permutations(perm, k):
    
    #Initial definitions
    A = AbstractPolynomialRing(QQ)
    K = A.demazure_basis_on_vectors()
    n = len(perm)
    
    # permutations
    omega = Permutation([n-i for i in xrange(n)])
    zeta = [perm[i] for i in xrange(k)]
    zeta.sort()
    zeta.reverse()
    zeta2 = [perm[i] for i in xrange(k,n)]
    zeta2.sort()
    zeta2.reverse()
    zeta.extend(zeta2)
    zeta = Permutation(zeta)
    
    #paths
    hatpis = (zeta * omega).reduced_word()
    pis = (perm * zeta.inverse()).reduced_word()
    
    #computation
    pol = K(list(omega))
    if(len(hatpis)!=0):
        pol = pol.apply_reduced_word(hatpis, method="hatpi")
    if(len(pis)!=0):
        pol = pol.apply_reduced_word(pis, method="pi")
    
    # perm set from polynomial
    perms = [Permutation(key) for (key,c) in pol.elements()]
    return perms
\end{lstlisting}
Sur l'exemple, cela donne
\begin{lstlisting}
sage: perm = Permutation([1,3,6,2,5,4])
sage: perms = product_permutations(perm, 4)
sage: perms
[[1, 5, 6, 2, 4, 3],
 [1, 4, 6, 3, 5, 2],
 [1, 5, 6, 4, 3, 2],
 [1, 3, 6, 5, 4, 2],
 [1, 4, 6, 5, 2, 3],
 [1, 4, 6, 2, 5, 3],
 [1, 5, 6, 3, 2, 4],
 [1, 5, 6, 2, 3, 4],
 [1, 5, 6, 4, 2, 3],
 [1, 4, 6, 5, 3, 2],
 [1, 3, 6, 4, 5, 2],
 [1, 3, 6, 2, 5, 4],
 [1, 5, 6, 3, 4, 2],
 [1, 3, 6, 5, 2, 4]]
\end{lstlisting}
L'application des opérateurs est une opération formelle sur les vecteurs. Dans notre implantation, elle est directement définie au niveau des polynômes clés : aucun changement de base n'est nécessaire. Ce second calcul est donc beaucoup plus rapide que le premier : on passe de 10 secondes à un dixième de seconde.
\begin{lstlisting}
sage: time( apply_operators(perm, 4))
CPU times: user 10.60 s, sys: 0.02 s, total: 10.62 s
Wall time: 10.73 s
-G(0, 1, 3, 1, 1, 0) - G(0, 1, 3, 2, 0, 0) + G(0, 1, 3, 2, 1, 0) + G(0, 1, 3, 0, 1, 0) + G(0, 2, 3, 1, 1, 0) + G(0, 2, 3, 2, 0, 0) - G(0, 2, 3, 2, 1, 0) - G(0, 2, 3, 0, 1, 0) - G(0, 3, 3, 0, 0, 0) + G(0, 3, 3, 1, 0, 0) - G(0, 3, 3, 1, 1, 0) - G(0, 3, 3, 2, 0, 0) + G(0, 3, 3, 2, 1, 0) + G(0, 3, 3, 0, 1, 0)
sage: time(product_permutations(perm,4))
CPU times: user 0.09 s, sys: 0.00 s, total: 0.09 s
Wall time: 0.09 s
[[1, 5, 6, 2, 4, 3],
 [1, 4, 6, 3, 5, 2],
 [1, 5, 6, 4, 3, 2],
 [1, 3, 6, 5, 4, 2],
 [1, 4, 6, 5, 2, 3],
 [1, 4, 6, 2, 5, 3],
 [1, 5, 6, 3, 2, 4],
 [1, 5, 6, 2, 3, 4],
 [1, 5, 6, 4, 2, 3],
 [1, 4, 6, 5, 3, 2],
 [1, 3, 6, 4, 5, 2],
 [1, 3, 6, 2, 5, 4],
 [1, 5, 6, 3, 4, 2],
 [1, 3, 6, 5, 2, 4]]
\end{lstlisting}

\`A présent, vérifions le résultat du théorème \ref{thm:polynomes_grothendieck:interval}. La fonction suivante génère l'intervalle de Bruhat entre deux permutations.
\begin{lstlisting}
def bruhat_interval(p1, p2, res = None):
    if res is None:
        res = []
    if(p1.bruhat_lequal(p2)):
        res.append(p1)
        if(p1.length() < p2.length()):
            succs = p1.bruhat_succ()
            for p in succs:
                if p not in res:
                    bruhat_interval(p,p2,res)
        return res
    else:
        return None
\end{lstlisting}
On teste que l'ensemble obtenu par l'application des opérateurs est bien un intervalle de l'ordre de Bruhat.
\begin{lstlisting}
def is_result_interval(perm, k):
    perms = product_permutations(perm, k)
    #get 1 max permutation
    pmax = perm
    for p in perms:
        if(pmax.bruhat_lequal(p)):
            pmax = p
            
    interval = bruhat_interval(perm,pmax)
    return set(interval) == set(perms)
\end{lstlisting}
Sur l'exemple précédent, cela donne :
\begin{lstlisting}
sage: perm = Permutation([1,3,6,2,5,4])
sage: is_result_interval(perm,4)
True
\end{lstlisting}
On peut tester sur l'ensemble des permutations de taille 4 et 5.
\begin{lstlisting}
def test_all_perms(size):
    for p in Permutations(size):
        for k in xrange(1,size):
            if not is_result_interval(p,k):
                print p,k
                return False
    return True
sage: test_all_perms(4)
True
sage: test_all_perms(5)
True
\end{lstlisting}

\cleardoublepage

\part{Treillis de ($m$-)Tamari et algèbres}
\label{part:tamari}

\chapter{Treillis sur les arbres binaires et algèbres de Hopf}
\index{arbres!binaires}
\index{treillis!de Tamari}
\index{ordre!de Tamari}
\index{algèbre!de Hopf}
\label{chap:tamari_prelim}

La troisième partie de ce mémoire est dédiée à l'étude d'un ordre sur les arbres binaires appelé \emph{treillis de Tamari} qui peut se décrire comme un quotient de l'ordre faible sur les permutations. Il apparaît en 1962 dans le travail de Tamari \cite{Tamari1} sous forme de structure d'ordre sur les parenthésages formels d'une expression. Tamari prouve plus tard que cet ordre est en fait un treillis \cite{Tamari2}. Par ailleurs, son diagramme de Hasse correspond à un polytope connu sous le nom d'associaèdre ou polytope de Stasheff. Le nombre de parenthésages formels d'une expression est compté par les nombres de Catalan et il existe donc de multiples objets combinatoires sur lesquels on peut décrire ce treillis \cite{Stanley}. Dans ce chapitre, nous en rappelons deux : les \emph{chemins de Dyck} et les \emph{arbres binaires}. Dans les deux cas, la relation de couverture se décrit par une opération très simple (cf. figures \ref{fig:tamari_prelim:rot-dyck} et \ref{fig:tamari_prelim:rot-tree}).

Les relations entre l'associaèdre et le permutoèdre sont souvent étudiées d'un point de vue géométrique. L'approche algébrique permet de démontrer un résultat très fort : l'ordre de Tamari est un quotient de l'ordre faible sur les permutations. Ce résultat apparaît déjà en filigrane dans les travaux de Björner et Wachs \cite{BW}. On en trouve une généralisation dans la théorie des treillis cambriens de Reading \cite{Cambrian}. Une explication directe est donnée dans \cite{PBT2}. Ce dernier article ainsi qu'un précédent de Loday et Ronco \cite{PBT1} font le lien avec une structure importante en combinatoire algébrique : les \emph{algèbres de Hopf}.

 Le principe d'une structure d'algèbre est la composition des objets. C'est-à-dire que si $A$ est un espace vectoriel d'objets combinatoires, on définit une application $A \otimes A \rightarrow A$. Si maintenant on définit une opération dite de \emph{coproduit} : $A \rightarrow A \otimes A$ et qu'elle vérifie une certaine compatibilité avec le produit, on obtient une structure de bigèbre. En termes combinatoires, cela revient à \emph{décomposer} les objets. Dans le cadre de la topologie algébrique du milieu du XXème siècle, Heinz Hopf étudia certaines bigèbres possédant une opération appelée \emph{antipode} qui généralise la notion de passage à l'inverse des groupes. C'est ce qu'on appelle les \emph{Algèbres de Hopf} \cite{Hopf1, Hopf2}.

En combinatoire, c'est surtout la structure de bigèbre qui est intéressante. Comme on travaille dans des espaces de dimension finie graduellement, l'existence de l'antipode est toujours vérifiée et on appelle donc ces espaces des \emph{algèbres de Hopf combinatoires}. Cette double structure de produit et coproduit est déjà implicitement présente dans le travail de MacMahon sur les fonctions symétriques \cite{MacMahon}. Puis dans les années 1970, Rota fait le lien avec la structure étudiée par les topologues et insiste sur l'importance des algèbres de Hopf en combinatoire \cite{Rota1,Rota2}. On pourra lire \cite{CHA} pour une bonne introduction sur le sujet.

La théorie des algèbres de Hopf combinatoires prend une nouvelle direction avec l'apparition en 1995 des fonctions symétriques non-commutatives qui feront l'objet d'une série d'articles sur plus de dix ans \cite{NCSF1,NCSF2,NCSF3,NCSF4,NCSF5,NCSF6,NCSF7}. C'est en fait une nouvelle méthode de calcul qui est décrite : à un objet combinatoire, on associe une somme de mots sur un alphabet non commutatif. Le produit et le coproduit ne se définissent plus sur les objets mais sur les mots ce qui en fait des opérations triviales. La difficulté réside alors dans la fonction qui à chaque objet associe une somme. C'est ce qu'on appelle la \emph{réalisation polynomiale}. Elle permet souvent de trivialiser des preuves jusqu'alors complexes et de donner le cadre pour des résultats plus importants. En particulier, dans \cite{NCSF6, NCSF7} les auteurs introduisent l'algèbre des fonctions quasi-symétriques libres $\FQSym$ dont les bases sont indexées par les permutations. Cette algèbre se révèle isomorphe à une algèbre de Hopf déjà connue sur les permutations, celle introduite par Malvenuto et Reutenauer \cite{MalReut}. 

Le lien avec le treillis de Tamari apparaît en 1998 lorsque Loday et Ronco décrivent une algèbre de Hopf sur les arbres binaires comme un quotient de l'algèbre de Malvenuto-Reutenauer \cite{PBT1}. Le produit dans cette algèbre s'interprète comme un intervalle du treillis de Tamari. De leurs résultats sur $\FQSym$, Hivert, Novelli et Thibon tirent alors une nouvelle interprétation de l'algèbre de Loday et Ronco comme un quotient de $\FQSym$ \cite{PBT2}. On la nomme $\PBT$. Un élément de la base de $\PBT$ est indexé par un arbre binaire. Il correspond à une somme sur un intervalle de l'ordre faible dans l'algèbre $\FQSym$. Les auteurs prouvent alors que l'ordre entre les intervalles à l'intérieur de l'ordre faible correspond en fait au treillis de Tamari sur les arbres binaires. On retrouve ainsi que le treillis de Tamari est un quotient et un sous-treillis de l'ordre faible. Dans le chapitre \ref{chap:tamari_intervalles}, nous utiliserons cette propriété pour prouver un nouveau résultat sur le treillis de Tamari.

Plus généralement, la troisième partie de ce mémoire est dédiée à l'étude de l'ordre de Tamari, à ses généralisations et aux structures algébriques qui lui sont liées. Ce premier chapitre sert d'introduction et nous rappelons en détail les résultats que nous venons d'évoquer. Dans le paragraphe \ref{sec:tamari_prelim:tamari}, nous commençons par définir l'ordre de Tamari sur les chemins de Dyck et les arbres binaires. Nous expliquons aussi en quoi il est un quotient et un sous-treillis de l'ordre faible. Le paragraphe \ref{sec:tamari_prelim:hopf} nous sert à définir les algèbres de Hopf dans le cadre combinatoire où nous les utilisons. Nous donnons le principe de la réalisation polynomiale à travers l'exemple de l'algèbre $\FQSym$. Enfin, dans le paragraphe \ref{sec:tamari_prelim:pbt}, nous décrivons l'algèbre $\PBT$ sur les arbres binaires en nous basant sur les résultats de \cite{PBT2}.

\section{Treillis de Tamari et ordre faible}
\label{sec:tamari_prelim:tamari}

\subsection{Définition : chemins de Dyck et arbres binaires}
\label{sub-sec:tamari_prelim:tamari:def}
\index{chemins de Dyck}
\index{treillis!de Tamari}
\index{ordre!de Tamari}

La définition historique de l'ordre de Tamari est donnée en termes de parenthésages. Nous donnons ici celle en termes de \emph{chemins de Dyck} qui est couramment utilisée.

\begin{Definition}
\label{def:tamari_prelim:dyck}
Un \emph{chemin de Dyck} de taille $n$ est un chemin dans le plan depuis l'origine jusqu'au point $(2n,0)$ formé de pas dits "montants" $(1,1)$ et de pas "descendants" $(1,-1)$ tel que le chemin ne descend jamais en dessous la ligne $y=0$. 
\end{Definition}

On identifie les chemins à des mots sur un alphabet binaire $\lbrace 1, 0 \rbrace$ où les pas montants sont représentés par des $1$ et les pas descendants par des $0$. On les appelle alors \emph{mots de Dyck}. Ils contiennent autant de 1 que de 0 et et dans chacun de leurs préfixes, le nombre de 1 est supérieur ou égal au nombre de 0. Un chemin est dit \emph{primitif} s'il n'a pas d'autres contacts avec la droite $y=0$ que son origine et son point final. 

\begin{Definition}
\label{def:tamari_prelim:dyck_rotation}
Soit $u$ un chemin de Dyck, tel que $u$ contienne un pas descendant $d$ suivi d'un chemin primitif $c$. Une rotation sur $u$ consiste à échanger le pas descendant $d$ avec le chemin $c$.
\end{Definition}

L'opération de rotation, illustrée figure \ref{fig:tamari_prelim:rot-dyck}, est la relation de couverture de l'ordre de Tamari. C'est-à-dire qu'un chemin $v$ est plus plus grand qu'un chemin $u$ si on peut atteindre $v$ par une série de rotations sur $u$. Cela définit bien un ordre et même un treillis \cite{Tamari2}, cf. figure \ref{fig:tamari_prelim:dyck-3-4}.

\begin{figure}[ht]
\input{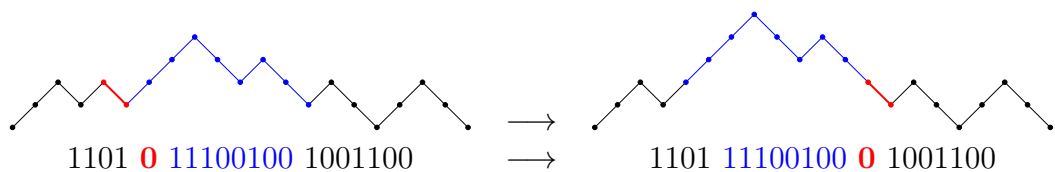}
\caption{Rotation sur les chemins de Dyck}
\label{fig:tamari_prelim:rot-dyck}
\end{figure}

\index{arbres!binaires}
Par la suite, nous utiliserons plutôt la définition de l'ordre de Tamari sur les arbres binaires. Un \emph{arbre binaire} se définit récursivement comme étant soit l'arbre vide $\emptyset$ (qu'on appelle aussi \emph{feuille}), soit un couple d'arbres binaires appelés \emph{sous-arbre droit} et \emph{sous-arbre gauche} greffés sur un \noeud racine. Si $T$ est un arbre formé de la racine $x$ et respectivement des sous-arbres, gauche et droit, $A$ et $B$, on écrit $T=x(A,B)$. L'opération de rotation existe aussi sur les arbres.

\begin{figure}[p]
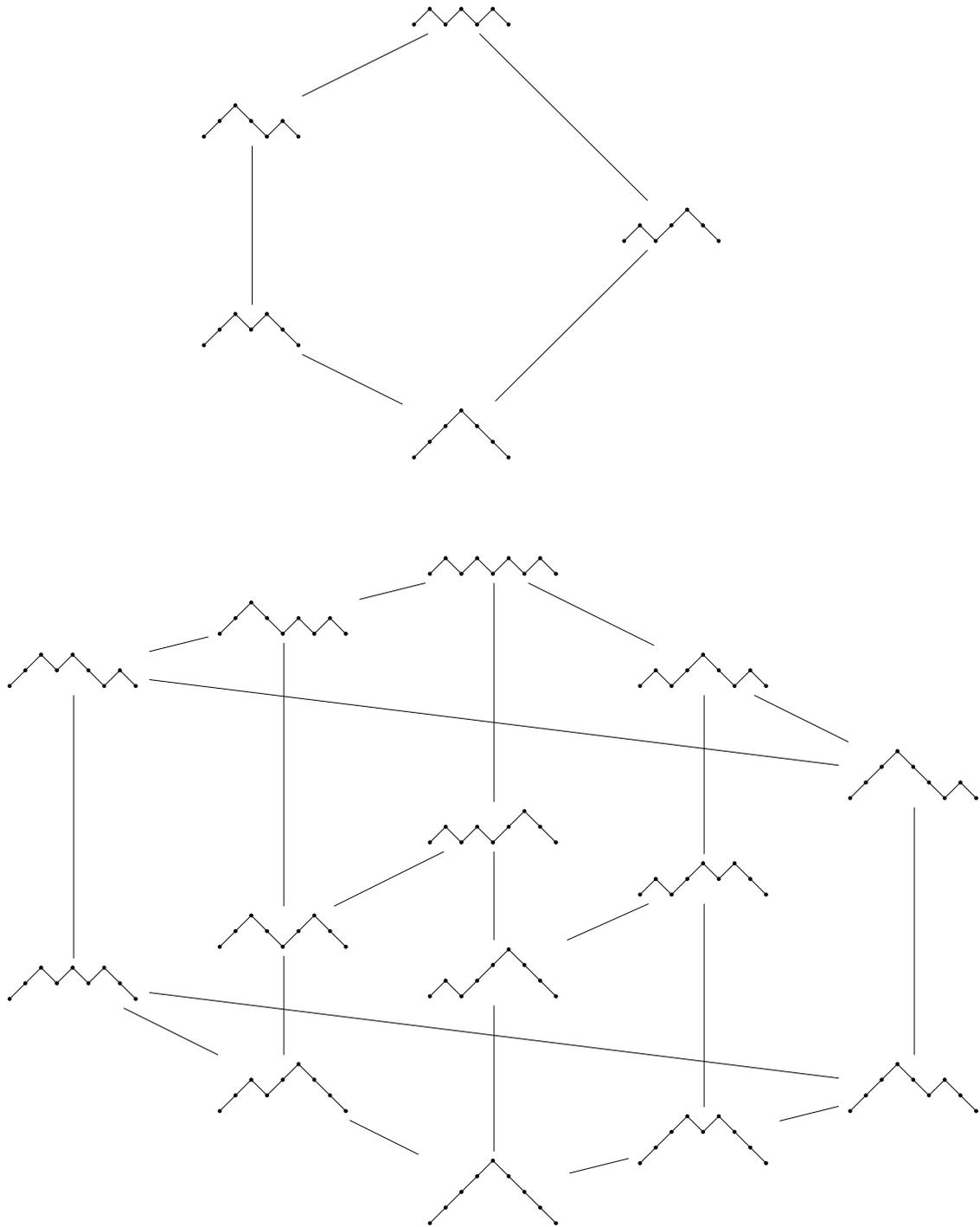

\begin{leftfullpage}
\centering
\input{includes/figures/tamari_dyck-3}
\vspace{3em}
\input{includes/figures/tamari_dyck-4}
\caption[Ordre de Tamari sur les chemins de Dyck de tailles 3 et 4]{Ordre de Tamari sur les chemins de Dyck de tailles 3 et 4, le plus petit élément est le chemin alterné en haut du diagramme.}
\label{fig:tamari_prelim:dyck-3-4}
\end{leftfullpage}
\end{figure}
\begin{figure}[p]
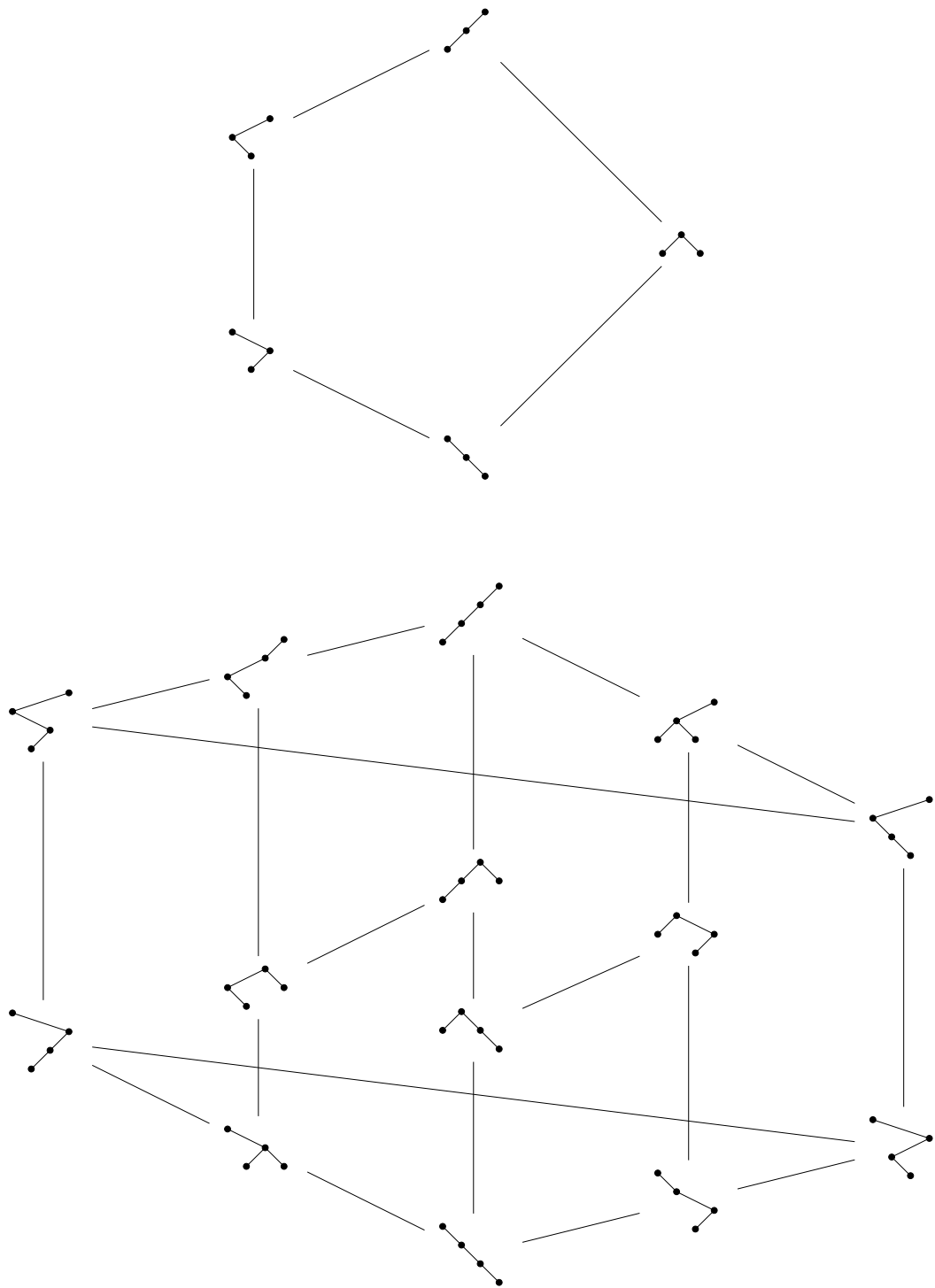

\begin{fullpage}
\centering
\input{includes/figures/tamari_trees-3}
\vspace{3em}
\input{includes/figures/tamari_trees-4}
\caption[Ordre de Tamari sur les arbres binaires de tailles 3 et 4]{Ordre de Tamari sur les arbres binaires de tailles 3 et 4, le plus petit élément est le peigne gauche en haut du diagramme.}
\label{fig:tamari_prelim:trees-3-4}
\end{fullpage}
\end{figure}

\begin{Definition}
Soit $y$ un \noeud d'un arbre binaire $T$ dont le sous-arbre gauche n'est pas vide et soit $x$ la racine de ce sous-arbre gauche. La \emph{rotation droite} de $T$ en $y$ est la  réécriture locale décrite par la figure \ref{fig:tamari_prelim:rot-tree}, c'est-à-dire que l'on remplace $y(x(A,B),C)$ par $x(A,y(B,C))$ où $A$,$B$ et $C$ sont des arbres binaires potentiellement vides.
\end{Definition}

\begin{figure}[ht]
\centering
\input{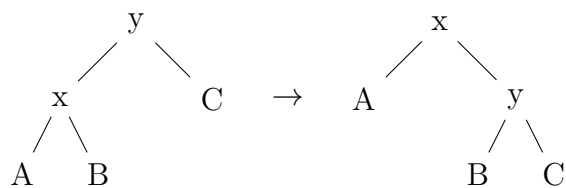}
\caption{Rotation sur les arbres binaires}
\label{fig:tamari_prelim:rot-tree}
\end{figure}

Cette opération sur les arbres binaires n'est pas spécifique à l'ordre de Tamari. En effet, elle est utilisée (en complément de son opération symétrique, la rotation gauche) pour \emph{équilibrer} les arbres binaires quand ils sont utilisés dans des algorithmes de tris de données \cite{AVL}. Si on la considère comme une relation de couverture, le poset obtenu est isomorphe à l'ordre obtenu sur les chemins de Dyck. C'est l'ordre de Tamari sur les arbres binaires illustré figure \ref{fig:tamari_prelim:trees-3-4}.

La correspondance entre les chemins de Dyck et les arbres binaires utilise la structure récursive des deux objets. Un chemin de Dyck $D$ est soit le chemin vide, soit s'écrit $D = D_1 1 D_2 0$ où $D_1$ et $D_2$ sont deux chemins de Dyck. On définit alors $T$, l'arbre correspondant à $D$ par $T=x(T_1,T_2)$ où $T_1$ et $T_2$ sont les arbres binaires correspondant respectivement à $D_1$ et $D_2$. La bijection est illustrée figure \ref{fig:tamari_prelim:dyck-tree}. On remarque que $D_1$ est le préfixe de $D$ jusqu'au dernier retour à 0. En particulier, si $D$ est un chemin primitif alors $D_1$ est vide. Le chemin $D_2$ est vide si $D$ se termine par le mot $10$. 

\begin{figure}[ht]
\centering
\input{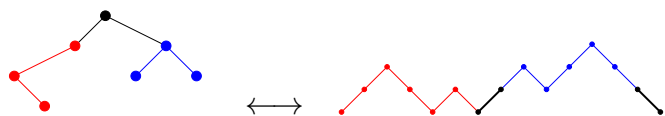}
\caption{Correspondance chemins de Dyck / arbres binaires}
\label{fig:tamari_prelim:dyck-tree}
\end{figure}

\subsection{Liens avec l'ordre faible}
\label{sub-sec:tamari_prelim:tamari:ordre_faible}
\index{ordre!faible}

Le lien avec l'ordre faible passe par un étiquetage particulier des arbres binaires, les \emph{arbres binaires de recherche}.

\begin{Definition}
\index{arbres!binaires!de recherche}
Un arbre binaire de recherche est un arbre binaire étiqueté vérifiant pour chaque \noeud $x$ la condition suivante : si $x$ est étiqueté par $k$ alors les \noeuds du sous-arbre gauche (resp. droit) de $x$ sont étiquetés par des entiers inférieurs ou égaux (resp. supérieurs) à $k$. 
\end{Definition}

\begin{figure}[ht]
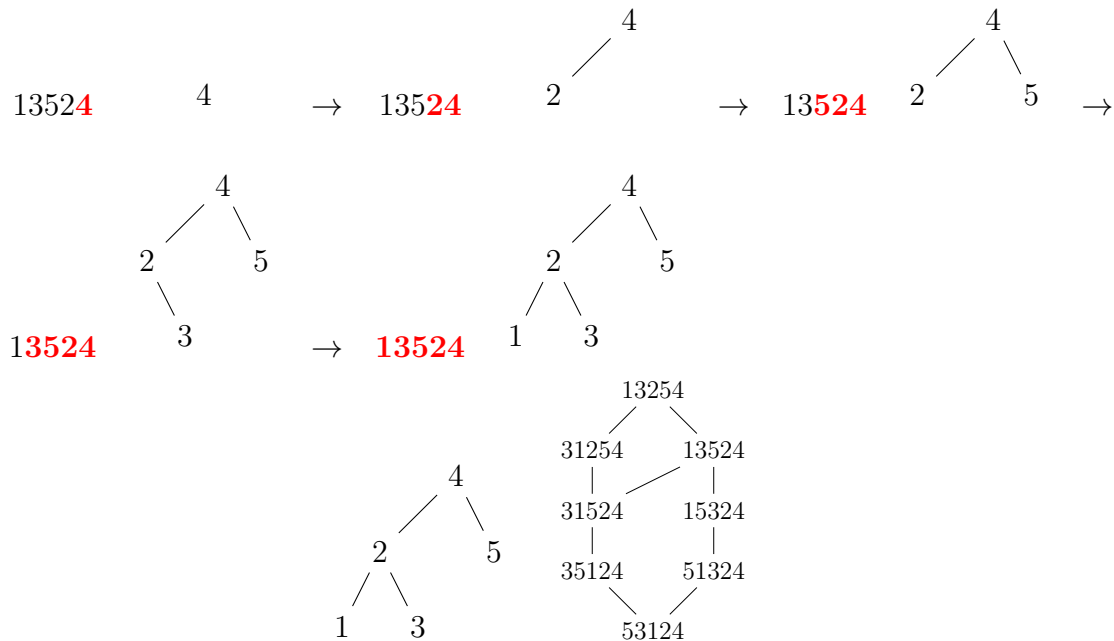

\centering
\input{includes/figures/tree-insertion}
\input{includes/figures/bst-example}
\caption{Insertion dans un arbre binaire de recherche et extensions linéaires}
\label{fig:tamari_prelim:bst-exemple}
\end{figure}

Un exemple d'arbre binaire de recherche est donné figure \ref{fig:tamari_prelim:bst-exemple}.
Pour un arbre binaire donné de taille $n$, il n'existe qu'un seul étiquetage, dit \emph{standard}, utilisant les entiers de $1$ à $n$ une fois chacun, tel que le résultat soit un arbre binaire de recherche. On identifie donc les arbres binaires non étiquetés et les arbres binaires de recherche standard. La structure d'arbre binaire de recherche est utilisée couramment en algorithmique pour le stockage d'ensembles ordonnés. En particulier, l'algorithme récursif d'insertion dans un arbre binaire de recherche est bien connu. L'insertion de l'entier $k$ dans l'arbre $T$ se fait de la façon suivante : si $T$ est vide alors $k$ devient la racine de $T$, sinon, si $k\leq racine(T)$ (resp. $k > racine(T)$) on insère $k$ dans le sous-arbre gauche (resp. droit) de $T$. 

On utilise cet algorithme pour associer un arbre binaire $\ABR(\sigma)$ à chaque permutation $\sigma$ : on insère successivement les entiers de la permutation de la droite vers la gauche dans un arbre $T$ vide au départ. Ce processus est illustré figure \ref{fig:tamari_prelim:bst-exemple}. La propriété suivante est déjà dans \cite{BW} et reprise dans \cite{PBT2} dans le cadre qui nous intéresse.

\begin{Proposition}
\index{classe sylvestre}
Les permutations qui donnent le même arbre binaire par insertion dans un arbre binaire de recherche sont les \emph{extensions linéaires} de cet arbre vu comme un poset. Elles forment un intervalle pour l'ordre faible droit. On appelle cet ensemble la \emph{classe sylvestre} de l'arbre.
\end{Proposition} 

\begin{figure}[p]
\begin{leftfullpage}
\centering
\input{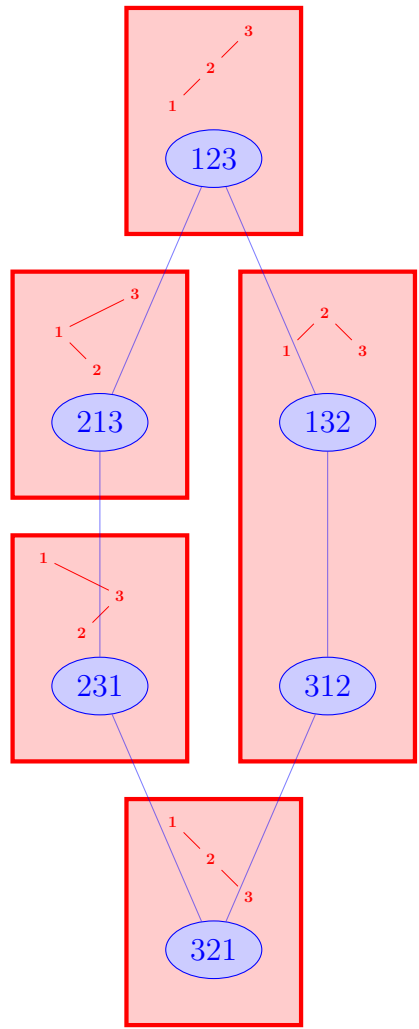}
\caption[L'ordre de Tamari comme quotient de l'ordre faible droit]{L'ordre de Tamari comme quotient de l'ordre faible droit : taille 3 sur la page de gauche et 4 sur la page de droite. Les permutations sont regroupées par classes d'équivalence. On pourra vérifier que si une relation existe entre deux permutations de classes différentes, alors les arbres binaires associés sont comparables dans l'ordre de Tamari.}
\label{fig:tamari_prelim:quotient}
\end{leftfullpage}
\end{figure}
\begin{figure}[p]
\begin{fullpage}
\centering
\input{includes/figures/tamari_quotient4}
\end{fullpage}
\end{figure}

En effet, un arbre binaire étiqueté peut être interprété comme un poset (où les éléments minimaux se trouvent en bas). Si $a$ se trouve dans le sous-arbre issu de $b$, on écrit $a \trprec b$. Une extension linéaire de l'arbre se définit alors comme au paragraphe \ref{sub-sec:prelim_posets:posets:ext} : si $a \trprec b$ alors $a$ se trouve avant $b$ dans l'extension linéaire. Dans la figure \ref{fig:tamari_prelim:bst-exemple}, on donne l'ensemble des extensions linéaires d'un arbre donné et on pourra vérifier qu'elles forment bien un intervalle pour l'ordre droit.

\`A partir de cette construction, on prouve une propriété très forte sur l'ordre faible et l'ordre de Tamari. On donnera une idée de la preuve donnée dans \cite{PBT2} à partir de la \emph{congruence sylvestre} au paragraphe \ref{sub-sec:tamari_prelim:pbt:sylv}.

\begin{Theoreme}
\label{thm:tamari_prelim:quotient}
L'ordre de Tamari est à la fois un sous-treillis de l'ordre faible et un quotient par la relation 
\begin{equation}
\sigma \equiv \mu \Leftrightarrow \ABR(\sigma) = \ABR(\mu).
\end{equation}
\end{Theoreme}

L'opération $\ABR(\sigma)$ découpe l'ordre faible en classes d'équivalences qui sont des intervalles pour l'ordre faible droit. Les éléments maximaux de ces intervalles sont les permutations qui évitent le motif $132$. Si on restreint l'ordre faible à ces permutations, on obtient un ordre isomorphe à l'ordre de Tamari. Cette propriété est aussi vraie sur les éléments minimaux des classes. Enfin l'ordre de Tamari est un quotient de l'ordre faible, ce qui s'exprime par

\begin{equation}
\sigma \leq \mu \Rightarrow \ABR(\sigma) \leq \ABR(\mu)
\end{equation}
où la relation à gauche est celle de l'ordre faible droit et à droite, celle de Tamari sur les arbres binaires. On pourra vérifier ces propriétés pour les tailles 3 et 4 dans la figure \ref{fig:tamari_prelim:quotient}.

On a vu que les ordres faibles, droit et gauche, étaient isomorphes. L'ordre de Tamari est donc aussi un sous-treillis et un treillis quotient de l'ordre faible gauche. Pour obtenir une description explicite, on utilise \emph{les arbres binaires décroissants}.

\begin{Definition}
\index{arbres!binaires!décroissants}
Un arbre binaire décroissant est un arbre binaire étiqueté tel que pour chaque \noeud $x$ étiqueté par $k$, les \noeuds de l'arbre issu de $x$ soient étiquetés par des entiers inférieurs ou égaux à $k$. 
\end{Definition}

Les permutations sont en bijection avec les arbres binaires décroissants. On obtient une permutation par un parcours infixe de l'arbre binaire décroissant (fils gauche, racine, fils droit). Réciproquement, à une permutation qui s'écrit $\sigma = unv$ où $n$ est la valeur maximale de $\sigma$, on fait correspondre récursivement l'arbre $T = n(T_u,T_v)$ où $T_u$ et $T_v$ sont les arbres des facteurs $u$ et $v$. On note cet arbre $\ABD(\sigma)$. La bijection est illustrée figure \ref{fig:tamari_prelim:decreasing}.

\begin{figure}[ht]
\centering
\input{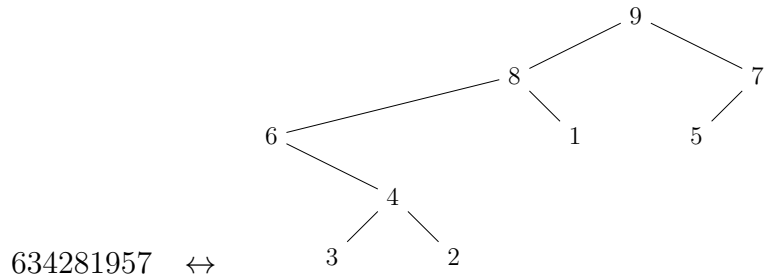}
\caption{Un arbre binaire décroissant et la permutation associée.}
\label{fig:tamari_prelim:decreasing}
\end{figure}

\begin{figure}[ht]
\centering
\input{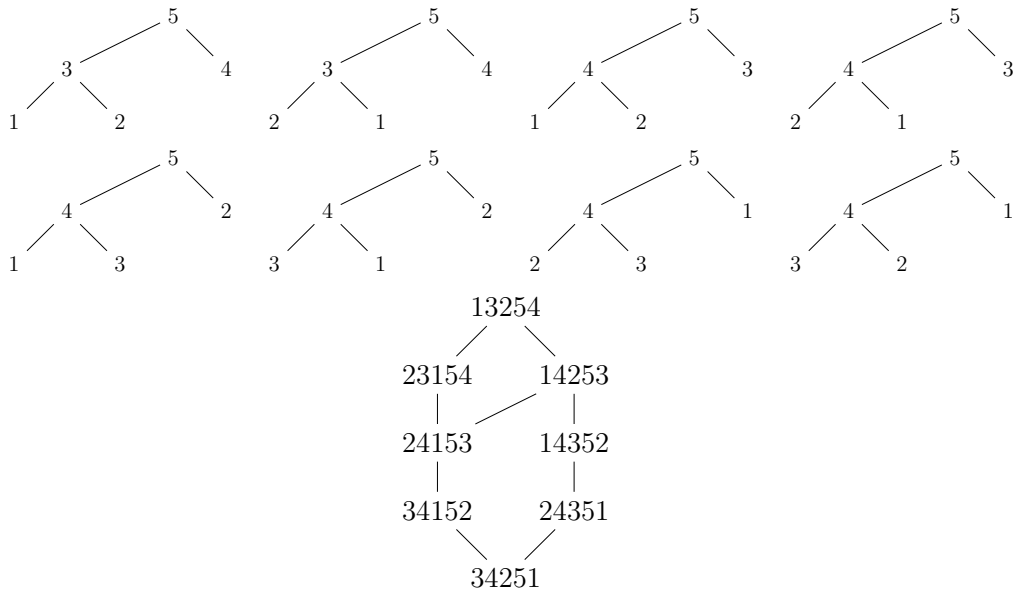}
\caption{Les arbres décroissants d'un arbre binaire donné et l'intervalle correspondant dans l'ordre gauche.}
\label{fig:tamari_prelim:arbres-decroissants}
\end{figure}

Le nombre d'étiquetages décroissants d'un arbre binaire donné est égal au nombre de ses extensions linéaires. Il est donné par la formule des équerres :
\begin{equation}
\frac{n!}{\prod_v h_v}
\end{equation}
où le produit est évalué sur les sommets $v$ de l'arbre et où $h_v$ est la taille du sous-arbre issu de $v$. Les permutations associées aux étiquetages décroissants d'un arbre binaire sont les inverses des extensions linéaires de son arbre binaire de recherche. Elles forment donc un intervalle de l'ordre faible gauche. Par ailleurs, l'ordre de Tamari est un quotient de l'ordre faible gauche par la relation $\sigma \equiv \mu$ si et seulement si les arbres décroissants associés à $\sigma$ et $\mu$ ont le même arbre binaire sous-jacent. Un exemple est donné figure \ref{fig:tamari_prelim:arbres-decroissants} d'un arbre binaire et de ses étiquetages décroissants.

\subsection{Arbres planaires}
\label{sub-sec:tamari_prelim:tamari:planaires}
\index{arbres!planaires}

L'ordre de Tamari peut aussi se décrire sur un autre type d'arbre : les \emph{arbres planaires enracinés}. Un arbre planaire se décrit récursivement comme étant un \noeud racine auquel est greffée une liste d'arbres planaires : les \emph{fils} du \noeud. La liste peut être vide et l'ordre des arbres dans la liste est important. La taille d'un arbre est donnée par son nombre de n\oe{}uds. Les arbres binaires de taille $n$ sont en bijection avec les arbres planaires de taille $n+1$ par l'opération suivante : soit $F$ l'arbre planaire associé à $T$, alors
\begin{enumerate}
\item Si $x$ est le fils gauche de $y$ dans $T$, alors $x$ est le \emph{frère gauche} de $y$ dans $F$
\item Si $x$ est le fils droit de $y$ dans $T$, alors $x$ est le fils de $y$ dans $F$.
\end{enumerate}
Récursivement, on part d'un \noeud racine $p_0$ dans $F$ et on insère les \noeuds de $T$. Si $T$ contient un unique \noeud $x$, il devient le fils de $p_0$. Sinon, on a $T = x(T_1,T_2)$, on insère d'abord $T_1$ puis le \noeud $x$ et enfin on insère $T_2$ dans l'arbre de racine $x$. Cette bijection met en lumière une structure récursive différente de la structure usuelle sur les arbres binaires : un arbre binaire est une liste d'autres arbres greffés sur sa branche gauche. La bijection est illustrée figure \ref{fig:tamari_prelim:binaire-planaire}.

La bijection peut aussi se faire directement avec les chemins de Dyck. Un chemin de Dyck s'écrit comme une suite de chemins primitifs, chaque chemin primitif correspondant à un fils de la racine. On lit sur les trois objets une statistique que nous utiliserons à plusieurs reprises : le nombre de \emph{\noeuds sur la branche gauche} de l'arbre binaire correspond au nombre de fils de la racine de l'arbre planaire et au nombre de chemins primitifs qui composent le chemin de Dyck qu'on appelle aussi nombre de \emph{retours à 0} \cite{FindStatTouchPoints, FindStatLeftBranch}.

\begin{figure}[ht]
\centering
\input{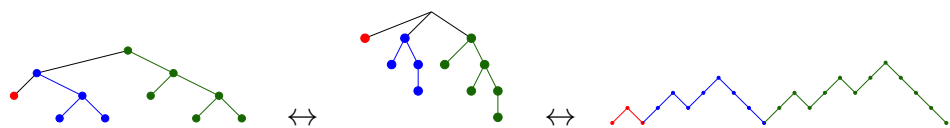}
\caption{Bijection arbre binaire / arbre planaire / chemin de Dyck}
\label{fig:tamari_prelim:binaire-planaire}
\end{figure}

\begin{figure}[p]
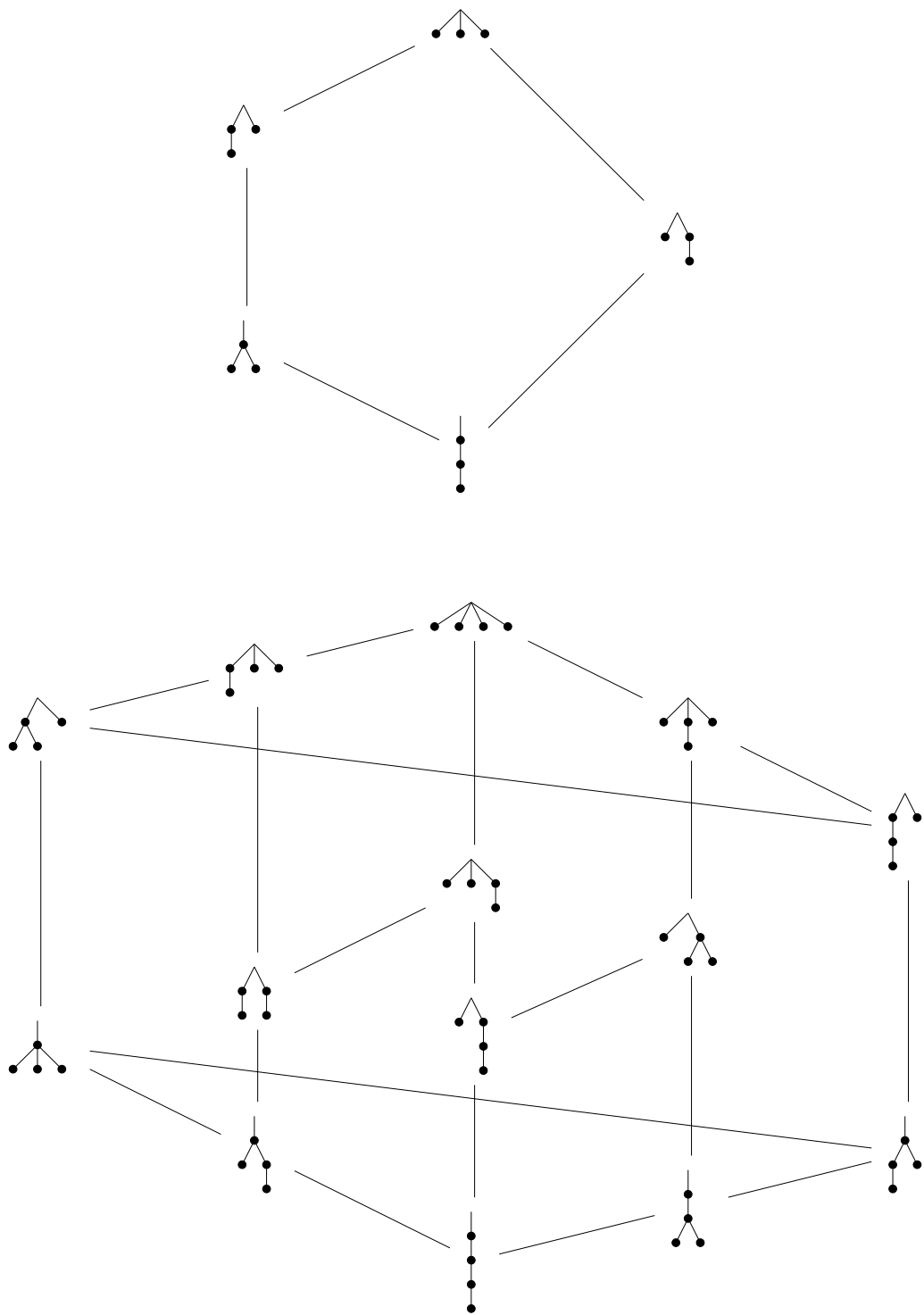

\centering
\input{includes/figures/tamari_planar-3}
\vspace{3em}
\input{includes/figures/tamari_planar-4}
\caption{Ordre de Tamari sur les arbres planaires, tailles 3 et 4}
\label{fig:tamari_prelim:planar-3-4}
\end{figure}

La relation de couverture donnée par la rotation se décrit simplement sur les arbres planaires comme le glissement d'un \noeud sur son frère gauche (cf. figure~\ref{fig:tamari_prelim:planar-3-4}). \`A partir de la bijection, on vérifie facilement que cette opération correspond bien aux rotations définies sur les chemins de Dyck et les arbres binaires.

Il est aussi possible d'effectuer la bijection symétrique entre les arbres binaires et les arbres planaires en transformant les \emph{fils droits} de l'arbre binaire en \emph{frères droits} de l'arbre planaire. Dans ce cas, le nombre de \noeuds sur la branche droite correspond au nombre de fils de la racine de l'arbre planaire. Ces deux bijections jouent un rôle important dans la suite de notre travail. Nous revenons dessus de façon détaillée dans le chapitre \ref{chap:tamari_intervalles}.

\section{Algèbres de Hopf combinatoires}
\label{sec:tamari_prelim:hopf}
\index{algèbre!de Hopf}

Le lien entre l'ordre de Tamari et l'ordre faible est en fait de nature \emph{algébrique} et passe par la définition de deux \emph{algèbres de Hopf} : $\FQSym$ et $\PBT$. Pour définir ces algèbres, nous avons besoin d'introduire les notions de base liées aux algèbres de Hopf en combinatoire.

Formellement, les algèbres de Hopf sont des objets algébriques assez complexes vérifiant de nombreux axiomes. Cependant, en combinatoire, les opérations de base que sont le \emph{produit} et le \emph{coproduit} se définissent très simplement sur les mots et les axiomes découlent alors naturellement. Plutôt que d'introduire un cadre formel pour en donner plus tard des exemples concrets, il nous a semblé plus naturel de partir des exemples de base pour introduire les notions algébriques qui y sont liées.

\subsection{Cogèbres et bigèbres : le doublement d'alphabet}
\label{sub-sec:tamari_prelim:hopf:def}

Soit $A$ un alphabet. On se place sur l'algèbre libre engendrée par $A$ que l'on note $\Aa = \KK \langle A \rangle$ où le produit est la concaténation des mots. Cette algèbre est tout simplement l'ensemble des combinaisons linéaires des mots sur $A$ aussi appelées \emph{polynômes non commutatifs}. Elle est clairement associative car

\begin{equation}
(u . v) . w = u . (v . w) = uvw
\end{equation}
pour $u,v,w \in A^*$. Le produit $\mu := .$ est une application linéaire $\mu : \Aa \otimes \Aa \rightarrow \Aa$ et l'associativité se traduit par le diagramme commutatif suivant :
\begin{equation}
\label{eq:tamari_prelim:diag_ass}
    \xymatrix{
        \Aa \otimes \Aa \otimes \Aa \ar[r]^-{I \otimes \mu} \ar[d]_-{\mu \otimes I} &
            \Aa \otimes \Aa \ar[d]^-\mu \\
        \Aa \otimes \Aa \ar[r]_-\mu &
        \Aa
    }
\end{equation}
où $I : \Aa \rightarrow \Aa$ est l'identité.

Le mot vide $\epsilon$ est donc l'élément neutre pour la multiplication, c'est-à-dire
\begin{equation}
\forall u \in A^*, u.\epsilon = \epsilon .u = u. 
\end{equation}
Le mot $\epsilon$ est l'unité de l'algèbre $\Aa$. Un élément $k \in \KK$ est assimilé à l'élément $k \epsilon \in \Aa$. On peut donc interpréter $\epsilon$ non pas comme un mot mais comme une application $\epsilon: \KK \rightarrow \Aa$. La propriété de l'unité s'exprime alors aussi par un diagramme.
\begin{equation}
\label{eq:tamari_prelim:diag_unit}
    \xymatrix{
        \Aa \otimes \KK \ar[r]^-{I \otimes \epsilon} \ar[dr]_-\simeq &
            \Aa \otimes \Aa \ar[d]^-p & \KK \otimes \Aa \ar[l]_-{\epsilon \otimes I} \ar[dl]^\simeq \\
        & \Aa &
    }
\end{equation}

\index{morphisme!d'algèbre}
\`A présent, si $A = \lbrace a, b, \dots \rbrace$, introduisons un nouvel alphabet $A' = \lbrace a', b', \dots \rbrace$ copie de l'alphabet $A$ et dont les lettres commutent avec celles de $A$. On définit l'opération $\Delta$ sur les lettres de $A$ par $\Delta(a) = a + a'$.  Puis on étend l'opération de telle sorte que $\Delta$ soit un morphisme d'algèbre, c'est-à-dire
\begin{equation}
\Delta(u.v) = \Delta(u).\Delta(v)
\end{equation}
pour $u,v \in A^*$. Par exemple, si $A = \lbrace a, b \rbrace$ on a
\begin{align}
\Delta(aab) &= (a + a')(a + a')(b + b') \\
&= aab + aab' + aa'b + aa'b' + a'ab + a'ab' + a'a'b + a'a'b' \\
&= aab + aab' + 2aba' + 2aa'b' + ba'a' + a'a'b'.
\end{align}

Un mot sur les alphabets $A$ et $A'$ se décompose en deux parties indépendantes : un mot sur $A$ et un mot sur $A'$. On peut le considérer comme un élément de $\Aa \otimes \Aa' \simeq \Aa \otimes \Aa$. L'opération $\Delta : \Aa \rightarrow \Aa \otimes \Aa$ est alors un morphisme d'algèbre tel que $\Delta(a) = a \otimes \epsilon + \epsilon \otimes a$. Le calcul précédent s'écrit 
\begin{equation}
\Delta(aab) = aab \otimes \epsilon + aa \otimes b + 2(ab \otimes a) + 2(a \otimes ab) + b \otimes aa + \epsilon \otimes aab.
\end{equation}

L'opération $\Delta$ est \emph{co-associative}. Pour un mot $u$, on a $\Delta(u) = \sum u_1 \otimes u_2$ sur les couples $(u_1,u_2)$ de sous-mots complémentaires de $u$. Alors $\sum \Delta(u_1) \otimes u_2 = \sum u_1 \otimes \Delta(u_2) = \sum u_1 \otimes u_2 \otimes u_3$, la somme sur les triplets complémentaires de sous-mots de $u$. Sur un exemple, cela donne

\begin{align}
(\Delta \otimes I)\Delta(ab) &= (\Delta \otimes I)\left( (a \otimes \epsilon) + (\epsilon \otimes a) \right) \left( (b \otimes \epsilon) + (\epsilon \otimes b) \right) \\
&= ( a \otimes \epsilon \otimes \epsilon + \epsilon \otimes a \otimes \epsilon + \epsilon \otimes \epsilon \otimes a)( b \otimes \epsilon \otimes \epsilon + \epsilon \otimes b \otimes \epsilon + \epsilon \otimes \epsilon \otimes b) \\
&= (I \otimes \Delta)\left( (a \otimes \epsilon) + (\epsilon \otimes a) \right) \left( (b \otimes \epsilon) + (\epsilon \otimes b) \right)\\
&= (I \otimes \Delta)\Delta(ab).
\end{align} 
On écrit le diagramme commutatif suivant
\begin{equation}
    \xymatrix{
        \Aa \ar[r]^-\Delta \ar[d]_-\Delta & \Aa \otimes \Aa \ar[d]^-{I \otimes \Delta} \\
        \Aa \otimes \Aa \ar[r]_-{\Delta \otimes I} & \Aa \otimes \Aa \otimes \Aa
    }.
\end{equation}
Ce diagramme est la version "renversée" du diagramme de l'associativité de l'algèbre \eqref{eq:tamari_prelim:diag_ass}. De la même façon, la \emph{co-unité} de $\Aa$ est l'application linéaire $c : \Aa \rightarrow \KK$ telle que $c(u) = 0$ si $u \neq \epsilon$ et $c(\epsilon) = 1$, on a
\begin{equation}
    \xymatrix{
        \Aa \otimes \KK & \Aa \otimes \Aa \ar[r]^-{c \otimes I} \ar[l]_-{I \otimes c} &
            \KK \otimes \Aa \\
        & \Aa \ar[lu]^-\simeq \ar[u]_-\Delta \ar[ru]_-\simeq &
    }
\end{equation}
version renversée du diagramme \eqref{eq:tamari_prelim:diag_unit} de l'unité de l'algèbre. Munie de l'opération $\Delta$, l'espace vectoriel $\Aa$ possède une structure de \emph{cogèbre}. La compatibilité entre $\Delta$ et $\mu$ en fait une \emph{bigèbre}.

\begin{Definition}
\index{cogèbre}
\index{bigèbre}
\index{algèbre!de Hopf}
Soit $\Aa$ un espace vectoriel. Si $\Aa$ est muni d'un produit associatif $\mu : \Aa \otimes \Aa \rightarrow \Aa$ et d'une unité $\epsilon : \KK \rightarrow \Aa$, on dit que $\Aa$ est une \emph{algèbre}. Si par ailleurs, $\Aa$ est munie d'un \emph{coproduit} $\Delta : \Aa \rightarrow \Aa \otimes \Aa$ co-associatif et d'une co-unité $c: \Aa \rightarrow \KK$ alors $\Aa$ est une \emph{cogèbre}.

Enfin, si $\Aa$ est à la fois une algèbre et une cogèbre et que $\Delta$ et $c$ sont des \emph{morphismes d'algèbres} alors, $\Aa$ est une \emph{bigèbre}.
\end{Definition}

En tant qu'espace vectoriel, $\Aa$ est \emph{gradué}, c'est-à-dire que

\begin{equation}
\index{algèbre!graduée}
\Aa = \bigoplus_{n \in \NN} \Aa^n 
\end{equation}
où $\Aa^n$ est l'espace vectoriel sur les mots de taille $n$. En tant que bigèbre, $\Aa$ est aussi graduée ce qui signifie que son produit $\mu$ et son coproduit $\Delta$ vérifient
\begin{align}
\mu(\Aa^n \otimes \Aa^m) &\subset \Aa^{n+m},\\
\Delta(\Aa^k) &\subset \bigoplus_{n+m=k} \Aa^n \otimes \Aa^m.
\end{align}
Par ailleurs la dimension de $\Aa^0$ est 1 (le mot vide $\epsilon$) ce qui signifie que $\Aa$ est \emph{connexe}. On peut alors prouver que \emph{l'antipode} de $\Aa$ est bien définie \cite{Hopf1, Hopf2}. Nous ne détaillerons pas cette propriété car nous n'en aurons pas besoin. C'est elle qui justifie l'appellation \emph{algèbre de Hopf}. Par la suite, nous n'étudierons que des bigèbres graduées et connexes et parlerons donc toujours \emph{d'algèbres de Hopf combinatoires}.

Si l'alphabet $A$ n'est pas commutatif, le produit $\mu$ non plus. On a $u.v \neq v.u$. Cependant, le coproduit $\Delta$ que nous avons défini est \emph{co-commutatif} : soit $\omega : \Aa \otimes \Aa \rightarrow \Aa \otimes \Aa$ l'application définie par $\omega(u \otimes v) = v \otimes u$, alors $\omega \circ \Delta = \Delta$. Par la suite, on définira d'autres coproduits qui n'auront pas cette propriété.

L'exemple des polynômes non commutatifs, s'il parait trivial, n'en est pas moins fondamental. En effet, il est souvent fastidieux de prouver tous les axiomes relatifs au produit et au coproduit sur une bigèbre. Dans le cas de $\Aa$, ce sont des propriétés élémentaires. Une technique est alors d'exprimer une algèbre combinatoire en fonction de $\Aa$. Cela revient à associer à chaque objet un développement sous forme de polynômes (commutatifs ou non). Pour prouver que l'espace en question possède une structure d'algèbre de Hopf, on prouve que la famille de polynômes obtenue est stable par les opérations de produit et de coproduit. Cette technique est appelée la \emph{réalisation polynomiale}, nous en donnons un exemple au paragraphe suivant avec l'algèbre $\FQSym$.

\begin{Remarque}
Les fonctions symétriques que nous avons étudiées dans le chapitre \ref{chap:polynomes_action} possèdent aussi une structure d'algèbre de Hopf et la définition du coproduit passe aussi par un doublement d'alphabet.
\end{Remarque}

\subsection{L'algèbre des fonctions quasi-symétriques libres}
\label{sub-sec:tamari_prelim:hopf:fqsym}
\index{FQSym}

Pour illustrer le principe de la réalisation polynomiale, nous expliquons la construction de l'algèbre de Hopf des fonctions quasi-symétriques libres $\FQSym$ comme cela a été fait dans \cite{NCSF6,NCSF7}. Cette algèbre est isomorphe à l'algèbre définie par Malvenuto et Reutenauer sur les permutations \cite{MalReut}.

\begin{Definition}
Soit $u = u_1 \dots u_n$ un mot sur l'alphabet ordonné $A$. L'action d'une permutation $\sigma$ sur $u$ est donnée par
\begin{equation}
u \bullet \sigma = u_{\sigma_1}\dots u_{\sigma_n}.
\end{equation}
 \emph{L’exécution} de $u$, $\exec(u)$, est la permutation $\sigma$ de longueur minimale telle que $u \bullet \sigma$ soit ordonné c'est-à-dire $u_{\sigma_1} \leq u_{\sigma_2} \leq \dots \leq u_{\sigma_n}$. 
\end{Definition}

Par exemple, si $u=baa$ alors $\exec(u) = 231$, et on a $u \bullet \sigma = aab$. L'exécution d'un mot $u$ dépend uniquement de ses inversions, c'est-à-dire du standardisé de $u$ tel que nous l'avons défini paragraphe \ref{sub-sec:prelim_posets:struct_elementaires:operations}. On a $\exec(u) = \exec(\std(u))$. Par ailleurs, pour une permutation $\nu$ alors $\nu \bullet \sigma = \nu \circ \sigma$ et donc $\exec(\nu) = \nu^{-1}$. On a donc
\begin{equation}
\exec(u) = \std(u)^{-1}.
\end{equation}

On se place à présent sur un alphabet infini et on définit la série
\begin{equation}
\BF_\sigma := \sum_{\exec(u) = \sigma} u.
\end{equation}

On a alors le résultat suivant \cite{NCSF6}.

\begin{Proposition}
Le produit $\BF_\sigma \BF_\mu$ pour $\sigma \in \Sym{n}$ et $\mu \in \Sym{m}$ s'exprime comme une somme d'éléments $\BF_{\nu}$ avec $\nu \in \Sym{n+m}$. Plus précisément
\begin{equation}
\label{eq:tamari_prelim:fqsym_prod_f}
\BF_\sigma \BF_\mu = \sum_{\nu \in \sigma \cshuffle \mu} \BF_{\nu}.
\end{equation}
\end{Proposition}

On rappelle que $\cshuffle$ est le produit de mélange décalé sur les permutations que nous avons défini paragraphe \ref{sub-sec:prelim_posets:struct_elementaires:alg}. Voyons le résultat sur un exemple. Soit $\sigma = 21$ et $\mu = 1$. Comme le résultat sera une combinaison linéaire de permutations de taille 3, on peut se contenter du développement de $\BF_\sigma$ et $\BF_\mu$ sur $A = \lbrace a,b,c \rbrace$. On a alors
\begin{align}
\BF_{21} &= ba + ca + cb, \\
\BF_{1} &= a + b + c, \\
\BF_{21} \BF_{1} &= baa + bab + bac + caa + cab + cac + cba + cbb + cbc \\
&= (baa +caa + cab +cbb) + (bab + bac + cac +cbc) + cba \\
&= \BF_{231} + \BF_{213} + \BF_{321}.
\end{align}

On a donc une structure d'algèbre sur les éléments $\BF_\sigma$. La définition du coproduit n'est pas celle du paragraphe \ref{sub-sec:tamari_prelim:hopf:def} mais utilise toujours un doublement d'alphabet. Soit $B$ un second alphabet ordonné infini tel que les lettres de $A$ soient considérées plus petites que les lettres de $B$ et qu'elles commutent avec les lettres de $B$. L'alphabet formé de cette façon est noté $A \cp B$. Le coproduit sur $\BF_\sigma$ consiste à développer la somme sur $A \cp B$ plutôt que sur $A$. Par exemple, si $A = \lbrace a, b, c \rbrace$ et $B = \lbrace a', b', c' \rbrace$,

\begin{align}
\BF_{231}(A \cp B) &= baa +caa + cab +cbb +  \\
\nonumber
&+ a'aa + a'ab + a'bb + b'aa + \dots \\
\nonumber
&+  b'aa' + c'aa' + c'ab'+ \dots \\
\nonumber
&+ b'a'a' + c'a'a' + c'a'b' + c'b'b'.
\end{align}

Comme les lettres de $A$ et $B$ commutent, on peut les réordonner pour séparer les deux alphabets et exprimer la somme dans $\Aa \otimes \Aa$ comme nous l'avons fait paragraphe \ref{sub-sec:tamari_prelim:hopf:def}. 

\begin{align}
\Delta(\BF_{231}) &= (baa \otimes \epsilon) + (caa \otimes \epsilon) + (cab \otimes \epsilon) + (cbb \otimes \epsilon) \\
\nonumber
&+ (aa \otimes a) + (ab \otimes a) + (bb \otimes a) + (aa \otimes b) + \dots \\
\nonumber
&+  (a \otimes ba) + (a \otimes ca) + (a \otimes cb) + \dots \\
\nonumber
&+ (\epsilon \otimes baa) + (\epsilon \otimes caa) + (\epsilon \otimes cab) + (\epsilon \otimes cbb).
\end{align}

Les séries qui apparaissent à gauche et à droite du signe $\otimes$ correspondent à des éléments $\BF$ et on a
\begin{equation}
\Delta(\BF_{231}) = \BF_{231} \otimes 1 + \BF_{12} \otimes \BF_{1} + \BF_{1} \otimes \BF_{21} + 1 \otimes \BF_{231}.
\end{equation}
Ici, $1$ désigne l'unité de l'algèbre, c'est-à-dire $1 = \BF_{\epsilon} = \epsilon$. De façon générale, on a

\begin{Proposition}
\begin{equation}
\label{eq:tamari_prelim:fqsym_coprod_f}
\Delta(\BF_\sigma) := \BF_\sigma(A \cp B)  = \sum_{\sigma = u.v} \BF_{\std(u)} \otimes \BF_{\std(v)}.
\end{equation}
\end{Proposition}

Ce résultat est prouvé dans \cite{NCSF6} et permet d'obtenir la proposition suivante.

\begin{Proposition}
L'algèbre des éléments $(\BF_{\sigma})$ munie du produit de mélange décalé \eqref{eq:tamari_prelim:fqsym_prod_f} et du coproduit \eqref{eq:tamari_prelim:fqsym_coprod_f} est une algèbre de Hopf.
\end{Proposition}

On note cette algèbre $\FQSym$. Ce résultat était déjà donné dans \cite{MalReut} et prouvé directement sur les permutations. Dans \cite{NCSF6}, les auteurs utilisent comme nous l'avons vu le développement des éléments $\BF_\sigma$ comme sommes de mots. Dans ce cas, la preuve devient beaucoup plus simple car il ne reste à prouver que la stabilité des éléments $\BF_\sigma$ par le produit et le coproduit. En effet, la co-associativité du coproduit est triviale car $(A \cp B) \cp C = A \cp (B \cp C)$. De même, la compatibilité du produit et du coproduit s'obtient facilement car, sur des alphabets infinis, le développement de $\BF_\sigma(A \cp B)\BF_\mu(A \cp B)$ est équivalent au développement de $\BF_\sigma(A)\BF_\mu(A)$. Notons que dans le cas de $\FQSym$, le coproduit n'est plus le même qu'au paragraphe \ref{sub-sec:tamari_prelim:hopf:def} et n'est plus cocommutatif.

\subsection{Ordre faible, dualité et autres bases}
\label{sub-sec:tamari_prelim:hopf:dualite}
\index{algèbre!duale}

Les éléments $(\BF_\sigma)$ sont appelés la \emph{base fondamentale} de l'algèbre $\FQSym$. Il est possible de définir d'autres bases sur $\FQSym$ en particulier par l'opération de \emph{dualité}. La notion du dual d'une algèbre de Hopf étend celle du dual d'un espace vectoriel. Soit $E$ un espace vectoriel de base $B$. Son dual est l'espace $E^*$ des formes linéaires sur $E$. Comme on travaille en dimension finie (ou graduellement finie), $E^*$ est isomorphe à $E$. Sa base est donnée par $B^*$ tel que $b_i^*(bj) = \delta_{i,j}$ pour $b_i^* \in B^*$ et $b_j \in B$. On utilise la notation du crochet de dualité $\langle \phi, x \rangle := \phi(x)$ pour $\phi \in E^*$ et $x \in E$. En particulier, si $\phi \in E^*$ et $b_i \in B$, alors $\langle \phi , b_i \rangle$ est le coefficient de $b_i^*$ dans $\phi$.

Soit $(\Aa, ., \Delta)$, une algèbre de Hopf, son algèbre duale est $(\Aa^*,.^*, \Delta^*)$ où $\Aa^*$ est le dual de $\Aa$ en tant qu'espace vectoriel et où le produit $.^*$ et le coproduit $\Delta^*$ sont définis par

\begin{align}
\label{eq:tamari_prelim:dual_coprod}
\langle \Delta^*(z), x \otimes y \rangle &= \langle z, x.y \rangle & \forall x,y \in \Aa, z \in \Aa^*, \\
\label{eq:tamari_prelim:dual_prod}
 \langle y.^*z, x \rangle &= \langle y \otimes z, \Delta(x) \rangle & \forall x \in \Aa, y,z \in \Aa^*.
\end{align}

Soit $(\BG_\sigma)$ la base duale de $(\BF_\sigma)$. Par \eqref{eq:tamari_prelim:dual_prod}, on a
\begin{align}
\BG_\sigma \BG_\mu &= \sum_{\BF_\sigma \otimes \BF_\mu \in \Delta(\BF_\nu)} \BG_\nu, \\
&= \sum_{\substack{\nu = u.v \\ \std(u) = \sigma \\ \std(v) = \mu}} \BG_\nu.
\end{align}
Et par \eqref{eq:tamari_prelim:dual_coprod},
\begin{align}
\Delta(\BG_\sigma) &= \sum_{\sigma \in \mu \cshuffle \nu} \BG_\mu \otimes \BG_\nu.
\end{align}

Le coproduit de $\BF_\sigma$ revenait à sommer sur les découpages de la permutation $\sigma$ en deux blocs en fonction d'une position $k$. Sur la base $\BG_\sigma$, on découpe en fonction d'une valeur $k$ : la partie gauche du produit tensoriel est formée par le sous-mot des valeurs inférieures ou égales à $k$ et la partie droite, par le sous-mot standardisé des valeurs supérieures à $k$. De même, pour le produit de $\BF_\sigma$ et $\BF_\mu$ avec $\sigma \in \Sym{k_1}$ et $\mu \in \Sym{k_2}$, on sommait sur les $\binom{k_1 + k_2}{k_1}$ façons de choisir $k_1$ positions où placer les valeurs de $\sigma$. Dans le produit $\BG_\sigma \BG_\mu$, on choisit les $k_1$ premières valeurs et leur ordre est donnée par $\sigma$.

Voyons sur un exemple,
\begin{align}
\BG_{21} \BG_{1} &= \BG_{\red{21}3} + \BG_{\red{31}2} + \BG_{\red{32}1}, \\
\Delta(\BG_{312}) &= \BG_{312} \otimes 1 + \BG_{12} \otimes \BG_{1} + \BG_{1} \otimes \BG_{21} + 1 \otimes \BG_{312}.
\end{align}

Cette dualité positions / valeurs rappelle la relation entre les ordres faibles droits et gauches sur les permutations. Et en effet, on prouve sur le développement en mot que le produit et le coproduit des éléments $\BF_{\sigma^{-1}}$ vérifient les relations \eqref{eq:tamari_prelim:dual_coprod} et \eqref{eq:tamari_prelim:dual_prod}, ce qui nous donne le résultat suivant.

\begin{Proposition}
En tant qu'algèbre de Hopf, $\FQSym$ est \emph{auto-duale} ce qui signifie qu'il existe un isomorphisme de bigèbre entre $\FQSym$ et $\FQSym^*$. De plus, l'application $\BG_\sigma \rightarrow \BF_{\sigma^{-1}}$ est un isomorphisme explicite.
\end{Proposition}

Par la suite, on identifiera $\FQSym$ et son dual. On a alors que $\BG_\sigma = \BF_{\sigma^{-1}} = \sum_{\std(u) = \sigma} u$ est une autre base de $\FQSym$, duale de la base $\BF$. Le produit de $\BF_\sigma$ et $\BF_\mu$ est le produit de mélange décalé. C'est une somme sur l'intervalle de l'ordre faible droit entre $\sigma . \Dec{\mu}$ et $\Dec{\mu}.\sigma$ où $\Dec{\mu}$ est la permutation $\mu$ décalée de $|\sigma|$. Le produit $\BG_\mu \BG_\sigma$ est une somme sur l'intervalle de l'ordre faible gauche entre $\sigma . \Dec{\mu}$ et $\Dec{\sigma} . \mu$. Par exemple,

\begin{align}
\BF_{231} \BF_{12} &= \sum_{\nu \in [23145,45231]_R} \BF_\nu \\
\BG_{312} \BG_{12} &= \sum_{\nu \in [31245,53412]_L} \BG_\nu.
\end{align} 

On définit deux autres bases de $\FQSym$, les fonctions \emph{élémentaires} $\BE$ et \emph{homogène} $\BH$ par
\begin{align}
\BE^\sigma &:= \sum_{\sigma' \geq_R \sigma} \BF_{\sigma'} \\
\BH^\sigma &:= \sum_{\sigma' \leq_R \sigma} \BF_{\sigma'}.
\end{align}

Ce sont des sommes sur des intervalles respectivement initiaux et finaux de l'ordre droit. On a
\begin{align}
\BE^{\sigma} \BE^{\mu} &= \BE^{\sigma . \Dec{\mu}} \\
\BH^{\sigma} \BH^{\mu} &= \BH^{\Dec{\mu}.\sigma}
\end{align}
où $\Dec{\mu}$ est la permutation $\mu$ décalée de $|\sigma|$. Ces bases sont dites \emph{multiplicatives} car le résultat du produit est donné par un unique élément.

\section{L'algèbre de Hopf sur les arbres binaires $\PBT$}
\label{sec:tamari_prelim:pbt}
\index{PBT}

La structure de l'algèbre $\FQSym$ est fortement liée à l'ordre faible. On a vu que l'ordre de Tamari était un sous-treillis et un treillis quotient de l'ordre faible. Cette relation passe en fait par le lien algébrique entre l'algèbre $\FQSym$ sur les permutations et l'algèbre $\PBT$ sur les arbres binaires que nous allons définir maintenant. 

\subsection{Congruence sylvestre}
\label{sub-sec:tamari_prelim:pbt:sylv}
\index{congruence sylvestre}

Dans le paragraphe \ref{sub-sec:tamari_prelim:tamari:ordre_faible}, nous avons décrit l'algorithme qui à chaque permutation $\sigma$ associe son arbre binaire de recherche $\ABR(\sigma)$. L'algorithme est basé sur l'opération bien connue d'insertion dans un arbre binaire de recherche et s'applique aussi aux mots sur un alphabet ordonné qui ne sont pas des permutations. \`A chaque mot $u$, on associe $\ABR(u)$ un arbre binaire de recherche étiqueté avec les lettres de $u$. Deux mots $u$ et $v$ ont le même arbre binaire de recherche s'ils sont liés par la \emph{congruence sylvestre} \cite{PBT2}.

\begin{Definition}
\label{def:tamari_prelim:congru_sylvestre}
Deux mots $w_1$ et $w_2$ sont dits \emph{adjacents par la relation sylvestre} si on a
\begin{align}
w_1 = u~ac~v~b~w,~~~ w_2 = u~ca~v~b~w
\end{align}
où $u,v,w$ sont des mots et $a,b,c$ des lettres telles que $a \leq b < c$. 

La \emph{congruence sylvestre} est la clôture transitive de la relation d'adjacence sylvestre. On a $u \equiv v$ s'il existe une chaîne de mots 
\begin{equation}
u = w_1, w_2, \dots, w_k = v
\end{equation}
telle que $w_i$ soit adjacent à $w_{i+1}$ pour $1 \leq i < k$.
\end{Definition}

Cette relation est clairement une relation de congruence. De plus, elle est compatible avec la concaténation : si $u \equiv u'$ et $v \equiv v'$ alors $u.v \equiv u'.v'$. Par ailleurs si $I = \lbrace a_i, a_{i+1}, \dots, a_{i+k} \rbrace$ est un intervalle de l'alphabet $A$ et si $u \equiv v$ alors $u_I \equiv v_I$ où $u_I$ (resp. $v_I$) est le sous-mot de $u$ (resp. $v$) restreint aux lettres de $I$. On dit que la congruence sylvestre est compatible à la restriction aux intervalles. Enfin, on a $u \equiv v$ si et seulement si $\std(u) \equiv \std(v)$.

On a que $\ABR(u)$ a la même forme, c'est-à-dire le même arbre binaire non étiqueté sous-jacent, que $\ABR(\std(u))$. Par ailleurs, la construction de l'arbre binaire décroissant $\ABD$ d'une permutation s'étend elle aussi aux mots en posant $\ABD(u) = \ABD(\std(u))$. On pose alors
\begin{equation}
\BP_T = \sum_{\shape(\ABD(u))=T} u,
\end{equation} 
où $T$ est un arbre binaire et $\shape(\ABD(u))$ la forme de l'arbre étiqueté $\ABD(u)$. Pour un mot $u$ tel que $\std(u) = \sigma$, on a d'après \cite[Lemme 11]{PBT2} que $\ABR(u)$ a la même forme que $\ABR(\sigma)$ et $\ABD(\sigma^{-1})$ ce qui permet d'exprimer $\BP_T$ comme une somme sur $\FQSym$,
\begin{equation}
\BP_T = \sum_{\shape(\ABR(\sigma)) = T} \BF_\sigma.
\end{equation}

On définit ainsi une sous-algèbre de Hopf de $\FQSym$. En effet, la stabilité du produit et celle du coproduit sont données par la compatibilité de la congruence sylvestre avec la déstandardisation et la restriction aux intervalles. Par ailleurs, cette algèbre est isomorphe à celle définie sur les arbres binaires par Loday et Ronco \cite{PBT1}. On la note $\PBT$. On a par exemple,

\begin{equation}
\label{eq:tamari_prelim:ex_pbt}
\BP_{
\scalebox{.5}{\input{includes/figures/trees/T4-7}}
}  = \BF_{2143} + \BF_{2413} + \BF_{4213}.
\end{equation}

\begin{Definition}
Soit $T$ un arbre binaire. Le \emph{mot canonique} de $T$ est la permutation $\omega_T$ qui correspond à la lecture suivante de l'arbre binaire de recherche de $T$ : fils droit, fils gauche, racine.
\end{Definition}

\index{classe sylvestre}
Dans l'exemple \eqref{eq:tamari_prelim:ex_pbt}, le mot canonique est $4213$. Soit $\sigma$ une permutation, on appelle les permutations ayant le même arbre binaire que $\sigma$ la \emph{classe sylvestre} de $\sigma$. Comme on l'a vu paragraphe \ref{sub-sec:tamari_prelim:tamari:ordre_faible}, les classes sylvestres sont des extensions linéaires d'arbres et forment donc des intervalles de l'ordre faible. La définition de ces classes en terme de la congruence sylvestre donnée par la définition \ref{def:tamari_prelim:congru_sylvestre} permet de prouver le théorème \ref{thm:tamari_prelim:quotient}. En effet, on prouve que les mots canoniques des classes sylvestres sont les permutations évitant 132 et que l'ordre restreint à ces permutations est isomorphe à l'ordre de Tamari. Enfin on prouve que si $\sigma \leq \sigma'$  avec $\ABR(\sigma) = T$ et $\ABR(\sigma') = T'$, alors $\omega_{T} \leq \omega_{T'}$. On trouvera la preuve complète dans \cite{PBT2}.

\subsection{Produit et coproduit}
\label{sub-sec:tamari_prelim:pbt:produit-coproduit}

Le développement en termes de $\FQSym$ des éléments $\BP_T$ permet d'obtenir des formules simples pour le produit et le coproduit.

\begin{Proposition}
\label{prop:tamari_prelim:pbt-product}
Soient $T_1$ et $T_2$ deux arbres binaires, le \emph{shuffle} de $T_1$ et $T_2$, noté $T_1 \shuffle T_2$, est l'ensemble  des arbres $T$ tel que $\omega_T$ apparaisse dans le produit de mélange décalé $\omega_{T_1} \cshuffle \omega_{T_2}$. On a alors
\begin{equation}
\label{eq:tamari_prelim:pbt-product}
\BP_{T_1} \BP_{T_2} = \sum_{T \in T_1 \shuffle T_2} \BP_T.
\end{equation}
\end{Proposition}

On trouve la preuve de cette propriété dans \cite{PBT2}. Elle vient du fait qu'une permutation canonique ne peut être issue que du produit de mélange de deux permutations canoniques. On peut alors indexer les éléments $\BP$ par des permutations canoniques plutôt que par des arbres. On a par exemple
\begin{equation}
\BP_{4213} \BP_{312} = \BP_{7421356} + \BP_{7452136} + \BP_{7456213} + \BP_{7542136} + \BP_{7546213} + \BP_{7564213}.
\end{equation}

Le développement de ce produit sur les $\BF_\sigma$ s'exprime comme une somme sur un intervalle de l'ordre faible droit. Dans l'exemple précédent c'est l'intervalle entre la permutation $2143576$, élément minimal de $\BP_{7421356}$ et $7564213$. C'est aussi un intervalle pour l'ordre de Tamari comme illustré figure \ref{fig:tamari_prelim:prod_pbt}.

\begin{figure}[ht]
\centering
\input{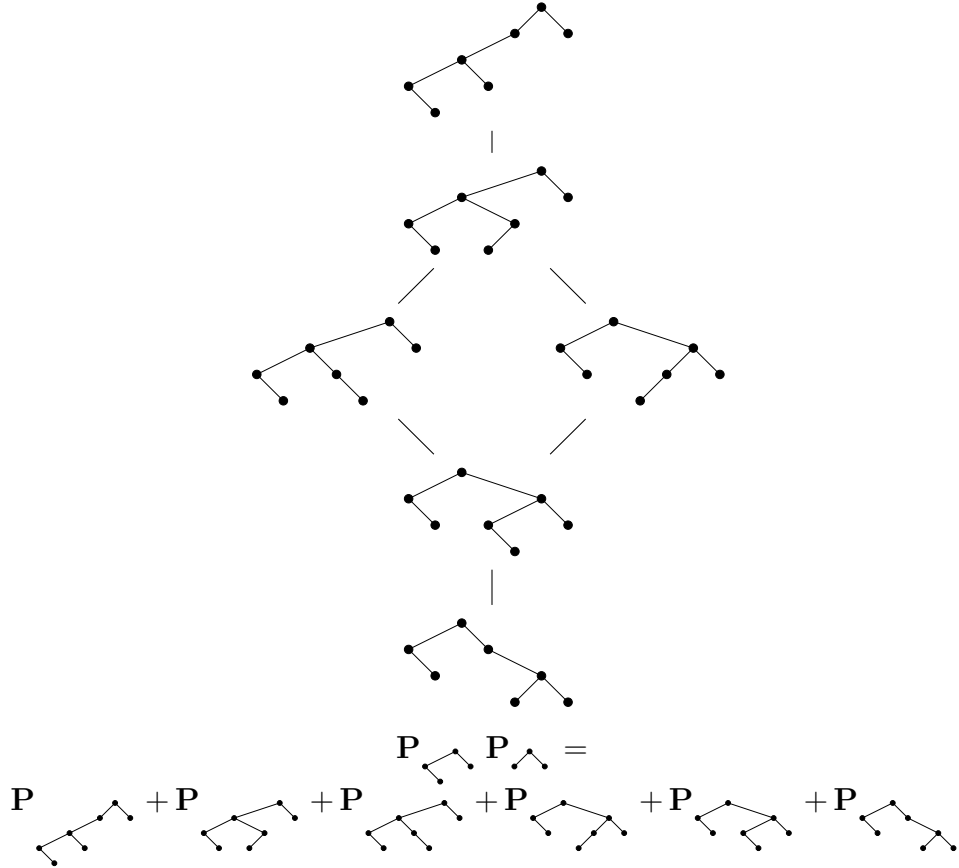}
\caption[Produit dans $\PBT$]{Le produit entre $T_1$ et $T_2$ donne l'intervalle entre l'arbre $T$ où $T_2$ est placé sous le dernier fils gauche de $T_1$ et $T'$ où $T_1$ est placé sous le dernier fils droit de $T_2$.}
\label{fig:tamari_prelim:prod_pbt}
\end{figure}

\begin{Proposition}
Soit $T$ un arbre binaire, le coproduit de $\BP_T$ est donné par
\begin{equation}
\Delta(\BP_T) = \sum \BP_{T'} \otimes \BP_{T''}
\end{equation}
sommé sur les couples d'arbres binaires $(T',T'')$ tels que $\omega_{T'} = \std(w_1)$, $\omega_{T''} = \std(w_2)$ avec $w_1.w_2$ appartenant à la classe sylvestre de $T$.
\end{Proposition}

Par exemple, la classe sylvestre de $\sigma = 4213$ est donnée par $\lbrace 4213, 2413, 2143 \rbrace$ et on a
\begin{align}
\Delta(\BP_{4213}) &= \BP_{4213} \otimes 1 + (\BP_{321} + \BP_{231} + \BP_{213}) \otimes \BP_{1} + (\BP_{21} + \BP_{12}) \otimes \BP_{12} + \BP_{21} \otimes \BP_{21} \\
\nonumber &+ \BP_{1} \otimes (\BP_{213} + \BP_{312}) + 1 \otimes \BP_{4213} 
\end{align}

\subsection{Bases multiplicatives}
\label{sub-sec:tamari_prelim:pbt:mult}

On définit deux autres bases sur $\PBT$, analogues des bases élémentaires et complètes de $\FQSym$ :

\begin{align}
\BH_T &= \sum_{T' \leq T} \BP_{T'}, \\
\BE_T &= \sum_{T' \geq T} \BP_{T'}
\end{align}
où l'ordre utilisé est l'ordre de Tamari sur les arbres binaires. On trouve dans \cite[Théorèmes 29 et 30]{PBT2} la preuve que ces bases sont multiplicatives. On a
\begin{equation}
\BH_{T_1} \BH_{T_2} = \BH_T
\end{equation} 
où $T$ est l'arbre obtenu en greffant $T_2$ à droite du fils le plus à droite de $T_1$. En particulier, $T$ est l'élément maximal de l'intervalle donné par $P_{T_1} P_{T_2}$. Et
\begin{equation}
\BE_{T_1} \BE_{T_2} = \BE_{T}
\end{equation}
où $T$ est l'arbre obtenu en greffant $T_1$ à gauche du fils le plus à gauche de $T_2$, c'est l'élément minimal de l'intervalle donné par $P_{T_1} P_{T_2}$. Par exemple,
\begin{align}
\BH_{\scalebox{0.5}{\input{includes/figures/trees/T4-7}}}
\BH_{\scalebox{0.5}{

%

\begin{tikzpicture}
\node (N0) at (0.750, 0.000){};
\node (N00) at (0.250, -0.500){};
\draw[Point] (N00) circle;
\node (N01) at (1.250, -0.500){};
\draw[Point] (N01) circle;
\draw (N0.center) -- (N00.center);
\draw (N0.center) -- (N01.center);
\draw[Point] (N0) circle;
\end{tikzpicture}}}
&= \BH_{\scalebox{0.5}{\input{includes/figures/trees/T7-ex6}}} \\
\BE_{\scalebox{0.5}{\input{includes/figures/trees/T4-7}}}
\BE_{\scalebox{0.5}{

%

\begin{tikzpicture}
\node (N0) at (0.750, 0.000){};
\node (N00) at (0.250, -0.500){};
\draw[Point] (N00) circle;
\node (N01) at (1.250, -0.500){};
\draw[Point] (N01) circle;
\draw (N0.center) -- (N00.center);
\draw (N0.center) -- (N01.center);
\draw[Point] (N0) circle;
\end{tikzpicture}}}
&= \BE_{\scalebox{0.5}{\input{includes/figures/trees/T7-ex1}}}
\end{align}

Enfin, tout comme $\FQSym$, l'algèbre $\PBT$ est auto-duale. L'isomorphisme est explicite et donné dans \cite{PBT2}.


\chapter{Intervalles de Tamari et énumération}
\label{chap:tamari_intervalles}
\index{treillis!de Tamari}
\index{intervalles!de Tamari}

Comme nous l'avons vu dans le chapitre précédent, les éléments $\BP_T$ de $\PBT$ s'expriment comme une somme de permutations dans $\FQSym$. Ces permutations appartiennent à la \emph{classe sylvestre} de $T$. Ce sont les extensions linéaires de l'arbre binaire et elles forment un intervalle de l'ordre faible. Les éléments $\BH_T$ et $\BE_T$ sont aussi des sommes sur des intervalles de l'ordre faible qui englobent cette fois plusieurs classes sylvestres. Plus précisément, ce sont les extensions linéaires des arbres $T' \leq T$ (resp. $T' \geq T$) pour l'ordre de Tamari. En fait, on peut exprimer ces intervalles initiaux et finaux comme les extensions linéaires des deux arbres planaires obtenus par la bijection décrite au paragraphe \ref{sub-sec:tamari_prelim:tamari:planaires}. Plus généralement, un intervalle de Tamari $[T_1, T_2]$ est encodé par un poset particulier dont les extensions linéaires correspondent aux classes sylvestres des arbres inclus dans $[T_1, T_2]$. Nous appelons ces posets les \emph{intervalles-posets} de Tamari et utilisons leurs propriétés combinatoires pour obtenir de nouveaux résultats sur le treillis de Tamari.

Dans \cite{Chap}, Chapoton démontre que le nombre d'intervalles dans le treillis de Tamari est donné par 

\begin{equation}
\label{eq:tamari_intervalles:formule_chap}
I_n = \frac{2(4n+1)!}{(n+1)!(3n+2)!}.
\end{equation}

Cette formule est obtenue par la résolution d'une équation fonctionnelle sur la série génératrice des intervalles de Tamari. Pour prouver que la série génératrice vérifie bien l'équation fonctionnelle, Chapoton utilise des arguments combinatoires. Nous proposons dans ce chapitre une nouvelle preuve de ce résultat utilisant les intervalles-posets. L'équation fonctionnelle donnée par Chapoton peut s'exprimer en fonction d'un opérateur bilinéaire qui s'interprète simplement en termes d'intervalles-posets. On note $\Phi(x,y)$ la série génératrice des intervalles de Tamari où $y$ compte la taille des arbres et $x$ le nombre de \noeuds sur la branche gauche du plus petit arbre de l'intervalle. On prouve que
\begin{equation}
\Phi(x,y) = \OB(\Phi, \Phi) + 1
\end{equation}
où
\begin{equation}
\label{eq:tamari_intervalles:def-B}
\OB(f,g) = xy f(x,y) \frac{x g(x,y) - g(1, y)}{x - 1}.
\end{equation}

Cela nous amène à définir le \emph{polynôme de Tamari} d'un arbre donné.

\begin{Definition}
\label{def:tamari_intervalles:tamari-polynomials}
\index{polynômes!de Tamari}
Soit $T$ un arbre binaire, son polynôme de Tamari $\OBT_T(x)$ est défini récursivement par
\begin{align}
\OBT_\emptyset &:= 1 \\
\OBT_T(x) &:= \OB_{y=1}(\OBT_L,\OBT_R)
\end{align}
où $L$ et $R$ sont respectivement les sous-arbres gauche et droit de $T$.
\end{Definition}

On prouve alors un résultat plus fin que la simple énumération des intervalles.

\begin{Theoreme}
\label{thm:tamari_intervalles:smaller-trees}
Soit $T$ un arbre binaire. Son polynôme de Tamari $\OBT_T(x)$ compte le nombre d'arbres inférieurs ou égaux à $T$ pour l'ordre de Tamari en fonction du nombre de \noeuds sur leur branche gauche. En particulier $\OBT_T(1)$ est le nombre d'arbres inférieurs ou égaux à $T$.

De façon symétrique, si $\tilde{\OBT}_T$ est défini en inversant les rôle des sous-arbres droit et gauche dans $\OBT_T$, alors $\tilde{\OBT}_T$ compte le nombre d'arbres supérieurs ou égaux à $T$ en fonction du nombre de \noeuds sur leur branche droite.  
\end{Theoreme}

\begin{figure}[ht]
\centering
\input{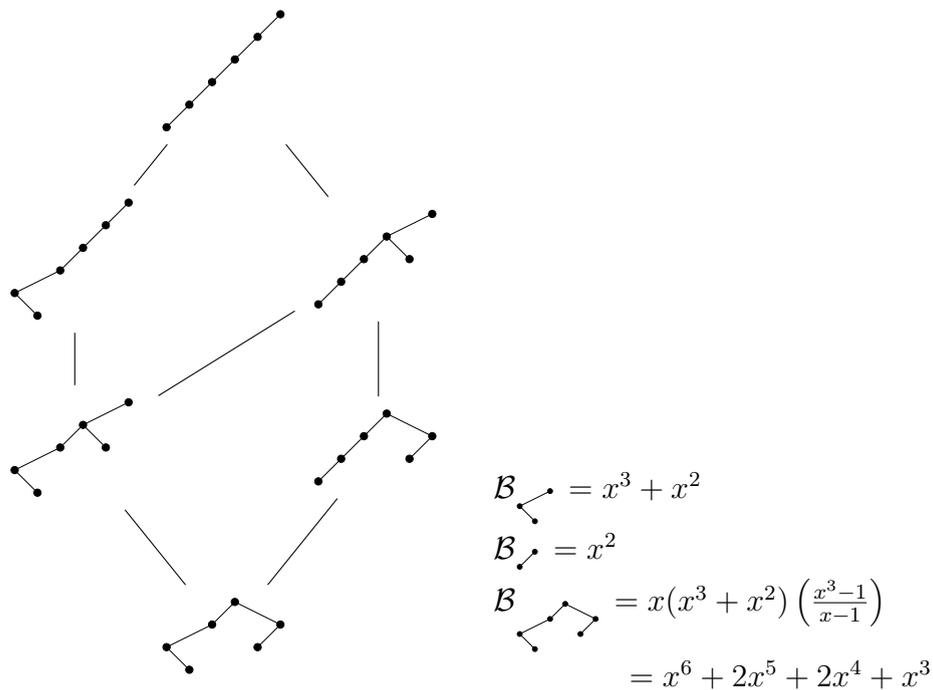}
\caption{Calcul du polynôme de Tamari d'un arbre et intervalle correspondant.}
\label{fig:tamari_intervalles:BTExemple-2}
\end{figure}

Un exemple de calcul du polynôme de Tamari et du résultat du théorème est donné figure \ref{fig:tamari_intervalles:BTExemple-2}. La preuve du théorème \ref{thm:tamari_intervalles:smaller-trees} est donné paragraphe \ref{sub-sec:tamari_intervalles:polynomes:smaller}. Nous commençons au paragraphe \ref{sec:tamari_intervalles:intervalles-posets} par définir les intervalles-posets et nous en donnons les principales propriétés. Dans le paragraphe \ref{sub-sec:tamari_intervalles:polynomes:composition}, nous décrivons une opération de composition sur les intervalles-posets. Nous utilisons cette opération, tout d'abord pour donner une nouvelle preuve du résultat de Chapoton sur la fonction génératrice des intervalles (paragraphe \ref{sub-sec:tamari_intervalles:polynomes:enumeration}) puis pour prouver le théorème~\ref{thm:tamari_intervalles:smaller-trees} (paragraphe \ref{sub-sec:tamari_intervalles:polynomes:smaller}). Au paragraphe \ref{sub-sec:tamari_intervalles:polynomials:bivar}, nous faisons le lien avec un autre résultat de Chapoton sur les flots d'arbres enracinés \cite{ChapBiVar}. Les résultats de ce chapitre découlent d'un travail fait en commun avec Grégory Chatel et sont publiés dans \cite{Me_Tamari}.

\section{Intervalles-posets de Tamari}
\label{sec:tamari_intervalles:intervalles-posets}
\index{intervalles-posets}

\subsection{For\^ets initiales et finales}
\label{sub-sec:tamari_intervalles:intervalles-posets:forets}
\index{forêt!initiale}
\index{forêt!finale}

La bijection entre les arbres binaires et les arbres planaires décrites au paragraphe \ref{sub-sec:tamari_prelim:tamari:planaires} peut aussi s'exprimer en terme de posets.

\begin{Definition}
Soit $T$ un arbre binaire. On identifie $T$ à son arbre binaire de recherche que l'on considère comme un poset. On note $a \trprec_T b$ si $a$ précède $b$ dans le poset c'est-à-dire si $a$ est dans le sous-arbre issu de $b$. Si $a \trprec_T b$ et $a < b$ alors $a$ est dans le sous-arbre gauche de $b$ et on dit que $a \trprec_T b$ est une relation \emph{croissante} de $b$. Si $b \trprec_T a$ alors $b$ est dans le sous-arbre droit de $a$ et on dit que $b \trprec_T a$ est une relation \emph{décroissante} de $T$. 

La \emph{forêt initiale} de $T$, notée $\inc(T)$, est le poset obtenu par la relation $\trprec_{\inc}$ définie telle que
\begin{equation}
a \trprec_{\inc} b \Leftrightarrow a < b \text{ et } a \trprec_T b.
\end{equation}
En d'autre termes, $a \trprec_{\inc} b$ si $a \trprec_T b$ est une relation croissante de $T$. Le poset $T$ est donc une extension de $\inc(T)$.

La \emph{forêt finale} de $T$, notée $\dec(T)$, est le poset obtenu par la relation $\trprec_{\dec}$ définie telle que
\begin{equation}
b \trprec_{\dec} a \Leftrightarrow a < b \text{ et } b \trprec_T a.
\end{equation}
On a donc que $b \trprec_{\dec} a$ si $b \trprec_T a$ est une relation décroissante de $T$.
\end{Definition}

Un exemple de la construction est donné figure \ref{fig:tamari_intervalles:forets}. 

\begin{figure}[ht]
\centering
\input{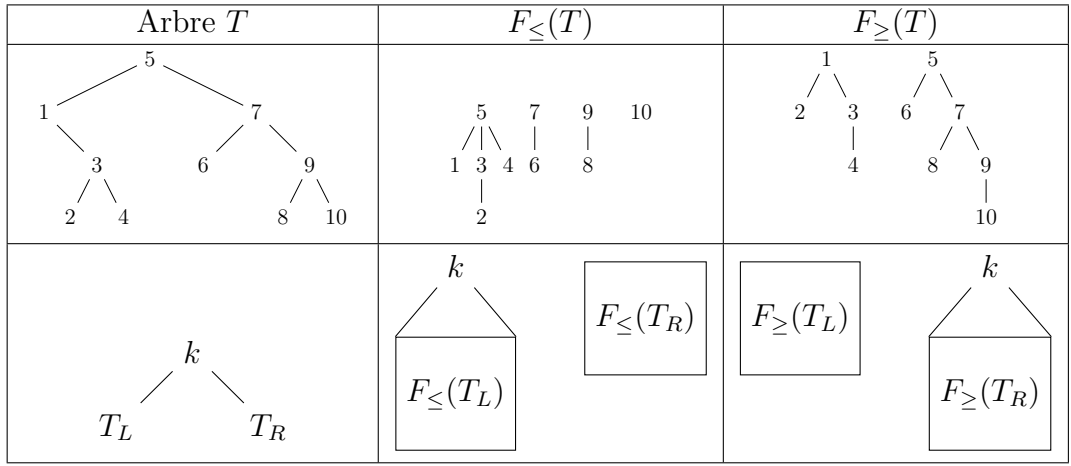}
\caption{Un arbre binaire et les forêts initiales et finales correspondantes.}
\label{fig:tamari_intervalles:forets}
\end{figure}

Les deux opérations sont en fait chacune des bijections : on peut retrouver l'arbre binaire à partir de sa forêt initiale ou de sa forêt finale. Ainsi, la bijection entre la forêt finale et l'arbre binaire est celle donnée entre les arbres planaires et les arbres binaires au paragraphe \ref{sub-sec:tamari_prelim:tamari:planaires}. \`A partir d'une forêt étiquetée, on obtient en effet un arbre planaire en supprimant les étiquettes et en rajoutant une racine commune aux arbres. La construction récursive de l'arbre binaire en fonction de sa forêt initiale ou finale est illustrée figure \ref{fig:tamari_intervalles:forets}. On donne à présent la condition nécessaire et suffisante sur l'étiquetage d'une forêt pour qu'il corresponde à l'étiquetage de l'arbre binaire de recherche correspondant.

\begin{Lemme}
  \label{lem:tamari_intervalles:carac-foret}
  Soit un poset étiqueté $F$. Alors $F$ est la forêt finale d'un arbre binaire $T$ si et seulement si $c \trprec_F a$ implique que $c>a$ et que $b \trprec_F a$ pour tout $b$ tel que $a < b < c$. De même, $F$ est la forêt initiale d'un arbre binaire $T$ si et seulement si $a \trprec_F c$ implique que $a<c$ et que $b \trprec_F c$ pour tout $b$ tel que $a < b < c$.
\end{Lemme}

\begin{proof}
On effectuera la preuve uniquement pour le cas de la forêt finale $\dec$. La preuve pour la forêt initiale est symétrique.

Tout d'abord, prouvons que si $F$ est la forêt finale d'un arbre binaire $T$, alors la condition est vérifiée. Soit $c > a$ tel que $c \trprec_F a$. Par construction, on a $c \trprec_T a$ ce qui signifie que $c$ est dans le sous-arbre droit de $a$ dans $T$. Soit $b$ tel que $a<b<c$. Trois configurations sont possibles : soit $a \trprec_T b$ et $a$ est dans le sous-arbre gauche de $b$, soit $a$ et $b$ ne sont pas comparables, soit $b \trprec_T a$ et $b$ est dans le sous-arbre droit de $a$. 

Supposons que $a$ et $b$ ne soient pas comparables dans $T$. Alors, il existe $b'$ tel que $a < b'< b$ avec $a$ dans le sous-arbre gauche de $b'$ et $b$ dans le sous-arbre droit de $b'$. Comme $c$ est dans le sous-arbre droit de $a$, il est aussi dans le sous-arbre gauche de $b'$. Or $b' < c$ ce qui contredit la règle de l'arbre binaire de recherche. Pour la même raison, $a$ ne peut pas être dans le sous-arbre gauche de $b$. On a donc que $b$ est dans le sous-arbre droit de $a$, c'est-à-dire $b \trprec_T a$. La forêt $F$ est formée par les relations décroissantes de $T$ et on a bien $b \trprec_F a$.

\`A présent, soit $F$ un poset étiqueté vérifiant la condition du lemme. Le poset $F$ se décompose en $r$ composantes connexes $F_1, F_2, \dots, F_r$. Pour chaque $F_i$, il existe un unique élément $x_i$ qu'on appelle la racine de $F_i$ tel que $y \trprec_F x_i$ pour tout $y \in F_i$. En effet, si $x, x'$ et $y$ sont des éléments de $F_i$ avec $y \trprec_F x$ et $y \trprec_F x'$, on a soit $x < x' < y$ et donc $x' \trprec_F x$ ou bien $x' < x < y$ et $x \trprec_F x'$. Comme toutes les relations de $F$ sont décroissantes, l'étiquette de $x_i$ est aussi minimale dans $F_i$ : $y > x_i$ pour tout $y \in F_i$. De plus, si $x_i$ et $x_j$ sont les racines de deux composantes connexes différentes, respectivement $F_i$ et $F_j$, alors $x_i < x_j$ implique que $y < z$ pour tout $y \in F_i$ et $z \in F_j$. En suivant le schéma de la figure \ref{fig:tamari_intervalles:forets}, on pose $k$ la racine de valeur maximale parmi $x_1, \dots, x_r$. En supprimant le sommet $k$ de sa composante connexe, on obtient un nouveau poset $F_L$ formé des fils de $k$ qui vérifie toujours la condition et dont toutes les étiquettes sont supérieures à $k$. Par ailleurs, le poset $F_R$ formé par les autres composantes connexes de $F$ vérifie lui aussi la condition et toutes ses étiquettes sont inférieures à $k$. On peut donc construire récursivement l'arbre binaire $T = k(T_L, T_R)$ où $T_L$ et $T_R$ sont obtenus respectivement par $F_L$ et $F_R$. Par construction, $T$ est un arbre binaire de recherche et $F= \dec(T)$.
\end{proof}

\begin{Proposition}
\label{prop:tamari_intervalles:ext-lin-foret}
Les extensions linéaires de la forêt finale $\dec(T)$ d'un arbre binaire $T$ sont exactement les classes sylvestres des arbres $T'\geq T$ pour l'ordre de Tamari. De même, les extensions linéaires de la forêt initiale $\inc(T)$ sont les classes sylvestres des arbres $T' \leq T$.
\end{Proposition}

\begin{proof}
On effectue la preuve uniquement pour $\dec(T)$. Par symétrie de l'ordre faible et de l'ordre de Tamari, le résultat est aussi vrai pour $\inc(T)$. Soit $\alpha_T$ l'élément minimal de la classe sylvestre de $T$. On veut prouver que les extensions linéaires de $\dec(T)$ correspondent à l'intervalle $[\alpha_T, \omega]$ où $\omega$ est la permutation maximale. Comme l'ordre de Tamari est un quotient de l'ordre faible, cela prouve entièrement le résultat.

Le poset $\dec(T)$ ne contient que des relations décroissantes $b \trprec_{\dec} a$ avec $b > a$. Les extensions linéaires de $\dec(T)$ sont exactement les permutations contenant toutes les coinversions $(a,b)$ telles que $b \trprec_{\dec} a$. En effet, par définition les extensions linéaires de $\dec(F)$ contiennent toutes ces coinversions. C'est aussi une condition suffisante. Soit  $\sigma$ une permutation non extension linéaire de $\dec(T)$. Alors il existe une relation $b \trprec_{\dec} a$ avec $b > a$ et $a$ avant $b$ dans $\sigma$. La permutation $\sigma$ ne contient pas la covinversion $(a,b)$.

Enfin, $\alpha_T$ ne contient pas d'autres coinversions que les relations de $\dec(T)$. En effet, on lit $\alpha_T$ sur l'arbre binaire de recherche $T$ par un parcours suffixe : fils gauche, fils droit, racine. Soit $b > a$ telle que $\dec(T)$ ne contient pas la relation $b \trprec_{\dec} a$. 
Alors $b$ n'est pas dans le sous-arbre droit de $a$. On a soit que $a$ est dans le sous-arbre gauche de $b$, soit que $a$ est dans le sous-arbre gauche d'un élément $b'$ dont $b$ est dans le sous-arbre droit. Dans tous les cas, $a$ est lu avant $b$ dans $\alpha_T$.

Pour conclure, rappelons la règle de comparaison des éléments dans l'ordre faible droit donnée au paragraphe \ref{sub-sec:prelim_groupe_sym:perm:treillis} : une permutation $\sigma$ est plus petite qu'une permutation $\mu$ si les coinversions de $\sigma$ sont incluses dans les coinversions de $\mu$. Les extensions linéaires de $\dec(T)$ sont exactement les permutations dont les coinversions contiennent celles de $\alpha_T$.
\end{proof}

\subsection{Définition des intervalles-posets}
\label{sub-sec:tamari_intervalles:intervalles-posets:def}

\begin{figure}[p]
\begin{leftfullpage}
\centering
\input{includes/figures/forest-intersection}
\caption[Construction d'un intervalle poset]{Sur la page de gauche : la construction d'un intervalle poset à partir des forêts initiales et finales. Dans la dernière image, on a supprimé les relations redondantes.

Sur la page de droite : l'intervalle de Tamari correspondant. On pourra vérifier qu'une extension linéaire d'un arbre de l'intervalle correspond toujours à une extention linéaire du poset et vice versa.}
\label{fig:tamari_intervalles:intervalle-poset}
\end{leftfullpage}
\end{figure}

\begin{figure}[p]
\centering
\input{includes/figures/interval-forest-intersection}
\end{figure}

Soit $[T_1, T_2]$ un intervalle de Tamari. Si $\sigma$ est une extension linéaire de $\dec(T_1)$ alors $\sigma$ appartient à la classe sylvestre d'un arbre $T'\geq T_1$. Maintenant, si $\sigma$ est aussi une extension linéaire de $\inc(T_2)$, alors on a $T' \leq T_2$. On peut donc encoder l'intervalle $[T_1, T_2]$ par les relations des deux posets $\dec(T_1)$ et $\inc(T_2)$. 

\begin{Definition}
\label{def:tamari_intervalles:intervalles-posets}
Un intervalle-poset $(P, \trprec)$ est un poset sur les entiers de 1 à $n$ tel que les conditions suivantes soient respectées :
\begin{enumerate}
    \item si $a < c$ et $a \trprec c$ alors pour tout $b$ tel que $a < b < c$, on a $b \trprec c$,
    \item si $a < c$ et $c \trprec a$ alors pour tout $b$ tel que $a < b < c$, on a $b \trprec a$.
\end{enumerate}
\end{Definition}

\begin{Proposition}
Les intervalles-posets sont en bijection avec les intervalles de Tamari.

Plus précisément, à chaque intervalle-poset $P$ correspond un couple d'arbres $T_1 \leq T_2$ tel que les extensions linaires de $P$ soient exactement les extensions linéaires des arbres $T' \in [T_1, T2]$. 

En particulier, les intervalles-posets sont les seuls posets étiquetés dont les extensions linéaires forment des intervalles de l'ordre faible droit $[\alpha_{T_1}, \omega_{T_2}]$ où $\alpha_{T_1}$ est l'élément minimal d'une classe sylvestre $T_1$ et $\omega_{T_2}$, l'élément maximal d'une classe $T_2$. 
\end{Proposition}

\begin{proof}
Soit un intervalle de Tamari $[T_1, T_2]$. Comme $T_1 \leq T_2$, par la proposition \ref{prop:tamari_intervalles:ext-lin-foret}, les extensions linéaires de $T_1$ en particulier vérifient à la fois les relations des posets $\dec(T_1)$ et $\inc(T_2)$. Ces deux posets sont donc compatibles dans le sens où il n'existe pas de relations contradictoires : $a \trprec_{\dec} b$ et $a \trsucc_{\inc} b$. On forme alors le poset $P$ contenant à la fois les relations de $\dec(T_1)$ et $\inc(T_2)$. Par le lemme \ref{lem:tamari_intervalles:carac-foret}, $P$ possède les deux conditions qui en font un intervalle-poset.

\`A présent, soit $P$ un intervalle-poset. Soit $F_1$ le poset formé par les relations décroissantes de $P$ : $b \trprec_{F_1} a$ si $b > a$ et $b \trprec_P a$. Et soit $F_2$ le poset formé par les relations croissantes de $P$. Par le lemme \ref{lem:tamari_intervalles:carac-foret}, les posets $F_1$ et $F_2$ sont respectivement les forêts finales et initiales de deux arbres binaires de recherche $T_1$ et $T_2$. Soit $\sigma$ une extension linéaire de $P$ et $T' = \ABR(\sigma)$. On a que $\sigma$ est aussi une extension linéaire de $F_1$ et donc $T_1 \leq T'$ par la proposition \ref{prop:tamari_intervalles:ext-lin-foret}. Et $\sigma$ est une extension linéaire de $F_2$ d'où $T' \leq T_2$. On a donc $T_1 \leq T_2$, et le poset $P$ correspond aux extensions linéaires des arbres de l'intervalle $[T_1, T_2]$.
\end{proof}

Un exemple de la construction avec l'intervalle correspondant est donné figure~\ref{fig:tamari_intervalles:intervalle-poset}. La bijection permet d'identifier les intervalles de Tamari aux intervalles-posets. Un arbre binaire de recherche $T$ est un intervalle-poset particulier qui correspond à $[T,T]$. De même, les forêts initiales et finales sont des cas particuliers d'intervalles-posets. Ces objets combinatoires sont facilement maniables et programmables et on y lit de nombreuses propriétés.

\begin{Proposition}
\label{prop:tamari_intervalles:comb-prop}
\begin{enumerate}[label=(\roman{*}), ref=(\roman{*})]
\item Soient $I_1$ et $I_2$ deux intervalles-posets. L'intersection de $I_1$ et $I_2$ est non vide si et seulement si les relations de $I_1$ ne contredisent pas celles de $I_2$. Dans ce cas, l'intersection est aussi un intervalle, elle est donnée par $I_3$ l'intervalle-poset contenant les relations à la fois de $I_1$ et $I_2$.
\label{prop:tamari_intervalles:comb-prop-intersect}
\item Un intervalle  $I_1 := \left[ T_1, T_1' \right] $ contient l'intervalle $I_2 := \left[ T_2, T_2' \right]$, c'est-à-dire $T_1 \leq T_2$ et $T_1' \geq T_2'$, si et seulement si $I_2$ est une extension de $I_1$ ($I_2$ contient les relations de $I_1$ et éventuellement d'autres)
\label{prop:tamari_intervalles:comb-prop-inclusion}
\item Si $I_1 := \left[ T_1, T_1' \right]$  alors $I_2 = \left[ T_2, T_1' \right]$ tel que $T_2 \geq T_1$ si et seulement si $I_2$ est une extension de $I_1$ et que les relations supplémentaires de $I_2$ sont décroissantes. De façon symétrique, $I_3 = \left[ T_1, T_3 \right]$ tel que $T_3 \leq T_1'$ si et seulement si $I_3$ est une extension de $I_1$ et que les relations supplémentaires de $I_3$ sont croissantes.
\label{prop:tamari_intervalles:comb-prop-minmax-inclusion}
\end{enumerate}
\end{Proposition}

Toutes ces propriétés découlent directement de la construction des intervalles-posets.

\subsection{Lien avec $\PBT$}
\label{sub-sec:tamari_intervalles:intervalles-posets:PBT}
\index{PBT}

On a vu dans le chapitre \ref{chap:tamari_prelim} qu'une base de l'algèbre de Hopf $\PBT$ est donnée par les éléments $\BP_T$ :
\begin{equation}
\BP_T = \sum_{\ABR(\sigma) = T} \BF_\sigma.
\end{equation}

Dans le paragraphe \ref{sub-sec:tamari_prelim:pbt:mult}, on définit deux autres bases de $\PBT$ et d'après la définition des forêts finales et initiales, on a que
\begin{align}
\BH_T &= \sum_{\sigma \in \inc(T)} \BF_\sigma, \\
\BE_T &= \sum_{\sigma \in \dec(T)} \BF_\sigma,
\end{align}
où on écrit $\sigma \in \inc(T)$ (resp. $\dec(T)$) pour $\sigma$ une extension linéaire de $\inc(T)$ (resp. $\dec(T)$). Par ailleurs, on a vu que le produit dans $\PBT$ s'exprime comme une somme sur un intervalle de l'ordre de Tamari et que $\BH$ et $\BE$ sont des bases multiplicatives. Définissons à présent de nouveaux éléments particuliers de $\PBT$ indexés par les intervalles,
\begin{equation}
\BI_{[T_1, T_2]} = \sum_{T \in [T_1,T_2]} \BP_{T}.
\end{equation}

Notons qu'un élément de la base $\BP_T$ est égal à $\BI_{[T,T]}$ et que donc les éléments $\BI$ ne forment pas une base de $\PBT$. Cependant, si on note $P_{[T_1,T_2]}$ l'intervalle-poset qui correspond à $[T_1,T_2]$ alors on a
\begin{equation}
\BI_{[T_1,T_2]} = \sum_{\sigma \in P_{[T_1,T_2]}} \BF_\sigma
\end{equation}
et le produit s'exprime de façon simple.

\begin{Proposition}
On a
\begin{equation}
\BI_{P} \BI_{P'} = \BI_{P\ConcDec P'}
\end{equation}
où $P\ConcDec P'$ est la concaténation décalée du poset $P'$ au poset $P$. C'est-à-dire que $P \ConcDec P'$ contient le poset $P$ et le poset $P'$ où toutes les étiquettes ont été décalées de $|P|$.
\end{Proposition}

\begin{proof}
La démonstration est immédiate. Soient $P$ et $P'$ les intervalles-posets correspondant respectivement à $ [T_1,T_2]$ et $[T_1',T_2']$, alors
\begin{equation}
\BI_P \BI_{P'} = \left( \sum_{\sigma \in P} \BF_\sigma \right) \left( \sum_{\mu \in P'} \BF_\mu \right).
\end{equation}
Les extensions linéaires d'intervalles-posets forment des intervalles de l'ordre droit. C'est un résultat connu que le produit de deux sommes sur des intervalles dans $\FQSym$ est toujours une somme sur un intervalle \cite{PBT2}. Soient $Q_1$ l'arbre minimal du produit $\BP_{T_1} \BP_{T_1'}$ et $Q_2$ l'arbre maximal de $\BP_{T_2} \BP_{T_2'}$. De façon claire, on a
\begin{equation}
\BI_P \BI_{P'} = \sum_{\alpha_{Q_1} \leq \sigma \leq \omega_{Q_2}} \BF_\sigma
\end{equation}
où $\alpha_{Q_1}$ est la permutation minimale de la classe sylvestre de $Q_1$ et $\omega_{Q_2}$ la permutation maximale de la classe sylvestre de $Q_2$. Comme l'ordre de Tamari est un quotient de l'ordre droit, ce produit correspond à la somme des éléments $\BP_{Q'}$ pour $ Q' \in [Q_1, Q_2]$. Nous avons donné la construction des arbres $Q_1$ et $Q_2$ dans la figure~\ref{fig:tamari_prelim:prod_pbt} quand nous avons expliqué le produit dans $\PBT$. On en déduit directement que l'intervalle-poset $P_{[Q_1,Q_2]} = P \ConcDec P'$.
\end{proof}

Par exemple,
\begin{align}
\BI_{\scalebox{0.6}{

\def \hlev{0.8}
\def \wlev{0.5}

\begin{tikzpicture}
\node(N1) at (0,0) {1};
\node(N2) at (0.5 * \wlev,-1 * \hlev) {2};
\node(N3) at (\wlev,0) {3};
\draw (N2) -- (N3);
\draw (N2) -- (N1);
\end{tikzpicture}
}} 
= \BI_{\left[ \scalebox{0.4}{


\begin{tikzpicture}
\node (N0) at (1.250, 0.000){};
\node (N00) at (0.250, -0.500){};
\node (N001) at (0.750, -1.000){};
\draw[Point] (N001) circle;
\draw (N00.center) -- (N001.center);
\draw[Point] (N00) circle;
\draw (N0.center) -- (N00.center);
\draw[Point] (N0) circle;
\end{tikzpicture}},\scalebox{0.4}{


\begin{tikzpicture}
\node (N0) at (0.250, 0.000){};
\node (N01) at (1.250, -0.500){};
\node (N010) at (0.750, -1.000){};
\draw[Point] (N010) circle;
\draw (N01.center) -- (N010.center);
\draw[Point] (N01) circle;
\draw (N0.center) -- (N01.center);
\draw[Point] (N0) circle;
\end{tikzpicture}} \right]} &= \BP_{213} + \BP_{231}, \\
\BI_{\scalebox{0.6}{

\def \wlev{0.5}
\def \hlev{0.8}

\begin{tikzpicture}
\node(N1) at (0,0) {1};
\node(N2) at (\wlev,0) {2};
\node(N3) at (\wlev,-1 * \hlev) {3};
\draw (N3) -- (N2);
\end{tikzpicture}
}} 
= \BI_{\left[ \scalebox{0.4}{

%

\begin{tikzpicture}
\node (N0) at (0.750, 0.000){};
\node (N00) at (0.250, -0.500){};
\draw[Point] (N00) circle;
\node (N01) at (1.250, -0.500){};
\draw[Point] (N01) circle;
\draw (N0.center) -- (N00.center);
\draw (N0.center) -- (N01.center);
\draw[Point] (N0) circle;
\end{tikzpicture}},\scalebox{0.4}{


\begin{tikzpicture}
\node (N0) at (0.250, 0.000){};
\node (N01) at (0.750, -0.500){};
\node (N011) at (1.250, -1.000){};
\draw[Point] (N011) circle;
\draw (N01.center) -- (N011.center);
\draw[Point] (N01) circle;
\draw (N0.center) -- (N01.center);
\draw[Point] (N0) circle;
\end{tikzpicture}} \right]} &= \BP_{312} + \BP_{321}
\end{align}
où les éléments $\BP$ sont indexés par leur mot canonique $\omega_T$, permutation maximale de la classe sylvestre. Alors, en appliquant la formule du produit dans $\PBT$ de la proposition \ref{prop:tamari_prelim:pbt-product}, on obtient
\begin{align}
\BI_{\scalebox{0.6}{

\def \hlev{0.8}
\def \wlev{0.5}

\begin{tikzpicture}
\node(N1) at (0,0) {1};
\node(N2) at (0.5 * \wlev,-1 * \hlev) {2};
\node(N3) at (\wlev,0) {3};
\draw (N2) -- (N3);
\draw (N2) -- (N1);
\end{tikzpicture}
}} 
\BI_{\scalebox{0.6}{

\def \wlev{0.5}
\def \hlev{0.8}

\begin{tikzpicture}
\node(N1) at (0,0) {1};
\node(N2) at (\wlev,0) {2};
\node(N3) at (\wlev,-1 * \hlev) {3};
\draw (N3) -- (N2);
\end{tikzpicture}
}} &=
(\BP_{213} + \BP_{231})(\BP_{312} + \BP_{321}) \\
&= \BP_{621345} + \BP_{642135} + \BP_{645213} + \BP_{652134} + \BP_{654213} \\ \nonumber
&+ \BP_{623145} + \BP_{623415} + \BP_{623451} + \BP_{642315} + \BP_{642351} \\ \nonumber
&+ \BP_{645231} + \BP_{652314} + \BP_{652341} + \BP_{654231} \\
&= \BI_{\scalebox{0.6}{
\def \hlev{0.8}
\def \wlev{0.6}

\begin{tikzpicture}
\node(N1) at (0,0) {1};
\node(N2) at (0.5 * \wlev,-1 * \hlev) {2};
\node(N3) at (\wlev,0) {3};
\node(N4) at (2 * \wlev,0) {4};
\node(N5) at (3 * \wlev,0) {5};
\node(N6) at (3 * \wlev,-1 * \hlev) {6};
\draw (N2) -- (N3);
\draw (N2) -- (N1);
\draw (N6) -- (N5);
\end{tikzpicture}
}}.
\end{align}

\section{Polynômes de Tamari}
\label{sec:tamari_intervalles:polynomes}
\index{polynômes!de Tamari}

\subsection{Composition des intervalles-posets}
\label{sub-sec:tamari_intervalles:polynomes:composition}

Soit $\phi(y)$, la série génératrice des intervalles de Tamari,
\begin{equation}
\phi(y) = \sum_{n \geq 0} I_n y^n
\end{equation}
où $I_n$ est le nombre d'intervalles sur des arbres de taille $n$. Les premières valeurs sont données par \cite{OEISIntervalles}
\begin{equation}
\phi(y) = 1 + y + 3 y^2 + 13 y^3 + 68 y^4 + \dots~.
\end{equation}

Dans \cite{Chap}, Chapoton donne une version raffinée de $\phi$,
\begin{equation}
\Phi(x,y) = \sum_{\substack{n \geq 0 \\ m \geq 0}}I_{n,m} x^m y^n
\end{equation}
où $I_{n,m}$ est le nombre d'intervalles $[T_1,T_2]$ sur des arbres de taille $n$ tel que $T_1$ possède $m$ \noeuds sur sa branche gauche. On a
\begin{equation}
\Phi(x,y) = 1 + xy + (x + 2x^2)y^2 + (3x + 5x^2 + 5x^3)y^3 + \dots~.
\end{equation}

Comme nous l'avons vu dans au paragraphe \ref{sub-sec:tamari_prelim:tamari:planaires}, la statistique du nombre de \noeuds sur la branche gauche de $T$ se lit aussi sur l'arbre planaire correspondant à $T$. C'est le nombre de fils de la racine de l'arbre planaire ou le nombre de retours à 0 sur le chemin de Dyck \cite{FindStatLeftBranch,FindStatTouchPoints}. Sur la forêt finale $\dec(T)$, c'est le nombre d'arbres, c'est-à-dire son nombre de composantes connexes. 

\begin{Definition}
\label{def:tamari_intervalles:stat}
Soit un intervalle $[T_1,T_2]$ et $I$ son intervalle-poset, on note 
\begin{enumerate}
\item $\Isize(I)$ le nombre de \noeuds dans $I$, c'est-à-dire la taille des arbres $T_1$ et $T_2$.
\item $\Itrees(I)$ le nombre d'arbres de $\dec(I)$ la forêt formée en conservant uniquement les relations décroissantes de $I$.
\end{enumerate} 
Enfin, on définit $\PI(I) = x^{\Itrees(I)}y^{\Isize(I)}$ et on étend $\PI$ par linéarité aux combinaisons linéaires d'intervalles-posets.
\end{Definition}

La série génératrice raffinée $\Phi$ sur les intervalles de Tamari s'exprime alors par
\begin{equation}
\Phi(x,y) = \sum_{I} \PI(I)
\end{equation}
sommée sur l'ensemble des intervalles-posets. On prouve le théorème suivant.

\begin{Theoreme}
\label{thm:tamari_intervalles:equation-fonct}
La série génératrice $\Phi(x,y)$ vérifie l'équation fonctionnelle 
\begin{equation}
\label{eq:tamari_intervalles:equation-fonct}
\Phi(x,y) = \OB(\Phi,\Phi) + 1
\end{equation}
où 
\begin{equation}
\OB(f,g) = xy f(x,y) \frac{x g(x,y) - g(1, y)}{x - 1}.
\end{equation}
\end{Theoreme}

Ce théorème est prouvé par Chapoton dans \cite{Chap}. La formulation est légèrement différente, dans la série génératrice donnée en \cite[formule (6)]{Chap}, le degré de $x$ diffère de 1 et la série ne compte pas l'intervalle de taille 0. Dans le paragraphe~\ref{sub-sec:tamari_intervalles:polynomes:enumeration}, nous donnons une nouvelle preuve de ce théorème. Nous utilisons pour cela une opération de composition sur les intervalles-posets.

\begin{Definition}
\label{def:tamari_intervalles:composition}
Soient $I_1$ et $I_2$ deux intervalles-posets de tailles respectives $k_1$ et $k_2$. Alors $\BB(I_1,I_2)$ est la somme formelle de tous les intervalles-posets de taille $k_1 + k_2 + 1$ tels que
\begin{enumerate}[label=(\roman{*}), ref=(\roman{*})]
\item les relations entre les sommets $1, \dots, k_1$ sont celles de $I_1$,
\label{def:tamari_intervalles:composition:cond:I1}
\item les relations entre les sommets $k_1 + 2, \dots, k_1 + k_2 + 1$ sont celles de $I_2$ décalées de $k_1 + 1$,
\label{def:tamari_intervalles:composition:cond:I2}
\item on a $i \trprec k_1 + 1$ pour tout $i \leq k_1$,
\label{def:tamari_intervalles:composition:cond:incr}
\item il n'existe aucune relation $k_1+i \trprec j$ pour $j>k_1+1$
\label{def:tamari_intervalles:composition:cond:decr}
\end{enumerate}
On appelle cette opération la composition des intervalles et on l'étend par bilinéarité à toutes les sommes formelles d'intervalles-posets.
\end{Definition}

\begin{figure}[ht]
\centering
\input{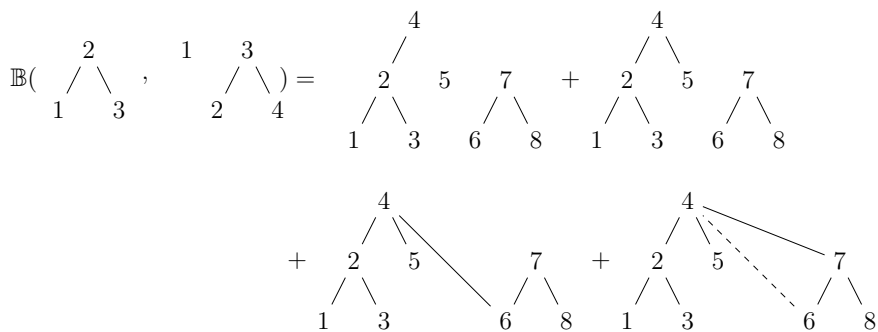}
\caption[Composition des intervalles-posets]{Composition des intervalles-posets. Les quatre termes correspondent à l'ajout de respectivement 0, 1, 2 et 3 relations décroissantes entre le second poset et le sommet 4. Dans le dernier terme, on a ajouté 3 relations : la relation $6 \trprec 4$ a été mise en pointillés car elle peut être obtenue par transitivité.}
\label{fig:tamari_intervalles:composition}
\end{figure}

La somme que l'on obtient correspond à toutes les façons d'ajouter des relations décroissantes entre le second poset et le nouveau sommet $k_1 + 1$, comme on peut le voir  figure \ref{fig:tamari_intervalles:composition}. En particulier, il n'y a aucune relation entre les sommets $1, \dots, k_1$ du premier poset et les sommets $k_1 +2, \dots, k_1 + k_2 +1$ du second poset. En effet, la condition \ref{def:tamari_intervalles:composition:cond:incr} interdit toute relation $j \trprec i$ avec $i < k_1+1 < j$ car cela impliquerait par la définition \ref{def:tamari_intervalles:intervalles-posets} que $k_1 +1 \trprec i$. Par ailleurs, la condition \ref{def:tamari_intervalles:composition:cond:decr} interdit toute relation $i \trprec j$ car cela impliquerait $k_1 + 1 \trprec j$. 

Le nombre d'éléments dans la somme est donné par $\Itrees(I_2) + 1$. En effet, si $x_1 < x_2 < \dots < x_m$ sont les racines des arbres de $\dec(I_2)$, on ne peut rajouter une relation $x_i \trprec k_1 +1$ que si on a $x_j \trprec k_1 +1$ pour tout $j < i$. On a donc 
\begin{equation}
\label{eq:tamari_intervalles;composition-simple}
\BB(I_1,I_2) = \sum_{0 \leq i \leq m} P_i
\end{equation}
où $P_i$ est l'intervalle-poset où on a rajouté exactement $i$ relations décroissantes : $x_j \trprec k_1 +1$ pour $j \leq i$.

\begin{Proposition}
Soit $I_1$ l'intervalle-poset de taille $k_1$ correspondant à l'intervalle $[T_1, T_1']$  et $I_2$ l'intervalle-poset de taille $k_2$ correspondant à $[T_2, T_2']$. Soient $k = k_1 + 1$ et 
\begin{enumerate}
\item $Q_\alpha$, l'arbre $T_2$ auquel on a greffé $k(T_1,\emptyset)$ à gauche de son \noeud le plus à gauche,
\item $Q_\omega$, l'arbre $k(T_1,T_2)$,
\item et $Q'$, l'arbre $k(T_1',T_2')$.
\end{enumerate}
On a
\begin{equation}
\BB(I_1, I_2) = \sum_{Q \in [Q_\alpha,Q_\omega]} P_{[Q,Q']}
\end{equation}
où $ P_{[Q,Q']}$ est l'intervalle-poset correspondant à $[Q,Q']$.
\end{Proposition}

\begin{proof}
La composition de $I_1$ et $I_2$ est une somme d'intervalle-posets $P_0, \dots P_{m}$ où $m=\Itrees(I_2)$ et où $P_i$ est l'intervalle-poset où on a ajouté exactement $i$ relations décroissantes. L'arbre maximum de tous les intervalles est le même car ils ont les mêmes relations croissantes, c'est $Q' =k(T_1',T_2')$. La forêt finale de $P_0$, $\dec(P_0)$ contient $\Itrees(I_1) + \Itrees(I_2) + 1$ arbres : les \noeuds sur la branche gauche de son arbre minimal sont exactement ceux de $T_1$, puis $k$, puis ceux de $T_2$, ce qui correspond à $Q_\alpha$. Soit $Q_i$ l'arbre minimal de $P_i$. Pour passer de $P_i$ à $P_{i+1}$, on rajoute une relation décroissante vers $k$ ce qui revient à effectuer une rotation entre le \noeud $k$ de $Q_i$ et sa racine. Le procédé termine quand l'arbre $T_2$ est entièrement passé à droite du \noeud $k$. On obtient alors l'arbre $Q_{m} = Q_\omega$.

Notons que l'intervalle entre $Q_\alpha$ et $Q_\omega$ est en fait une chaîne saturée : $Q_\alpha = Q_0 \lessdot Q_1 \lessdot \dots \lessdot Q_{m} = Q_\omega$. 
\end{proof}

\begin{figure}[p]
\centering

\def \fpath{includes/figures/}

\begin{tabular}{|c|c|}
\hline
Intervalle-poset & Intervalle correspondant \\
\hline
\input{\fpath interval-posets/I3-ex3}
&
\scalebox{0.6}{\input{\fpath trees/T3-4-n}},\scalebox{0.6}{\input{\fpath trees/T3-4-n}}

\\
\hline
\input{\fpath interval-posets/I4-ex1}
&
\scalebox{0.6}{\input{\fpath trees/T4-6-n}},\scalebox{0.6}{\input{\fpath trees/T4-12-n}} 
\\
\hline
\input{\fpath interval-posets/I8-ex1}
&
\scalebox{0.6}{\input{\fpath trees/T8-ex1-n}},\scalebox{0.6}{\input{\fpath trees/T8-ex2-n}}
\\
\hline
\input{\fpath interval-posets/I8-ex2}
&
\scalebox{0.6}{\input{\fpath trees/T8-ex3-n}},\scalebox{0.6}{\input{\fpath trees/T8-ex2-n}}
\\
\hline
\input{\fpath interval-posets/I8-ex3}
&
\scalebox{0.6}{\input{\fpath trees/T8-ex4-n}},\scalebox{0.6}{\input{\fpath trees/T8-ex2-n}}
\\
\hline
\input{\fpath interval-posets/I8-ex4}
&
\scalebox{0.6}{\input{\fpath trees/T8-ex5-n}},\scalebox{0.6}{\input{\fpath trees/T8-ex2-n}}
\\
\hline

\end{tabular}
\caption{Interprétation en termes d'intervalles de la composition des intervalles-posets.}
\label{fig:tamari_intervalles:composition-2}
\end{figure}

En exemple, on a repris le calcul de la figure \ref{fig:tamari_intervalles:composition} et on donne son interprétation en termes d'intervalles dans la figure \ref{fig:tamari_intervalles:composition-2}.

La composition est en fait formée de deux opérations distinctes : le produit gauche $\pleft$ et le produit droit $\pright$.

\begin{Definition}
\label{def:tamari_intervalles:left-right-product}
Soient $I_1$ et $I_2$ deux intervalles-posets tels que $\Itrees(I_2) = m$, et $\alpha$ est l'étiquette de valeur minimale de $I_2$ et $\omega$ l'étiquette de valeur maximale de $I_1$, alors
\begin{enumerate}
\item $I_1 \pleft I_2$ est l'intervalle obtenu par la concaténation décalée de $I_1$ et $I_2$ et l'ajout des relations croissantes $x \trprec \alpha$ pour tout $x \in I_1$.
\item $I_1 \pright I_2$ est la somme des $m + 1$ intervalles-posets $P_0, P_1, \dots, P_m$ où $P_i$ est la concaténation décalée de $I_1$ et $I_2$ où l'on a ajouté $i$ relations décroissantes $x_j \trprec \omega$ pour $j \leq i$ où $x_1 < \dots < x_m$ sont les les racines des arbres de $\dec(I_2)$.
\end{enumerate} 
\end{Definition}

Par exemple,

\begin{align}
\label{eq:tamari_intervalles:left-product}
\scalebox{0.8}{

\def \hlev{0.8}
\def \wlev{0.5}

\begin{tikzpicture}
\node(N1) at (0,0) {1};
\node(N2) at (0.5 * \wlev,-1 * \hlev) {2};
\node(N3) at (\wlev,0) {3};
\draw (N2) -- (N3);
\draw (N2) -- (N1);
\end{tikzpicture}
}
\pleft
\scalebox{0.8}{

\def \wlev{0.5}
\def \hlev{0.8}

\begin{tikzpicture}
\node(N1) at (0,0) {1};
\node(N2) at (\wlev,0) {2};
\node(N3) at (\wlev,-1 * \hlev) {3};
\draw (N3) -- (N2);
\end{tikzpicture}
} &=
\scalebox{0.8}{\input{includes/figures/interval-posets/I6-ex2}} \\
\label{eq:tamari-intervalles:right-product}
\scalebox{0.8}{

\def \hlev{0.8}
\def \wlev{0.5}

\begin{tikzpicture}
\node(N1) at (0,0) {1};
\node(N2) at (0.5 * \wlev,-1 * \hlev) {2};
\node(N3) at (\wlev,0) {3};
\draw (N2) -- (N3);
\draw (N2) -- (N1);
\end{tikzpicture}
}
\pright
\scalebox{0.8}{

\def \wlev{0.5}
\def \hlev{0.8}

\begin{tikzpicture}
\node(N1) at (0,0) {1};
\node(N2) at (\wlev,0) {2};
\node(N3) at (\wlev,-1 * \hlev) {3};
\draw (N3) -- (N2);
\end{tikzpicture}
} &=
\scalebox{0.8}{

\def \wlev{0.5}
\def \hlev{0.8}

\begin{tikzpicture}
\node(N1) at (0,0) {1};
\node(N2) at (0.5 * \wlev,-1 * \hlev) {2};
\node(N3) at (\wlev,0) {3};

\node(N4) at (2*\wlev,-1 *\hlev) {4};
\node(N5) at (3*\wlev,-1 * \hlev) {5};
\node(N6) at (3*\wlev,-2 *\hlev) {6};

\draw (N2) -- (N3);
\draw (N2) -- (N1);
\draw (N6) -- (N5);
\end{tikzpicture}
} +
\scalebox{0.8}{\input{includes/figures/interval-posets/I6-ex4}} +
\scalebox{0.8}{\input{includes/figures/interval-posets/I6-ex5}}.
\end{align}

\`A partir de la description de la composition donnée par \eqref{eq:tamari_intervalles;composition-simple}, on a clairement que
\begin{equation}
\BB(I_1, I_2) = I_1~\pleft~u~\pright~I_2
\end{equation}
où $u$ est l'intervalle-poset possédant un unique sommet. Notons que l'ordre des opérations ne modifie pas le résultat : $(I_1 \pleft u) \pright I_2 = I_1 \pleft (u \pright I_2)$.

\subsection{\'Enumération des intervalles}
\label{sub-sec:tamari_intervalles:polynomes:enumeration}

L'opérateur $\OB$ peut aussi se décomposer en deux opérations, gauche et droite,
\begin{align}
\label{eq:tamari_intervalles:polleft}
f \polleft g &:= fg \\
\label{eq:tamari_intervalles:polright}
f \polright g &:= f \Delta(g),
\end{align}
où 
\begin{equation}
\label{eq:tamari_intervalles:delta}
\Delta(g) := \frac{x g(x,y) - g(1,y)}{x - 1}.
\end{equation}
On a dans ce cas
\begin{equation}
\OB(f,g) = f \polleft xy \polright g.
\end{equation}

La composition des intervalles-posets est une interprétation combinatoire de l'opérateur $\OB$ défini dans le théorème \ref{thm:tamari_intervalles:equation-fonct}, ce qui s'exprime par la proposition suivante.

\begin{Proposition}
\label{prop:tamari_intervalles:equiv-composition}
Soient $I_1$ et $I_2$ deux intervalles-posets et $\PI$ l'application linéaire de la définition \ref{def:tamari_intervalles:stat}. Alors
\begin{align}
\label{eq:tamari_intervalles:equiv-pleft}
\PI(I_1 \pleft I_2) &= \PI(I_1) \polleft \PI(I_2), \\
\label{eq:tamari_intervalles:equiv-pright}
\PI(I_1 \pright I_2) &= \PI(I_1) \polright \PI(I_2),
\end{align}
et donc
\begin{equation}
\PI(\BB( I_1,  I_2)) = \OB(\PI(I_1), \PI(I_2)).
\end{equation}
\end{Proposition}

Par exemple, dans la figure \ref{fig:tamari_intervalles:BTExample}, on a $\PI(I_1) = x^2 y^3$ et $\PI(I_2) = x^3 y^4$ et on vérifie que $\PI(\BB(I_1,I_2)) = y^8(x^6 + x^5 + x^4 + x^3) = \OB(x^2 y ^3, x^3 y^4)$.

\begin{proof}
Soient $I_1$ et $I_2$ deux intervalles-posets. Le produit gauche $I_1 \pleft I_2$ est la concaténation décalée de $I_1$ et $I_2$ à laquelle on a rajouté des relations décroissantes. On a clairement 
\begin{equation}
\PI(I_1 \pleft I_2) = y^{\Isize(I_1) + \Isize(I_2)} x^{\Itrees(I_1) + \Itrees(I_2)} = \PI(I_1)\PI(I_2) 
\end{equation}
ce qui prouve \eqref{eq:tamari_intervalles:equiv-pleft}.

\`A présent, si $\Itrees(I_2) = m$, et que les racines de $\dec(I_2)$ sont $x_1 < x_2 < \dots <x_m$, on a par définition
\begin{equation}
I_1 \pright I_2 = \sum_{0 \leq i \leq m} P_i
\end{equation}
où on a ajouté exactement $i$ relations décroissantes entre les racines $x_1, \dots, x_i$ de $\dec(I_2)$ et l'étiquette de valeur maximale de $I_1$. On a $\Itrees(P_i) = \Itrees(I_1) + \Itrees(I_2) - i$ car chaque relation décroissante relie un arbre de $I_2$ à un arbre de $I_1$. Alors
\begin{align}
\PI(I_1 \pright I_2) &= y^{\Isize(I_1) + \Isize(I_2)}x^{\Itrees(I_1)}(1 + x + x^2 + \dots x^m) \\
&= y^{\Isize(I_1) + \Isize(I_2)}x^{\Itrees(I_1)} \frac{x^{m+1} - 1}{x - 1} \\
&= \PI(I_1) \polright \PI(I_2).
\end{align}
\end{proof}

Pour prouver le théorème \ref{thm:tamari_intervalles:equation-fonct}, nous avons encore besoin d'un résultat.

\begin{Proposition}
\label{prop:tamari_intervalles:unicity-composition}
Soit $I$ un intervalle-poset, alors il n'existe qu'un seul couple d'intervalles-posets $(I_1,I_2)$ tel que $I$ apparaisse dans la somme $\BB(I_1,I_2)$.
\end{Proposition}

\begin{proof}
Soit $I$ un intervalle-poset de taille $n$ et soit $k$ le sommet de $I$ dont l'étiquette est maximale vérifiant que pour tout $i < k$, on a $i \trprec k$ . Notons que le sommet 1 vérifie cette propriété et donc que $k$ existe toujours. On prouve que $I$ apparaît uniquement dans la composition des intervalles $I_1$ et $I_2$ où $I_1$ est le sous-poset de $I$ restreint à $1, \dots, k-1$ et $I_2$ le sous-poset de $I$ réétiqueté restreint à $k+1, \dots, n$. Si $k=1$ (resp. $k=n$) alors $I_1$ (resp. $I_2$) est le poset vide.

Pouvons d'abord que $I \in \BB(I_1,I_2)$. Les conditions \ref{def:tamari_intervalles:composition:cond:I1}, \ref{def:tamari_intervalles:composition:cond:I2} et \ref{def:tamari_intervalles:composition:cond:incr} de la définition \ref{def:tamari_intervalles:composition} sont vérifiées par construction. Si la condition \ref{def:tamari_intervalles:composition:cond:decr} n'est pas vérifiée, cela signifie qu'il existe une relation $k \trprec j$ avec $j>k$. Alors, par définition des intervalles-posets, on a aussi $\ell \trprec j$ pour tout $k<\ell<j$. Par ailleurs, $i \trprec k \trprec j$ pour tout $i < k$, et donc quelque soit $i < j$, $i \trprec j$. C'est impossible car $k$ est l'étiquette maximale vérifiant cette condition.

On a donc $I \in \BB(I_1,I_2)$. C'est le seul couple d'intervalles possible. En effet, supposons que $I \in \BB(I_1',I_2')$. Le sommet $k' = |I_1'| +1$ vérifie par définition que pour tout $i<k'$, on a $i \trprec k'$ et pour tout $j>k'$, $k' \ntrprec j$. C'est exactement la définition de $k$. On a donc $k'=k$ ce qui implique $I_1' = I_1$ et $I_2' = I_2$.
\end{proof}

\begin{proof}[Démonstration du théorème \ref{thm:tamari_intervalles:equation-fonct}]
Soit $\SI = \sum_{T_1 \leq T_2} P_{[T_1,T_2]}$ la série formelle des intervalles-posets. D'après la proposition \ref{prop:tamari_intervalles:unicity-composition}, on a
\begin{equation}
\SI = \BB(\SI,\SI) + \emptyset.
\end{equation}
Et par la proposition \ref{prop:tamari_intervalles:equiv-composition}, 
\begin{align}
\Phi &= \PI(\SI) \\
&= \PI(\BB(\SI,\SI)) + 1 \\
&= \OB(\Phi,\Phi) + 1.
\end{align}
\end{proof}

\subsection{Comptage des éléments inférieurs à un arbre}
\label{sub-sec:tamari_intervalles:polynomes:smaller}

En développant \eqref{eq:tamari_intervalles:equation-fonct}, on obtient 
\begin{align}
\Phi &= 1 + \OB(1,1) + \OB(\OB(1,1),1) + \OB(1,\OB(1,1)) + \dots \\
&= \sum_T y^{|T|}\OBT_T,
\end{align}
où $\OBT_T$ est le polynôme de Tamari de la définition \ref{def:tamari_intervalles:tamari-polynomials}. On prouve le théorème~\ref{thm:tamari_intervalles:smaller-trees} par la proposition suivante.

\begin{Proposition}
\label{prop:tamari_intervalles:smaller-trees}
Soit $T := k(T_L,T_R)$ un arbre binaire et $S_T := \sum_{T' \leq T} P_{[T',T]}$ la somme des intervalles-posets dont $T$ est l'arbre maximal. Alors on a  $S = \BB(S_{T_L},S_{T_R})$.
\end{Proposition}

\begin{proof}
Soit $T$ un arbre binaire de taille $n$ tel que $T=k(T_L,T_R)$. L'intervalle initial $[T_0,T]$ correspond à la forêt initiale de $T$, $\inc(T)$ qui est un intervalle-poset particulier. D'après la proposition \ref{prop:tamari_intervalles:comb-prop} \ref{prop:tamari_intervalles:comb-prop-minmax-inclusion},
la somme $S_T$ est la somme sur tous les intervalles-posets $I$ qui sont des extensions de $\inc(T)$ où les relations ajoutées sont décroissantes.

Soit $I$ un intervalle de la somme $S_T$. Soient $I_L$ et $I_R$ les sous-posets formés par la restriction de $I$ à respectivement $1, \dots, k-1$ et $k+1, \dots, n$. D'après la définition récursive des forêts initiales décrites en figure \ref{fig:tamari_intervalles:forets}, $I_L$ et $I_R$ sont des extensions de respectivement $\inc(T_L)$ et $\inc(T_R)$ où seules des relations décroissantes ont été ajoutées. On a alors $I_L \in S_{T_L}$ et $I_R \in S_{T_R}$. Enfin, on a clairement que $I \in \BB(I_L,I_R)$ car comme $I$ est une extension de $\inc(T)$ on a en particulier $i \trprec k$ pour $i < k$ et $k \ntrprec j$ pour $j > k$.

Inversement, si $I_L$ et $I_R$ sont deux éléments de respectivement $S_{T_L}$ et $S_{T_R}$ alors tout intervalle $I$ de $\BB(I_L,I_R)$ appartient à $S_T$ par construction car il est bien une extension de $\inc(T)$ où seules des relations décroissantes ont été ajoutées.
\end{proof}

Dans la figure \ref{fig:tamari_intervalles:BTExample}, on reprend l'exemple du calcul de la figure \ref{fig:tamari_intervalles:BTExemple-2} en détaillant la liste des intervalles-posets de $S_T = \sum_{T' \leq T} P_{[T',T]}$.

\begin{figure}[ht]
  \input{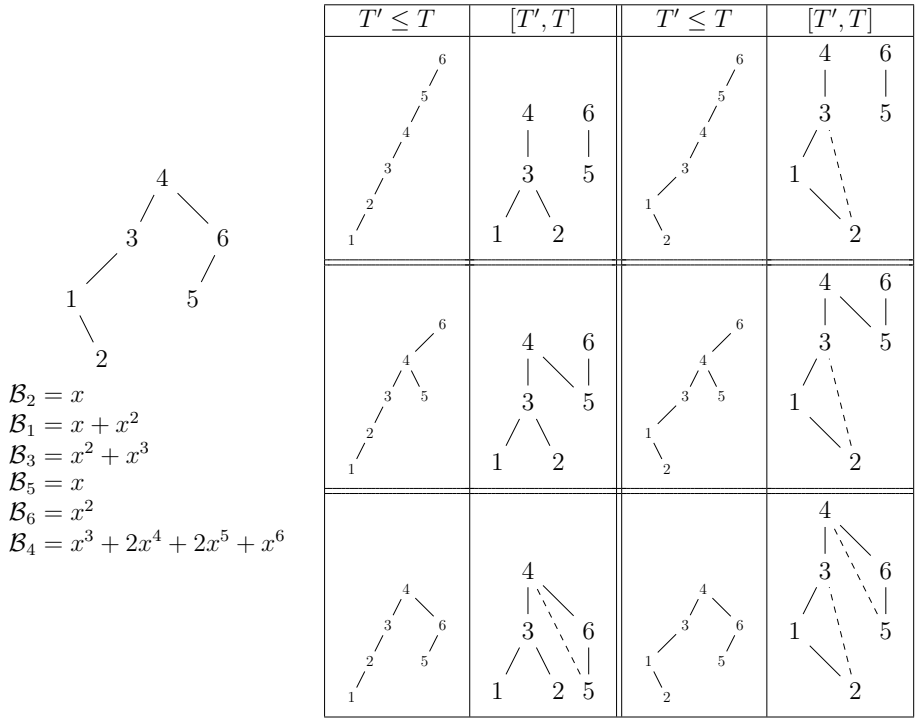}

  \caption{Exemple de calcul de $\OBT_T$ avec la liste des arbres plus petits et des intervalles-posets associés.}
  \label{fig:tamari_intervalles:BTExample}
\end{figure}

\begin{proof}[Démonstration du théorème \ref{thm:tamari_intervalles:smaller-trees}]
Compter le nombre d'arbres $T' \leq T$ en fonction du nombre de \noeuds sur la branche gauche de $T'$ revient à compter le nombre d'intervalles $I = [T',T]$ en fonction de $\Itrees(I)$. On souhaite donc prouver que $\OBT_T = \PI(S_T)$ où $S_T = \sum_{T' \leq T} P_{[T',T]}$. Cela se fait par récurrence sur la taille de $T$. Le cas initial est trivial. Soit $T=k(T_L,T_R)$, par hypothèse de récurrence on a que $\OBT_{T_L} = \PI(S_{T_L})$ et $\OBT_{T_R} = \PI(S_{T_R})$, alors les propositions \ref{prop:tamari_intervalles:equiv-composition} et \ref{prop:tamari_intervalles:smaller-trees} nous donnent
\begin{align}
\OBT_T &= \OB(\PI(S_{T_L}),\PI(S_{T_R})) \\
&= \PI(\BB(S_{T_L},S_{T_R})) \\
&= \PI(S_T).
\end{align}
\end{proof}

\subsection{Version bivariée}
\label{sub-sec:tamari_intervalles:polynomials:bivar}
\index{flots}

Dans un article très récent \cite{ChapBiVar}, Chapoton calcule des polynômes bivariés qui semblent correspondre à ceux que nous étudions. En calculant les premiers exemples de \cite[formule (7)]{ChapBiVar}, on remarque \cite{BiVar} que pour $b=1$ et $t = 1 - 1/x$ le polynôme défini est égal à $\OBT_T$ où $T$ est un arbre binaire dont le sous-arbre gauche est vide. Dans \cite{ChapBiVar}, le polynôme est indexé par un arbre enraciné. Cet arbre semble être la version non planaire de l'arbre planaire en bijection avec $T$ par $\dec(T)$.

Il est aussi possible d'ajouter un paramètre $b$ aux polynômes de la définition~\ref{def:tamari_intervalles:tamari-polynomials}. Pour un intervalle $[T',T]$, le paramètre $b$ compte le nombre de \noeuds de $T'$ qui possèdent un sous-arbre droit, ou de façon équivalente le nombre de \noeuds $x$ dans l'intervalle-poset $P_{[T',T]}$ qui ont une relation $y \trprec x$ avec $y>x$. On généralise alors l'opérateur $\PI$ en associant à chaque intervalle-poset un monôme en $x,y$ et $b$. On a 

\begin{equation}
\OB(f,g) = y \left( xbf \frac{x g - g_{x=1}}{x - 1} -bxfg + xfg \right),
\end{equation}
où $f$ et $g$ sont des polynômes en $x,y$ et $b$. La proposition \ref{prop:tamari_intervalles:equiv-composition} est toujours vérifiée car on ajoute un \noeud avec une relation décroissante dans tous les termes de la composition sauf un. Par exemple, pour le calcul présenté figure \ref{fig:tamari_intervalles:composition}, on a $\OB(y^3x^2b,y^4x^3b) = y^8(x^6b^2 + x^5b^3 + x^4b^3 + x^3b^3)$.

Avec cette définition du paramètre $b$, les polynômes bivariés $\OBT_T(x,b)$ semblent correspondre exactement aux polynômes calculés par Chapoton dans \cite{ChapBiVar} pris en $t = 1 - 1/x$. Le rapport entre les deux objets est le thème d'un travail commun entrepris avec Chapoton, Novelli et Thibon et sera étudié de façon plus complète dans une prochaine publication.


\chapter{Treillis de $m$-Tamari}
\label{chap:mtamari}
\index{treillis!de $m$-Tamari}

Dans un article récent \cite{BergmTamari}, Bergeron et Préville-Ratelle introduisent une famille de treillis généralisant le treillis de Tamari sur les chemins de Dyck. Pour un paramètre $m$ donné, on étudie l'ensemble des chemins dans le plan de $(0,0)$ à $(mn,n)$ formés de pas horizontaux $(1,0)$ et de pas verticaux $(0,1)$ et restant au dessus de la droite $y = \frac{x}{m}$. Dans le cas où $m=1$, ces chemins correspondent simplement à des chemins de Dyck où les pas montants ont été remplacés par des pas verticaux et les pas descendants par des pas horizontaux. Ce sont des objets combinatoires bien connus qui apparaissent en particulier dans le problème du scrutin et qui sont comptés par les nombres de $m$-Catalan,

\begin{equation}
\frac{1}{mn + 1} \binom{(m+1)n}{n}.
\end{equation}
Par la suite, on utilisera la dénomination anglo-saxonne \mpaths~pour l'ensemble de ces chemins.

\begin{figure}[ht]
\centering
\input{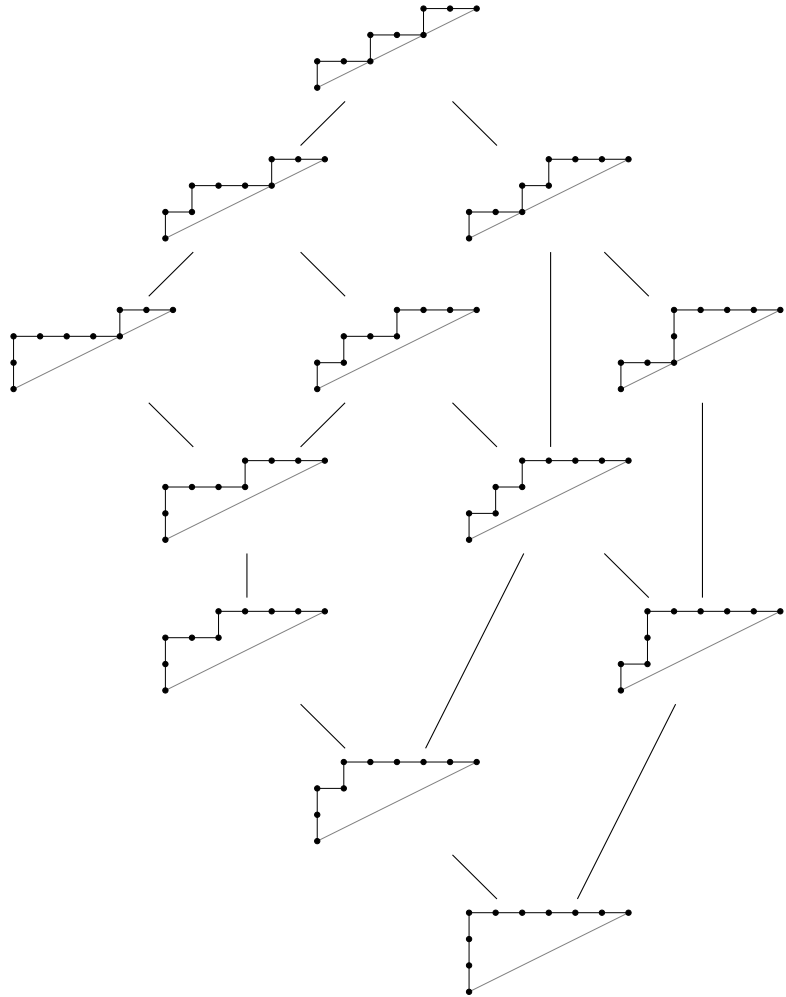}
\caption{Treillis de $m$-Tamari $\Tam{3}{2}$ sur les chemins.}
\label{fig:mtamari:mtam-3-2-paths}
\end{figure}

L'opération de rotation sur les chemins de Dyck (cf. figure \ref{fig:tamari_prelim:rot-dyck}) s'étend naturellement aux \mpaths. Elle induit aussi une structure de treillis \cite{BergmTamari} qu'on appelle \emph{treillis de $m$-Tamari}, un exemple est donné figure \ref{fig:mtamari:mtam-3-2-paths}. Quand $m=1$, on retrouve le cas classique du treillis de Tamari sur les chemins de Dyck. Les intervalles des treillis de $m$-Tamari ont été dénombrés dans \cite{mTamari}. La formule généralise celle de Chapoton \eqref{eq:tamari_intervalles:formule_chap}. On a
\begin{equation}
\label{eq:mtamari:formule_intervalles}
I_{n,m} = \frac{m+1}{n(mn +1)} \binom{(m+1)^2 n + m}{n - 1}
\end{equation}
où $I_{n,m}$ est le nombre d'intervalles dans $\Tamnm$, le treillis de $m$-Tamari pour les chemins de taille $n$. Comme dans le cas $m=1$, la formule est prouvée par la résolution d'une équation fonctionnelle sur la série génératrice des intervalles. On peut exprimer cette équation à l'aide d'un opérateur $m+1$-linéaire, généralisation de l'opérateur bilinéaire $\OB$ défini dans le chapitre \ref{chap:tamari_intervalles} \eqref{eq:tamari_intervalles:def-B}. Soit $\OBm$ l'opérateur défini par
\begin{equation}
\label{eq:mtamari:def-Bm}
\OBm(f,g_1, \dots, g_m) = fxy \Delta(g_1 \Delta(g_2 \Delta( \dots \Delta(g_m)) \dots )
\end{equation}
où $\Delta$ est la différence divisée déjà définie en \eqref{eq:tamari_intervalles:delta},
\begin{equation}
\Delta(g) = \frac{x g(x,y) - g(1,y)}{x - 1}.
\end{equation}
Si $\Phim(x,y)$ est la série génératrice des intervalles de $\Tamnm$ où $y$ compte la taille des chemins et $x$ la statistique des retours à 0 sur le chemin inférieur, on a \cite{mTamari}
\begin{equation}
\Phim(x,y) = 1 + \OBm(\Phim, \Phim, \dots, \Phim).
\end{equation}

La structure de l'équation fonctionnelle généralise donc directement celle du cas $m=1$. En développant l'expression, on obtient une somme sur les arbres $m+1$-aires. Cela laisse penser que les résultats du chapitre précédent peuvent se généraliser aux treillis $m$-Tamari. C'est en effet le cas et on obtient ainsi une nouvelle preuve que la série génératrice des intervalles de $\Tamnm$ vérifie bien l'équation fonctionnelle. Par ailleurs, en généralisant le théorème \ref{thm:tamari_intervalles:smaller-trees}, on obtient une formule comptant le nombre d'éléments plus petits ou égaux à un élément donné. La preuve de ces deux résultats est donné au paragraphe \ref{sec:mtamari:intervalles}. 

Par ailleurs, dans les récents articles traitant des treillis de $m$-Tamari, le problème était laissé ouvert de l'interprétation en termes d'arbres de ces treillis. La question était pour nous fondamentale car nos démonstrations dans le cas du treillis de Tamari classique utilisent comme base les arbres binaires et le lien avec l'ordre faible. Pour y répondre, nous utilisons le plongement naturel du treillis $\Tamnm$ dans le treillis de Tamari $\Tam{n \times m}{1}$ qu'avait déjà décrit \cite{mTamari}. Ainsi, les intervalles-posets définis au chapitre précédent se généralisent simplement : les intervalles de $m$-Tamari étant des cas particulier d'intervalles de Tamari.

Nous commençons donc par donner la description en termes d'arbres des treillis de $m$-Tamari dans le paragraphe \ref{sec:mtamari:def}. En particulier, nous décrivons les relations de couverture du treillis de $m$-Tamari sur les arbres $m+1$-aires. Nous utilisons ensuite cette description pour généraliser l'opération de composition des intervalles-posets que nous avons définie au chapitre précédent. Nous démontrons ainsi au paragraphe \ref{sec:mtamari:intervalles} les généralisations des théorèmes \ref{thm:tamari_intervalles:smaller-trees} et \ref{thm:tamari_intervalles:equation-fonct}. Enfin, dans le paragraphe \ref{sec:mtamari:alg}, nous décrivons des structures d'algèbres de Hopf généralisant les cas de $\FQSym$ et $\PBT$ à des versions dites "$m$". Les résultats de ce chapitre sont issus d'un travail commun avec Grégory Chatel, Jean-Christophe Novelli et Jean-Yves Thibon et feront prochainement l'objet de publications.   

\section{Les treillis $m$-Tamari sur les arbres}
\label{sec:mtamari:def}

\subsection{Définition sur les chemins}
\label{sub-sec:mtamari:def:chemins}
\index{ballot paths}

\begin{Definition}
\label{def:mtamari:m-ballot-paths}
Un \mpath~de taille $n$ est un chemin dans le plan depuis l'origine $(0,0)$ jusqu'au point $(nm,n)$ formé de pas "montants" verticaux $(0,1)$ et de pas "descendants" horizontaux $(1,0)$ tel que le chemin reste toujours au dessus de la droite $y = \frac{x}{m}$.  
\end{Definition}

Tout comme les chemins de Dyck, les \mpaths~peuvent s'interpréter comme des mots sur un alphabet binaire $\lbrace 1, 0 \rbrace$ où les pas montants sont codés par la lettre 1 et les pas descendants par 0. De même, un chemin est dit \emph{primitif} s'il n'a pas d'autres contacts avec la droite $y = \frac{x}{m}$ que ses extrémités. La \emph{rotation} est alors définie comme dans la définition \ref{def:tamari_prelim:dyck_rotation} et illustrée figure \ref{fig:mtamari:mpath-rot}.

\begin{figure}[ht]
\centering
\input{includes/figures/rotation_mpath}
\caption{Rotations sur les \mpaths.}
\label{fig:mtamari:mpath-rot}
\end{figure}

Si on considère la rotation sur les chemins comme une relation de couverture, l'ordre induit est un treillis qui généralise l'ordre de Tamari usuel \cite{BergmTamari}. En exemple, on donne l'ordre sur les \kpaths{2} de taille 3, $\Tam{3}{2}$ figure \ref{fig:mtamari:mtam-3-2-paths}.

En remplaçant chaque pas montant par une suite de $m$ pas montants, on peut faire correspondre injectivement un chemin de Dyck de taille $m.n$ à chaque \mpath. L'ensemble obtenu est composé de chemins de Dyck dont les cardinaux des suites de pas montants sont divisibles par $m$. On appelle ces objets les chemins de $m$-Dyck. Un exemple de la correspondance est donnée figure \ref{fig:mtamari:m-dyck}.

\begin{figure}[ht]
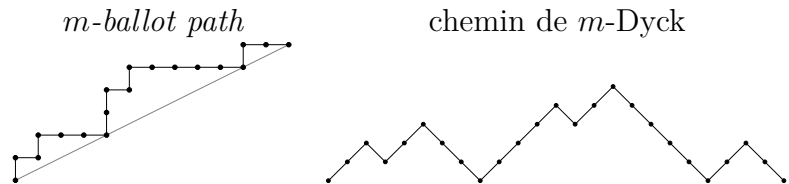

\centering

\def \fpath{includes/figures/}

\begin{tabular}{cc}
\mpath & chemin de $m$-Dyck \\
\scalebox{0.3}{\input{\fpath mpaths/P6-2-ex1}} &
\scalebox{0.25}{\input{\fpath dyck/D12-ex1}}
\end{tabular}
\caption{Un \mpath~et son chemin de $m$-Dyck correspondant}
\label{fig:mtamari:m-dyck}
\end{figure}

Le passage d'un \mpath~à son chemin de $m$-Dyck est compatible avec la rotation. On en déduit la propriété suivante que l'on trouve déjà dans \cite{mTamari} :

\begin{Proposition}
Le treillis de $m$-Tamari $\Tamnm$ est isomorphe à l'idéal supérieur de $\Tam{n\times m}{1}$ engendré par le chemin de Dyck $(1^m 0^m)^n$ (cf. figure \ref{fig:mtamari:minimal}).  
\end{Proposition}

De cette observation triviale, on déduit la plupart des propriétés des treillis de $m$-Tamari. 

\subsection{Arbres $m$-binaires}
\label{sub-sec:mtamari:def:mbinaires}
\index{arbres!$m$-binaires}

Nous avons décrit dans le paragraphe \ref{sub-sec:tamari_prelim:tamari:def} la bijection entre les arbres binaires et les chemins de Dyck. L'image du chemin de $m$-Dyck minimal est donnée par un arbre binaire qu'on appelle \emph{peigne-$(n,m)$}.

\begin{Definition}
\label{def:mtamari:peigne-nm}
Le peigne-$(n,m)$ est l'arbre binaire $T$ de taille $n\times m$ tel que $T$ possède $n$ \noeuds sur sa branche gauche et tel que le sous-arbre droit de chacun de ces \noeuds soit une suite de $m-1$ fils droits sans fils gauches (cf. figure \ref{fig:mtamari:minimal}).
\end{Definition}

\begin{figure}[ht]
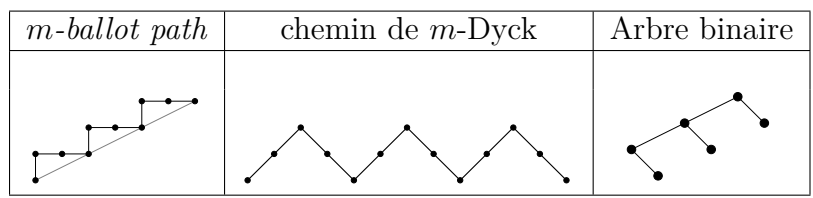

\centering

\def \fpath{includes/figures/}

\begin{tabular}{|c|c|c|}
  \hline
  \mpath & chemin de $m$-Dyck & Arbre binaire \\
  \hline
  & & \\
  \scalebox{0.35}{
  \input{\fpath mpaths/P3-2-a}
    } &
  \scalebox{0.35}{
    \input{\fpath dyck/P3-2-a}
  }&
  \scalebox{0.7}{
    \input{\fpath mbinary/P3-2-a}
  }\\\hline
\end{tabular}

\caption{\'Elément minimal de $\Tam{3}{2}$ en tant que \mpath, chemin de Dyck et arbre binaire (peigne-$(3,2)$).}
\label{fig:mtamari:minimal}
\end{figure}

Le treillis de $m$-Tamari est donc identifiable à l'idéal supérieur engendré par le peigne-$(n,m)$. Ces arbres sont comptés par les nombres de $m$-Catalan et possèdent une structure non plus binaire mais $m+1$-aire. 

\begin{figure}[ht]
\centering
\input{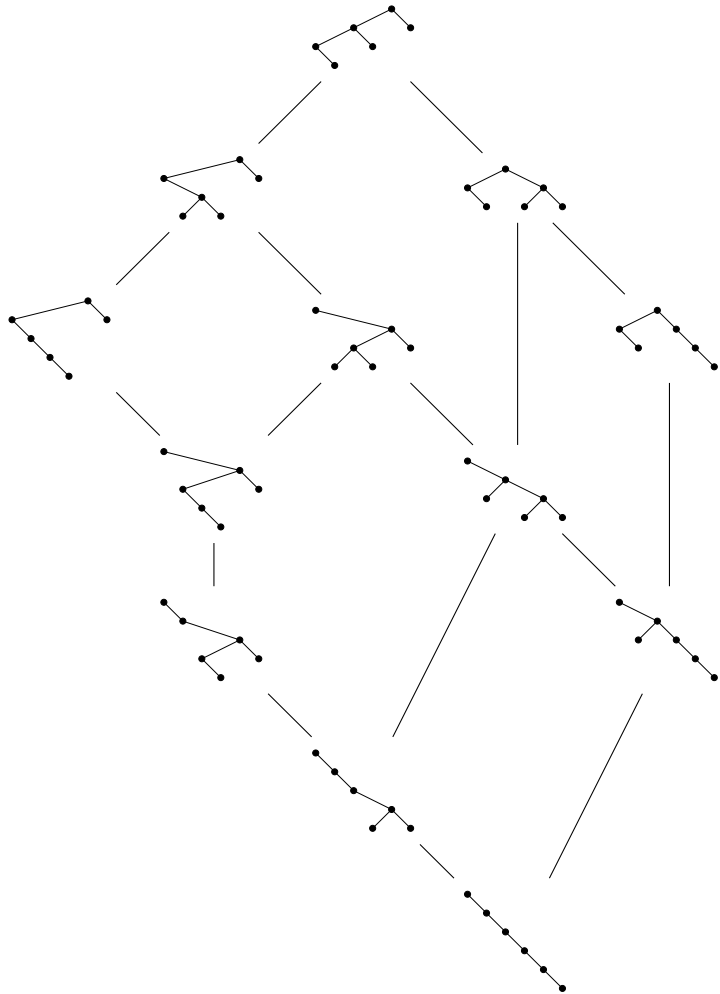}
\caption{Treillis de $m$-Tamari $\Tam{3}{2}$ sur les arbres $m$-binaires}
\label{fig:mtamari:mtam-3-2-mbinary}
\end{figure}

\begin{figure}[ht!]
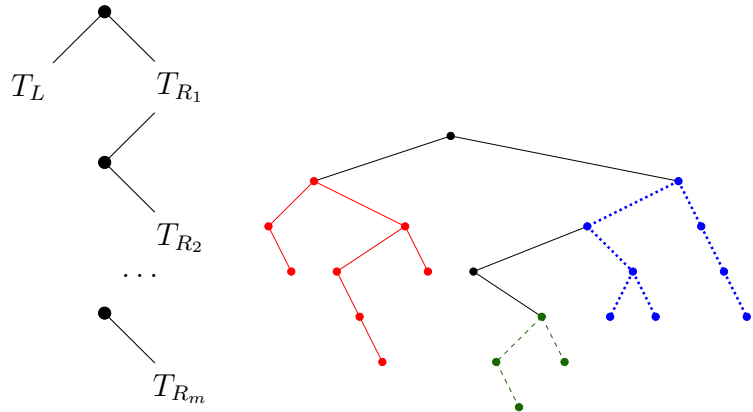

\centering
\begin{tabular}{cc}
\input{includes/figures/m-binary-trees}
&
\scalebox{0.6}{\input{includes/figures/exemple-arbre-m-binaire}}
\end{tabular}
\caption[Structure des arbres $m$-binaires]{Structure des arbres $m$-binaires. A droite, un exemple pour $m=2$ : $T_L$ est en rouge, $T_{R_1}$ en bleu (pointillés) et $T_{R_2}$ en vert (tirets).}
\label{fig:mtamari:mbinary}
\end{figure}

\begin{Definition}
On appelle \emph{arbres $m$-binaires} les arbres binaires qui vérifient la structure récursive décrite par la figure \ref{fig:mtamari:mbinary}. Un arbre $m$-binaire est soit l'arbre vide, soit un ensemble de $m+1$ arbres $m$-binaires $T_L, T_{R_1}, T_{R_2}, \dots, T_{R_m}$, greffés sur $m$ \noeuds racines. 

L'arbre $T_L$ est le sous-arbre gauche du premier \noeud racine. Le sous-arbre droit est formé de $T_{R_1}$ auquel on a greffé le deuxième \noeud racine à gauche de son \noeud le plus à gauche. Le sous-arbre droit de ce \noeud racine est alors $T_{R_2}$ auquel on a greffé le troisième \noeud racine et ainsi de suite. Si un arbre $T_{R_i}$ est vide, alors le \noeud racine $i+1$ est directement le fils droit du \noeud racine $i$.
\end{Definition}

\pagebreak

\begin{Proposition}
\label{prop:mtamari:carac-mbinaire}
Un arbre binaire $T$ est un arbre $m$-binaire si et seulement si il est de taille $n \times m$ et que son arbre binaire de recherche vérifie :
\begin{align}
\label{eq:mtamari:mbinaire-cond}
\nonumber
m &\trprec_T m-1 \trprec_T \dots \trprec_T 1 ,\\
\nonumber
2m &\trprec_T 2m-1 \trprec_T \dots \trprec_T m+1, \\
\dots \\
\nonumber
n.m &\trprec_T n.m-1 \trprec_T \dots \trprec_T (n-1).m +1.
\end{align}
\end{Proposition}

On peut vérifier la propriété sur la figure \ref{fig:mtamari:exemple-m-binaire} qui est l'arbre binaire de recherche de l'exemple donné figure \ref{fig:mtamari:mbinary}.

\begin{figure}[ht]
\centering
\scalebox{0.7}{
\input{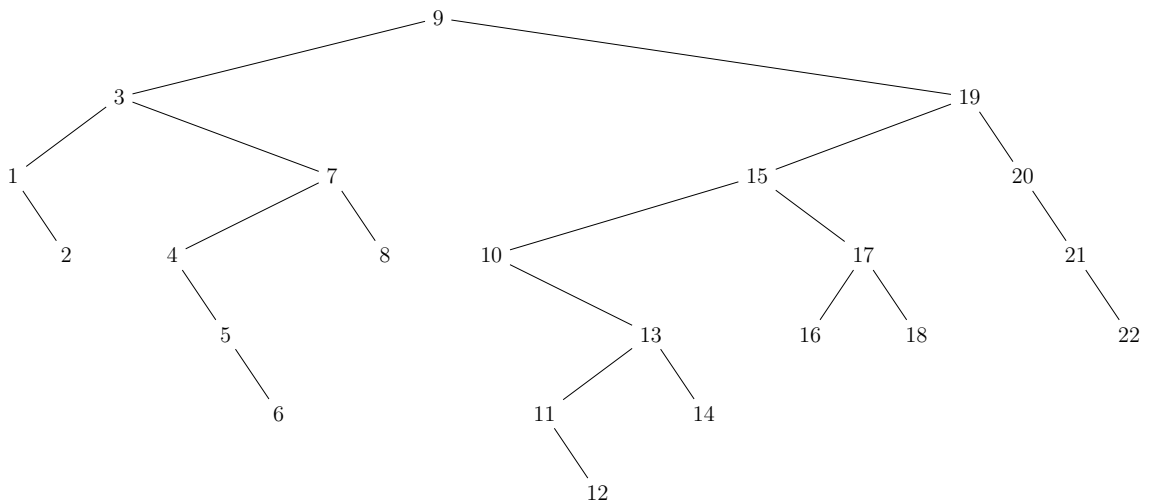}
}
\caption[Arbre binaire de recherche d'un arbre $m$-binaire]{Arbre binaire de recherche d'un arbre $m$-binaire. On peut vérifier qu'on a toujours $2k \trprec 2k -1$.}
\label{fig:mtamari:exemple-m-binaire}
\end{figure}

\begin{proof}
On prouve la propriété par récurrence sur $n$. Soit $T$ un arbre $m$-binaire composé de $T_L, T_{R_1}, \dots, T_{R_m}$. Supposons que $T_L, T_{R_1}, \dots, T_{R_m}$ sont des arbres $m$-binaires et qu'ils vérifient la propriété. Alors l'arbre $T$ la vérifie aussi. La racine de $T$ est étiquetée par $|T_L| +1$ et comme $T_L$ est $m$-binaire, alors $|T_L| = k.m$ pour un certain $k \in \NN$. L'étiquette de la racine de $T$ est donc $x = k.m + 1$ et pour tout $j>x$, on a $j \trprec_T x$. C'est vrai en particulier pour $x+1, x+2, \dots, x+m-1$. Par ailleurs, au sein de $T$, l'étiquetage de $T_L$ reste inchangé, l'étiquetage de $T_{R_m}$ est décalé de $|T_L| + m$, celui de $T_{R_{m-1}}$ de $|T_L| + m + |T_{R_m}|$ etc. Les étiquetages sont décalés uniquement par des multiples de $m$ et donc si la propriété \eqref{eq:mtamari:mbinaire-cond} était respectée dans $T_L$, $T_{R_1}, \dots, T_{R_m}$, elle l'est toujours dans $T$.

\`A présent, soit $T$ un arbre binaire de recherche vérifiant \eqref{eq:mtamari:mbinaire-cond}. Soit $x$ la racine de $T$. Comme $x$ ne précède aucun élément de $T$, on a forcément que $x = k.m + 1$ pour un certain $0 \leq k < n$. Soit $T_L$ le sous arbre gauche de $T$, alors $|T_L| = km$ et par récurrence, $T_L$ est un arbre $m$-binaire. On a que $x+1 \trprec_T x$, c'est-à-dire $x+1$ est dans le sous-arbre droit de $x$. C'est le \noeud le plus à gauche du sous-arbre droit. Soit $T_{R_1}$, l'arbre placé entre $x$ et $x+1$. Pour $0 \leq a < n$, et $ 1 \leq b \leq m$, on a que le \noeud $y = a.m +b$ est dans $T_{R_1}$ si et seulement si 
 tous les \noeuds $a.m + m \trprec_T a.m + m-1 \trprec_T \dots \trprec_T a.m + 1$ sont aussi dans $T_{R_1}$. L'arbre est donc d'une taille multiple de $m$ et vérifie \eqref{eq:mtamari:mbinaire-cond}, c'est un arbre $m$-binaire par récurrence. Le même raisonnement s'applique à $T_{R_2}$, $T_{R_3}$, \dots, $T_{R_m}$, et $T$ vérifie bien la structure récursive.
\end{proof}

\begin{Proposition}
L'idéal engendré par le peigne-$(n,m)$ est l'ensemble des arbres $m$-binaires.
\end{Proposition}

\begin{proof}
La forêt finale du peigne-$(n,m)$ est exactement le poset formé par les relations \eqref{eq:mtamari:mbinaire-cond} et on a prouvé par la proposition \ref{prop:mtamari:carac-mbinaire} que les arbres $m$-binaires étaient exactement ceux dont la forêt finale était une extension du poset \eqref{eq:mtamari:mbinaire-cond}. Cela prouve le résultat par les propriétés des intervalles-posets (proposition~\ref{prop:tamari_intervalles:comb-prop}).
\end{proof}

Cette description de $m$-Tamari sur les arbres $m$-binaires nous permet de généraliser les résultats du chapitre \ref{chap:tamari_intervalles}. Par ailleurs, on peut utiliser les arbres $m$-binaires comme outil pour construire l'ordre de $m$-Tamari sur les arbres $m+1$-aires. La figure \ref{fig:mtamari:mbinary} donne $\Tam{3}{2}$ sur les arbres $m$-binaires.

\begin{figure}[ht]
\centering
\input{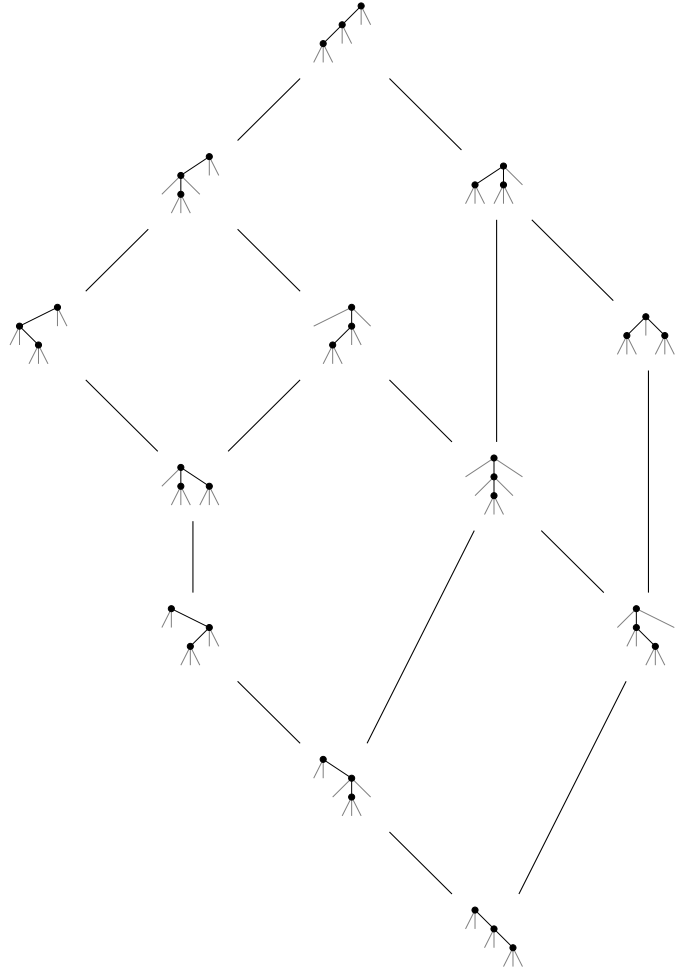}
\caption{Treillis de $m$-Tamari $\Tam{3}{2}$ sur les arbres ternaires}
\label{fig:mtamari:mtam-3-2-mary}
\end{figure}

\subsection{Arbres $m+1$-aires}
\label{sub-sec:mtamari:def:arbres-m-aires}

Les arbres $m$-binaires ont une structure $m+1$-aire : on peut donc leur associer un arbre $m+1$-aire. Si $T$ est un arbre $m$-binaire non vide, $T$ est composé de $T_L, T_{R_1}, \dots, T_{R_m}$, on lui associe l'arbre $\tilde{T}$ dont les sous-arbres sont dans cet ordre $\tilde{T_L}, \tilde{T_{R_1}}, \dots, \tilde{T_{R_m}}$. Un exemple est donné en figure \ref{fig:mtamari:path-tree}. 

La bijection entre les chemins de Dyck et les arbres binaires donne le passage d'un chemin de $m$-Dyck à un arbre $m$-binaire et donc par extension d'un \mpaths~à un arbre $m+1$-aire. Les \mpaths~admettent aussi une structure $m+1$-aire. Un \mpath~$D$ s'exprime récursivement comme 
\begin{equation}
D = D_L~1~D_{R_m}~0~D_{R_{m-1}}~\dots 0~D_{R_1}~0
\end{equation}
et cette structure reflète exactement celle de l'arbre correspondant, comme on peut le voir dans la figure \ref{fig:mtamari:path-tree}.

Cette bijection permet de représenter l'ordre de $m$-Tamari sur les arbres $m+1$-aires, ce que nous avons fait dans les figures \ref{fig:mtamari:mtam-2-2} et \ref{fig:mtamari:mtam-3-2-mary}. La relation de couverture se comprend à partir des arbres $m$-binaires. A partir d'un arbre $m$-binaire, deux types de rotations sont possibles : entre la racine de l'arbre $T_L$ et la racine de $T$, et entre un \noeud racine de $T$ et le \noeud le plus à gauche d'un arbre $T_{R_i}$. Ces deux rotations donnent deux types de transformations sur les arbres $m+1$-aires : passage de la branche gauche à une des branches droites et passage d'une branche droite à une autre. On donne le schéma général de ces deux opérations sur les arbres $m$-binaires et leur traduction en terme d'arbres $m+1$-aires dans figures \ref{fig:mtamari:rotation-type1} et \ref{fig:mtamari:rotation-type2}.

\begin{figure}[ht]
\centering
\input{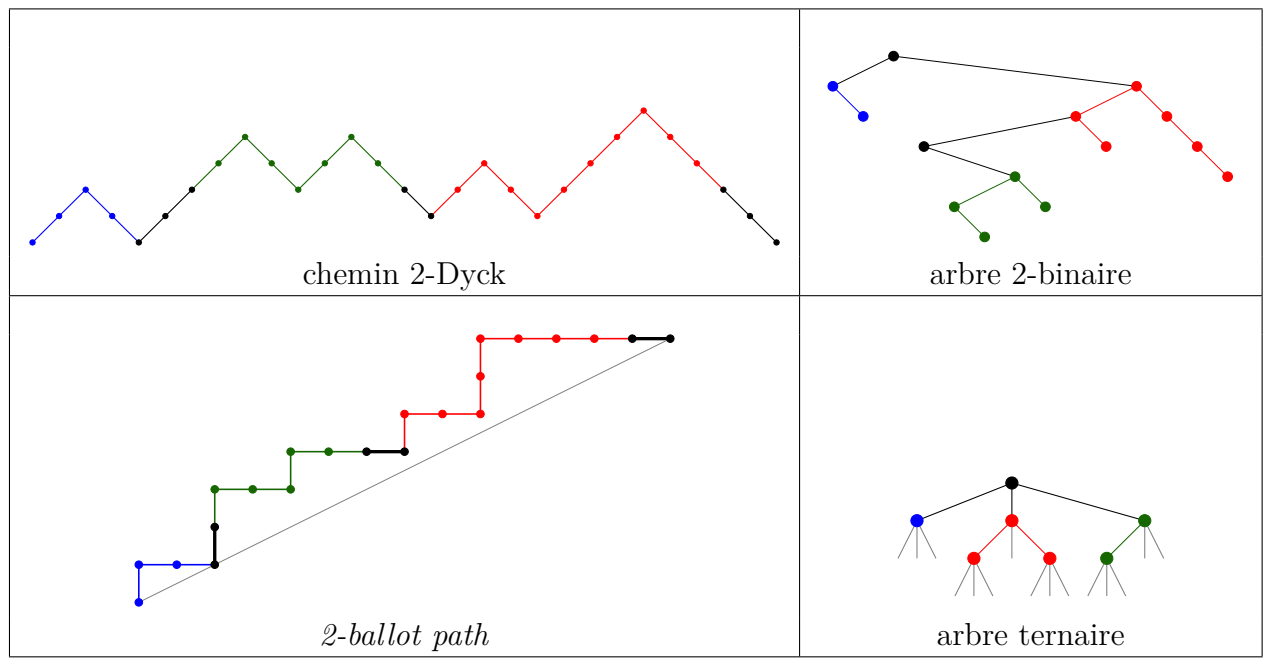}
\caption[Correspondance entre les arbres ternaires et les \kpaths{2}]{Correspondance entre les arbres ternaires et les \kpaths{2}. Quatre représentations du même objet de $\Tam{7}{2}$}
\label{fig:mtamari:path-tree}
\end{figure}

\begin{figure}[h!]
\centering
\input{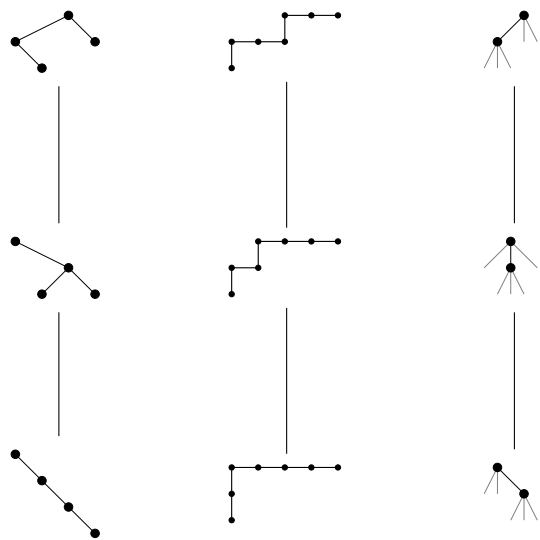}
\caption{Treillis de $m$-Tamari $\Tam{2}{2}$ sur les \mpaths, les arbres $m$-binaires et les arbres ternaires.}
\label{fig:mtamari:mtam-2-2}
\end{figure}

\begin{figure}[p]
\begin{leftfullpage}
\centering
\begin{tabular}{c}
Schéma général \\
\input{includes/figures/mtamari-rotation-type1}
\\
\\
 \\
 Exemple \\
 \input{includes/figures/mtamari-exemple-rotation1}
\end{tabular}
\end{leftfullpage}
\caption{Rotation de type 1 sur les arbres $m$-binaires et ternaires.}
\label{fig:mtamari:rotation-type1}
\end{figure}

\begin{figure}[p]
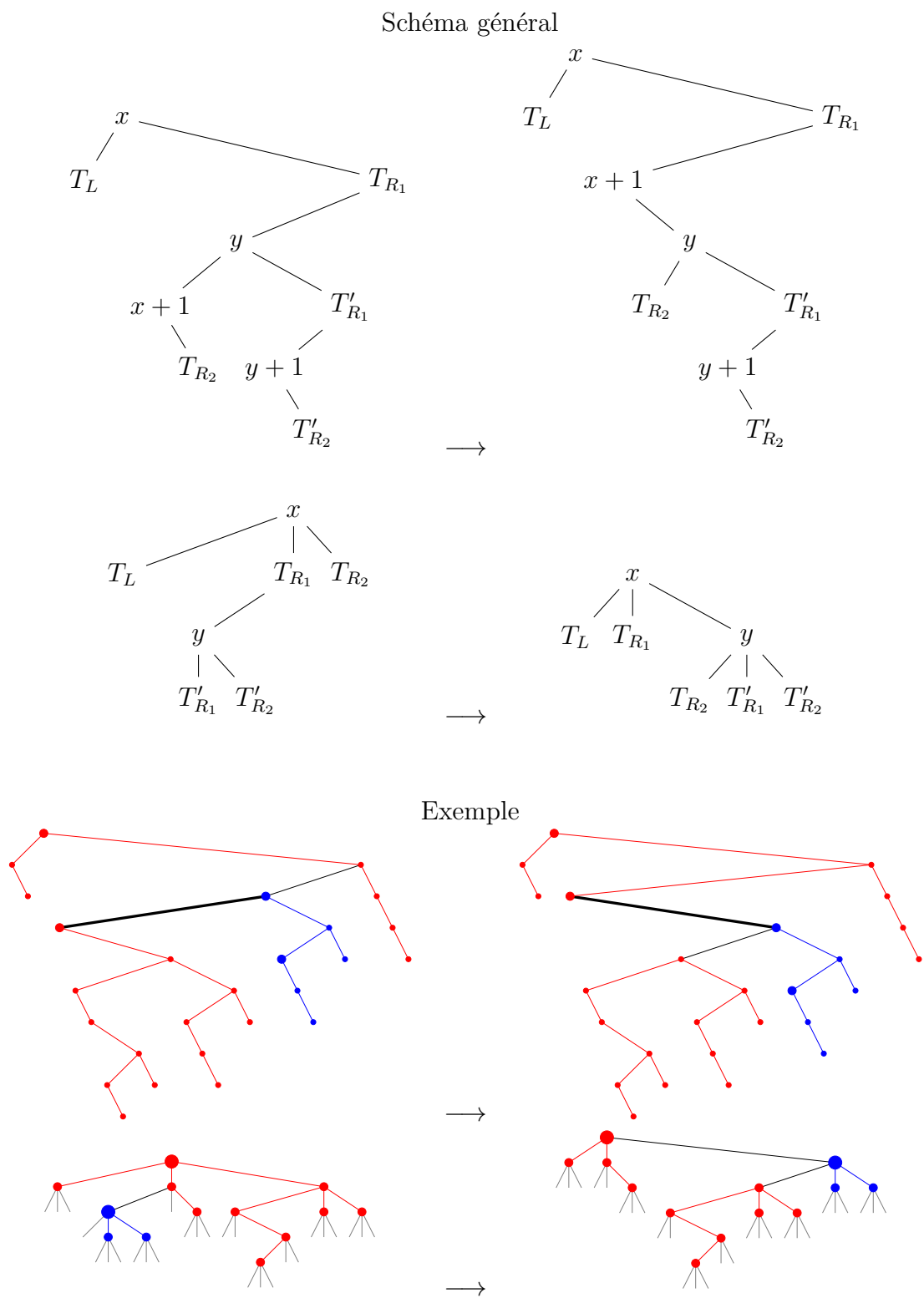

\centering
\begin{tabular}{c}
Schéma général \\
\input{includes/figures/mtamari-rotation-type2} \\
\\
\\
Exemple \\
\input{includes/figures/mtamari-exemple-rotation2}
\end{tabular}
\caption{Rotation de type 2 sur les arbres $m$-binaires et ternaires.}
\label{fig:mtamari:rotation-type2}
\end{figure}

\section{Intervalles}
\label{sec:mtamari:intervalles}
\index{intervalles!de $m$-Tamari}

Dans cette section, on donne les preuves des deux théorèmes suivants qui généralisent les théorèmes \ref{thm:tamari_intervalles:equation-fonct} et \ref{thm:tamari_intervalles:smaller-trees} au cas $m$-Tamari.

\begin{Theoreme}
\label{thm:mtamari:equation-fonct}
Soit $\Phim(x,y)$ la série génératrice des intervalles de $m$-Tamari où $y$ compte la taille des chemins et $x$ le nombre de retours à 0 du chemin inférieur (nombre de contacts avec la droite $y=\frac{x}{m}$ après l'origine). Alors
\begin{equation}
\Phim = \OBm(\Phim, \dots, \Phim) + 1
\end{equation}
où $\OBm$ est l'opérateur $m+1$-linéaire défini par 
\begin{equation}
\label{eq:mtamari:def-Bm2}
\OBm(f,g_1, \dots, g_m) = f \polleft xy \polright (g_1 \polright ( \dots \polright (g_{m-1} \polright g_m) ) \dots).
\end{equation}
avec $\polleft$ et $\polright$ les produits gauches et droits définis au chapitre précédent \eqref{eq:tamari_intervalles:polleft} et \eqref{eq:tamari_intervalles:polright}.
\end{Theoreme}

La définition de l'opérateur $\OBm$ \eqref{eq:mtamari:def-Bm2} est une réécriture de celle donnée en \eqref{eq:mtamari:def-Bm} et ce théorème est une reformulation de la proposition 8 de \cite{mTamari} dont nous proposons une nouvelle preuve. Le théorème suivant est un nouveau résultat sur les treillis $m$-Tamari.

\begin{Theoreme}
\label{thm:matamari:smaller-trees}
Soit $T$ un arbre $m+1$-aire, on définit récursivement le polynôme $\OBTm{T}$ par
\begin{align}
\OBTm{\emptyset} &= 1 \\
\OBTm{T} &= \OBm_{y=1}(\OBTm{T_L},\OBTm{T_{R_1}}, \dots, \OBTm{T_{R_m}}) 
\end{align}
où $T_L,T_{R_1}, \dots, T_{R_m}$ sont les sous-arbres de $T$. Alors $\OBTm{T}$ compte le nombre d'éléments inférieurs ou égaux à $T$ dans le treillis $\Tamnm$ en fonction du nombre de \noeuds sur la branche gauche de $T$, ou de façon équivalente en fonction du nombre de retours à 0 dans le chemin correspondant à l'arbre $T$. En particulier, $\OBTm{T}(1)$ est le nombre d'éléments inférieurs ou égaux à $T$.
\end{Theoreme}

Un exemple du calcul de $\OBTm{T}$ est donné figure~\ref{fig:mtamari:BmT-exemple}.

\begin{figure}[ht]
\centering
\input{includes/figures/mtamari_BmTExemple}
\caption[Exemple du calcul de $\OBTm{T}$]{Exemple du calcul de $\OBTm{T}$. On obtient $\OBTm{T}(1) = 5$ pour l'arbre $T$ en bas du graphe. On peut vérifier sur la figure que la puissance de $x$ correspond aux nombres de \noeuds sur la branche gauche des arbres ou au nombre de retours à 0 sur les chemins.}
\label{fig:mtamari:BmT-exemple}
\end{figure}

\subsection{Composition des $m$-intervalles-posets}
\label{sub-sec:mtamari:intervalles:composition}

\begin{Definition}
Un $m$-intervalle-poset est un intervalle-poset de taille $n \times m$ vérifiant 
\begin{align}
\label{eq:mtamari:m-intervalles-cond}
\nonumber
m &\trprec m-1 \trprec \dots \trprec 1, \\
\nonumber
2m &\trprec 2m-1 \trprec \dots \trprec m+1, \\
\dots \\
\nonumber
n.m &\trprec n.m-1 \trprec \dots \trprec (n-1).m +1.
\end{align}
\end{Definition}

\begin{Proposition}
Les $m$-intervalles-posets de taille $n$ sont en bijection avec les intervalles de $\Tam{n}{m}$.
\end{Proposition}
\begin{proof}
Un $m$-intervalle-poset $I$ correspond à un intervalle $[T_1,T_2]$ de Tamari $n\times m$ . D'après la proposition \ref{prop:mtamari:carac-mbinaire}, l'arbre $T_1$ est un arbre $m$-binaire. Comme $T_2 \geq T_1$, alors $T_2$ est aussi $m$-binaire et $I$ correspond donc à un intervalle de $\Tam{n}{m}$.
\end{proof}

Par ailleurs, le nombre de \noeuds sur la branche gauche d'un arbre $m$-binaire est le même que sur son arbre $m+1$-aire associé et correspond au nombre de retours à 0 du \mpath. Soit l'opérateur défini sur les $m$-intervalles-posets,
\begin{equation}
\PIm(I) := x^{\Itrees(I)}y^{\frac{\Isize(I)}{m}}.
\end{equation}
Alors,
\begin{equation}
\Phim(x,y) = \sum_{I} \PIm(I)
\end{equation}
sommé sur l'ensemble des $m$-intervalles-posets.

La composition $\BB$ de deux $m$-intervalles-posets ne donne pas une somme sur des $m$-intervalles-posets : les tailles ne sont plus des multiples de $m$. Il faut définir une $m$-composition qui soit $m+1$-linéaire et généralise la composition $\BB$. Si on se contente de traduire l'expression \eqref{eq:mtamari:def-Bm2} en transformant les produits droites $\polright$ et gauches $\polleft$ en respectivement $\pright$ et $\pleft$, on ne génère pas l'ensemble des $m$-intervalles-posets. Cependant, on remarque que
\begin{equation}
g_1 \polright g_2 = \frac{(xy \polright g_2) \polleft g_1}{xy}.
\end{equation}
Cette expression reflète la structure des arbres $m$-binaires (cf. figure \ref{fig:mtamari:mbinary}). On peut récrire \eqref{eq:mtamari:def-Bm2} à partir de cette observation. Par exemple, pour $m=3$ on a
\begin{align}
\OBk{3}(f,g_1,g_2,g_3) &= f \polleft xy \polright \left( g_1 \polright \left( g_2 \polright g_3 \right) \right) \\
&= f \polleft xy \polright \frac{1}{xy} \left( \left( xy \polright \frac{1}{xy}\left( (xy \polright g_3) \polleft g_2 \right) \right) \polleft g_1 \right) \\
&= \frac{1}{y^2} \left( f \polleft xy \polrightx \left( \left( xy \polrightx \left( (xy \polright g_3) \polleft g_2 \right) \right) \polleft g_1 \right) \right)
\end{align}
où
\begin{align}
f \polrightx g &:= f \polright (\frac{g}{x}) \\
&= f \Delta(\frac{g}{x}).
\end{align}
L'opération $\polrightx$ se traduit sur les intervalles-posets.

\begin{Definition}
\label{def:mtamari:polrightx}
Soient $I_1$ et $I_2$ deux intervalles-posets tels que $\Itrees(I_2) = k$. On note $y$ le label maximal des \noeuds de $I_1$ et $x_1, \dots, x_k$ les racines des arbres de $\dec(I_2)$. Alors $I_1 \prightx I_2$ est la somme des $k$ intervalles-posets $P_1, \dots, P_k$ où $P_i$ est la concaténation décalée de $I_1$ et $I_2$ où l'on a ajouté exactement $i$ relations décroissantes : $x_j \trprec y$ pour $j\leq i$.
\end{Definition}

La somme $I_1 \prightx I_2$ est la somme $I_1 \pright I_2$ de la définition \ref{def:tamari_intervalles:left-right-product} moins le poset $P_0$, la concaténation décalée de $I_1$ et $I_2$ auquel aucune relation décroissante n'a été ajoutée. En particulier, tous les intervalles-posets obtenus admettent la relation $y+1 \trprec y$.

\begin{Proposition}
\label{prop:mtamari:mcomposition}
On définit l'opérateur $m+1$-linéaire sur les $m$-intervalles-posets $\BBm$ par
\begin{equation}
\BBm(I_L, I_{R_1}, I_{R_2}, \dots, I_{R_m}) = I_L \pleft~u~\prightx \left( (u~\prightx (u \prightx \dots ((u \pright I_{R_m}) \pleft I_{R_{m-1}}) \pleft \dots )\pleft I_{R_1} \right) 
\end{equation}
où $u$ est l'intervalle-poset contenant un unique sommet et $I_L, I_{R_1}, I_{R_2}, \dots, I_{R_m}$ sont des $m$-intervalles-posets. Récursivement, la définition se lit
\begin{equation}
\BBm(I_L, I_{R_1}, \dots, I_{R_m}) = I_L \pleft \BR(I_{R_1}, \dots, I_{R_m})
\end{equation}
avec 
\begin{align}
\BR(I) &= u \pright I, \\
\BR(I_1, \dots, I_k) &= u \prightx  \left( \BR(I_2, \dots, I_k)  \pleft I_1 \right).
\end{align}
 Alors la somme obtenue est une somme de $m$-intervalles-posets. On appelle cette opération la $m$-composition.
\end{Proposition}

\begin{proof}
D'abord remarquons qu'on a composé avec l'intervalle-poset $u$ exactement $m$ fois et donc on a rajouté $m$ sommets en plus des sommets de $I_L, I_{R_1}, I_{R_2}, \dots, I_{R_m}$ : la taille des intervalles obtenus est bien un multiple de $m$. 

La première opération que l'on effectue est $\BR(I_{R_m}) = u \pright I_{R_m}$. C'est une somme d'intervalles-posets de taille $1 + |I_{R_m}|$, les labels de $I_{R_m}$ ont été décalés de 1. Le calcul suivant est
\begin{equation}
\BR(I_{R_{m-1}}, I_{R_m}) =  u \prightx \left( \BR(I_{R_m}) \pleft I_{R_{m-1}} \right).
\end{equation}

Le calcul $\BR(I_{R_m}) \pleft I_{R_{m-1}}$ revient à rattacher $I_{R_{m-1}}$ aux intervalles-posets de $u \pright I_{R_m}$ sans ajouter aucune relation décroissante. Les étiquettes de $I_{R_{m-1}}$ sont décalées de  $1 + |I_{R_m}|$. Lorsqu'on effectue $ u \prightx \left( \BR(I_{R_m}) \pleft I_{R_{m-1}} \right)$, on obtient alors une somme d'intervalles-posets qui ont tous la relation $2 \trprec 1$. Les étiquettes de $I_{R_m}$ ont été décalées de 2 et celles de $I_{R_{m-1}}$ de $2 + |I_{R_m}|$.

En réitérant cette opération, on obtient que $\BR(I_{R_1}, \dots, I_{R_m})$ est une somme d'intervalles-posets possédant les relations $m \trprec m-1 \trprec \dots \trprec 1$ avec les labels de $I_{R_m}$ décalés de $m$, ceux de $I_{R_{m-1}}$ décalés de $m + |I_{R_m}|$, etc, jusqu'à ceux de $I_{R_1}$ décalés de $m + |I_{R_2}| + \dots + |I_{R_m}|$. En particulier, $\BR(I_{R_1}, \dots, I_{R_m})$ est un $m$-intervalle-poset. On a alors aussi que $I_L \pleft \BR(I_{R_1}, \dots, I_{R_m})$ est un $m$-intervalle-poset car les labels du deuxième poset sont décalés d'un multiple de $m$ car $I_L$ est un $m$-intervalle-poset.
\end{proof}

On donne en exemple un calcul détaillé pour $m=2$.


\def \fpath{includes/figures/interval-posets/}
\def \sscale{0.8}

\begin{align}
\BBk{2} \left( 
\blue{\scalebox{\sscale}{

\def \wlev{0.5}
\def \hlev{0.8}

\begin{tikzpicture}
\node(N1) at (0,0) {1};
\node(N2) at (0,-1  * \hlev) {2};
\draw (N2) -- (N1);
\end{tikzpicture}
}},
\red{\scalebox{\sscale}{

\def \wlev{0.5}
\def \hlev{0.8}

\begin{tikzpicture}
\node(N1) at (0,0) {1};
\node(N2) at (0,-1 * \hlev) {2};
\node(N3) at (\wlev,0) {3};
\node(N4) at (\wlev,-1 * \hlev) {4};
\draw (N2) -- (N1);
\draw (N4) -- (N3);
\draw (N2) -- (N3);
\end{tikzpicture}
}},
\green{\scalebox{\sscale}{}}
\right) &=
\blue{\scalebox{\sscale}{}} \pleft
\left(
u \prightx \left( 
\left( u \pright \green{\scalebox{\sscale}{}} \right) 
\pleft
\red{\scalebox{\sscale}{}} \right) 
\right)
\end{align}
\begin{align}
u \pright \green{\scalebox{\sscale}{}} &= 
\scalebox{\sscale}{

\def \wlev{0.5}
\def \hlev{0.8}

\begin{tikzpicture}

\node(N1) at (0,\hlev) {1};
\green{
\node(N2) at (\wlev,0) {2};
\node(N3) at (\wlev,-1 * \hlev) {3};
\draw (N3) -- (N2);
}
\end{tikzpicture}
}
+ \scalebox{\sscale}{

\def \wlev{0.5}
\def \hlev{0.8}

\begin{tikzpicture}

\node(N1) at (0,\hlev) {1};
\green{
\node(N2) at (\wlev,0) {2};
\node(N3) at (\wlev,-1 * \hlev) {3};
\draw (N3) -- (N2);
}
\draw(N2) -- (N1);
\end{tikzpicture}
} \\
\left( u \pright \green{\scalebox{\sscale}{}} \right)
\pleft
\red{\scalebox{\sscale}{}} &=
 \scalebox{\sscale}{\input{\fpath I7-ex1}}
+  \scalebox{\sscale}{\input{\fpath I7-ex2}}
\end{align}
\begin{align}
\label{eq:mtamari:exemple-prightx}
u \prightx \left( 
\left( u \pright \green{\scalebox{\sscale}{}} \right) 
\pleft
\red{\scalebox{\sscale}{}} \right) &=
\scalebox{\sscale}{\input{\fpath I8-ex5}}
+ \scalebox{\sscale}{\input{\fpath I8-ex6}}
+ \scalebox{\sscale}{\input{\fpath I8-ex7}}
+ \scalebox{\sscale}{\input{\fpath I8-ex8}} \\
&+ \scalebox{\sscale}{\input{\fpath I8-ex9}}
+ \scalebox{\sscale}{\input{\fpath I8-ex10}}
+ \scalebox{\sscale}{\input{\fpath I8-ex11}}
\end{align}
\begin{align}
\label{eq:mtamari:exemple-mcomp}
\BBk{2} \left( 
\blue{\scalebox{\sscale}{}},
\red{\scalebox{\sscale}{}},
\green{\scalebox{\sscale}{}}
\right) &=
\scalebox{\sscale}{\input{\fpath I10-ex1}}
+ \scalebox{\sscale}{\input{\fpath I10-ex2}}
+ \scalebox{\sscale}{\input{\fpath I10-ex3}}
+ \scalebox{\sscale}{\input{\fpath I10-ex4}} \\
& + \scalebox{\sscale}{\input{\fpath I10-ex5}}
+ \scalebox{\sscale}{\input{\fpath I10-ex6}}
+ \scalebox{\sscale}{\input{\fpath I10-ex7}}
\end{align}

\begin{Proposition}
\label{prop:mtamari:desc-mcomp}
Soient $I_L, I_{R_1}, \dots, I_{R_m}$ des $m$-intervalles-posets, on définit $I_0$ par
\begin{enumerate}[label=(\roman{*}), ref=(\roman{*})]
\item $I_0$ est une extension de la concaténation décalée dans cet ordre de $I_L$, $r$, $I_{R_m}, I_{R_{m-1}}, \dots, I_{R_1}$ où $r$ est le poset $m \trprec m-1 \trprec \dots \trprec 1$.
\label{def:matamari:mcomp:cond:I1}
\item Soit $k = |I_L| + 1$, on a $i \trprec k$ pour tout $i \in I_L$
\label{def:mtamari:mcomp:cond:increasingIL}
\item Pour tout $j$, $1 \leq j <m$ si $I_{R_j}$ n'est pas vide, soit $a_j$ le label minimal de $I_{R_j}$. On a $i \trprec a_j$ pour tout $i$ tel que $a_j > i > k+j$
\label{def:mtamari:mcomp:cond:increasingIR}
\item $I_0$ ne possède pas d'autres relations que celles décrites ci-dessus.
\end{enumerate}
 Alors $\BBm(I_L, I_{R_1}, \dots, I_{R_m})$ est la somme des $m$-intervalles-posets $I$ de taille $m + |I_L| + |I_{R_1}| + \dots + |I_{R_m}|$ qui sont des extensions de $I_0$ où l'on a rajouté uniquement des relations décroissantes et tels que $I_L, I_{R_1}, \dots, I_{R_m}$ soient toujours des sous-posets de $I$ (on ne rajoute pas de relations à l'intérieur des posets qu'on compose).
\end{Proposition}

\begin{proof}
La construction de $\BBm(I_L, I_{R_1}, \dots, I_{R_m})$ suit la structure d'un arbre $m$-binaire. Soient $T_L, T_{R_1}, \dots, T_{R_m}$ et $T_L', T_{R_1}', \dots, T_{R_m}'$ les arbres $m$-binaires respectivement minimaux et maximaux des intervalles $I_L,\allowbreak I_{R_1}, \dots, I_{R_m}$. Et soit $T$ l'arbre minimal de $I_0$ et $T'$ son arbre maximal. Les relations croissantes de $I_0$ font de $T'$ l'arbre $m$-binaire formé par $T_L', T_{R_1}', \dots, T_{R_m}'$ comme dans la figure \ref{fig:mtamari:mbinary}. C'est l'arbre maximal commun à tous les intervalles-posets obtenus par $\BBm(I_L, I_{R_1}, \dots, I_{R_m})$. En effet, les relations croissantes sont les mêmes pour tous les intervalles : $\pleft$ correspond à une greffe sur la gauche et $\pright$ ainsi que $\prightx$ à une greffe sur la droite. Plus précisément, $I_0$ est l'intervalle de $\BBm(I_L, I_{R_1}, \dots, I_{R_m})$ avec le minimum de relations décroissantes. En effet, en ce qui concerne les relations décroissantes, $I_0$ est par définition la concaténation de $I_L$, $r$, $I_{R_m}, I_{R_{m-1}}, \dots, I_{R_1}$ où $r$ est le poset $m \trprec m-1 \trprec \dots \trprec 1$ ce qui est exactement ce que l'on obtient à partir de $\BBm(I_L, I_{R_1}, \dots, I_{R_m})$.

\`A présent, les intervalles vérifiant la proposition \ref{prop:mtamari:desc-mcomp} sont toutes les façons d'ajouter des relations décroissants à $I_0$ vers les sommets $k+m-1 \trprec k+m-2 \trprec \dots \trprec k$. En effet, les relations croissantes rendent impossible l'ajout de relations entre les intervalles $I_L,\allowbreak I_{R_1}, \dots, I_{R_m}$. Les définitions de $\pright$ et $\prightx$ donnent par définition toutes les façon d'ajouter ses relations. 
\end{proof}

\subsection{\'Enumération des intervalles}
\label{sub-sec:mtamari:intervalles:denombrement}

\begin{Proposition}
\label{prop:mtamari:equiv-mcomposition}
Soient $I_L,I_{R_1}, \dots, I_{R_m}$ des $m$-intervalles-posets. Alors
\begin{equation}
\PIm(\BBm(I_L,I_{R_1}, \dots, I_{R_m})) = \OBm(\PIm(I_L),\PIm(I_{R_1}), \dots, \PIm(I_{R_m}))
\end{equation}
\end{Proposition}

\begin{proof}
Il suffit de prouver que 
\begin{equation}
\label{eq:mtamari:equiv-prightx}
\PI(I_1 \prightx I_2) = \PI(I_1) \polrightx \PI(I_2)
\end{equation}
pour $I_1$ et $I_2$ deux intervalles-posets. En effet, soit $Y= y^{\frac{1}{m}}$ et $I$ un $m$-intervalle-poset de taille $n.m$, on a
\begin{equation}
\PIm(I)(x,y) = \PI(I)(x,Y).
\end{equation}
Alors si \eqref{eq:mtamari:equiv-prightx} est vérifiée, comme on a aussi \eqref{eq:tamari_intervalles:equiv-pleft} et \eqref{eq:tamari_intervalles:equiv-pright}, on a
\begin{align}
&\PIm(\BBm(I_L, I_{R_1}, \dots, I_{R_m})) = \PI \left( \BBm(I_L, I_{R_1}, \dots, I_{R_m}) \right)(x,Y) \\
&= \PI \left( I_L \pleft~u~\prightx \left( (u~\prightx \dots ((u \pright I_{R_m}) \pleft I_{R_{m-1}}) \pleft \dots )\pleft I_{R_1} \right) \right)(x,Y) \\
&= \PI(I_L) \polleft~x.Y~\polrightx \left( (x.Y~\polrightx \dots ((x.Y \polright \PI(I_{R_m})) \polleft \PI(I_{R_{m-1}})) \polleft \dots )\polleft \PI(I_{R_1}) \right) \\
&= Y^{m-1} \PI(I_L) \polleft x.Y \polright \left( \PI(I_{R_1}) \polright \dots \polright (\PI(I_{R_{m-1}}) \polright \PI(I_{R_m}))) \dots \right) \\
&= \OBm(\PI(I_L),\PI(I_{R_1}), \dots, \PI(I_{R_m}))(x,Y) \\
&= \OBm(\PIm(I_L),\PIm(I_{R_1}), \dots, \PIm(I_{R_m})).
\end{align}

On prouve donc \eqref{eq:mtamari:equiv-prightx}. Si $\Itrees(I_2) = k$, alors
\begin{align}
\Delta \left( \frac{\PI(I_2)}{x} \right) &= \Delta(y^{\Isize(I_2)} x^{k-1}) \\
&= y^{\Isize(I_2)} (1 + x + x^2 + \dots + x^{k-1}), \\
\PI(I_1) \polrightx \PI(I_2) &= y^{\Isize(I_1) + \Isize(I_2)} x^{\Itrees(I_1)}(1 + x + x^2 + \dots + x^{k-1})
\end{align}
Par ailleurs, $I_1 \prightx I_2$ est la somme des intervalles-posets $P_i$, $1 \leq i \leq k$, où $\Isize(P_i) = \Isize(I_1) + \Isize(I_2)$ et $\Itrees(P_i) = \Itrees(I_1) + k -i$ comme on l'a vu dans la preuve de la proposition \ref{prop:tamari_intervalles:equiv-composition}. 
\end{proof}

On peut vérifier la correspondance \eqref{eq:mtamari:equiv-prightx} sur l'exemple \eqref{eq:mtamari:exemple-prightx}.
\begin{align}
xy \polrightx y^7(x^4 + x^3) &= y^8x(1+x+x^2 +x^3 + 1 + x +x^2) \\
&= y^8(2x + 2x^2 + 2x^3 + x^4)  
\end{align}
Et en calculant
\begin{align}
\OBm(xy,x^2y^2,xy) &= xy \polleft (xy \polright (x^2y^2 \polright xy)) \\
&= y^5 x  (x \polright x^2(1+x)) \\
&= y^5x^2 (1 + x +x^2 + 1 + x + x^2 + x^3) \\
&= y^5(2x^2 + 2x^3 + 2x^4 + x^5)
\end{align}
on retrouve bien le résultat de \eqref{eq:mtamari:exemple-mcomp}.

\begin{Proposition}
\label{prop:mtamari:unicity-mcomposition}
Soit $I$ un $m$-intervalle-poset. Il existe une unique liste $I_L, I_{R_1},\allowbreak \dots, I_{R_m}$ tel que $I \in \BBm(I_L, I_{R_1},\allowbreak \dots, I_{R_m})$.
\end{Proposition}

\begin{proof}
On définit $k$ de la même façon que dans la preuve de la proposition~\ref{prop:tamari_intervalles:unicity-composition} : $k$ est le label maximal tel que $i \trprec k$ pour tout $i < k$. Pour les mêmes raisons, $k$ est unique. Pour $1 \leq j < m$, soit $a_j$ l'étiquette minimale telle que $k + j +1 \trprec a_j$ et $k + j \ntrprec a_j$. S'il n'existe pas de tel sommet, alors $a_j = \emptyset$. Enfin, on pose $a_m = k + m$ si $k + m-1 \ntrprec k+m$ sinon $a_m = \emptyset$.  

Les sommets $a_1, \dots, a_m$ vérifient exactement la condition \ref{def:mtamari:mcomp:cond:increasingIR} de la proposition~\ref{prop:mtamari:desc-mcomp}. Ils nous donnent un découpage de $I$ en sous-posets. Si $a_j = \emptyset$, on pose $I_{R_j} = \emptyset$, sinon $I_{R_j}$ est le sous-poset de $I$ dont $a_j$ est le label minimal. 

Toutes les conditions de la proposition \ref{prop:mtamari:desc-mcomp} sont vérifiées et on a $I \in \BBm(I_L,\allowbreak I_{R_1}, \dots, I_{R_m})$. Par ailleurs, les sommets $a_1, \dots, a_m$ sont les seuls à vérifier la condition \ref{def:mtamari:mcomp:cond:increasingIR} de la proposition \ref{prop:mtamari:desc-mcomp} sans ajouter de relations croissantes au poset $I_0$ et forment donc le seul découpage possible.
\end{proof}

\begin{proof}[Démonstration du théorème \ref{thm:mtamari:equation-fonct}]
La preuve est immédiate par les propositions \ref{prop:mtamari:equiv-mcomposition} et \ref{prop:mtamari:unicity-mcomposition} par le même raisonnement que pour le cas $m=1$ donné par le théorème \ref{thm:tamari_intervalles:equation-fonct}.
\end{proof}

\subsection{Comptage des éléments inférieurs à un arbre}
\label{sub-sec:mtamari:intervalles:smaller}

\begin{Proposition}
\label{prop:mtamari:smaller_trees}
Soit $T$ un arbre $m$-binaire et $S_T := \sum_{T' \leq T} P_{[T',T]}$ sommé sur les arbres $m$-binaires $T' \leq T$. C'est la somme des $m$-intervalles-posets dont $T$ est l'arbre supérieur. Alors, si $T$ est composé des arbres $m$-binaires $T_L, T_{R_1}, \dots, T_{R_m}$ on a $S_T = \BBm(S_{T_L}, S_{T_{R_1}}, \dots, S_{T_{R_m}})$.
\end{Proposition}

\begin{proof}
Soit $I_0$ l'intervalle $[T_0, T]$ où $T_0$ est le peigne-$(n,m)$, l'arbre $m$-binaire minimal. Les relations croissantes de $I_0$ sont celles de $T$ et les relations décroissantes sont \eqref{eq:mtamari:mbinaire-cond}. On effectue sur $I$ le découpage en sous-posets suivant : $I_L$ est le sous-poset sur les labels $1, \dots, |T_L|$, $I_{R_m}$ est le sous-poset sur les labels $|T_L| + m +1, \dots, |T_L| + m + |T_{R_m}|$, $I_{R_{m-1}}$ est le sous-poset sur les labels $|T_L| + m + |T_{R_m}| + 1, \dots, |T_L| + m + |T_{R_m}| + |T_{R_{m-1}}|$, etc. Alors, au vu de la structure de $T$, les $m$-intervalles-posets $I_L, I_{R_1}, \dots, I_{R_m}$ sont respectivement les intervalles initiaux de $m$-Tamari des arbres $T_L, T_{R_1}, \dots, T_{R_m}$. 

Soit $P$ un intervalle de la somme $S_T$, c'est une extension de $I_0$ où l'on a rajouté uniquement des relations décroissantes. Si on effectue le même découpage sur $P$ que sur $I_0$, alors $P_L, P_{R_1}, \dots, P_{R_m}$ sont des extensions de respectivement $I_L, I_{R_1}, \dots, I_{R_m}$ et donc appartiennent aux sommes respectives $S_{T_L}, S_{T_{R_1}}, \dots, S_{T_{R_m}}$. Enfin, comme les relations croissantes de $P$ sont celles de $T$, la structure de $T$ fait que $P \in \BBm(P_L, P_{R_1}, \dots, P_{R_m})$ par la proposition \ref{prop:mtamari:desc-mcomp}. 

Par ailleurs, si $P_L, P_{R_1}, \dots, P_{R_m}$ sont des éléments de respectivement $S_L, S_{R_1},\allowbreak \dots, S_{R_m}$ alors les relations croissantes des éléments de $\BBm(P_L, P_{R_1}, \dots, P_{R_m})$ sont bien celles de $P$ ce qui en fait des éléments de $S_T$.
\end{proof}

\begin{proof}[Démonstration du théorème \ref{thm:matamari:smaller-trees}]
Comme pour le théorème \ref{thm:tamari_intervalles:smaller-trees}, on souhaite prouver que 
\begin{equation}
\OBTm{T} = \PIm(S_T).
\end{equation}
Le résultat est donné par récurrence sur $n$ par les propositions \ref{prop:mtamari:equiv-mcomposition} et \ref{prop:mtamari:smaller_trees}
\begin{align}
\OBTm{T} &= \OBm(\OBTm{T_L}, \OBTm{T_{R_1}}, \dots, \OBTm{T_{R_m}}) \\
&= \OBm(\PIm(S_{T_L}), \PIm(S_{T_{R_1}}), \dots, \PIm(S_{T_{R_m}})) \\
&= \PIm\left( \BBm( S_{T_L}, S_{T_{R_1}}, \dots, S_{T_{R_m}} ) \right) \\
&= \PIm(S_T).
\end{align}
\end{proof}

\section{Structures algébriques $"m"$}
\label{sec:mtamari:alg}

Dans le chapitre \ref{chap:tamari_prelim}, nous avons expliqué comment le treillis de Tamari peut être construit comme un quotient de l'ordre faible sur les permutations. La relation s'exprime en termes algébriques : l'algèbre de Hopf des arbres binaires $\PBT$ est une sous-algèbre de l'algèbre de Hopf sur les permutations $\FQSym$. Il est possible de généraliser ces structures au cas $m$, c'est ce que nous présentons dans cette section. Ces résultats sont les premiers issus d'un travail commun avec Jean-Christophe Novelli, Jean-Yves Thibon et Grégory Chatel. Ils seront complétés et proposés à la publication prochainement. 

\subsection{Treillis sur les permutations bègues}
\label{sub-sec:mtamari:alg:begues}

\begin{Definition}
\label{def:mtamari:begues}
Une permutation $m$-bègue de taille $n$ est une permutation du mot $1^m 2^m \dots n^m$, c'est-à-dire un mot de taille $n \times m$ où l'on trouve $m$ fois chacune des lettres $1, \dots, n$.
\end{Definition}

Par exemple, le mot $221313$ est une permutation $2$-bègue de taille 3. Soit $u = u_1 \dots u_{n.m}$, tel que $u_i < u_{i+1}$, on définit une relation de couverture par $u \lessdot v$ où $v = u_1 \dots u_{i+1} u_i \dots u_{n.m}$. Cette relation correspond simplement à l'action d'une transposition simple sur $u$. Elle définit une structure de treillis sur les permutations bègues. En effet, on plonge l'ensemble des permutations $m$-bègues de taille $n$ dans l'ensemble des permutations de taille $n.m$ par la standardisation des mots. Le mot $1^m 2^m \dots n^m$ correspond à la permutation identité. Le treillis des permutations bègues est l'idéal inférieur de l'ordre droit engendré par le standardisé du mot $n^m (n-1)^m \dots 2^m 1^m$ comme on peut le voir dans la figure \ref{fig:mtamari:begues}.

\begin{figure}[ht]
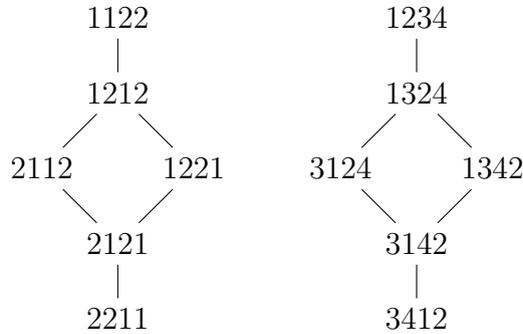

\centering
\begin{tabular}{cc}
\input{includes/figures/mbegues-2-2} &
\input{includes/figures/mbegues-perms-2-2}
\end{tabular}
\caption[Treillis des permutations $2$-bègues de taille 2]{Treillis des permutations $2$-bègues de taille 2 en tant que sous-treillis de l'ordre droit de taille 4}
\label{fig:mtamari:begues}
\end{figure} 

Comme pour les permutations classiques, on peut associer à chaque permutation bègue un arbre binaire de recherche par l'algorithme d'insertion $\ABR$ décrit au paragraphe \ref{sub-sec:tamari_prelim:tamari:ordre_faible}. Si $\mu$ est une permutation bègue, l'arbre binaire obtenu par $\ABR(\mu)$ est le même que celui obtenu par $\ABR(\std(\mu))$. L'ensemble des permutations bègues se découpe donc lui aussi en classes sylvestres et le treillis quotient sur les classes sylvestres correspond à un sous-treillis de Tamari $n \times m$. C'est l'idéal inférieur engendré par $\ABR(n^m \dots 2^m 1^m)$. 

L'ordre droit ainsi que l'ordre de Tamari sont symétriques : ils sont isomorphes à leurs ordres inverses. L'ordre inverse $\lret$ d'un ordre $\leq$ est défini par : $a~\lret~b$ si et seulement si $b \leq a$. Dans le cas de l'ordre faible, cela correspond au retournement de la permutation, on a $621345 \leq 643251$ et $152346 \leq 543126$. Pour le treillis de Tamari, l'inverse de l'ordre correspond à la symétrie gauche-droite sur les arbres binaires (échange des fils gauches et droits récursivement). On a

\begin{align}
\scalebox{0.4}{\input{includes/figures/trees/T4-3}} \leq 
\scalebox{0.4}{\input{includes/figures/trees/T4-13}}, \\
\scalebox{0.4}{\input{includes/figures/trees/T4-2}} \leq 
\scalebox{0.4}{\input{includes/figures/trees/T4-12}}.
\end{align}

Le treillis sur les permutations $m$-bègues de tailles $n$ est un idéal inférieur de l'ordre droit $n \times m$. Le quotient du treillis des permutations $m$-bègues par la relation sylvestre est donc un idéal inférieur de l'ordre de Tamari $n \times m$. Son inverse est isomorphe à l'idéal supérieur correspondant, c'est-à-dire au treillis de $m$-Tamari. En effet, le symétrisé gauche-droite de l'arbre binaire de recherche de $n^m \dots 2^m 1^m$ est le peigne-$(n,m)$ de la définition \ref{def:mtamari:peigne-nm}. Autrement dit, à partir du treillis inverse des permutations $m$-bègues, on associe à chaque permutation bègue $\mu$ le symétrique de son arbre $\ABR(\mu)$ et on obtient ainsi le treillis de $m$-Tamari sur les arbres $m$-binaires. Le treillis de $m$-tamari est donc un quotient du treillis des permutations $m$-bègues comme illustré en figure \ref{fig:mtamari:quotient-begues}.

\begin{figure}[ht]
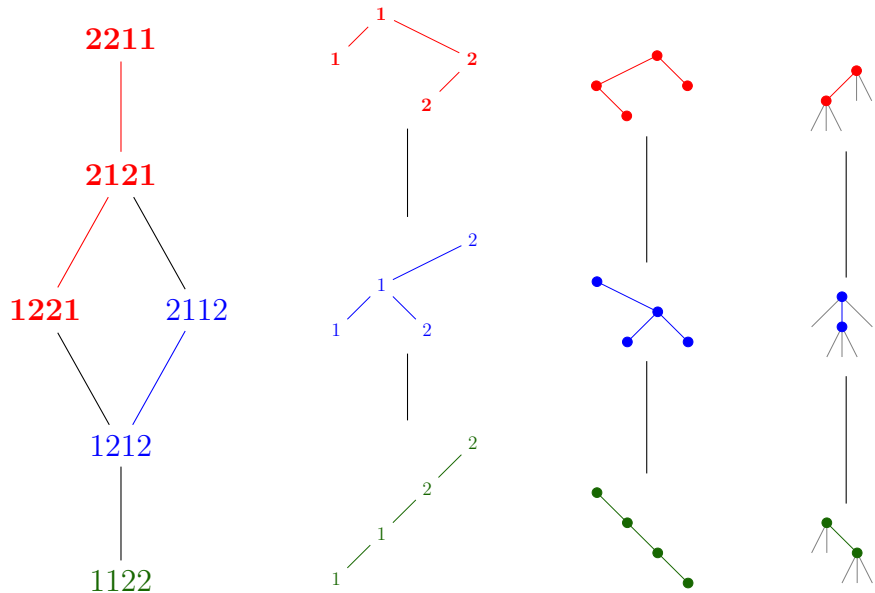

\centering
\begin{tabular}{cccc}
\input{includes/figures/mbegues-2-2-inv} &
\input{includes/figures/mTamari-2-2-inv} &
\input{includes/figures/mTamari-2-2-mbinary} &
\input{includes/figures/mTamari-2-2-ternary}
\end{tabular}
\caption[Le treillis des permutations $m$-bègues quotienté par la relation sylvestre]{Le treillis des permutations $m$-bègues quotienté par la relation sylvestre. De la gauche vers la droite : les permutations 2-bègues, leurs arbres binaires de recherches, l'arbre symétrisé et l'arbre ternaire correspondant.}
\label{fig:mtamari:quotient-begues}
\end{figure}

\subsection{Généralisation de $\FQSym$ à  $\FQSymm$}
\label{sub-sec:mtamari:alg:fqsymm}

Tout comme dans le cas $m=1$, la relation entre le treillis de $m$-Tamari et les permutations $m$-bègues s'exprime aussi en termes d'algèbres de Hopf. 

\begin{Definition}
Un mot $m$-bègue sur un alphabet $A$ est un mot tel que chaque lettre de $A$ apparaît un nombre multiple de $m$ fois.
\end{Definition}

Par exemple, les mots $ababaa$ et $abbbba$ sont des mots 2-bègues sur l'alphabet $A = \lbrace a, b \rbrace$. Un mot $m$-bègue a toujours un nombre de lettres égal à $m.n$ avec $n \in \NN$. La taille du mot $m$-bègue est donnée par $n$, le nombre de lettres divisé par $m$. Soit $\sigma$ la permutation standardisée d'un mot $m$-bègue $u$ de taille $n$. On a que $\sigma$ est inférieure à la permutation standardisée de $n^m(n-1)^m\dots2^m 1^m$. En d'autres termes, $\sigma$ est l'image par standardisation d'une permutation $m$-bègue. C'est ce qu'on appelle le $m$-standardisé du mot $u$.

\begin{Definition}
Le $m$-standardisé d'un mot $m$-bègue $u$, noté $\stdm(u)$ est l'unique permutation $m$-bègue $\sigma$ telle que
\begin{equation}
\std(u) = \std(\sigma)
\end{equation}
\end{Definition}

Par exemple, les $m$-standardisés de $ababaa$ et $abbbba$ sont respectivement $131322$ et $122331$. L'algorithme de $m$-standardisation est très similaire à celui de la standardisation. Au lieu des lettres $1, 2, \dots , n \times m$, on utilise $m$ fois la lettre $1$ puis $m$ fois la lettre $2$ etc. Comme le mot est $m$-bègue, on obtient bien une permutation $m$-bègue $\sigma$ avec $\std(\sigma) = \std(u)$. En particulier, toutes les lettres de $u$ numérotées par un même nombre dans $\sigma$ sont égales. L'ensemble des mots $m$-bègues est stable par la concaténation. On définit alors une algèbre sur les mots $m$-bègues dont le produit est la concaténation. On va définir une sous-algèbre $\FQSymm$ de l'algèbre sur les mots $m$-bègues.

\begin{Definition}
Soit $\sigma$ une permutation $m$-bègue. Alors
\begin{equation}
\BGm{\sigma} = \sum_{\stdm(u) = \sigma} u
\end{equation}
sommé sur les mots $m$-bègues.
\end{Definition}

Par exemple, sur l'alphabet $A = \lbrace a,b,c \rbrace$,
\begin{align}
\BGm{1221} &= abba + acca + bccb, \\
\BGm{11} &= aa + bb + cc.
\end{align}
Et on a
\begin{align}
\BGm{1221} \BGm{11} &= abbaaa + abbabb + abbacc + accaaa + accabb \\
\nonumber
&+ accacc + bccbaa + bccbbb + bccbcc \\
&= \BGm{122133} + \BGm{133122} + \BGm{233211}.
\end{align}
Plus généralement, si on se place sur un alphabet infini, le produit est donné par la proposition suivante.
\begin{Proposition}
\label{prop:mtamari:prod-fqsymm}
Soient $\sigma$ et $\mu$ des permutations $m$-bègues, alors
\begin{equation}
\label{eq:mtamari:prod-fqsymm}
\BGm{\sigma}\BGm{\mu} = \sum \BGm{\nu}
\end{equation}
sommé sur les permutations $m$-bègues $\nu$ qui s'écrivent $\nu = u.v$ avec $u$ et $v$ deux mots $m$-bègues tels que $\stdm(u) = \sigma$ et $\stdm(v) = \mu$.
\end{Proposition}
\begin{proof}
La preuve est très similaire à celle du produit dans $\FQSym$. 

Soit $\nu$ une permutation $m$-bègue qui s'écrit $\nu = u.v$ avec $u$ et $v$ deux mots $m$-bègues tels que $\stdm(u) = \sigma$ et $\stdm(v) = \mu$. On veut prouver que tous les mots de la somme $\BGm{\nu}$ apparaissent dans le produit $\BGm{\sigma}\BGm{\mu}$.Comme on est sur un alphabet infini, la somme $\BGm{\nu}$ n'est jamais vide. Soit $\nu'$ un mot de la somme $\BGm{\nu}$ et $u'$ son facteur gauche de taille $|u|$.  Le mot $u$ est $m$-bègue par définition et contient exactement $m$ fois chacune de ses lettres. Comme $\nu$ est le $m$-standardisé de $\nu'$, si $\nu'(i) = a$ et $\nu(i) = k$ alors pour tout $j$ tel que $\nu(j)=k$, on a $\nu'(j) = a$. C'est en particulier vrai pour les lettres de $u$ et $u'$ et donc $u'$ est un mot $m$-bègue. Par ailleurs, on a que les inversions de $u'$ sont celles de $u$ car $\stdm(\nu') = \nu$ et donc $\std(u') = \std(u) = \std(\sigma)$. On en déduit que $\stdm(u') = \sigma$. De la même façon, si $v'$ est le facteur droit de taille $|v|$ de $\nu'$, alors $v'$ est un mot $m$-bègue avec $\stdm(v') = \mu$. Donc $u'$ et $v'$ apparaissent respectivement dans $\BGm{\sigma}$ et $\BGm{\mu}$.

Par ailleurs, le $m$-standardisé d'un mot $u'.v'$ du produit $\BGm{\sigma}\BGm{\mu}$ vérifie clairement les propriétés de la somme \eqref{eq:mtamari:prod-fqsymm} et la proposition est vérifiée.
\end{proof}

Comme pour $\FQSym$, on définit un coproduit sur les éléments $\BGm{\sigma}$ en utilisant la somme ordonnée des alphabets $A \cp B$ où les lettres de $A$ et $B$ commutent. Par exemple, si $A = \lbrace a, b, c, \dots \rbrace$ et $B = \lbrace a', b', c',  \dots \rbrace$.

\begin{align}
\BGm{121233} &= ababbb + ababcc + acaccc +bcbccc + \dots \\
\nonumber
&+ abab\sred{a'a'} + acac\sred{a'a'} + bcbc\sred{a'a'} + abab\sred{b'b'} + \dots \\
\nonumber
&+ a\sred{a'}a\sred{a'a'a'} + a\sred{a'}a\sred{a'b'b'} + a\sred{b'}a\sred{b'b'b'} + b\sred{b'}b\sred{b'b'b'} + \dots \\
\nonumber
&+ \sred{a'b'a'b'b'b'} + \sred{a'b'a'b'c'c'} + \sred{a'c'a'c'c'c'} + \sred{b'c'b'c'c'c'} + \dots \\
&= \BGm{121233} \otimes 1 + \BGm{1212} \otimes \BGm{11} + \BGm{11} \otimes \BGm{1122} + 1 \otimes \BGm{121233}.
\end{align}

De façon générale, on a 
\begin{Proposition}
\label{prop:mtamari:coprod-fqsymm}
\begin{align}
\label{eq:mtamari:coprod-fqsymm}
\Delta(\BGm{\sigma}) &:= \BGm{\sigma}(A \cp B) \\
&= \sum \BGm{\mu} \otimes \BGm{\nu}
\end{align}
sommé sur les permutations $m$-bègues $\mu$ et $\nu$ telles que pour un certain $k \leq n$, $\mu$ soit le sous-mot de $\sigma$ formé des lettres $a\leq k$ et $\nu$ soit le sous-mot $m$-standardisé formé des lettres $b>k$. 
\end{Proposition}
\begin{proof}
Soit $u$ un mot de $\BGm{\sigma}$ sur $A \cp B$. On a $u = v \otimes v'$ avec $v$ un mot sur $A$ et $v'$ un mot sur $B$. Comme $u$ est un mot $m$-bègue, chacune de ses lettres apparaît un nombre multiple de $m$ fois. En particulier, c'est le cas pour les lettres de $v$ et $v'$ qui sont donc des mots $m$-bègues sur respectivement $A$ et $B$. Par ailleurs, comme toutes les lettres de $B$ sont supérieures aux lettres de $A$, on a que toutes les lettres de $v$ sont inférieures à celles de $v'$. Soit $i$ tel que $u(i)$ soit la plus grande lettre de $v$. Alors on a $\sigma(i) = k$ et $\stdm(v)$ est le sous-mot formé par les lettres inférieures ou égales à $k$ de $\sigma$. De même, les lettres correspondant à $v'$ dans $\sigma$ sont celles supérieures strictement à $k$.

Soit à présent un élément $v \otimes v'$ avec $v$ un mot $m$-bègue de $A$ et $v'$ un mot $m$-bègue de $B$ et tel qu'il existe $k$ avec $\stdm(v)$ le sous-mot de $\sigma$ formé des lettres inférieures ou égales à $k$ et $\stdm(v')$ le sous-mot $m$-standardisé formé des lettres supérieures à $k$. Supposons que les lettres $1, \dots k$ sont en positions respectives $i_{1,1}, i_{1,2}, \dots, i_{1,m}, i_{2,1} \dots, i_{k,m}$. On forme le mot $u$ tel que le sous-mot $u(i_{1,1})\dots u(i_{k,m})$ soit $v$ et le sous-mot des positions complémentaires soit $v'$. Le mot $u$ est un mot $m$-bègue de $A \cp B$. Comme les lettres de $v'$ sont supérieures aux lettres de $v$, ses inversions sont exactement celles de $\sigma$ et on a donc $\stdm(u) = \sigma$, c'est à dire $u$ apparaît dans la somme $\BGm{\sigma}$. 
\end{proof}

Soit $\FQSymm$ la sous-algèbre de l'algèbre des mots $m$-bègues engendrée par les éléments $(\BGm{\sigma})$. La proposition \ref{prop:mtamari:prod-fqsymm} nous assure que c'est bien une algèbre et nous donne une formule directe pour le produit. De la proposition \ref{prop:mtamari:coprod-fqsymm}, on déduit que $\FQSymm$ munie du produit \eqref{eq:mtamari:prod-fqsymm} et du coproduit \eqref{eq:mtamari:coprod-fqsymm} est une algèbre de Hopf. La preuve est la même que dans le cas de $\FQSym$. La co-associativité est donné par $(A \cp B) \cp C = A \cp (B \cp C)$ et la compatibilité du produit et du coproduit par le fait que le développement de $\BGm{\sigma}(A \cp B)\BGm{\mu}(A \cp B)$ est équivalent à celui de $\BGm{\sigma}(A)\BGm{\mu}(A)$.

\subsection{Dual et $\PBTm$}
\label{sub-sec:mtamari:alg:pbtm}

On note $(\BFm{\sigma})$ la base duale en tant qu'algèbre de Hopf de $(\BGm{\sigma})$ comme on l'a définie par le crochet de dualité au début du paragraphe \ref{sub-sec:tamari_prelim:hopf:dualite}. On a par définition
\begin{equation}
\BFm{\sigma} \BFm{\mu} = \sum_{\BGm{\sigma} \otimes \BGm{\mu} \in \Delta(\BGm{\nu})} \BFm{\nu}.
\end{equation}
On a que $\BGm{\sigma} \otimes \BGm{\mu} \in \Delta(\BGm{\nu})$ si $\sigma$ correspond au sous-mot de $\nu$ formé des lettres inférieures ou égales à $k=|\sigma|$ et si $\mu$ correspond au sous-mot $m$-standardisé des lettres supérieures à $k$. C'est donc toutes les façons d'intercaler les lettres de $\sigma$ et les lettres de $\mu$ auxquelles on a ajouté $|\sigma|$ en conservant l'ordre des lettres dans $\sigma$ et $\mu$. En d'autres termes
\begin{equation}
\BFm{\sigma} \BFm{\mu} = \sum_{\nu \in \sigma \cshuffle \mu} \BFm{\nu}.
\end{equation}

Pour le coproduit, on a 
\begin{equation}
\Delta(\BFm{\sigma}) = \BFm{\mu} \otimes \BFm{\nu}
\end{equation}
sommé sur les permutations $m$-bègues $\mu$ et $\nu$ telles que pour un certain $k$, $\mu = \stdm(\sigma_1 \dots \sigma_k)$ et $\nu = \stdm(\sigma_{k+1} \dots \sigma_n)$. Ce sont toutes les façons de découper la permutation $\sigma$ en deux facteurs qui sont eux-mêmes des mots $m$-bègues. Par exemple
\begin{equation}
\BFm{121233} = \BFm{121233} \otimes 1 + \BFm{1212} \otimes \BFm{11} + 1 \otimes \BFm{121233}.
\end{equation}
Quand aucune coupure non triviale ne découpe $\sigma$ en deux mots $m$-bègues, on a $\BFm{\sigma} = \BFm{\sigma} \otimes 1 + 1 \otimes \BFm{\sigma}$, par exemple
\begin{equation}
\BFm{121323} = \BFm{121323} \otimes 1 + 1 \otimes \BFm{121323}.
\end{equation} 

En terme d'algèbre, le dual de $\FQSymm$ est isomorphe à une sous-algèbre de $\FQSym$. \`A un élément $\BFm{\sigma}$, on associe $\BF_{\std(\sigma)}$ et le produit est bien le même que celui défini par \eqref{eq:tamari_prelim:fqsym_prod_f}. Cependant, ce n'est pas un isomorphisme d'algèbre de Hopf. L'ensemble des permutations $m$-bègues n'est pas stable par le coproduit de $\FQSym$. La définition du coproduit dans $\FQSymm$ correspond à celle de $\FQSym$ où l'on a "supprimé" les éléments qui ne correspondent plus à des permutations $m$-bègues. 

On définit à présent une sous-algèbre de $\FQSymm$ qui généralise la construction de $\PBT$. Soit $T$ un arbre $m$-binaire, on note $\Ti$ son symétrisé gauche-droite, on pose alors
\begin{equation}
\BPm{T} := \sum_{\ABR(\sigma) = \Ti} \BFm{\sigma},
\end{equation}
la somme des éléments $\BFm{\sigma}$ avec $\sigma$ une permutation $m$-bègue dont l'arbre binaire de recherche est $\Ti$. Par exemple, on peut lire sur la figure \ref{fig:mtamari:quotient-begues},
\begin{align}
\BPm{\scalebox{0.5}{\input{includes/figures/trees/T4-7}}} &= \BFm{2211} + \BFm{2121} + \BFm{1221}, \\
\BPm{\scalebox{0.5}{\input{includes/figures/trees/T4-12}}} &= \BFm{2112} + \BFm{1212}, \\
\BPm{\scalebox{0.5}{\input{includes/figures/trees/T4-14}}} &= \BFm{1122}.
\end{align}

Comme en tant qu'algèbre, $\FQSymm$ est une sous-algèbre de $\FQSym$, on a directement que $\PBTm$ est une sous-algèbre de $\FQSymm$ car $\PBT$ est une sous-algèbre de $\FQSym$. Le produit est donné par \eqref{eq:tamari_prelim:pbt-product}, c'est à dire sur $\PBTm$,
\begin{equation}
\BPm{T_1} \BPm{T_2} = \sum_{\Ti \in \Ti_1 \shuffle \Ti_2} \BPm{T}.
\end{equation}
On indice les éléments par la $\omega_T$, la permutation $m$-bègue maximale de la classe sylvestre de $\Ti$. Le produit $\BPm{T_1} \BPm{T_2}$ revient à effectuer le shuffle $\omega_{T_1} \cshuffle \omega_{T_2}$ et à conserver uniquement les éléments qui évitent le motif $132$. Par exemple
\begin{align}
\BPm{\scalebox{0.5}{\input{includes/figures/trees/T4-12}}} 
\BPm{\scalebox{0.5}{


\begin{tikzpicture}
\node (N0) at (0.250, 0.000){};
\node (N00) at (0.750, -0.500){};
\draw[Point] (N00) circle;
\draw (N0.center) -- (N00.center);
\draw[Point] (N0) circle;
\end{tikzpicture}}}
&= \BPm{2112}\BPm{11} \\
&= \BPm{211233} + \BPm{321123} + \BPm{332112} \\
&= \BPm{\scalebox{0.5}{\input{includes/figures/mbinary/P3-2-k}}}
+  \BPm{\scalebox{0.5}{\input{includes/figures/mbinary/P3-2-h}}}
+  \BPm{\scalebox{0.5}{\input{includes/figures/mbinary/P3-2-c}}}.
\end{align}
De par les propriétés de $\PBT$, c'est un intervalle de Tamari et donc de $m$-Tamari. Plus précisément $\BPm{T_1} \BPm{T_2}$ est l'intervalle de $m$-Tamari entre l'arbre $m$-binaire $T_1$ auquel on a greffé $T_2$ à gauche de son fils le plus à gauche et $T_2$ auquel on a greffé $T_1$ à droite de son fils le plus à droite. On peut aussi faire la construction directement sur les arbres $m+1$-aires correspondants aux arbres $m$-binaires. Dans l'exemple précédent, la somme se fait sur l'intervalle
\begin{equation}
\left[ \scalebox{0.5}{\input{includes/figures/mbinary/P3-2-c}}, 
\scalebox{0.5}{\input{includes/figures/mbinary/P3-2-k}} \right],
\end{equation}
soit en terme d'arbres ternaires, l'intervalle
\begin{equation}
\left[ \scalebox{0.5}{\input{includes/figures/mary/P3-2-c}}, 
\scalebox{0.5}{\input{includes/figures/mary/P3-2-k}} \right].
\end{equation}

Assez clairement, $\PBTm$ est aussi une sous-algèbre de Hopf de $\FQSymm$. Le coproduit est donné par
\begin{equation}
\Delta(\BPm{T}) = \sum \BPm{T'} \otimes \BPm{T''}
\end{equation}
sommé sur les couples d'arbres $m$-binaires $(T', T'')$ tels que $\omega_{T'} = \stdm(\omega_1)$ et $\omega_{T''} = \stdm(\omega_2)$ avec $\omega_1 \omega_2$ appartenant à la classe sylvestre de $\Ti$. Par exemple,

\begin{align}
\Delta(\BPm{\scalebox{0.5}{\input{includes/figures/mbinary/P3-2-k}}}) 
&= \Delta( \BFm{121233} + \BFm{211233} ) \\
&= \BFm{121233} \otimes 1 + \BFm{1212} \otimes \BFm{11} + 1 \otimes \BFm{121233} \\
\nonumber
&+ \BFm{211233} \otimes 1 + \BFm{2112} \otimes \BFm{11} + 1 \otimes \BFm{211233} \\
&= \BPm{\scalebox{0.5}{\input{includes/figures/mbinary/P3-2-k}}} \otimes 1 
+ \BPm{\scalebox{0.5}{\input{includes/figures/trees/T4-12}}} \otimes 
\BPm{\scalebox{0.5}{


\begin{tikzpicture}
\node (N0) at (0.250, 0.000){};
\node (N00) at (0.750, -0.500){};
\draw[Point] (N00) circle;
\draw (N0.center) -- (N00.center);
\draw[Point] (N0) circle;
\end{tikzpicture}}} 
+1 \otimes \BPm{\scalebox{0.5}{\input{includes/figures/mbinary/P3-2-k}}}.
\end{align}
Le fait que les éléments du coproduit soient bien des éléments de $\PBTm$ vient directement des propriétés du coproduit sur $\PBT$. Soit $\BFm{\mu} \otimes \BFm{\nu}$ un élément du coproduit de $\BPm{T}$. On sait par le coproduit de $\PBT$ que pour tout mot $\mu'$ de la classe sylvestre de $\mu$, il existe $\sigma$ tel que $\BF_{\std(\mu')} \otimes \BF_{\nu}$ soit dans le coproduit pour $\FQSym$ de $\BF_{\sigma}$ avec $\BF_{\sigma} \in \BP_{\Ti}$. Comme $T$ est un arbre $m$-binaire, $\sigma$ correspond à une permutation $m$-bègue $\sigma'$. Et comme $\mu$ est un mot $m$-bègue, $\mu'$ l'est aussi.  Le découpage de $\sigma'$  en $\mu'$ et $\nu$ est valide pour $\FQSymm$ et on a $\BFm{\mu'} \otimes \BFm{\nu} \in \Delta(\BFm{\sigma})$.

\cleardoublepage

\lhead[\oldstylenums \thepage]{Perspectives}
\rhead[Perspectives]{\oldstylenums \thepage}
\addcontentsline{toc}{chapter}{Perspectives}

\chapter*{Perspectives}

\section*{Produits d'opérateurs}

Dans le chapitre \ref{chap:polynomes_grothendieck} nous étudions un produit particulier d'opérateurs de différences divisées isobares $\pi$ et $\hpi$. Nous prouvons qu'il se développe comme une somme sur un intervalle de l'ordre de Bruhat. De façon générale, le développement de produits mélangeant les deux types d'opérateurs $\pi$ et $\hpi$ est une question ouverte. Il serait intéressant de rechercher quels autres cas particuliers se développent en un intervalle. On peut aussi poser la question dans l'autre sens : quels sont les intervalles de l'ordre de Bruhat que l'on peut obtenir à partir d'un produit d'opérateurs ? Le développement d'un produit $\pi_\sigma$ en somme d'éléments $\hpi_\mu$ est en effet une des méthodes plus efficaces pour générer des intervalles initiaux de l'ordre de Bruhat, il serait intéressant d'obtenir une méthode pour d'autres types d'intervalles.  Par ailleurs, comme on peut le lire dans \cite{LenartPostnikov}, d'autres produits de polynômes de Grothendieck se développent grâce à des manipulations dans l'ordre de Bruhat. On peut se demander dans quels cas les ensembles obtenus peuvent se décrire comme des intervalles.

\section*{Intervalles-posets}

Dans le chapitre \ref{chap:tamari_intervalles}, nous introduisons un nouvel objet combinatoire, les intervalles-posets, en bijection avec les intervalles de Tamari. Nous utilisons ces objets pour obtenir une nouvelle formule sur les treillis de Tamari et $m$-Tamari dénombrant les éléments plus petits ou égaux à un arbre donné. Nous pensons que ces objets sont un outil intéressant pour traiter toutes les questions relatives aux intervalles des treillis de Tamari. En particulier, dans \cite{triangulations}, les auteurs décrivent une bijection entre les intervalles de Tamari et les triangulations. Cette bijection doit pouvoir s'interpréter en termes d'intervalles-posets. Ainsi on pourrait espérer une preuve bijective directe du nombre d'intervalles qui soit généralisable aux treillis de $m$-Tamari.

Par ailleurs, les auteurs de \cite{mTamari} utilisent une seconde statistique sur les intervalles qu'ils appellent \emph{montée initiale} du chemin supérieur. Ils obtiennent alors une distribution symétrique du nombre de points de contact sur le chemin inférieur et de la montée initiale du chemin supérieur. Pour l'ordre de Tamari classique, cette symétrie s'explique simplement par la symétrie de l'ordre lui-même. Mais la bijection ne s'étend pas directement aux treillis de $m$-Tamari. Les statistiques de la montée initiale et du nombre de points de contact s'interprètent aussi en termes d'arbres et d'intervalles-posets. Il semble alors envisageable de rechercher la bijection explicite sur les intervalles-posets qui explique la distribution symétrique des deux statistiques.

\section*{Polynômes de Tamari}

Le calcul des polynômes de Tamari définis dans le chapitre \ref{chap:tamari_intervalles} fait intervenir deux opérations $\polleft$ et $\polright$. Nous souhaitons étudier les relations et la combinatoire liées à ces opérations. Le produit gauche $\polleft$ est le produit classique. En particulier, il est associatif et commutatif. Le produit droit $\polright$ n'est ni commutatif, ni associatif. Entre eux les opérateurs vérifient clairement
\begin{equation}
(f \polleft g) \polright h = f \polleft (g \polright h).
\end{equation}
Existe-t-il d'autres relations ? Par ailleurs, deux arbres binaires peuvent avoir le même polynôme : le produit gauche étant commutatif, on peut permuter entre eux les sous-arbres issus des branches gauches. Expérimentalement, il semble que ce soit la seule possibilité. Le nombre de polynômes serait donc égal au nombre d'arbres enracinés. On rejoint alors les résultats sur les flots obtenus par Chapoton \cite{ChapBiVar}. Le contexte pour le calcul des flots semble très différent du cadre dans lequel nous avons introduit les polynômes. Non seulement nous souhaitons montrer que ces deux familles de polynômes sont les mêmes, comme nous l'avons évoqué dans le paragraphe \ref{sub-sec:tamari_intervalles:polynomials:bivar}. Mais nous voulons comprendre le lien entre les deux théories. Une piste serait de décrire une version non commutative des polynômes pour exprimer le calcul que nous effectuons dans une algèbre de Hopf connue.

\section*{Structures "$m$"}

Comme nous l'avons vu dans le chapitre \ref{chap:mtamari}, le treillis de Tamari se généralise en treillis de $m$-Tamari. Comme on peut exprimer les treillis de $m$-Tamari comme des idéaux du treillis de Tamari classique, ils restent des quotients de certains idéaux de l'ordre faible. C'est ce que nous avons appelé le treillis des permutations $m$-bègues dans le paragraphe \ref{sec:mtamari:alg}. Il semble possible de définir un treillis quotient des permutations $m$-bègues dont le treillis de $m$-Tamari serait lui-même un quotient. Les sommets de ce treillis seraient donnés par certaines \emph{chaînes de permutations} elles-mêmes en bijection avec les arbres $m+1$-aires décroissants. Nous appelons chaîne de permutations une liste de $m$ permutations $\sigma^{(1)} \leq \sigma^{(2)} \leq \dots \leq \sigma^{(m)}$ ordonnées pour l'ordre faible droit. Nous ne considérons ensuite que les chaînes $c$ telles que si $\sigma$ et $\mu$ appartiennent à $c$ avec $\sigma \leq \mu$, alors $\sigma^{-1} \mu$ évite le motif $312$. L'ensemble de ces chaînes munies d'une relation d'ordre généralisant l'ordre faible formerait un treillis. Par une relation de réécriture raffinant la réécriture sylvestre, il semble alors possible de découper les classes sylvestres des permutations $m$-bègues tel que l'ordre induit entre ces nouvelles classes soit celui entre chaînes de permutations. En figure \ref{fig:matamari:quotient-chaines}, nous donnons une illustration de cette construction en dessinant le treillis de $m$-Tamari comme un quotient du treillis sur les chaînes de permutations.

Cette recherche est un travail en cours et nous donnerons prochainement les démonstrations formelles de ces résultats. Il reste alors à comprendre le rôle de ce nouveaux treillis et les liens avec les objets existants. En particulier, ce treillis correspond-il à une algèbre ? Dans le cas du treillis de Tamari classique, les arbres binaires décroissants sont en bijection avec les permutations. Ils permettent de construire l'ordre de Tamari comme quotient de l'ordre faible gauche. Les arbres $m+1$-aires décroissants pourraient donc être compris comme une généralisation de cette construction et permettraient peut-être d'obtenir de nouveaux résultats sur les structures algébriques "$m$".

\begin{figure}[p]
\centering
\input{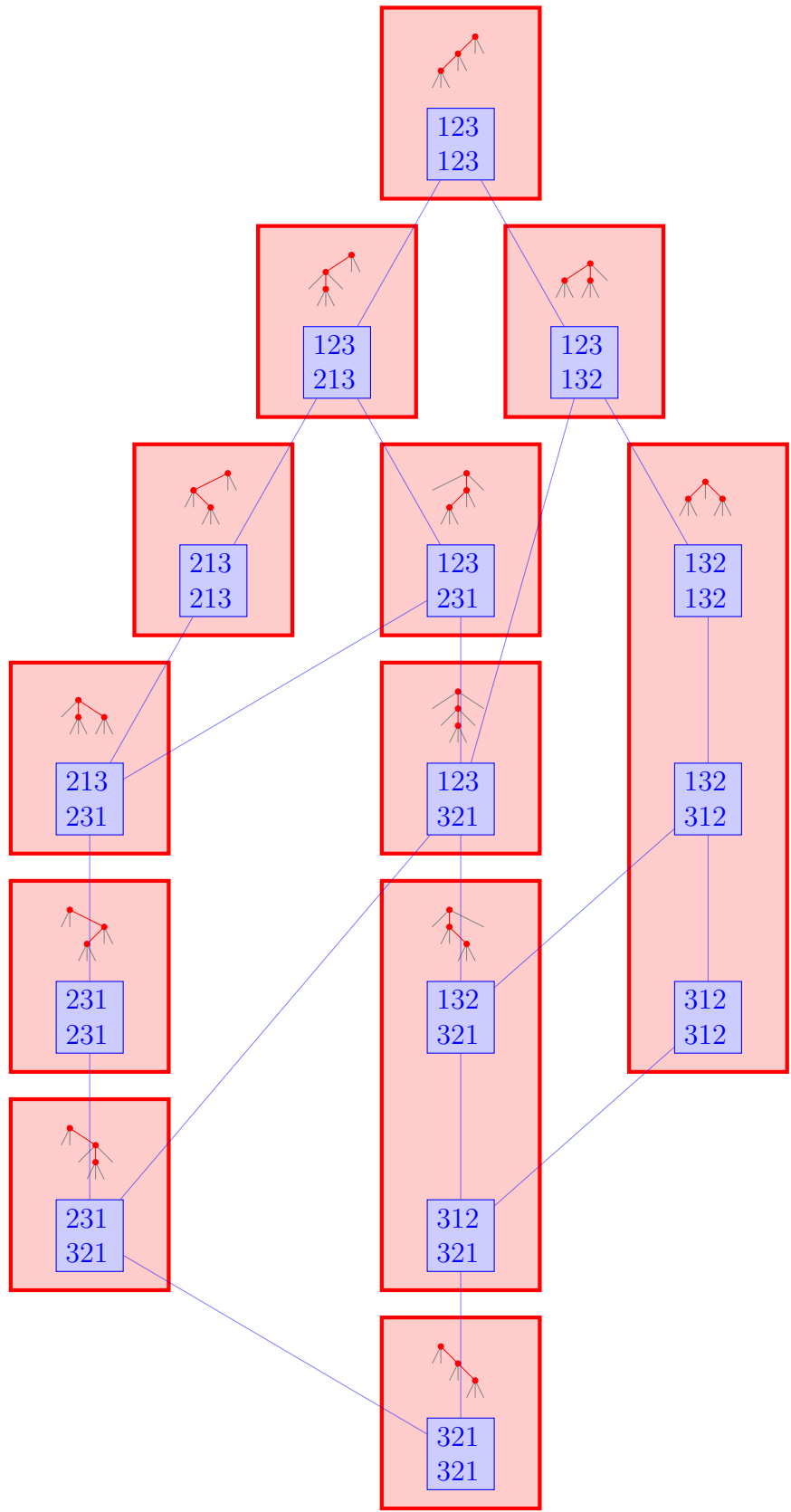}
\caption{Le treillis de $m$-Tamari comme quotient d'un treillis sur les chaînes de permutations}
\label{fig:matamari:quotient-chaines}
\end{figure}
\cleardoublepage


\backmatter
\lhead[\oldstylenums \thepage]{\rightmark}
\rhead[\leftmark]{\oldstylenums \thepage}

\lhead[\oldstylenums \thepage]{Bibliographie}
\rhead[Bibliographie]{\oldstylenums \thepage}
\bibliographystyle{alpha}
\bibliography{these}
\cleardoublepage

\lhead[\oldstylenums \thepage]{Index}
\rhead[Index]{\oldstylenums \thepage}
\printindex
\cleardoublepage

\thispagestyle{empty}
\begin{center}
    {\bf \Titre \quad --- \quad \TitreEN}
\end{center}

\end{document}